\documentclass[12pt,reqno]{amsart}
\usepackage{amsmath}
\usepackage{amsfonts}
\usepackage{amsthm}
\usepackage{amssymb} 
 \usepackage[all,cmtip]{xy}
\usepackage{color}
\usepackage{hyperref}
\usepackage[alphabetic]{amsrefs}
\usepackage{slashed}
\usepackage{textcomp}
\usepackage{latexsym}
\usepackage{scalerel}
\usepackage{adjustbox}
\usepackage[OT2,T1]{fontenc}
\DeclareSymbolFont{cyrletters}{OT2}{wncyr}{m}{n}
\DeclareFontFamily{OT1}{rsfs}{}

\RequirePackage{mathrsfs}
\RequirePackage{newpxtext}
\RequirePackage[euler-digits, euler-hat-accent]{eulervm}
\DeclareFontEncoding{LS1}{}{}
\DeclareFontSubstitution{LS1}{stix2}{m}{n}
\DeclareSymbolFont{CormaExOperators}{LS1}{stix2}{m}{n}
\DeclareSymbolFont{CormaMathbb}{LS1}{stix2bb}{m}{n}
\DeclareSymbolFontAlphabet{\mathbb}{CormaMathbb}
\DeclareMathSymbol{\CormaQED}{\mathord}{CormaExOperators}{"D1}

\usepackage{verbatim}
\usepackage{graphicx}
\usepackage[top=1in,bottom=1in,left=1in,right=1in]{geometry}
\usepackage{nomencl}
\usepackage{tikz}
\usepackage{tikz-3dplot}
\usepackage{tikz-cd}
\tikzcdset{arrow style=tikz,
    diagrams={>={Straight Barb[scale=.8]}},
}

\usepackage{stmaryrd}
\usepackage{mathtools}
\usepackage[nameinlink,capitalize]{cleveref}
\usepackage{autonum}
\usepackage[skins]{tcolorbox}

\allowdisplaybreaks[4]

\makeatletter
\def\@secnumfont{\bfseries}
\makeatother

\Crefname{equation}{}{}
\crefname{equation}{}{}
\Crefname{section}{\S}{\S\S}
\crefname{section}{\S}{\S\S}

\AtBeginDocument{
\newcommand{\symup}{\mathrm}
\newcommand{\symbf}{\mathbf}
\newcommand{\symbb}{\mathbb}

\newcommand{\symcal}{\mathcal}

\newcommand{\symfrak}{\mathfrak}

\newcommand{\bA}{{\symbf{A}}}
\newcommand{\bB}{{\symbf{B}}}
\newcommand{\bC}{{\symbf{C}}}

\newcommand{\bJ}{{\symbf{J}}}

\newcommand{\bFe}{{\symbf{e}}}

\newcommand{\bFv}{{\symbf{v}}}

\newcommand{\bbA}{{\symbb{A}}}

\newcommand{\bbC}{{\symbb{C}}}

\newcommand{\bbF}{{\symbb{F}}}

\newcommand{\bbN}{{\symbb{N}}}

\newcommand{\bbP}{{\symbb{P}}}
\newcommand{\bbQ}{{\symbb{Q}}}
\newcommand{\bbR}{{\symbb{R}}}

\newcommand{\bbT}{{\symbb{T}}}

\newcommand{\bbZ}{{\symbb{Z}}}

\newcommand{\cM}{{\symcal{M}}}

\newcommand{\cO}{{\symcal{O}}}
\newcommand{\cP}{{\symcal{P}}}

\newcommand{\FRf}{{\symfrak{f}}}

\newcommand{\FRo}{{\symfrak{o}}}

\newcommand{\FRI}{{\symfrak{I}}}

\newcommand{\FRR}{{\symfrak{R}}}
\newcommand{\FRS}{{\symfrak{S}}}
\newcommand{\FRT}{{\symfrak{T}}}

\newcommand{\x}{\times}

\newcommand{\longto}{\longrightarrow}

\newcommand{\defeq}{\coloneqq}
\let\eqdef\undefined
\newcommand{\eqdef}{\eqqcolon}
\newcommand{\ul}{\underline}

\newcommand{\notion}{\emph}
\mathchardef\mathhyphen="2D 

\providecommand{\given}{}
\newcommand\SetSymbol[1][]{%
    \nonscript\:#1\vert
    \allowbreak
    \nonscript\:
    \mathopen{}}
\DeclarePairedDelimiterX{\Set}[1]{\{}{\}}{%
    \renewcommand\given{\SetSymbol[\delimsize]}
    #1
}
\DeclarePairedDelimiterX{\Stack}[1]{[}{]}{%
    #1
}

\newcommand{\OneHalf}{{\frac{1}{2}}}

\newcommand{\dd}{\symup{d}} 

\DeclarePairedDelimiterX{\abs}[1]{\lvert}{\rvert}{%
    \ifblank{#1}{\:\cdot\:}{#1}
}
\DeclarePairedDelimiterX{\norm}[1]{\lVert}{\rVert}{%
    \ifblank{#1}{\:\cdot\:}{#1}
}
\DeclarePairedDelimiterX{\lrangle}[1]{\langle}{\rangle}{%
    \ifblank{#1}{\:\cdot\:}{#1}
}
\DeclarePairedDelimiterX{\powser}[1]{[\![}{]\!]}{%
    \ifblank{#1}{\:\cdot\:}{#1}
}
\DeclarePairedDelimiterX{\lauser}[1]{(\!(}{)\!)}{%
    \ifblank{#1}{\:\cdot\:}{#1}
}
\newcommand{\Cc}[1][]{%
    \ifblank{#1}{%
        C_{\symup{c}}
    }{
        C_{\symup{c}}^{#1}
    }
}
\newcommand{\CcInf}{\Cc[\infty]}

\DeclareMathOperator{\Hom}{Hom} 
\DeclareMathOperator{\End}{End} 
\DeclareMathOperator{\Tr}{Tr} 
\newcommand{\id}{\symup{id}}
\newcommand{\Id}{\symup{id}}

\newcommand{\Gm}[1][]{%
    \ifblank{#1}{%
        {\symbb{G}}_{\symup{m}}
    }{
        {\symbb{G}}_{\symup{m},#1}
    }
}

\DeclareMathOperator{\val}{val} 
\newcommand{\Frob}{\sigma} 

\newcommand{\TypeD}{\symup{D}}

\newcommand{\dual}{\check} 
\newcommand{\Spin}{\symup{Spin}} 
\DeclareMathOperator{\Mat}{Mat} 
\DeclarePairedDelimiterX\Pair[2]{\langle}{\rangle}{#1,#2}
\DeclarePairedDelimiterX\IPair[2]{(}{)}{#1,#2}
\DeclareMathOperator{\Diag}{diag} 

\newcommand{\Rt}{\alpha} 
\newcommand{\Wt}{\varpi} 

\newtcbox{\TODO}{enhanced,nobeforeafter,tcbox raise
    base,boxrule=0.4pt,top=-1pt,bottom=-1pt,
    right=.5mm,left=12mm,arc=1pt,boxsep=2pt,before upper={\vphantom{dlg}},
    colframe=red!75!black,coltext=red!50!black,colback=red!5!white,
    overlay={\begin{tcbclipinterior}\fill[red!75!black] (frame.south west)
    rectangle node[text=white,font=\sffamily\bfseries\tiny
    ] {TODO} ([xshift=11.5mm]frame.north west);\end{tcbclipinterior}}}

\robustify{\TODO}

\pdfstringdefDisableCommands{%
    \def\TODO#1{'#1'}%
}

\newtcolorbox{TODObox}{colback=red!5!white,colframe=red!75!black,coltext=red!50!black,fonttitle=\sffamily\bfseries,title=TODO}

\newtcbox{\FIXME}{enhanced,nobeforeafter,tcbox raise
    base,boxrule=0.4pt,top=-1pt,bottom=-1pt,
    right=.5mm,left=12mm,arc=1pt,boxsep=2pt,before upper={\vphantom{dlg}},
    colframe=orange!75!black,coltext=orange!50!black,colback=orange!5!white,
    overlay={\begin{tcbclipinterior}\fill[orange!75!black] (frame.south west)
    rectangle node[text=white,font=\sffamily\bfseries\tiny
    ] {FIXME} ([xshift=11.5mm]frame.north west);\end{tcbclipinterior}}}

\robustify{\FIXME}

\pdfstringdefDisableCommands{%
    \def\FIXME#1{'#1'}%
}

\newtcolorbox{FIXMEbox}{colback=orange!5!white,colframe=orange!75!black,coltext=orange!50!black,fonttitle=\sffamily\bfseries,title=FIXME}

\definecolor{poop-brown}{RGB}{122,89,1}
\newtcbox{\POOP}{enhanced,nobeforeafter,tcbox raise
    base,boxrule=0.4pt,top=-1pt,bottom=-1pt,
    right=.5mm,left=12mm,arc=1pt,boxsep=2pt,before upper={\vphantom{dlg}},
    colframe=poop-brown!75!black,coltext=poop-brown!50!black,colback=poop-brown!5!white,
    overlay={\begin{tcbclipinterior}\fill[poop-brown!75!black] (frame.south west)
    rectangle node[text=white,font=\sffamily\bfseries\tiny
    ] {POOP} ([xshift=11.5mm]frame.north west);\end{tcbclipinterior}}}

\robustify{\POOP}

\pdfstringdefDisableCommands{%
    \def\POOP#1{'#1'}%
}

\newtcolorbox{POOPbox}{breakable,pad at break*=1mm,colback=poop-brown!5!white,colframe=poop-brown!75!black,coltext=poop-brown,fonttitle=\sffamily\bfseries,title=POOP}

\newtcbox{\COMMENT}{enhanced,nobeforeafter,tcbox raise
    base,boxrule=0.4pt,top=-1pt,bottom=-1pt,
    right=.5mm,left=18mm,arc=1pt,boxsep=2pt,before upper={\vphantom{dlg}},
    colframe=teal!75!black,coltext=teal!50!black,colback=teal!5!white,
    overlay={\begin{tcbclipinterior}\fill[teal!75!black] (frame.south west)
    rectangle node[text=white,font=\sffamily\bfseries\tiny
    ] {COMMENT} ([xshift=17.5mm]frame.north west);\end{tcbclipinterior}}}

\robustify{\COMMENT}

\pdfstringdefDisableCommands{%
    \def\COMMENT#1{'#1'}%
}

\newtcolorbox{COMMENTbox}{colback=teal!5!white,colframe=teal!75!black,coltext=teal!50!black,fonttitle=\sffamily\bfseries,title=COMMENT}

\newtcbox{\QUESTION}{enhanced,nobeforeafter,tcbox raise
    base,boxrule=0.4pt,top=-1pt,bottom=-1pt,
    right=.5mm,left=18mm,arc=1pt,boxsep=2pt,before upper={\vphantom{dlg}},
    colframe=yellow!75!black,coltext=yellow!50!black,colback=yellow!5!white,
    overlay={\begin{tcbclipinterior}\fill[yellow!75!black] (frame.south west)
    rectangle node[text=white,font=\sffamily\bfseries\tiny
    ] {QUESTION} ([xshift=17.5mm]frame.north west);\end{tcbclipinterior}}}

\robustify{\QUESTION}

\pdfstringdefDisableCommands{%
    \def\QUESTION#1{'#1'}%
}

\newtcolorbox{QUESTIONbox}{colback=yellow!5!white,colframe=yellow!75!black,coltext=yellow!50!black,fonttitle=\sffamily\bfseries,title=QUESTION}

\newtcbox{\ANSWER}{enhanced,nobeforeafter,tcbox raise
    base,boxrule=0.4pt,top=-1pt,bottom=-1pt,
    right=.5mm,left=15mm,arc=1pt,boxsep=2pt,before upper={\vphantom{dlg}},
    colframe=green!75!black,coltext=green!50!black,colback=green!5!white,
    overlay={\begin{tcbclipinterior}\fill[green!75!black] (frame.south west)
    rectangle node[text=white,font=\sffamily\bfseries\tiny
    ] {ANSWER} ([xshift=14.5mm]frame.north west);\end{tcbclipinterior}}}

\robustify{\ANSWER}

\pdfstringdefDisableCommands{%
    \def\ANSWER#1{'#1'}%
}

\newtcolorbox{ANSWERbox}{colback=green!5!white,colframe=green!75!black,coltext=green!50!black,fonttitle=\sffamily\bfseries,title=ANSWER}

    \newcommand{\Av}{\symup{Av}}
    \DeclarePairedDelimiterX\Contra[1]{\langle}{\rangle}{#1}

    \newcommand{\EN}[1]{\mathbb{#1}}
    \newcommand{\IEN}[1]{\bar{\EN{#1}}}
    \newcommand{\FRTV}{\FRT}
     \newcommand{\FRTE}{\FRT^{\mathrm{or}}}
    \newcommand{\POne}{\mathbb{P}^1}
    
    \newcommand{\bx}{\mathbf{x}}
    \newcommand{\blank}{\,\cdot\,}
    \newcommand{\LGr}{\mathsf{LGr}}
    \newcommand{\Clif}{\mathrm{Cl}}
    
    \newcommand{\sgn}{\mathrm{sgn}}
    \newcommand{\pFq}[2]{\vphantom{F}_#1F_#2}
    \renewcommand{\TypeD}{\mathsf{D}}
    \newcommand{\SU}{\mathrm{SU}}
    \DeclareMathOperator*{\bigboxtimes}{{\scalerel*{\boxtimes}{\bigotimes}}}
    \DeclareMathOperator*{\displaybigboxtimes}{\raisebox{-.125em}{\adjustbox{margin=.125em}{\scalebox{1.25}{\scalerel*{\boxtimes}{\bigotimes}}}}}
}
\DeclareFontEncoding{OT2}{}{}
  
\newcommand{\II}{\mathbf{I}}

\newcommand{\Spinor}{P_\pm}
\newcommand{\RO}{\mathcal{R}}
\newcommand{\SO}{\mathrm{SO}}
\newcommand{\Veven}{S^{\textrm{ev}}}
\newcommand{\Vodd}{S^{\textrm{odd}}}
\newcommand{\Z}{\mathbb{Z}}

\newcommand{\FT}{\mathcal{F}}
\newcommand{\OddFunctions}{\tilde{\TypeD}_6}

\newcommand{\SL}{\mathrm{SL}}

\newcommand{\PrincipalSeries}{\mathcal{X}}

\newcommand{\Schwartz}{\mathcal{S}}

\newcommand{\Pisymbol}{\boldsymbol{\{}\Pi\boldsymbol{\}}}
\newcommand{\PisymbolPrime}{\boldsymbol{\{}\Pi'\boldsymbol{\}}}
\newcommand{\chisymbol}{\boldsymbol{\{}\chi\boldsymbol{\}}}
\newcommand{\chisymbolBig} {\begin{Bmatrix} 
        \chi_{\V{12}} & \chi_{\V{13}} & \chi_{\V{14}} \\ 
        \chi_{\V{34}} & \chi_{\V{24}} & \chi_{\V{23}} 
    \end{Bmatrix}}

\newcommand{\ad}{\mathrm{ad}}

\newcommand{\taunorm}{\tau^{\mathrm{norm}}}

\newcommand{\C}{\mathbb{C}}

\newcommand{\R}{\mathbb{R}}

\newcommand{\PGL}{\mathrm{PGL}}

 \DeclareFontShape{OT1}{rsfs}{n}{it}{<-> rsfs10}{}
\DeclareMathAlphabet{\mathscr}{OT1}{rsfs}{n}{it}
\newcommand{\GL}{\mathrm{GL}}

\newcommand{\tetraV}{\mathbf{V}} 
\newcommand{\tetraE}{\mathbf{E}} 
\newcommand{\tetraO}{\mathbf{O}}
\newcommand{\V}[1]{\mathbf{#1}}

\DeclareMathOperator{\VI}{I^{\tetraV}}
\DeclareMathOperator{\EI}{I^{\tetraE}}

\newcommand{\CT}{\mathrm{T}^*}

\numberwithin{equation}{subsection}
\newtheorem{theorem}[subsubsection]{Theorem}
\newtheorem{proposition}[subsubsection]{Proposition}
\newtheorem{lemma}[subsubsection]{Lemma}

\theoremstyle{definition}
\newtheorem{definition}[subsubsection]{Definition}
\newtheorem{warning}[subsubsection]{Warning}

\newtheorem{question}[subsubsection]{Question}

\theoremstyle{remark}
\newtheorem{remark}[subsubsection]{Remark}
\newtheorem*{claim}{Claim}

\begin{document}  

\title{The Tetrahedral (or $6j$) Symbol}
\author{Akshay Venkatesh and X. Griffin Wang}

\begin{abstract}

    We will attach a scalar invariant to a tetrahedron whose edges are labelled by irreducible
    representations of a ternary orthogonal group $\mathrm{SO}_3$ over a local field.
    This generalizes the \(6j\) symbol whose
    theory was developed by Racah, Wigner, and Regge.
    
    We give several formulas for this invariant, including in terms of hypergeometric-type integrals
    and functions, and show that it admits a symmetry by the the \(23040\)-element   Weyl group of \(\mathrm{Spin}_{12}\).
    We then interpret  these results
    in terms of relative Langlands duality,  where the dual story comes from
      the action of $\mathrm{Spin}_{12}$ on a $16$-dimensional cone of spinors.

\end{abstract}

\maketitle
\setcounter{tocdepth}{1}
\tableofcontents

\newpage 
\section{Introduction}\label{sec:introduction}

The \(6j\) symbol 
$\begin{Bmatrix} 
    j_1 & j_2 &  j_3\\ 
    j_4 &  j_5 &  j_6 
\end{Bmatrix}$
is an invariant attached to a tetrahedron $\mathcal{T}$ with integral side
lengths $j_i$, where the upper and lower rows index opposite edges (as in \Cref{fig:tetrahedron}). 
The tetrahedron does not have to be Euclidean, but for this introduction we will
assume so for the ease of mind (see \Cref{rmk:Euclideanness_unnecessary}).
In truth,  each  $j$  is  indexing  the
 irreducible $(2j+1)$-dimensional representation $V_{j}$ of the Euclidean rotation group
$\SO_3$.\footnote{In fact, the $6j $
symbol is defined for half-integral $j$, which correspond to representations
of the double cover $\SU_2$ of $\SO_3$.  For the purposes of this paper, however,
restricting to integral $j$ gives a simpler presentation of the theory. This is inessential; see
\Cref {sub:symbol_for_SU2} and \Cref{remarkSU} for further discussion.
}
The origin of the $6j$ symbol lies in the work of Racah and Wigner on addition
theorems for (quantum) angular momentum.  The theory was further developed by
Regge, who discovered the extraordinary ``Regge symmetries'' \cite{Re59} of $6j$
symbols, and  from these deduced new geometric symmetries of tetrahedra
\cites{PoRe69,Ro99}.
Since then,  the $6j$ symbol has  resurfaced  
in special function theory, topology and number theory
(see \Cref{sec:Further} for some references). 
We will review the classical definition, and then proceed to discuss in more detail
what we do in this paper. 

\begin{figure}
    \tdplotsetmaincoords{115}{100}
    \begin{tikzpicture}[tdplot_main_coords, scale=2.5]

        \coordinate (A) at (1, 1, 1);
        \coordinate (B) at (1, -1, -1);
        \coordinate (C) at (-1, 1, -1);
        \coordinate (D) at (-1, -1, 1);

        \draw[thick] (A) -- (B);
        \draw[thick] (A) -- (C);
        \draw[thick] (B) -- (D);
        \draw[thick, dashed] (C) -- (D);

        \draw[thick] (A) -- (D);
        \draw[thick] (B) -- (C);

        \path (A) -- (B) node[midway, below right] {$j_1$};

        \path (A) -- (C) node[midway, right] {$j_2$};

        \path (A) -- (D) node[midway, above left] {$j_3$};

        \path (C) -- (D) node[midway, above right] {$j_4$};

        \path (B) -- (D) node[midway, above left] {$j_5$};

        \path (B) -- (C) node[midway, below] {$j_6$};

    \end{tikzpicture}
    \caption{A tetrahedron with labeled edges}
    \label[figure]{fig:tetrahedron}
\end{figure}
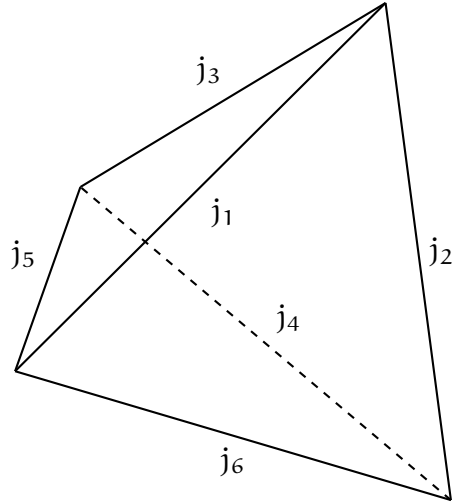
\subsection{The classical definition.}
\label{Intro}

\begin{quote}
   {\em \ldots I hardly ever take up Dr. Frankland's exceedingly valuable
   ``Notes for Chemical Students,'' which are drawn up exclusively on the basis
of Kekul{\'e}'s exquisite conception of valence, without deriving suggestions for
new researches in the theory of algebraical forms.} --- James Joseph Sylvester,
   {\em ``Chemistry and Algebra.''}
\end{quote}

These words of Sylvester relate to a graphical calculus for invariant theory. 
The definition of the 6j symbol  is animated by the same spirit;
were some strange beast, whose only language was the invariant theory of binary forms,
to be   confronted with the idea of a tetrahedron, we think 
it would surely rediscover the definition that follows.

We can realize the irreducible representation $V_j$ of $\SO_3$ of dimension
$2j+1$ by considering the space of homogeneous polynomials of degree $j$ in
three variables $x,y,z$, and restricting them to the sphere $x^2+y^2+z^2=1$. In
what follows, it is more convenient to consider only real polynomials. By
integrating over the sphere we get a rotation-invariant inner product on $V_j$.

Now, a classical theorem of invariant theory asserts that the triple product
\begin{align}
    V_a \otimes V_b \otimes V_c
\end{align}
admits a   nonzero $\mathrm{SO}_3$-invariant
vector if and only if   $a,b$ and $c$ satisfy the triangle inequality, i.e. form
the sides of a Euclidean triangle. When this happens, the invariant vector is
unique up to scaling, and can be explicitly described, with reference to the
model just discussed, as the dual of the functional
\begin{align}
    (v_1, v_2, v_3) \longmapsto \int_{S^2} v_1 v_2 v_3.
\end{align}

If a triangle with integer side lengths, then, indexes an invariant
vector, what can we extract from a tetrahedron $\mathcal{T}$ of integer sides?
When $j,j', j''$ correspond to
the side lengths of a face of $\mathcal{T}$, we can distinguish an invariant
vector inside $V_{j} \otimes V_{j'} \otimes V_{j''}$, which we normalize to have length
$1$; by working inside the real form, this distinguishes the invariant vector up to sign.
Tensoring these vectors together for all four faces,  we arrive at a vector
inside
\begin{align}
    V_{j_1} \otimes V_{j_1} \otimes V_{j_2} \otimes V_{j_2} \otimes \dots V_{j_6} \otimes V_{j_6}
\end{align}
and contract, using the inner product on each $V_e$. We arrive
at a real-valued invariant; this is, up to sign, 
the classical $6j$ symbol. 

\begin{remark}
    \label[remark]{rmk:Euclideanness_unnecessary}
    Observe that this definition makes sense in slightly more generality: it does
    not require that the various $j$s form the lengths of a Euclidean tetrahedron,
    only that the triangle inequalities are satisfied for all faces.  
    This implies that the $j$s are the side lengths of a tetrahedron that is either Euclidean, flat, 
    or Minkowskian.    Minkowskian means that we can find a tetrahedron  in $\mathbb{R}^3$ such
    that the metric   $(dx)^2+(dy)^2-(dz)^2$ is positive definite along each face, 
    and the side lengths  for this same metric
    are equal to the $j$s.  
\end{remark}

\subsection{What we do in this paper, and why.}
Our goal is to set up and study the definitions above in a broader context.

\subsubsection{What?}
First of all, 
we will allow $\SO_3$ to mean the automorphisms
of {\em any} nondegenerate ternary quadratic form;
thus, for example, we allow also $x^2+y^2-z^2$,
which results in a noncompact group $\mathrm{SO}_{2,1}(\mathbb{R})$.\footnote{
This group may be more familiar
in its isomorphic realization the group $ \mathrm{PGL}_2(\mathbb{R})$
of projective linear transformations of the plane.}
The result of this substitution is that the $j$-parameters now can vary continuously;
informally, the input  may be either a Euclidean or a Lorentzian tetrahedron. 

Secondly, and perhaps more disorienting to the 
reader familiar with the classical definition, 
  we will allow the real numbers to be replaced by any local field $F$,
  for example, the complex numbers or the $p$-adic numbers. 
Informally speaking, this further enlarges the domain of permissible $j$-parameters.

 We propose to rename the symbol, in this context
at least, the ``tetrahedral symbol,'' which seems more evocative than the
traditional name,   and, at least, does no further  injustice to the pioneers.

\subsubsection{Why?}\label{Why} It seems
to us that, in this more general context, the theory both becomes richer in its own right, and 
also acquires interesting new connections to other areas of mathematics.   Thus, for example:

\begin{itemize}
\item[(a)]  As alluded to above, the $6j$ symbol possesses an unexpectedly
large symmetry group --- a group of order $144$, isomorphic to $\FRS_4 \times
\FRS_3$ (where \(\FRS_n\) is the symmetric group of degree \(n\)).
In the general context, the tetrahedral symbol will  acquire  symmetry
under a much larger group of order $23040$, which is isomorphic to the   the Weyl group
of $\mathrm{Spin}_{12}$.  

\item[(b)] The tetrahedral symbol in the general context possesses a variety
of  novel and beautiful integral representations, some of which seem to us
much simpler
  than any integral representation for the original $6j$ symbol.

\item[(c)]  When $F$ is a local field,  and the representations in question are unramified,
it becomes possible to evaluate {\em explicitly} the tetrahedral symbol in terms
of the geometry of a certain remarkable spinor cone on which the group $\mathrm{Spin}_{12}$ acts.
In this way, we see the spin group itself, rather than only the Weyl group
that was manifested in (a).

\item[(d)] The setting in which we develop the theory --- namely,
the representation of real and $p$-adic groups --- is also the setting
of the theory of automorphic forms and the Langlands program. 
As we will see, the tetrahedral symbol 
has {\em already} played an interesting unacknowledged role in the former theory
--- as the 
kernel underlying certain ``spectral reciprocity'' formulas; and we will
offer various proposals concerning its broader role in the Langlands program.
\end{itemize}

\subsubsection{Connection to existing work}
The idea of generalizing the $6j$ symbol to the cases of $F=\bbR$ or $F=\bbC$ 
is   not a new one. Indeed, 
the analogues of $6j$ symbols have been studied for the groups $\mathrm{SL}_2(\mathbb{R})$
and $\mathrm{SL}_2(\mathbb{C})$ by several authors, both in the context of special function theory, and
of mathematical physics. Among other things, these works define versions of the $6j$ symbols and give a number
of formulas of hypergeometric type, closely related to our
\Cref{sec:hypergeometric_formulas_for_the_tetrahedral_symbol} in the case $F=\bbR$ or $\bbC$.

 We discuss these papers   in a little more detail in \Cref{sub:symbol_for_SU2}.
 Broadly speaking, the main point of overlap is point (b) from \Cref{Why}. 
However, our approach to the theory also has a somewhat different emphasis,
in that we have sought to give a presentation separating 
abstract aspects from computational aspects.  Thus our definition of the symbol 
is somewhat different to prior work; it uses
no explicit formulas and is
manifestly invariant by tetrahedral symmetries. This simplicity comes at a price
--- more effort is needed to get to explicit formulas.

\subsection{A summary of the paper} 

To try to
bring out the beauty of the subject matter, we have to some extent separated
statements from proofs; in the first part of the paper, the reader will find
statements of the theorems, but some proofs are only sketched, with details
given in the second part.  
We summarize briefly the contents of this first part. 

\begin{itemize}
    \item In \Cref{sec:representation_theoretic_definitions}  we give a more
        precise version of the discussion above, and explain how to extend it to
        the case of general $\mathrm{SO}_3$. 
           In this general context, the $V_i$s become infinite-dimensional;
        nonetheless, there is a simple rearrangement of the definition that
        avoids analytic difficulty. 

        \item In \Cref{D6Tetrahedron} we set up various notation that connect
        tetrahedral geometry to the geometry of the root system $\TypeD_6$
        and the associated group $\mathrm{Spin}_{12}$.
   
    \item In \Cref{sec:MainTheorems} we formulate the main theorems:
    \begin{itemize}
        \item  \Cref{thm:W_D6_symmetry} proves that the tetrahedral symbol, for principal series,
        enjoys a $W(\TypeD_6)$-symmetry;
        \item \Cref{thm:main_duality_theorem} 
        evaluates the tetrahedral symbol, in the unramified case, in terms of the
        geometry of a spinor cone. 
    \end{itemize} 
    \item In   \Cref{sec:integral_formulas_for_tetra_symbol} we give several
        formulas of geometric nature for the tetrahedral symbol, in particular
        \Cref{prop:WignerIntegral} as an integral of characters,
        \Cref{prop:VertexIntegral} as an integral of spherical functions, and
        \Cref{prop:edge_integral} as an integral over moduli of six points on
        $\mathbb{P}^1$. We regard these formulas as having  intrinsic interest,
        besides their usage to prove the theorems of \Cref{sec:MainTheorems};
        the same comment goes for the next section too. 
    \item In \Cref{sec:hypergeometric_formulas_for_the_tetrahedral_symbol} we
        give hypergeometric formulas for the tetrahedral symbol, in particular
        \Cref{thm:6j=Hypergeometric}. In the case $F=\R$ we will express the
        result as a sum of $\pFq{4}{3}$ hypergeometric series evaluated at $1$.
    \item In \Cref{TetrahedralAsInterface} we explain how the study of the tetrahedral symbol,
        and in particular our theorem computing it in terms of a spinor cone,
        fits into  the story of relative Langlands duality. 
    \item In \Cref{sec:Further} we rather briefly discuss a number of
        interesting topics: the unitary integral transform defined by the
        tetrahedral symbol and its role in number theory; difference equations;
        and corresponding questions in geometric representation theory.
\end{itemize}

\subsection{Acknowledgements}
The first-named author (A.V.) would like to thank Andre Reznikov, for two
decades of inspiration and friendship, which included many conversations around
the present subject matter; in particular, it was Reznikov's encouragement that
led us to really look carefully at the definition of the $6j$ symbol.

The second-named author (X.G.W.) would like to thank Minh-Tam Trinh for
many spontaneous discussions; and his relentless pursuit of
creativity in math has always been inspirational throughout the years.

Both of us thank Danii Rudenko for interesting discussions during his visit to IAS. 
We would also like to thank Tulio Regge for inspiration.

\section{Review of harmonic analysis on a local field} \label{GammaReview}

\emph{The reader should skip this section and refer to it as needed.}

Recall that a \notion{local field} is a field that is equipped with a
multiplicative \emph{absolute value} $\abs{\blank}$, where we require
$\abs{\blank}$ to satisfy the triangle inequality and induce a locally compact
topology.\footnote{As we recall below, the topology in fact determines a
canonical absolute value; usually, one therefore thinks of the topology as part
of the datum of a local field, but not the absolute value.} Such a field is
isomorphic to either the real numbers, the complex numbers, a finite extension
of the $p$-adic numbers, or a finite extension of the field of Laurent series
over a finite field.

To facilitate the discussions throughout the paper, it is necessary to introduce
some common notations from harmonic analysis of groups defined over a local
field. 
 We mainly focus on the case where the group is the multiplicative group \(\Gm\).

\subsection{Characters, Haar measures, and absolute values}
\label{Characters And Measures}
Let \(F\) be a local field, and if \(F\) is nonarchimedean we let \(\cO\) be
its valuation ring with uniformizer \(\varpi\) and residue field \(k=\bbF_q\).
We fix, once and for all, a nontrivial additive character
\begin{align}
    \Psi=\Psi_F\colon F\longto \bbC^\x,
\end{align}
as well as an additive Haar measure \(\dd\mu\) on \(F\), such that the Fourier
transform \(\FT\) on \(F\) with respect to \(\Psi\) and \(\dd\mu\) is
involutive: in other words, \((\FT^2f)(x)=f(-x)\). In this paper, we will make
the following choices for \(\dd\mu\):
\begin{enumerate}
    \item when \(F\) is nonarchimedean, \(\cO\) has measure \(1\);
    \item when \(F=\bbR\), the unit interval \([0,1]\) has measure \(1\);
    \item when \(F=\bbC\), the unit square \([0,1]\x [0,i]\) has measure \(1\).
\end{enumerate}
We also normalize the absolute value \(\abs{\blank}\) on \(F\) 
to be the factor by which dilation scales the additive Haar measure, or more
explicitly:
\begin{enumerate}
    \item when \(F\) is nonarchimedean, \(\abs{\varpi}=q^{-1}\);
    \item when \(F=\bbR\), \(\abs{x}=\sgn(x)x\);
    \item when \(F=\bbC\), \(\abs{z}=z\bar{z}\).
\end{enumerate}
Then \(\dd\mu/\abs{\blank}\) is a multiplicative Haar measure on \(F^\x\).

Accordingly, the character \(\Psi\) is as follows:
\begin{enumerate}
    \item when \(F\) is \(\bbQ_p\), the choice of uniformizer \(\varpi\) induces
        a group isomorphism \(F/\cO\cong \mu_{p^\infty}\) (the \(p\)-power roots
        of unity in \(\bbC^\x\)), and we let \(\Psi(x)=x\bmod \cO\);
    \item when \(F\) is a finite extension of \(\bbQ_p\), we have
        \(\Psi(x)=\Psi_{\bbQ_p}(\Tr_{F/\bbQ_p}(\varpi^{-d}x))\), where the
        \(\cO\)-module \(\varpi^{-d}\cO\) is precisely the set of elements \(y\)
        such that \(\Tr_{F/\bbQ_p}(y\cO)\subset \bbZ_p\);
    \item when \(F=\bbF_q\lauser{\varpi}\), \(\Psi\) is the composition of
        projecting to the coefficient of \(\varpi^{-1}\), and an isomorphism
        \(\bbF_q\cong \mu_q\subset \bbC^\x\);
    \item when \(F=\bbR\), \(\Psi(x)=e^{-2\pi i x}\);
    \item when \(F=\bbC\), \(\Psi(z)=e^{-\pi i(z+\bar{z})}=e^{-2\pi i\Re(z)}\).
\end{enumerate}
Note that with these choices, \(\Psi\) is \emph{unramified} when \(F\) is
nonarchimedean, namely \(\Psi(x)=1\) if and only if \(x\in\cO\).

Following the conventions of number theory, we shall call a
continuous homomorphism \(F^\x\to \bbC^\x\) a \notion{quasi-character} (instead
of a character per conventions of group theory) of
\(F^\x\). A unitary quasi-character (i.e., its image lands in the unit circle)
is called a \notion{character} of \(F^\x\). 
The choice of measure on $F^{\times}$ induces also
a measure on the group of characters of $F^{\x}$
endowed with its natural locally compact topology, in such 
a way that the Fourier inversion formula holds.

We make the following notational definitions:
\begin{definition}
    Given a quasi-character \(\chi\colon F^\x\to \bbC^\x\) of \(F^\x\) and
    \(s\in\bbC\), we define its \emph{\(s\)-twist} to be the quasi-character
    \begin{align}
        \chi_s\defeq \chi\abs{\blank}^s.
    \end{align}
    To avoid ambiguity, we adopt the convention that \(\chi_s^{n}\defeq
    (\chi_{s})^{n}=(\chi^n)_{ns}\) for any integer \(n\). Namely, \emph{raising to
    a power always has lower priority than \(s\)-twisting.}
    We also define the shorthands \(\chi_+\defeq \chi_{\OneHalf}\),
    \(\chi_-\defeq \chi_{-\OneHalf}\), \(\chi_+^{-1}\defeq
    (\chi_+)^{-1}=(\chi^{-1})_-\), \(\chi_{++}\defeq (\chi_+)_+\), and so on.
\end{definition}

\subsection{\(\gamma\), \(L\), and \(\epsilon\)-factors}

\label{sub:intro_review_three_factors}
Let \(\chi\) be a quasi-character of \(F^\x\). There are three meromorphic
functions on \(\bbC\) attached to \(\chi\): the \notion{\(\gamma\)-factor},
the \notion{\(L\)-function} (or \notion{\(L\)-factor}),  and the
\notion{\(\epsilon\)-factor}, which will be used frequently throughout this
paper. The most important of the three, for us, will be the $\gamma$-factor.

These will be denoted by  $\gamma(s, \chi)$, $L(s, \chi)$ and $\epsilon(s,
\chi)$ respectively. They are all compatible with twisting, in the sense that
 \begin{align}
    \gamma(s+t,\chi)&=\gamma(s,\chi_t)=\gamma(0,\chi_{s+t}),\\
    L(s+t,\chi)&=L(s,\chi_t)=L(0,\chi_{s+t}),\\
    \epsilon(s+t,\chi)&=\epsilon(s,\chi_t)=\epsilon(0,\chi_{s+t}),
\end{align}
and we will denote their values at $s=0$
by $\gamma(\chi)$, $L(\chi)$, $\epsilon(\chi)$, \emph{when defined},
that is to say, when $s=0$ is not a pole of the meromorphic function. 
When \(\chi=\abs{\blank}^s\), we also use the shorthand \(\gamma(s)\defeq
\gamma(\abs{\blank}^s)\), and similarly for \(L(s)\), and \(\epsilon(s)\).

\subsubsection{The $\gamma$-factor as the Fourier transform of a quasi-character}
the $\gamma$-factor $\gamma(s, \chi)$
tells us what the Fourier transform of a quasi-character is. 
Its value $\gamma(\chi)=\gamma(0, \chi)$ at $s=0$ is characterized by the following equality:
\begin{align}
    \label{Tate2gamma}
    \FT(\chi^{-1}) = \gamma(\chi)  \chi_{-1}.
\end{align}
\textit{A priori}, the left hand side is a distribution; the assertion is that, when $\gamma(s, \chi)$
does not have a pole at $s=0$, the left-hand side
is represented by the function on the right-hand side.
Homogeneity arguments already imply that \(\FT(\chi^{-1})\) and \(\chi_{-1}\)
are multiples of one another, 
so the only question has to do with the scalar, and that is what $\gamma$ tells us.
\begin{remark}
    The \(\gamma\)-factor is often characterized in number theory by means of
    the following equality which is essentially a restatement of
    \Cref{Tate2gamma}:
    \begin{equation}
        \label{Tate}
        \int_F \check{\Phi}(x)  \chi^{-1}(x)\abs{x}^{1-s}  \frac{\dd x}{\abs{x}}
        = \gamma(s, \chi) \int_F \Phi(x) \chi(x) \abs{x}^{s} \frac{\dd x}{\abs{x}},
    \end{equation}
    for $\Phi$ a Schwartz-Bruhat function on $F$ (this means a Schwartz function
    for $F$ archimedean, and a locally constant function of compact support
    otherwise), and \(\check{\Phi}\defeq\FT(\Phi)\).
\end{remark}

\subsubsection{Evaluation of the $\gamma$-factor} \label{Gamma Evaluation}
 
It is not difficult to directly evaluate $\gamma(s)$.
For example, take the case $F=\mathbb{R}$;
one readily computes that, writing $\II=2\pi i$ and $\bar{\II}=-2\pi i$,
\begin{equation} \label{gammaReval}
    \gamma(s) = \frac{1}{(\II^{-s}+\bar{\II}^{-s}) \Gamma(s)}
    = (\II^{s-1}+\bar{\II}^{s-1})\Gamma(1-s).
\end{equation}

However, there is a more elegant
way to rewrite this as a ratio of two $\Gamma$-funtions that
reflects better the involutive property of Fourier transform, and also
generalizes well to local fields because \(\Gamma\)-functions are just special
cases of \(L\)-functions.

The \notion{\(L\)-function} or \notion{\(L\)-factor} attached to \(\chi\) is
defined as:
\begin{align}
    \label{eq:Ldef}
    L(s, \chi) \defeq
    \begin{cases}
        \frac{1}{1 - \chi(\varpi)q^{-s}} & \text{if $F$ is nonarchimedean and $\chi$ is unramified}, \\
        1 & \text{if $F$ is nonarchimedean and $\chi$ is ramified}, \\
        \pi^{-\frac{s+t+c}{2}}\Gamma\left(\frac{s+t+c}{2}\right)
          & \text{if $F = \mathbb{R}$ and $\chi(x) = \abs{x}^t\sgn(x)^c$ where $c \in \Set{0,1}$}, \\
        2(2\pi)^{-(s+t)}\Gamma(s+t) & \text{if $F = \mathbb{C}$ and $\chi(z) = \abs{z}^{t}$}.
    \end{cases}
\end{align}
Then one always has
\begin{align} \label{gammaab}
    \gamma(s, \chi) = \frac{L(1-s, \chi^{-1})}{L(s, \chi)} \times (ab^s)
\end{align}
for unique $a,b \in \mathbb{C}$; we write
$\epsilon(s, \chi) = a b^s$ for this $a$
and $b$.

Said differently, 
  \(\epsilon\)-factor is defined in terms of the \(L\)-factor and
the \(\gamma\)-factor by means of
\begin{align} \label{epsilon_and_gamma_and_L}
    \epsilon(s,\chi)\defeq\gamma(s,\chi)\frac{L(s,\chi)}{L(1-s,\chi^{-1})}.
\end{align}
In the nonarchimedean case, we have (assuming $\Psi$ is unramified as we have chosen),
\begin{equation} \label{epsilon unramified case}
    \epsilon(s, \chi) =  a \cdot q^{-f (s-\OneHalf)},
\end{equation}
where $a$ has absolute value $1$ and \(f\in\bbN\) is the exponent of the
conductor of \(\chi\). The reader can refer to \cite[(3.2)]{Ta79} for
more details.

\subsubsection{Other equalities for the $\gamma$-factor}
There are a variety of equivalent forms of \Cref{Tate2gamma}
that we record for reference. 
 Replacing \(\chi\) by \(\chi_s\), we have
\begin{align} \label{Tate3gamma}
    \gamma(s,\chi)= \FT(\chi_{s}^{-1})(1)
    =\int_F\chi^{-1}(x)\abs{x}^{-s}\Psi(x)\dd x.
\end{align}

Applying Fourier transform again to \eqref{Tate2gamma} and using the involutive
property, we have
\begin{align} \label{CheckChiEquation}
    \check{\chi}\defeq\FT(\chi)=\bigl(\chi(-1)\gamma(1,\chi)\chi_1\bigr)^{-1}.
\end{align}
which can be written symmetrically as
   $ \FT(\chi_-)=\chi(-1)\gamma(\OneHalf,\chi)^{-1}\chi_+^{-1}$.
If we apply $\FT$ again to 
 \eqref{Tate2gamma} we find that
\begin{align} \label{GammaSymmetry}
\gamma(0, \chi) \gamma(1, \chi^{-1}) =  \chi(-1),
\end{align}
and so also $\gamma(\OneHalf,\chi) \gamma(\OneHalf,\chi^{-1})=\chi(-1)$,
and moreover combining this with \eqref{Tate3gamma} we arrive at
\begin{align}
    \label{Tate4gamma}
    \gamma(s)^{-1} = \int_{F} \abs{x}^{s-1} \Psi(x) \dd x.
\end{align}

\subsubsection{Multisets of characters} \label{MultisetNotation}
Lastly, we use the following notational conventions for \emph{multi-sets} of
characters \(X, Y\): we let \(X^{-1}\) (resp.~\(X_s\), \(X_+\), \(X_-\), etc.)
to be the multi-set consisting of characters \(\chi^{-1}\) (resp.~\(\chi_s\),
\(\chi_+\), \(\chi_-\), etc.) for \(\chi\in X\), and
\begin{align}
    X\otimes Y\defeq \Set{xy\given x\in X, y\in Y}.
\end{align}
Similar to the single character case, \(X_s^{-1}\) means \((X_s)^{-1}\).
If \(X=\Set{\chi}\) is a singleton, we also use \(\chi\otimes Y\defeq X\otimes Y\).
In addition, we let
\begin{align}
    \gamma(s,X)\defeq \prod_{\chi\in X}\gamma(s,\chi),
\end{align}
and similarly for the \(L\)-factor or the \(\epsilon\)-factor.

For general groups other than \(\Gm\), the definition of \(\gamma\), \(L\), and
\(\epsilon\)-factors are more complicated and we only need them for a
small portion of the paper. For this reason we will postpone the discussion
until \Cref{sub:review_of_local_Langlands_correspondence}.

\subsection{Adjointness and isometric properties of the Fourier transform} \label{sec:Fadjoint}
To avoid any sign confusions we write these out. 
For $\Phi_i$  Schwartz--Bruhat functions on $F$, we have (writing simply $\check{\Phi}$
for the Fourier transform of $\Phi$)
\begin{align}
    \int_F \Phi_1 \check{\Phi}_2 = \int_{F} \check{\Phi}_1 \Phi_2,
\end{align}
as it follows directly from the definition. If we replace above $\Phi_2$
by its Fourier transform, we arrive at
\begin{align}
    \int_{F} \Phi_1(x) \Phi_2(-x) = \int_{F} \check{\Phi}_1 \check{\Phi}_2.
\end{align}
Finally, replacing $\Phi_2$ by $\overline{\Phi_2(-x)}$, we get
\begin{align}
    \int_{F} \Phi_1 \overline{\Phi}_2
    = \int_{F} \check{\Phi}_1 \overline{\check{\Phi}_2}.
\end{align}

\subsection{Review of integration on projective spaces}
\label{ProjectiveIntegrationReview}
We will several times have occasions to integrate densities over projective
spaces, and we now set up relevant notations.

Let $W$ be a $k$-dimensional vector space over $F$.
We say a complex-valued function $\varphi$  on $W-\Set{0}$ is $(-k)$-homogeneous
if $\varphi(tw) = \abs{t}^{-k} \varphi(w)$ for nonzero $t \in F^{\times}$.
There is, up to scaling, a unique
$\GL(W)$-invariant functional on such functions, denoted by
\begin{align}
    \varphi \longmapsto \int_{\bbP W} \varphi 
\end{align}
which we regard as ``integration over $\bbP W$''.
It can be normalized by the following requirement, once we pick Haar measures on $W$ and $F^{\times}$:
for a Schwartz function $\Phi$ on $W$ itself, 
the function
$\bar\Phi(w) = \int_{F^{\times}} \abs{t}^k \Phi(tx) $
is $(-k)$-homogeneous, and we require
\begin{align}
    \int_{\bbP W} \bar\Phi = \int_{W} \Phi.
\end{align}

Now, fix coordinates $W \simeq F^n$, and suppose the Haar measures
on both $W$ and $F^{\times}$ are induced from a Haar meausre on $F$. 
Then one readily verifies
\begin{align} \label{Pnexplicit}
    \int_{\bbP^{n-1}} \varphi
    = \int_{F^{n-1}} \varphi(1, x_2, \ldots, x_n) \dd x_2 \cdots \dd x_n.
\end{align}

\part{Definitions and statements}

\section{Definition of the tetrahedral symbol}\label{sec:representation_theoretic_definitions}

In this section, we define the tetrahedral symbol.  The first two subsections,
\Cref{subsec:Tetra} and  \Cref{subsec:Group}, set up notational preliminaries
about tetrahedra and $\SO_3$ respectively. In  \Cref{Tetrahedraldatum} we
describe the datum defining the tetrahedral symbol, and the actual definition is
given in \Cref{sub:special definition} (under some mild simplifying
conditions) and \Cref{sub:general definition} (in general).  

\subsection{Tetrahedra} \label{subsec:Tetra}
Consider a tetrahedron, with the following labeling: we label the four vertices
by bold face numbers \(\V1,\V2,\V3,\V4\),  the six edges by unordered pairs of vertices, and the
twelve oriented edges by ordered pairs of distinct vertices; we denote these
sets by \(\tetraV, \tetraE, \tetraO\) respectively. For two vertices
\(i,j\in\tetraV\), we will use \(ij\) to denote either the associated oriented
edge (from \(i\) to \(j\)) or
the unoriented one and the context will make it clear which version we are
referring to.  
There are natural maps
\begin{align}
    \tetraO &\longto \tetraE,\quad ij \mapsto ij\\
    \tetraO &\longto \tetraV,\quad ij \mapsto i
\end{align}
which, respectively, assign to an oriented edge  either the underlying unoriented edge
or its source vertex.
For each \(i\in\tetraV\), we let \(\tetraO_i\subset\tetraO\) be
the oriented edges with source \(i\), that is the preimage of $i$
under the second map above.

\subsection{Group-theoretic setup}\label{subsec:Group}
Let
\begin{align}
    \RO = \SO_3(F)
\end{align}
be the special orthogonal group of a nondegenerate ternary quadratic form over the local field \(F\).
Note that we allow an arbitrary form, not only a split one;
therefore, in the case of $F=\bbR$, the group
$\RO$ is either compact $\SO_3$ or $\PGL_2(\bbR)$,
and more generally $\RO$ is either $\PGL_2(F)$
or the projective group of units in a quaternion algebra over $F$.
 Now define:
\begin{align} \label{eq: group basic notation}
    G = \RO^{\tetraO},\quad
    D = \RO^{\tetraE}, \quad
    H =\RO^{\tetraV}.
\end{align}
Here we regard $D, H$ as subgroups of $G$,
by means of the maps 
$\tetraO \to \tetraE$ and $\tetraO \to \tetraV$.
We may visualize  elements of $G \simeq \RO^{12}$ as
$12$-tuples of elements of $\RO$ thus:
\begin{align}
    \begin{bmatrix}
        g_{\V{12}} & g_{\V{13}} & g_{\V{14}} & g_{\V{23}} & g_{\V{24}} & g_{\V{34}} \\
        g_{\V{21}} & g_{\V{31}} & g_{\V{41}} & g_{\V{32}} & g_{\V{42}} & g_{\V{43}}
    \end{bmatrix},
\end{align}
and then $D \simeq \RO^6$ and $H \simeq \RO^4$ correspond to subgroups:
\begin{align}
    D = \begin{bmatrix}
        a  & b & c & d  & e  & f \\
        a  & b & c & d  & e  & f 
    \end{bmatrix},\quad
    H = \begin{bmatrix}
        x  & x & x & y & y & z \\
        y & z & w & z & w & w
    \end{bmatrix}.
\end{align}

We fix a  Haar measure on \(\RO\) and so also on \(G, D, H\), etc as follows: if
\(\RO\) is compact, then we let its volume be \(1\); 
  in the nonarchimedean  split case we normalize the Haar measure so that \(\SO_3(\cO)\) has volume
\(1\).
For the remaining case,  when \(\RO=\PGL_2(\bbR)\) or \(\PGL_2(\bbC)\), we first
fix an algebraic volume form on $\PGL_2$, identifying it by
means of the map $g \mapsto (g \cdot 0, g \cdot 1, g \cdot \infty)$ with
$(\mathbb{P}^1)^3$ minus diagonals; on this space, we consider the unique algebraic volume
form  $\omega$ whose restriction to $(\mathbb{A}^1)^3$ minus diagonals is
\begin{equation}
    \label{omegaA1def}
    \omega \defeq \frac{ \dd x_1 \wedge \dd x_2 \wedge \dd x_3}
    {(x_1-x_2)(x_2-x_3)(x_3-x_1)}.
\end{equation}
One readily verifies that this is $\RO$-invariant.
Then $\abs{\omega}$ defines an volume form on the $F$-points of $\RO$.

Note that this construction works in the case of $F$ nonarchimedean too, but it
gives a {\em different} measure: it would assign to the $\cO$-points the volume
$(1-q^{-2})$.
Sometimes we will need to use both the Haar measure and \eqref{omegaA1def} for
nonarchimedean \(F\), and so to emphasize their distinction, we will also use
\(\dd^{\bbP}g\) to denote the measure induced by \eqref{omegaA1def}.
For convenience, we define
\begin{align}
    \label{eqn:Haar_vs_P13_measures}
    \nu^{\bbP}\defeq\frac{\dd^\bbP g}{\dd g}= \begin{cases}
        (1-q^{-2}) & \RO=\PGL_2(F), F\text{ nonarchimedean},\\
        1 & \text{otherwise}.
    \end{cases}
\end{align}

\subsection{The tetrahedral datum} \label{Tetrahedraldatum}

We denote by $\Pi$
an assignment of {\em an irreducible smooth representation of $\RO$
to each unoriented edge of the tetrahedron.}\footnote{
    The word ``smooth'', in the nonarchimedean case,
    means that each vector has open stabilizer. In the archimedean case,
    it connotes that the underlying vector space for $\RO$
    has a Fr{\'e}chet topology such that the map $g \mapsto g \cdot v$
is smooth. Subtle issues of topology, however, will be almost irrelevant for us.}
The representation will, of course, matter only up to isomorphism.
To this we will attach an invariant $\Pisymbol$ which is
a complex number defined up to sign. The matter of fixing the sign is an interesting
one, which we return to at various points, in particular 
\Cref{SignRemark1} and part (2) of \Cref{thm:W_D6_symmetry}.

Denoting the assignment $\Pi$ by $e \mapsto \pi_e$, we
will also refer to $\pi_{ij}$
for an oriented edge $ij \in \tetraO$ by means of the natural map
$\tetraO \to \tetraE$,
and taking the external tensor product of all $\pi_{ij} $ produces an
irreducible representation of \(G\), denoted simply by \(\Pi_G\).

Now, each $\pi_e$ admits an invariant {\em symmetric} self-pairing
\begin{align}
    \IPair{-}{-} \colon \pi_e \times \pi_e \longto \bbC.
\end{align}
Such a pairing always exists (\cite[Theorems~2.18, 5.11, 6.2]{JaLa70});
making clever use of  multiplicity one subgroups, Dipendra Prasad
proved  that it is also always symmetric
(\cite[Corollary~2]{Prasad99b} and \cite[Proposition~2]{Prasad99}).\footnote{
    The latter reference covers only the {\em generic} representations; but 
    by  \cite[Theorems~2.13, 5.13, 6.3]{JaLa70}
    all irreducible representations are generic except for the finite-dimensional ones.
    The statement can be checked by hand in that remaining case, although we will not use it.
} 
Moreover,  any two  such pairings
are equivalent under a rescaling of the underlying space of $\pi_e$;
and by Schur's lemma, the automorphisms of the pair 
\((\pi_{e},\IPair{-}{-})\) reduce to multiplication by \(\pm 1\).\footnote{While
    this seems \textit{ad hoc}, this is a special case of a
construction that works for any split reductive group, see discussion of duality
in \cite{BZSV24}.} We will frequently refer
to such a pairing as a \emph{rigidification}, because it reduces
the automorphism group of $\pi$ from $\mathbb{C}^{\times}$ to $\pm 1$.

Fix such a self-pairing for each $\pi_e$, which then induces a pairing between
\(\pi_{ij}\) and \(\pi_{ji}\); these induce a $D$-invariant linear
\notion{contraction map}
\begin{align}
    \Contra{-}\colon \Pi_G\longto \bbC,
\end{align}
which, concretely, is given by
\begin{align}
    \bigotimes_{ij \in \tetraO} v_{ij}
    \in \Pi_G \longmapsto  \prod_{i<j}  \IPair{v_{ij}}{v_{ji}}.
\end{align}

We will also denote this by $\Lambda^D$ to emphasize its $D$-invariance.

\subsection{Definition of the tetrahedral symbol when $\Pi_G$ is tempered}
\label{sub:special definition}

\subsubsection{Definition in the compact case}\label{sub:definition_in_the_compact_case}

We begin with the definition of the tetrahedral symbol in ``the compact case'',
that is to say, when $\RO$ is compact (cf. \Cref{Intro}). In this
case all the representations $\pi_{ij}$ are finite dimensional. 
We put \(\pi_{\tetraO_i} \defeq \bigboxtimes_{ij\in\tetraO_i}\pi_{ij}\), a
representation of \(\RO^{\tetraO_i} \simeq \RO^3\). For example,
\begin{align}
    \pi_{\tetraO_\V1}
    = \pi_{\V{12}}\boxtimes\pi_{\V{13}}\boxtimes\pi_{\V{14}}.
\end{align}
Then \(\Pi_G=\bigboxtimes_{i\in\tetraV}\pi_{\tetraO_i}\),
and the pairings on each $\pi_e$
induce also a self-pairing on each $\pi_{\tetraO_i}$.
It will be very convenient to make use of the following:

\begin{lemma} \label[lemma]{PrasadLemma}
    There exists a real structure $\pi_e^{\R}$
    on which  
    $\IPair{-}{-}$ defines a real inner product.
\end{lemma}
\begin{proof}
    This is well-known, but we write out the argument
    for later reference. Fix a unitary inner product on $\pi_e$,
    which we denote by $\Pair{x}{y}$.
    Necessarily $\IPair{x}{y} = \Pair{x}{C y}$ 
    for some complex-antilinear
    $C: \pi_e \rightarrow \pi_e$ which commutes with the group action, and one readily
    sees that $C^2=\pm 1$. Symmetry of the pairing implies
    that $\Pair{x}{Cy}=\Pair{y}{Cx}$;
    taking $x=Cy$ we deduce that $C^2$ is positive, thus must be $+1$.
    Then the fixed points of $C$ gives the desired real structure, since
    $x= \frac{x+Cx}{2} + i \frac{x-Cx}{2i}$.
\end{proof}

We suppose that each \(\pi_{\tetraO_i}\) admits a nonzero \(\RO\)-invariant
vector \(v_i\); otherwise we will define $\Pisymbol \defeq 0$.\footnote{It would be  more proper to regard the symbol as \emph{undefined} in
those cases. We adopt this convention, however,  in order to avoid having to repeatedly say ``or
$\Pisymbol$ is undefined'' in various statements; the
same convention is followed in the theory of $6j$ symbols.}
Then the self-pairing $\IPair{v_i}{v_i}$ is nonzero, because the self-pairing is
positive definite on a real structure for $\pi_{\tetraO_i}$ and a suitable
multiple of $v_i$ is real for this structure. We therefore normalize \(v_i\) so
that \(\IPair{v_i}{v_i}=1\) and let
\begin{equation}
    \label{eqn:def_of_invariant_vector_V}
    V = v_{\V1} \otimes v_{\V2} \otimes v_{\V3} \otimes v_{\V4} \in \Pi_G,
\end{equation} 
which, having fixed pairings, is uniquely specified up to a sign. 
Contract \(V\) to obtain what we shall call the \notion{tetrahedral
symbol}
\begin{equation}
    \label{eqn:unnormalized_6j_symbol_def}
    \Pisymbol \defeq \Contra{V} \in\bbC.
\end{equation}
 This depends on our choice of pairings only up to an overall $\pm$ sign. 

In the classical case when \(\RO\) is the compact real \(\mathrm{SO}_3\)
and each \(\pi_{ij}\) is indexed by its highest weight, this is up to sign the standard
definition of the classical \(6j\) symbol;
this will follow from our computations in \Cref{sub:vertex_formula}.

\subsubsection{Definition in the tempered case}\label{sub:definition_in_the_non_compact_case}
We now drop the assumption that $\RO$ is compact,
but assume that each $\pi_e$ is {\em tempered}.
Recall that an irreducible representation of a group over a local
field is called \emph{tempered} if it is ``weakly
contained'' in the regular representation $L^2(\RO)$; for details
on what this means, see \cite{CowlingHaagerupHowe}.
As we discuss further in \Cref{DefinitionNontempered}, tempered reprsentations
can be considered, roughly speaking, as a ``real form''
of a complex variety parameterizing all representations.   
In the compact case, all irreducible representations are tempered.

Every tempered representation admits an $\RO$-invariant inner product;
in particular, by inspecting the proof, we see that Lemma \ref{PrasadLemma} continues to apply:
there exists a real structure $\pi_e^{\R}$  and a $\RO$-invariant
inner product upon it.

In place of \(D\) and \(H\)-invariant \emph{vectors} in \(\Pi_G\), we will use
\(D\) and \(H\)-invariant \emph{functionals}. 
The space of $D$-invariant functionals is always exactly
one-dimensional, spanned by $\Lambda^D$
defined as before; the space of $H$-invariant functionals
is zero or one (see \cite{PrasadCompositio}).
Although unnecessary for our immediate purposes, the two possibilities
are distinguished by the  signs of various $\epsilon$-factors, as explained
in the just-quoted work. 
 
If a nonzero $H$-invariant functional exists, there are 
  two natural ways to  construct it. 
First of all, we can normalize one, up to sign, by the following rule:
\begin{equation}
    \label{eqn:def_of_Lambda_H}
    \Lambda^H(v_1) \Lambda^H(v_2) = \int_{H}  \IPair{h v_1}{v_2}\dd h,\quad v_1,v_2\in\Pi_G.
\end{equation}
\emph{Note the pairing that is used here, and indeed everywhere unless
explicitly stated otherwise, is the self-duality pairing, not an inner product.}

The integral is absolutely convergent on account of the
assumption that $\Pi_G$ is tempered. In this case,
$\Lambda^H$ is nonzero if and only if the space of $H$-invariant functionals
is nonzero, as is proven in a more general context in \cite{SV17}.
On the other hand, we can start with 
\(\Lambda^D\) and just \emph{average it to be \(H\)-invariant}:
\begin{align}  \label{eqn:LambdaHprimedef}
    (\Lambda^H)' \colon v \longmapsto  \int_{H \cap D \backslash H}
    \Lambda^D(hv)\dd h.
\end{align}
This is also absolutely convergent under the assumption of temperedness: 
see \Cref{Hprimedefconvergence}.
 By the multiplicity-one property, \(\Lambda^H\) and \((\Lambda^H)'\) are
proportional to one another; we define the   tetrahedral symbol \(\Pisymbol\) to be the proportionality
factor:
\begin{equation}
    \label{cPi2}
    (\Lambda^H)' = \Pisymbol \Lambda^H.
\end{equation} 

This agrees with the definition given in the compact case. Indeed, define $V$ as in
\eqref{eqn:def_of_invariant_vector_V}; it is straightforward to show that
\(\Lambda^H=\IPair{V}{-}\) satisfies \eqref{eqn:def_of_Lambda_H}.
Thus, on the one hand,
\(\Lambda^H(V) = 1\) by definition, and on the other hand, \((\Lambda^H)'(V)=\Contra{V}\),
which coincides with \eqref{eqn:unnormalized_6j_symbol_def}.

\begin{warning}
    The definition just given is \emph{dual to the definition given in
    \Cref{Intro}}. Relative to that discussion, we have swapped the role of
    vertices versus faces; or to put it differently, by swapping the upper row with
    the lower row in the classical \(6j\) notation.
\end{warning}

\subsection{Definition  of the tetrahedral symbol in the general case}
\label{sub:general definition}

\subsubsection{Principal series and the classification of tempered representations}
\label{DefinitionNontempered}
 
We can  extend the definition beyond tempered representations by a process of
analytic continuation. This discussion is only relevant in the case when $\RO$
is noncompact, which we shall therefore assume.

We must first recall the notion of principal series representation
and  the classification of tempered representations.
Principal series are a twisted
version of functions on the projecive line $\mathbb{P}^1_F$. 
Namely, let $\chi$ be a quasi-character of $F^{\times}$,
and consider the space $\pi_{\chi}$ of functions 
on $F^2-\{0\}$ that are  
  homogeneous of degree $(\chi^2)_{-1} = \chi_{-}^2 = (\chi_{-})^2$, that is, a function
on $F^2 - \Set{0}$ satisfying
\begin{align}
    f (\lambda \cdot z) = \chi^2(\lambda) \abs{\lambda}^{-1} f(z).
\end{align}
Then $\RO$, identified with $\PGL_2(F)$, 
acts on such functions by means of 
\begin{align}
    g \cdot f\colon z \longmapsto f(z\tilde{g})\chi^{-1}_{-}(\det\tilde{g})
    =f(z\tilde{g})(\chi^{-1})_{\OneHalf}(\det\tilde{g}),
\end{align} 
where \(\tilde{g}\in\GL_2(F)\) is an arbitrary lift of \(g\), the choice of
which has no effect on the action.
This construction yields 
an association
\begin{align}
    \text{quasi-characters \(\chi\) of $F^{\times}$} \longrightarrow
    \text{smooth representations $\pi_{\chi}$ of $\RO$}.
\end{align}
The resulting representation $\pi_{\chi}$ is called a \notion{principal series representation}.
It is not irreducible in general, but it is if \(\chi\) is away from certain
discrete subset of all quasi-characters.
Moreover, $\pi_{\chi}$ and $\pi_{\chi^{-1}}$ have the same semi-simplification.
For $\chi$ a character, that is to say, a unitary quasi-character, in
particular, $\pi_{\chi}$ is always irreducible, tempered, and
$\pi_{\chi} \simeq \pi_{\chi^{-1}}$.

With this setup, a tempered irreducible representation of $\RO$ is either:
\begin{enumerate}
    \item of the form $\pi_{\chi}$, where $\chi$ is a character  uniquely determined
        up to the substitution $\chi\leftrightarrow \chi^{-1}$; or
    \item isomorphic to a direct summand of $L^2(\RO)$; these
        form a countable set  of irreducible representations called the \notion{discrete series}.
\end{enumerate}

\subsubsection{Definition  of the tetrahedral symbol in the general case} \label{Definition-general}
We continue to suppose that $\RO$ is noncompact. 
Let  \(\cP_0\) be the set of isomorphism classes of irreducible tempered
representations. Because of the classification above, we can think of \(\cP_0\) as a subset  
of points of a complex analytic variety \(\cP\):
\begin{align} \label{Pdef}
    \cP = \Set*{\text{discrete series}}
    \coprod \Set*{\text{quasi-characters of }F^{\x}\text{ up to inversion}}.
\end{align}
In fact, $\cP_0$ is a subset of the real points $\cP(\mathbb{R})$
for a natural real structure on $\cP$. What is more important
for us is that an analytic function on $\cP$ that vanishes
on $\cP_0$ is identically vanishing; therefore, there is at most one way
to extend a function from $\cP_0$ to a meromorphic function on $\cP$.
The tetrahedral symbol thus extends, after squaring to remove sign ambiguity:

\begin{proposition}
    \label[proposition]{prop:AC}
    The function $\Pisymbol^2$ extends to a meromorphic
    function on $\mathcal{P}^\tetraE$.
\end{proposition}
The proof is given in  \Cref{sub:pf_of_AC}.
It is based on studying the  
the asymptotic behavior of the integrand in \eqref{eqn:LambdaHprimedef} that enters into the definition of $\Pisymbol$,
which in all cases is very simple; for example in the nonarchimedean case
it is a geometric progression, in suitable coordinates.

\begin{remark}
    Recall that $\Pisymbol$ is nonzero only when $\Pi_G$ admits an $H$-invariant
    functional. Whether this is so   depends only on the component of
    $\cP^{\tetraE}$ to which $\Pi$ belongs. Consequently, the function
    $\Pisymbol$ is simply identically zero on some components of
    $\cP^{\tetraE}$; and  these can be identified by means of $\epsilon$-factors
    using the results of Prasad, see \Cref{EpsilonFactors}.
\end{remark}

\begin{remark}
    One way to interpret the proposition as giving an extension of  
    tetrahedral symbol from tempered representations to all irreducible ones.
    Indeed, there is an identification of sets 
    \begin{equation}
        \label{cPid}
        \cP \simeq \Set{\text{isomorphism classes of irreducible representations of $\RO$}}
    \end{equation}
    defined as follows: we associate to $\chi$ as above $\pi_{\chi}$ if it is
    irreducible, and, otherwise, the unique finite-dimensional subquotient of
    $\pi_{\chi}$ (\cite[Theorems~3.3, 5.11, 6.2]{JaLa70}).
    It is likely that the function $\Pisymbol$ thus extended
    coincides, even for nontempered representations, with a function defined
    by means of suitably regularizing the integrals appearing in
    \Cref{sub:definition_in_the_non_compact_case}. However, we do not examine this
    in the current paper, and the identification \eqref{cPid} will play no 
    further role.  
\end{remark}

\subsection{The sign ambiguity}\label{SignRemark1}

Note that, even when $\Pisymbol$ is defined, it is defined only up to a
sign.  The reader can, at a first reading, ignore all remarks that pertain to this sign, 
and still grasp most of the   content of the paper. 
But,  as we now discuss, the sign is subtle and interesting.

Unlike the classical situation, \emph{this sign ambiguity is essential}:
the meromorphic function $\Pisymbol^2$ described  in the above proposition does
not in general admit a meromorphic square root. Of course, we could redefine
the symbol to be $\Pisymbol^2$ instead of $\Pisymbol$ but we prefer not to do so for
two reasons: first of all, it is $\Pisymbol$ and not its square that corresponds
to the $6j$ symbol; and secondly, the choice of sign is actually very
interesting.

As a general convention, when we prove a formula of the form
\begin{align}
    \Pisymbol = \cdots,
\end{align}
we always regard the equality as being up to sign.
However, in all  important such instances, and particularly in
\Crefrange{sec:MainTheorems}{sec:hypergeometric_formulas_for_the_tetrahedral_symbol},
{\em the formula in fact will give more:} 
it gives a mechanism to resolve the sign ambiguity,   in the sense that we will
produce an explicit meromorphic function $\gamma$ which belongs to the same
square class as $\Pisymbol^2$ ---   that is to say, $\gamma$
differs from $\Pisymbol^2$ by the square of a meromorphic function, and
and therefore $\sqrt{\gamma} \cdot \Pisymbol$ can be
globally defined, not only up to sign.
Similarly, when we describe the symmetry properties of $\Pisymbol$,
we will describe precisely all the signs (part (c) of \Cref{thm:W_D6_symmetry}),
although the reader may prefer to ignore this at a first reading (part (a) of the
same theorem).

All our work elucidating signs comes, however, at a price: one
must make a choice of orientation of edges; and 
in order to obtain the nicest formulas we even need to use
a slightly strange one, see the diagram \eqref{eqn:chosen_oriented_tetrahedron}.

\section{$\TypeD_6$ and the tetrahedron} \label{D6Tetrahedron}

It has been observed by various authors, and perhaps brought into the greatest clarity
  by Rudenko \cite{Rudenko22}, that the geometry of tetrahedron
is related to the root system $\TypeD_6$. As we will
soon see, the tetrahedral symbol reflects the geometry of the associated group,
and not just its root system.

In the present section, we will set up various notation that will
make this connection clearer.  For us, the connection
between the tetrahedron and $\TypeD_6$ will be  ``carried'' by
a homomorphism of rank $6$ free abelian groups
$\OddFunctions \hookrightarrow \TypeD_6$  
to be defined in  \Cref{sub:D6_lattice}.
In \Cref{GroupTheory} we promote this injection of free abelian
groups to a homomorphism of compact Lie groups,
and in \Cref{GroupTheory2} we will examine how the Weyl group $W(\TypeD_6)$
interacts with the symmetries of the tetrahedron.

\subsection{The lattice $\TypeD_6$ and its tetrahedral avatar $\OddFunctions$}
\label{sub:D6_lattice}

The \notion{coroot lattice of type $\TypeD_n$}
will be understood to be the sublattice
of $\Z^{2n}$ consisting of elements of the form
\begin{align} \label{DnDef}
    (x_1, \ldots, x_n, -x_n, \ldots, -x_1),\quad \sum_{i=1}^n x_i\in 2\bbZ.
\end{align}
Here, and in what follows, a \notion{lattice} is simply a free abelian group of
finite rank.
It is equipped with the reflection
group \(W(\TypeD_n) = \Set{\pm 1}^{n-1}\rtimes\FRS_n\)  
obtained by permutations of the $x_i$
and changes of an \emph{even}  number of signs.
It also contains a distinguished $W(\TypeD_n)$-invariant
set of \emph{coroots}, namely, all those vectors $\alpha$
that satisfy $(\alpha,\alpha)=2$ for the standard Euclidean inner product, and the
reflections through their orthogonal hyperplanes
generate $W$.

Following \cite[Plate IV]{Bo02}, we identify \(\TypeD_n\) with the coweight
lattice for the simply-connected group \(\Spin_{2n}\) by projecting to the
first \(n\) coordinates. Here $\Spin_{2n}$ means
the universal (two-fold) cover of
$\SO_{2n}(\bbC)$; it is a complex semisimple group.  The reader who prefers compact Lie groups
can equally well work with its maximal compact subgroup,
which is similarly described as the the universal (two-fold) cover of the compact group $\SO_{2n}(\bbR)$.
The fact that \(\pi_1(\SO_{2n}(\bbC))=\pi_1(\SO_{2n}(\bbR))=\bbZ/2\) for
\(n>1\) may be found in \cite[Theorems~13.6 and 24.1]{Bu13}.

To relate $\TypeD_6$ to the tetrahedron, we consider  the 
 set of odd integral-valued functions on oriented edges
\begin{align} \label{OddFunctionsDef}
    \OddFunctions \defeq \Set{f\colon\tetraO\to\bbZ\given f(ij)=-f(ji)},
\end{align}
where we say that a function, with domain the
oriented edges of the tetrahedron, is {\em odd}
if inverting an edge negates (or inverts, where appropriate) the value of the function.
As a convention,
we will denote such a function $f$ by means of the $2 \times 3$ matrix
of values as follows:
\begin{equation}
    \label{eqn:definitive_2by3_matrix}
    \begin{bmatrix}
        f(\V{12}) & f(\V{13}) & f(\V{14}) \\ 
        f(\V{34}) & f(\V{24}) & f(\V{23})
    \end{bmatrix}.
\end{equation}

Then $\OddFunctions$
is a lattice of rank $6$, and there is an injection
\begin{align}
    \label{OddFunctions} 
    \OddFunctions &\longto \TypeD_6
\end{align}
which sends a function $f$ to the element of $\TypeD_6$
whose  twelve coordinates are {\em sums of values of $f$ on opposite edges.}
There are three pairs of opposite edges, but each comes with four possible orientations; thus we get twelve such sums in all,
symmetric under negation.
For definiteness, we take the first six coordinates of the element
of $\TypeD_6$
to be 
\begin{align} \label{TheIsogeny}
    f(\V 1\V 2)\pm f(\V 3 \V 4),\quad
    f(\V 1\V 3)\pm f(\V 2 \V 4),\quad
    f(\V 1\V 4)\pm f(\V 2 \V 3),
\end{align}
in the given order; 
the last six coordinates are uniquely determined by the first six.

\subsection{The associated isogeny  $\tau$ of compact, or reductive, groups} \label{GroupTheory}   The isogeny of lattices has a more group-theoretic manifestation
which will play an important role in
\Cref{TetrahedralAsInterface}. Namely,
there is a commutative diagram
\begin{equation} \label{taudef2}
    \begin{tikzcd}
        \OddFunctions \otimes \Gm \ar[r]\ar[d] & \TypeD_6 \otimes
        \Gm \ar[d]\\
        \SL_2^{\tetraE} \ar[r, "\tau"] & \Spin_{12}
    \end{tikzcd}
\end{equation}
where the top row is an isogeny of tori induced from \eqref{OddFunctions},
and the bottom row is a homomorphism of reductive groups in which these tori are maximal.

Here, $\Gm$ denotes the multiplicative group,  i.e. the algebraic
group whose complex points are simply $\bbC^{\times}$,
and
the notation $(\blank)\otimes\Gm$ simply means that we replace integer
variables in $(\blank)$ by $\Gm$-valued ones; thus, for example,
$\OddFunctions \otimes \Gm$ can be considered as functions from
oriented edges to $\Gm$ satisfying $f(ij) f(ji)=1.$ Also, as mentioned
earlier, the reader who prefers compact groups to reductive algebraic groups may
harmlessly replace $\Gm$ by the unit circle, $\SL_2$ by $\mathrm{SU}_2$, and
$\Spin_{12}$ by the compact group of the same name.

To construct $\tau$,  let \(\bbC^2_{ij}\) be a two-dimensional vector space attached to
edge \(ij\) with standard basis \(\Set{\bFe_{ij},\bFe_{ji}}\).
Introduce  a copy of $\SL_2$ indexed by $ij$, namely, the unimodular automorphisms
of  \(\bbC^2_{ij}\); its cocharacter group is then identified
with pairs of integers $a_{ij}, a_{ji}$ satisfying $a_{ij}+a_{ji}=0$.
Taking the product over edges $\tetraE$
we arrive at a model for $SL_2^{\tetraE}$
whose maximal torus is canonically  identified with $\OddFunctions \otimes
\Gm$.
Now, for each pair of opposite edges $\Set{ij, kl}$,
taking tensor product of the defining representations 
induces a homomorphism\footnote{Here, to be precise,
    \(\SO_4\) is the \emph{split} form attached to the symmetric  pairing
    in which
    \begin{align}
        \Pair{\bFe_{ij}\otimes \bFe_{kl}}{\bFe_{ji}\otimes \bFe_{lk}} = 1,
        \Pair{\bFe_{ij}\otimes \bFe_{lk}}{ \bFe_{ji}\otimes \bFe_{kl}} = -1,
    \end{align}
and all other pairings zero. }
\begin{align}
    \SL_2^{ij}\x\SL_2^{kl}\longto\SO_4\subset \GL(\bbC^2_{ij}\boxtimes\bbC^2_{kl}),
\end{align}

Applying this construction to the three sets of opposite pairs ---
using  the same order that has been used in \eqref{TheIsogeny}, that is: \(\Set{\V1\V2,\V3\V4}\),
\(\Set{\V1\V3,\V2\V4}\), \(\Set{\V1\V4,\V2\V3}\) ---
we arrive at a map $\SL_2^{\tetraE} \rightarrow \SO_{12}$,
which lifts uniquely to the desired map $\tau: \SL_2^{\tetraE} \rightarrow \Spin_{12}$
of \eqref{taudef2}. Note that $\tau$ is not an
embedding: it has a finite kernel \(\Set{\pm 1}^2\).

\subsection{$W(\TypeD_6)$ and the symmetries of the tetrahedron} \label{GroupTheory2}

By means of  \eqref{TheIsogeny} we understand $W(\TypeD_6)$ to act
on $\OddFunctions \otimes \mathbb{R}$. 
This action interacts richly with the geometry of the tetrahedron. For example
the constraint on dihedral angles is $W(\TypeD_6)$-invariant, see
\Cref{PVol}.

We shall now 
label several important subgroups, which together generate  $W(\TypeD_6)$:
the orientation reversals,   tetrahedral symmetries (which come in two versions: the evident ones, and ones that use an orientation),  and the Regge symmetries:

\begin{itemize}
    \item   The group $\FRI  \simeq (\Z/2)^6$
        of orientation reversals: 
        for each $ij \in \tetraO$ there exists
        a unique element of $W(\TypeD_6)$
        which acts on   
        $\OddFunctions \otimes \mathbb{R}$  by ``reversing the orientation of $ij$,''
        i.e. negating (or inverting, when appropriate) the value of any function
        on $ij$.\footnote{These {\em should not be confused} with the
            elements of $W(\TypeD_6)$ that switch the signs of some coordinates
        \(x_i\). The latter is a group isomorphic to $(\Z/2)^5$.} Such elements 
        generate an elementary abelian subgroup  $\FRI \subset W(\TypeD_6)$ 
        of order $2^6$.

    \item  The group \(\FRTV\simeq\FRS_4\) of tetrahedral symmetries. These
        arise from physical symmetries of the tetrahedron, acting by permuting
        vertices in the evident fashion; for
        example,  \((12)\in\FRS_4\) will \emph{negate} the value of
        \(f\in\OddFunctions \otimes \bbR\) at \(\V{12}\).

    \item The \emph{subset} 
        \(\FRR=\Set{r_{\EN{1}\EN{4}},r_{\EN{2}\EN{5}},r_{\EN{3}\EN{6}}}\)
        of \emph{Regge symmetries}, indexed by pairs of opposite edges.
        These  ``exotic'' symmetries  were written down by
        Regge; they do not form a group,
        {\em and they depend on choice of an orientation}, that is to say,
        a  splitting $\tetraE \rightarrow \tetraO$.
        \begin{itemize}
            \item To elegantly resolve an important sign ambiguity down the line
                (see \Cref{SignRemark1,thm:W_D6_symmetry}),  we will choose not
                the ``dictionary'' order, i.e. coming from the order on the
                natural numbers $\V{1} < \V{2} < \V{3} < \V{4}$, but rather the
                ``vortex''orientation\footnote{There are multiple choices
                possible, but they are not completely random and the
                ``admissible'' choices all lead to the same end.} defined by the
                following diagram  \eqref{eqn:chosen_oriented_tetrahedron},  where
                we also use blackboard bold numbers \(\EN{1}\), etc. to denote each
                oriented edge.\footnote{ We do apologize in advance for using
                different font shapes maybe a little excessively.}
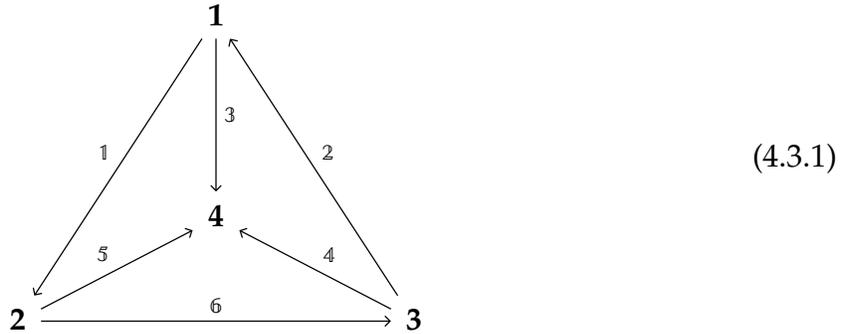
                \begin{equation}
                    \label{eqn:chosen_oriented_tetrahedron}
                    \begin{tikzcd}
                        && \V1
                        \ar[llddd,"\EN{1}"']\ar[dd,"\EN{3}"] && \\
                        && \phantom{x} &&\\
                        && \V4  &&\\
                        \V2\ar[rrrr,"\EN{6}"]\ar[rru,"\EN{5}"] && &&
                        \V3\ar[llu,"\EN{4}"']\ar[lluuu,"\EN{2}"']
                    \end{tikzcd}
                \end{equation}
            \item Note that the orientation is almost according to the
                dictionary order in \(\tetraV\), with one exception: the edge
                \(\V{31}\) is favored over \(\V{13}\), making the outer triangle
                an oriented cycle rather than a simplex.
            \item For short we write \(r_{\EN{1}\EN{4}}\)
                for the element associated to $\Set{\V{12}, \V{34}}$, which is defined as follows:
                for $f \in \OddFunctions \otimes \bbR$, let $a,b,c,d$ be its
                values at   $\V{31}, \V{14}, \V{23}, \V{24}$, i.e. the remaining
                four edges, oriented according to the chosen orientation. Then
                \(r_{\EN{1}\EN{4}}(f)\) has the same values at $\V{12}$ and
                $\V{34}$, but $a,b,c,d$ are modified according to
                \begin{align}
                    a^* = s-a,\quad b^*=s-b,\quad  c^*=s-c,\quad  d^*=s-d,
                \end{align}
                where $s \defeq \frac{a+b+c+d}{2}$ is the ``semi-perimeter.'' 
                We proceed similarly for \(r_{\EN{2}\EN{5}}\) and \(r_{\EN{3}\EN{6}}\).
        \end{itemize}

    \item For technical use only: the group $\FRTE \simeq \FRS_4$ of oriented tetrahedral symmetries.
        These also arise from physical symmetries of the tetrahedron, but
        \emph{taking account of orientation.}   
        We identify $\OddFunctions$ with
        functions on the set $\tetraE$ using the vortex orientation
        of  \eqref{eqn:chosen_oriented_tetrahedron}, and then \(\FRTE\) acts
        in the natural way on $\tetraE$. Thus, for example, the transposition
        $(12) \in \FRS_4$ acting on $f \in \OddFunctions$ leaves its value at
        $\V{12}$ unchanged. We call \(\FRTE\) the
        \notion{oriented tetrahedral group}.
 .

\end{itemize}

To avoid potential confusion when compared with notations such as
\eqref{eqn:definitive_2by3_matrix}, we emphasize again that

\begin{quote}
    \emph{the vortex orientation 
    is to be used only 
    when we consider the
    symmetries \(\FRTE\) and \(\FRR\), and   in particular when considering sign issues
    associated with them. By default, in all other instances unless specified otherwise, we use
    the ordering orientation. }
\end{quote}

We now discuss some of the group theory of how
these various groups interact. For the moment,
the reader can skip this discussion, and refer back to it as necessary:

\begin{lemma}
    \label[lemma]{lem:some_subgroups_of_WD6}
    Together with the group of orientation reversals \(\FRI\), either 
    the group of tetrahedral symmetries $\FRTV$ or the group of oriented tetrahedral symmetries \(\FRTE\) generate
    the same subgroup of \(W(\TypeD_6)\); this subgroup is isomorphic to $(\Z/2)^{6} \rtimes \FRS_4$.
\end{lemma}

\begin{lemma}
    \label[lemma]{lem:Regge_group_in_a_narrow_sense}
    Suppose we put the oriented edges in the following matrix:
    \begin{align}
        \begin{bmatrix}
            \V{12} & \V{31} & \V{14}\\
            \V{34} & \V{24} & \V{23}
        \end{bmatrix}
        =\begin{bmatrix}
            \EN{1} & \EN{2} & \EN{3}\\
            \EN{4} & \EN{5} & \EN{6}
        \end{bmatrix},
    \end{align}
    and let
    \begin{enumerate}
        \item \(r_i\) be the Regge symmetry that fixes the \(i\)-th column (so that
            \(r_1=r_{\EN{1}\EN{4}}\), and so on);
        \item \(h_i\in\FRTE\) fixes the \(i\)-th
            column and swaps the two rows in the other columns;
        \item \(v_i\in\FRTE\) fixes the \(i\)-th column and swap
            the two columns other than the \(i\)-th.
    \end{enumerate}
    Then the group generated by \(\FRTE\) and \(\FRR\) is isomorphic to
    \(\FRS_3\x \FRS_4\), with \(\FRS_3\) generated by \(\Set{h_ir_i=r_ih_i}\),
    and the commuting \(\FRS_4\) generated by \(\Set{v_ir_i=r_iv_i}\).
\end{lemma}
\begin{proof}
    Relabel coordinates $x_1, \dots, x_6$ in
    $\TypeD_6$ (cf. \eqref{DnDef}) as $y_1^{\pm}, y_{2}^{\pm}, y_3^{\pm}$, so
    that $y_1^{\pm}$ correspond to $\V{12}\pm\V{34}$, \(y_2^\pm\) correspond to
    \(\V{31}\pm\V{24}\), and so on.
    Then $r_1,h_1, v_1$  act as follows:
     $r_1$ swaps $y_2^+$ with $y_3^+$ and negates $y_2^-, y_3^-$;  
     $h_1$ negates $y_2^{-}$ and $y_3^{-}$, and $v_1$ swaps $y_2^{\pm}$
     with $y_3^{\pm}$.
    Then $v_1 r_1$ (resp. $h_1 r_1$) fixes four of the coordinates and
    acts as a reflection in the remaining two:
    \begin{align}
        \begin{bmatrix}
            y_1^+ & y_2^+ & y_3^+ \\ y_1^- & y_2^- & y_3^-
        \end{bmatrix} 
        \longmapsto
        \begin{bmatrix}
            y_1^+ & y_2^+ & y_3^+ \\ y_1^- & -y_3^- & -y_2^-
        \end{bmatrix}
        \text{ resp. }
        \begin{bmatrix}
            y_1^+ & y_3^+ & y_2^+ \\ y_1^- & y_2^- & y_3^-
        \end{bmatrix}.
    \end{align}
    The action of the other $v_ir_i$s and $h_ir_i$s is similar. 
    The result is now an easy exercise.
\end{proof}

 \begin{lemma} 
    \label[lemma]{lem:generators_of_WD6}
    Any of the following collections generate all of $W(\TypeD_6)$:
    \begin{enumerate}
        \item the group $\FRI$ and subset \(\FRR\);
        \item the groups \(\FRI\), \(\FRTE\), and any single element in
            \(\FRR\);
        \item the groups \(\FRI\), \(\FRTV\), and any single element in
            \(\FRR\).
    \end{enumerate}
\end{lemma}
\begin{proof}
    We use the notation in the proof of
    \Cref{lem:Regge_group_in_a_narrow_sense}.
    Considering the permutation action on the $y^\pm$s and ignoring signs gives
    a map $        W(\TypeD_6) \longrightarrow  \FRS_6$. We will first of all describe the images of $\FRI, \FRR, \FRTE$
    under this map:
    \begin{itemize}
        \item Orientation reversal symmetries   have image in $\FRS_6$
            given by a transposition that switches $y_i^+$ with $y_i^-$
            for one value of $i$.  
            Consequently, the image of of $\FRI$ in $\FRS_6$
            is isomorphic to $(\Z/2)^3$.

        \item  Regge symmetries in \(\FRR\) induce transpositions on 
        the $y^+$s while fixing all $y^-$ coordinates.  

        \item  Oriented tetrahedral symmetries in \(\FRTE\): these 
        stabilize the $y^+$s and $y^-$s and act on them according
        to the {\em same} permutation of  the index set $\Set{1,2,3}$.
    \end{itemize}

    Conjugating orientation reversal by Regge symmetries, we can produce any
    transposition swapping  any one of  $\{y_1^+, y_2^+, y_3^+\}$ with any one of
    $\{y_1^-, y_2^-, y_3^-\}$, and then further conjugating by orientation symmetries,
    we can produce all transpositions.  So, the   group generated by $ \FRI$
    and $\FRR$ surjects onto $\FRS_6$,
    with kernel a permutation-invariant subgroup of $(\Z/2)^6_0$
    (the subscript means elements of sum zero). The order of this kernel is at
    least $8$, since $\FRI$ has order $2^6$ but its image in $\FRS_6$ has order
    $2^3$; therefore, this kernel is all of $(\Z/2)^6_0$.

    Clearly \(\FRR\) is contained in the group generated by a single Regge
    symmetry and \(\FRTE\), so the second claim follows. The third claim
    then follows from \Cref{lem:some_subgroups_of_WD6}.
\end{proof}

\begin{remark}  \label[remark]{PVol}
    Here are two manifestations of $W(\TypeD_6)$ in the geometry of Euclidean tetrahedra. 

    Firstly, the volume function is $W(\TypeD_6)$-invariant:
    consider a Euclidean tetrahedron with edge lengths $\ell_{ij}$.
    It is natural to consider the element of the Lie algebra
    of $\SL_2^{\tetraE}$ whose diagonal entries, in the $ij$-th coordinates, are
    $\pm \ell_{ij}$.
    Its image under $\tau$ is then an element of the Lie algebra
    $\mathfrak{spin}_{12}$.
    Now, the squared volume of the tetrahedron is a polynomial in $\ell_{ij}^2$;
    in fact, {\em this volume-square comes from a conjugacy-invariant polynomial on
    $\mathfrak{spin}_{12}$.} We first learned of this from an answer of A. Goucher on the website Math Overflow:
    \cite{GouMO}.
 
    Secondly, the constraint on dihedral angles is $W(\TypeD_6)$-invariant:
    for such a tetrahedron, choose an odd function $\theta_{ij}$ on oriented
    edges so that $\abs{\theta_{ij}}$ gives the $ij$ dihedral angle and put
    \begin{align}
        \mathbf{x} = \text{image of $e^{i \theta_{ij}} \in
            \OddFunctions \otimes \mathbb{C}^{\times}$ inside $\TypeD_6 \otimes
        \mathbb{C}^{\times}$}.
    \end{align}
    Then \cite[\S~3, (2)]{Poonen2020} says precisely that the constraint to
    form the dihedral angles of a Euclidean tetrahedron takes the form $P \equiv 0$
    where $P$ is a {\em $W(\TypeD_6)$-invariant}  regular function on the
    torus $\TypeD_6 \otimes \mathbb{C}^{\times}$. In fact, $\TypeD_6 \otimes
    \mathbb{C}^{\times}$ is the maximal torus of the  group
    $\mathrm{Spin}_{12}$, and in this way we can regard $\mathbf{x}$ as an
    element of the compact group of this type. With this identification,  the
    polynomial $P$ is given by a linear combination of the characters of the
    trivial, half-spin, adjoint, and symmetric square representations. 
\end{remark}

\subsection{The ``vertex'' and ``face'' half-spin representations}
\label{sub:halfspin}

The spin representation of $\Spin_{12}$, and 
its relation to the geometry of the tetrahedron through the prior discussion,
will be particularly relevant for us later on, and we set up notation now.

Recall that, given an element of \(\Spin_{2n}\) whose image in
\(\SO_{2n}\) has standard eigenvalues \(\lambda_1^{\pm}, \ldots,
\lambda_n^{\pm}\), its eigenvalues in the
two half-spin representations are collectively given by
\begin{align}
    \prod_{i=1}^n \lambda_i^{\pm \OneHalf}.
\end{align}
One half-spin involves those eigenvalues with an
\emph{even} number of \(-1\) signs, and the other half-spin involves the
remaining eigenvalues, i.e., those with an \emph{odd} number of \(-1\) signs.

In the case of $\Spin_{12}$, both the half-spin representations
are of dimension $32$. 
Each half-spin gives rise to a set of weights --- a collection of elements in
the dual of $\TypeD_6$;
and so also, by means of
$\OddFunctions \rightarrow \TypeD_6$,  to a set of $32$ weights in the dual of $\OddFunctions$.
For one of the half-spins, which we call $S$, this set of weights consists of all functionals 
$$f \mapsto \pm f(ij) \pm f(ik) \pm f(il),$$ that is, where
we sum $f$ over three edges that meet at a vertex.
For the other half-spin, the set of weights is similarly defined but
now involving three edges that span a face.   We will be interested in the first, or ``vertex'',
half-spin representation $S$. Note that:

\begin{itemize}
    \item The pullback of $S$ via $\tau: \SL_2^{\tetraE} \rightarrow \Spin_{12}$
        is given by the direct sum of four copies of $\mathbb{C}^2 \boxtimes
        \mathbb{C}^2 \boxtimes \mathbb{C}^2$,
        where we tensor together the standard representations for each triple of edges that meet in a common vertex. 

    \item  If we choose 
        $\bx \in \OddFunctions \otimes \Gm$, specified by a set of coordinates
        $x_{ij}$ satisfying $x_{ij}x_{ji}=1$,
        the eigenvalues of the corresponding toral elements on $S$ are
        the $32$ products
        \begin{align} \label{32}
            x_{ij}^\pm x_{ik}^\pm x_{il}^\pm,
        \end{align}
        where \(\Set{ij,ik,il}=\tetraO_i\) are edges sharing one vertex \(i\).
\end{itemize}

There is an important splitting 
into a direct sum of Lagrangian subspaces
\begin{align} \label{eqn:S_into_Splus_Sminus}
    S=S^+\oplus S^-,
\end{align}
stable under the diagonal torus of
$\SL_2^{\tetraE}$.
Namely, we take $S^+$ to contain  all eigenspaces arising from eigenvalues
\eqref{32} where there are at least two $+$ signs amongst the exponents;
similarly, we define $S^-$ to correspond to eigenvalues where there are at least
two negative signs amongst the exponents.

\section{The main theorems} \label{sec:MainTheorems}
 
In this section, we formulate the two main results of the present paper. The
first (\Cref{thm:W_D6_symmetry}) shows that $\Pisymbol$   enjoy a
$W(\TypeD_6)$-symmetry in the principal series case. The second
(\Cref{thm:main_duality_theorem}) gives an explicit evaluation when
$\Pi$ is additionally unramified. Both theorems give explicit
ways, within the contexts to which they apply, of resolving the sign ambiguity of $\Pisymbol$.

\subsection{Weyl symmetry for principal series}
\label{sub:the_principal_series_case}
Let us  assume $\RO$ is split and consider now the tetrahedral symbol on
principal series representations.
As we will see,  in this situation, we will be able to resolve
the sign ambiguity and also will find a symmetry by the Weyl group of $\TypeD_6$.

We will first fix a ``nice'' parameterization of
such representations by character-valued \emph{odd}
functions on oriented edges:
\begin{align} \label{PSdef}
    \PrincipalSeries \defeq  
    \Set{\chi_{ij}  \in \widehat{F^\x}\given
    \chi_{ij}\chi_{ji}=1},
\end{align}
where \(\widehat{F^\x}\) denotes the set of characters (i.e., unitary
quasi-characters) of \(F^\x\).  As observed
after \eqref{taudef2}, the space $\PrincipalSeries$
can be identified with the
\emph{tensor product of abelian groups}:
\begin{align} \label{CharacterOddFunction}
    \OddFunctions\otimes\widehat{F^\x}
    &\stackrel{\sim}{\longto}\PrincipalSeries,\\
    f\otimes\lambda &\longmapsto \chi_{ij}=\lambda^{f(ij)}
\end{align} 

For such $\chi_{ij}$, let \(\pi_{ij}\) be the
principal series attached to \(\chi_{ij}\), as has been defined in
\Cref{DefinitionNontempered}; note that the isomorphism class of
\(\pi_{ij}\) does not depend on the order of \(i\) and \(j\).
By analogy with the standard notation for classical \(6j\) symbols, we shall write
\begin{equation}
    \begin{Bmatrix} 
        \chi_{\V{12}} & \chi_{\V{13}} & \chi_{\V{14}} \\ 
        \chi_{\V{34}} & \chi_{\V{24}} & \chi_{\V{23}} 
    \end{Bmatrix}=\chisymbol
\end{equation}
for $\Pisymbol$ in this case. It turns out that $\Pisymbol$ is not identically
zero on any components where all $\pi$ are principal series. By \Cref{prop:AC},
the square of this symbol
extends from $\PrincipalSeries$ to a meromorphic function on the
``complexified'' space
$\PrincipalSeries_{\mathbb{C}}$, defined analogously to $\PrincipalSeries$
but dropping the unitarity requirement. 

In order to make better sense of this, we will need to consider the   map
\begin{equation} \label{InducedMap}
    \OddFunctions \otimes \widehat{F^{\times}} 
    \longto \TypeD_6 \otimes \widehat{F^{\times}}.
\end{equation}
induced by  $\OddFunctions \rightarrow \TypeD_6$ from
\eqref{TheIsogeny}.
The right hand side can be considered as collections of characters $\psi_1,
\ldots, \psi_6$, {\em together with  a chosen square root of the product $\psi_1\cdots
\psi_6$}; the morphism \eqref{InducedMap} sends $\chi_{ij}$ to the $\psi$s given
by $\chi_{\V{12}}\chi_{\V{34}}^{\pm}, \chi_{\V{13}} \chi_{\V{24}}^{\pm},
\chi_{\V{14}} \chi_{\V{23}}^{\pm}$ together with the square root of their
product given by $\chi_{\V{12}} \chi_{\V{13}}\chi_{\V{14}}$. In particular, all ``vertex''
products $\chi_{ij}^{\pm} \chi_{ik}^{\pm} \chi_{il}^{\pm}$, as well as all ``face''
products of the form $\chi_{ij}^{\pm} \chi_{jk}^{\pm} \chi_{ki}^{\pm}$, depend
only on the image of $\chi_{ij}$ under \eqref{InducedMap}.

To efficiently state our main results, we introduce the following set of $32$ characters:
\begin{align} \label{LSdef}
    \Sigma = \Set*{\chi_{ij}^{\pm } \chi_{ik}^{\pm } \chi_{il}^{\pm }\,\big|\, i\in\tetraV}.
\end{align}
where for each of the four $i \in \tetraV$, we allow all possible choices of the three signs.
This $\Sigma$ is visibly an avatar of the weights of the half-spin representation  
introduced in \Cref{sub:halfspin};
the connection will be made even clearer in \Cref{sub:review_of_local_Langlands_correspondence}.
We divide $\Sigma = \Sigma^+ \coprod \Sigma^-$,
where \(\Sigma^+\subset \Sigma\) contains all terms that involve either
three \(+1\)'s in the exponents, or two \(+1\)'s and one \(-1\), whereas
\(\Sigma^-=\Sigma-\Sigma^+\).
Write $\gamma(s,\Sigma), \gamma(s, \Sigma^+), \gamma(s, \Sigma^-)$
for the corresponding $\gamma$-functions, where we follow the notation of
\Cref{MultisetNotation} 
and define $\gamma(s, -)$ for a multiset to be the product of the constituent $\gamma$-functions; 
we have
\begin{align} \label{gammaS}
    \gamma(s, \Sigma) = \gamma(s, \Sigma^+)\gamma(s, \Sigma^-).
\end{align} 

The following is our first main theorem:
\begin{theorem}
    \label[theorem]{thm:W_D6_symmetry}
    We have the following results:
    \begin{enumerate}
        \item The value of $\chisymbol^2$ depends only on  
              three-fold products
               $\chi_{ij}^{\pm} \chi_{ik}^{\pm} \chi_{il}^{\pm}$ and in particular 
            the rule $\chi \mapsto \chisymbol^2$ factors through the map
            \begin{align}
                \label{eqn:Principal_to_type_D}
                \PrincipalSeries\simeq\OddFunctions\otimes\widehat{F^{\x}}
                \longto \TypeD_6\otimes\widehat{F^{\x}},
            \end{align} 
            in such a way that the extended function on the right hand side is
            invariant by the action of \(W=W(\TypeD_6)\)
            (cf.~\Cref{sub:D6_lattice} for the \(W\)-action).
        \item More precisely, the rule
            \begin{align} \label{eq:POE}
                \chi \longmapsto
                \begin{Bmatrix} 
                    \chi_{\V{12}} & \chi_{\V{13}} & \chi_{\V{14}} \\ 
                    \chi_{\V{34}} & \chi_{\V{24}} & \chi_{\V{23}}  
                \end{Bmatrix}\sqrt{\gamma\Bigl(\OneHalf,\Sigma^-\Bigr)}
            \end{align}
            can be globally defined on \(\TypeD_6\otimes\widehat{F^{\x}}\):
            there exists a meromorphic function $I$ defined on $
            \TypeD_6\otimes\widehat{F^{\x}}$ that agrees with the right-hand
            side for suitable choice of sign of $\sqrt{\gamma}$. This function
            can be chosen to
            depend only on products of the form \(\chi_{ij} \chi_{ik}
            \chi_{il}^{-1}\) and  satisfy the following $W$-equivariance:
            \begin{align}
                \label{eqn:the_cocycle_with_sign_formula}
                \frac{I(w^{-1}\chi)}{I(\chi)}
                =  \iota(w, \chi) \gamma\Bigl(\OneHalf,\Sigma^+\cap w(\Sigma^-)\Bigr),
            \end{align}
            where $\iota(w, \chi) \in \pm 1$ is characterized
            by the fact that the right hand side of \eqref{eqn:the_cocycle_with_sign_formula} is a $1$-cocycle
             of \(W\) valued in the multiplicative group of nonzero
            meromorphic functions on \(\TypeD_6\otimes\widehat{F^{\x}}\),
            and the following facts,  using notations as in \Cref{GroupTheory2}:
            \begin{enumerate}
                \item when \(w\in \FRTV\) is a tetrahedral symmetry or \(w\in\FRR\) is a Regge symmetry,
                    \(\iota(w,\chi)=1\);
                \item when \(w\) is the element \(s_{ij}\) that exchanges
                    \(\chi_{ij}\) with \(\chi_{ji}\),
                    \(\iota(w,\chi)=\chi_{ik}\chi_{il}\chi_{jk}\chi_{jl}(-1)\),
                    where \(\Set{i,j,k,l}=\tetraV\).
            \end{enumerate}
    \end{enumerate}
\end{theorem}

A proof is sketched in \Cref{SketchProof}. The proof proper appears
in \Cref{thmD6proofsection}.

\begin{remark}
    \begin{enumerate}
        \item We note that the extension of \(\chisymbol\) to
            \(\TypeD_6\otimes\widehat{F^{\x}}\) is no longer the tetrahedral
            symbol for \(\RO\), because \eqref{eqn:Principal_to_type_D} is not
            surjective. It would be interesting to interpret the extension in
            terms of tetrahedral symbols for a spin group rather than $\RO$,
            i.e., in the noncompact case, \(\SL_2\) rather than $\PGL_2$, cf.
            \Cref{sub:symbol_for_SU2}.
 
        \item It is not a formality that a choice of signs for \(\iota(w,\chi)\)
            exists. It is equivalent to the following fact: the $2$-cocycle on
            $W$ defined by
            \begin{align}
                (w,w') \longmapsto
                \biggl(\chi\mapsto \prod_{\substack{\psi \in \Sigma^+,
                w^{-1}\psi \in \Sigma^-\\ (ww')^{-1}\psi\in \Sigma^+}} \psi(-1)\biggr)
            \end{align}
            is cohomologically trivial.
        \item  We will define the function $I$ using a certain hypergeometric
            formula for the tetrahedral symbol (see
            \Cref{sub:edge_formula,sec:hypergeometric_formulas_for_the_tetrahedral_symbol}).
        \item The statement above does not \emph{directly} cover Regge's
            original symmetry of $6j$ symbols, which pertains to
            the case of $\RO$ \emph{compact}. We outline our
            expectations about this in \Cref{ClassicalRegge}.
    \end{enumerate}
\end{remark}

\subsection{Evaluation of $\Pisymbol$ in the unramified case} \label{Subsection:Unramified}
Suppose now that $F$ is nonarchimedean, and let \(\bbT\subset\bbC^\x\) be the unit
circle. There is an embedding of  groups
\begin{align}
    \bbT &\longto \widehat{F^{\times}}\\
    z &\longmapsto (x\mapsto z^{\val_F(x)}),
\end{align}
which, in words,  sends a complex number $z$ of absolute value $1$
to the character of $F^{\times}$ that sends a uniformizer to $z$.
The image of this homomorphism gives the unitary unramified characters
of $F^{\times}$, i.e., those that are trivial on the maximal
compact subgroup $\cO^{\times} \subset F^{\times}$.
The isomorphism \eqref{CharacterOddFunction}
identifies the subspace $\PrincipalSeries_0\subset \PrincipalSeries$ consisting of unramified
unitary characters with the tensor product
\begin{align} \label{PS0def}
    \PrincipalSeries_0 \simeq \OddFunctions \otimes \bbT,
\end{align}
and by \Cref{thm:W_D6_symmetry}, the function $\chisymbol$ (up to signs)
descends to a $W(\TypeD_6)$-invariant function on
\begin{align}
    \TypeD_6 \otimes \bbT \simeq  \text{the maximal torus of compact
    }\Spin_{12}(\bbR),
\end{align}
where the identification was that discussed after \Cref{DnDef}.

Now, the $W(\TypeD_6)$-invariant function on \(\TypeD_6\otimes\bbT\)
are precisely those arising by restriction of
class functions (i.e., 
  conjugacy-invariant continuous functions) on $\Spin_{12}(\bbR)$.
It is therefore, we hope, irresistible to ask
for a description of this class function in terms of the
representation theory of $\Spin_{12}(\bbR)$. Since the representation theory of
the compact group \(\Spin_{12}(\bbR)\) is equivalent to the theory of
\emph{algebraic} representations of its complexification (simply denoted as
\(\Spin_{12}\)), we will, from now on, switch to the latter to better align
ourselves with the more general discussions in
\Cref{TetrahedralAsInterface,sub:review_of_local_Langlands_correspondence}.

As discussed in \Cref{sub:halfspin},
there is a $32$-dimensional half-spin representation \(S\) of $\Spin_{12}$
whose weights pull back under $\OddFunctions \to \TypeD_6$
to the  $32$ linear functionals
\begin{align}
    \{ \pm f(ij)  \pm f(ik) \pm f(il) \}  
\end{align}
where \(\Set{ij,ik,il}=\tetraO_i\) are edges sharing starting vertex \(i\).   
It contains a distinguished $16$-dimensional
\emph{cone  of pure spinors}  $P \subset S$ (for details see \Cref{sub:the_dual_side})
viewed as a complex subvariety of \(S\). Its ring of regular functions
\(\bbC[P]\) then affords a weighted \(\bbC^\x=\Gm(\bbC)\)-action induced by the
scaling on \(S\).

\begin{theorem}
    \label[theorem]{thm:main_duality_theorem}
    Suppose $F$ is nonarchimedean and $\chi\in\PrincipalSeries_0$
    with image $\Frob\in \TypeD_6 \otimes \bbT$. Then
    \begin{align}
        \label{eqn:main_duality}
        \chisymbolBig  = \frac{ (1-q^{-2})^{3}}{\sqrt{L(\OneHalf,\Sigma)}}
        \Tr\bigl(q^{-\OneHalf}\Frob, \bbC[P]\bigr),
    \end{align}
    where \(q^{-\OneHalf}\) acts on \(\bbC[P]\) through the weighted action of
    \(\bbC^\x\), and $\Sigma$ is as in \eqref{LSdef}.
\end{theorem}

A proof is sketched in \Cref{SketchProof} and detailed  in
\Cref{sec:computations_for_unramified_principal_series}.

In words, up to normalizing factors (see \Cref{section:geometric_to_numerical}
for a geometric point of view concerning the factor \((1-q^{-2})^3\)), the tetrahedral symbol is
\emph{the weighted character of the Frobenius action on the algebraic function
ring of the spinor cone.} Why are the half-spin representation and the spinor
cone involved? Relative Langlands duality offers at least a context in which to
understand this, explaining, for example, why the spinor cone is actually
Lagrangian inside the half-spin; see \Cref{TetrahedralAsInterface} for
discussion.

\subsection{A sketch of the proofs of \Cref{thm:W_D6_symmetry,thm:main_duality_theorem}}
\label{SketchProof}
In \Cref{sub:edge_formula}, we will write down an explicit integral formula (the
``edge formula'') for $\chisymbol$;
this formula is proved in 
\Cref{sec:edge_formula_for_principal_series}  by explicit computation with principal series models,
and converted to a more concrete form in
   \eqref{eqn:hypergeometric_edge_integral}.
    This formula makes clear
that \(\chisymbol^2\) depends
only on characters of the form \(\chi_{ij}^\pm\chi_{ik}^\pm\chi_{il}^\pm\). 
  This
proves that \(\chisymbol^2\) extends over
\eqref{eqn:Principal_to_type_D}. 

For the \(W(\TypeD_6)\)-symmetry, we construct
explicit symmetries of \(\chisymbol\) that generate the whole Weyl group,
corresponding to the generators discussed in \Cref{GroupTheory2}.
 The least trivial ones are the Regge symmetries, which come
from applying a general Fourier duality result for hypergeometric-type
integrals; see Proposition \ref{prop:Fourier_duality}
for a succinct abstract formulation of that duality.
This proves the symmetry up to signs. The signs can then be pinned
down using the Mellin transform of the hypergeometric integral
from \Cref{sub:Mellin_of_tetrahedral_symbols}.
See \Cref{sec:proof_of_Weyl_symmetry} for the details.

Now consider \Cref{thm:main_duality_theorem}.
The right-hand side of \eqref{eqn:main_duality} can be computed using the Weyl
character formula after we decompose \(\bbC[P]\) into irreducible
\(\Spin_{12}\)-representations. To analyze the left-hand side
we rely on another integral formula (the
``vertex formula'')\footnote{In principle we might  try to use the hypergeometric formula as well, but 
the vertex formula seems to be easier to work with for this purpose. } for
\(\chisymbol\) that we will prove in \Cref{sub:spherical_vertex_formula};
it leads to a computation involving the Bruhat--Tits tree of
\(\PGL_2(F)\). A direct comparison of both sides is thus conceptually possible
yet seems a bit tedious to do even with a computer. We will instead compare both
sides in these steps:
\begin{enumerate}
    \item Showing that both sides are meromorphic with the same set of simple
        poles;
    \item Assisted by a computer, showing that their residues agree (which is
        significantly easier than comparing the whole expressions);
    \item Showing that the difference of both sides is a bounded function with
        value \(0\) at a single input, so it must be identically \(0\).
\end{enumerate}
For details, see \Cref{sec:computations_for_unramified_principal_series}. 

\section{Integral formulas for the tetrahedral symbol}\label{sec:integral_formulas_for_tetra_symbol}
In this section we will give a class of integral formulas for the
tetrahedral symbol.  They are based on the interaction between tetrahedral combinatorics
and the geometry of  certain \(\RO\)-spaces. 

They come in two classes, which can
be seen as dual to one another: a \notion{vertex integral} 
which takes as input an $\RO$-space $X$ and a collection
of $\RO$-invariant functions \(\varphi_{ij}: X^2 \rightarrow \bbC\) indexed by
(unoriented) edges, and an \notion{edge integral} which takes
as input a $\RO$-space $Y$ and a collection of $\RO$-invariant functions
\(\psi_i: Y^3 \rightarrow \bbC\) indexed by
vertices. These are given, respectively, by
\begin{align} \label{VertexEdgeDefinition}
    \VI(X,\varphi)\defeq \int_{\RO\backslash X^{\tetraV}} \prod_{ij \in \tetraE}  \varphi_{ij}(x_i, x_j),
    \text{ and }
    \EI(Y,\psi)\defeq\int_{\RO\backslash Y^{\tetraE}} \prod_{i \in \tetraV} \psi_{i}(y_{ij}, y_{ik}, y_{il}).
\end{align}

In both forms, the group \(\RO\) acts on \(X^{\tetraV}\) or \(Y^{\tetraE}\)
diagonally, and the integrands are \(\RO\)-invariant.
The symmetry between the two constructions is a little clearer if we describe
the situation in words: To each $X$-labelling of vertices, i.e. ``tetrahedron in
$X$,'' we can attach a number, namely, the product over edges $e$, of the values
of $\varphi_e$ on vertices incident with $e$. We obtain $\VI$ by integrating
this number over the space of $X$-labellings of vertices. For $\EI$ we just
switch  $X$ with $Y$, vertices with edges, and $\varphi$ with $\psi$.

\subsection{Vertex formula for $\RO$ compact}\label{sub:vertex_formula}

The following vertex formula, valid for \(\RO\) compact, was already given by
Wigner \cite[(27)]{wigner1965matrices}.

\begin{proposition}[Wigner]\label[proposition]{prop:WignerIntegral}
    Suppose that $\RO$ is compact.
    Then $\Pisymbol^2$
    is the vertex integral $\VI(X, \varphi)$ associated to $X=\RO$
    and \(\varphi_{ij}(g_i,g_j)\) the character of $\pi_{ij}$
    evaluated at $g_j^{-1}g_i$:
    \begin{align} \label{Wigner}
        \Pisymbol^2
        = \int_{\RO^{\tetraV}} \prod_{ij \in \tetraE}  \Tr\circ\pi_{ij}(g_j^{-1}g_i).
    \end{align}
\end{proposition}

\begin{proof}
    We follow the notation of \Cref{sub:definition_in_the_compact_case}.
    Let $\delta\in \Pi_G$ represent the contraction mapping, so that
    $\IPair{v}{\delta} = \Contra{v}$
    for all $v \in V$. By definition, $\Pisymbol = \IPair{V}{\delta}$.
    Now, for any vector $v, w$ we have
    \begin{align}
        \IPair{V}{v}\IPair{V}{w} = \int_{h \in H}\IPair{hv}{w},
    \end{align}
    for both sides determine $H \times H$-invariant
    functionals on $\Pi_G \otimes \Pi_G$
    with the same value $1$ at $V \otimes V$.
    Substitute $v=w=\delta$ to find:
    \begin{align}
        \Pisymbol^2 = \int_{H}\IPair{h\delta}{\delta}.
    \end{align}
    However, \Cref{lem:tracepair} below implies that for
    $h = (g_i)_{i \in \tetraV}$, we have
    \begin{align}
        \IPair{h\delta}{\delta} = \prod_{ij \in \tetraE} \Tr\circ\pi_{ij}(g_j^{-1}g_i),
    \end{align}
    which readily implies the desired formula.
\end{proof}

\begin{lemma}
    \label[lemma]{lem:tracepair}
    Suppose $W$ is a finite-dimensional complex vector space
    equipped with a nondegenerate symmetric pairing $\IPair{-}{-}$.
    Let $\delta \in W \otimes W$ represent the pairing, so that
    $\IPair{\delta}{v \otimes w} = \IPair{v}{w}$.
    Then
    \begin{equation}
        \IPair{(A\otimes B) \delta}{\delta} = \Tr(B^{-1}A)=\Tr(A^{-1}B)
    \end{equation}
    for any endomorphisms $A,B$ of $W$ that preserve
    the pairing.
\end{lemma}
\begin{proof}
    Fix an orthonormal basis $e_k$ 
    with respect to the self-pairing $\IPair{-}{-}$.
    Then  $\delta =\sum_k e_k \otimes e_k$ and
    \begin{align}
        \IPair{(A\otimes B) \delta}{\delta}
        &= \sum_{k}\IPair{A e_k\otimes Be_k}{\delta}\\
        &= \sum_k\IPair{Ae_k}{Be_k}\\
        &= \sum_k\IPair{B^{-1}Ae_k}{e_k}\\
        &= \Tr(B^{-1}A)
    \end{align}
    as claimed, and similarly for \(\Tr(A^{-1}B)\).
\end{proof}

\subsection{Vertex formula for $F$ nonarchimedean and $\pi$ unramified}
\label{sub:spherical_vertex_formula}

Suppose now that $\RO$ is noncompact and the $\pi_{ij}$ are
\emph{unramified} principal series (see \Cref{Subsection:Unramified}) induced
from \(\chi_{ij}\).
In this case, each $\pi_{ij}$ admits a one-dimensional space of vectors
invariant under the maximal compact subgroup
\begin{align}
    K \subset \RO.
\end{align}
Explicitly, in the non-archimedean case, with reference to an isomorphism $\RO
\simeq \PGL_2(F)$, we have $K \simeq \PGL_2(\cO)$.  We will henceforth
assume that $F$ is nonarchimedean (but see remark
after \eqref{edge equation 2} for the archimedean case). 

We shall describe a class of vertex formulas, valid for such representations.
Take \(X\) to be \(\RO/K\) where \(K\) is as above; 
this $X$  can be naturally given the structure of the set of vertices
of an infinite $q+1$-regular tree, where $q$ is the residue characteristic.
Fix \(v_{ij} \in \pi_{ij}\)
a nonzero  \(K\)-fixed vector normalized so that \(\IPair{v_{ij}}{v_{ij}} =1\);
this is possible because, again, the symmetric pairing is positive definite on a
real structure. Therefore $v_{ij}$ is uniquely specified up to sign. Now define
$\varphi_{ij} = X \x X \longto \bbC$ by the rule
\begin{align}
    \varphi_{ij}(g_i K, g_j K) =  \IPair{g_i v_{ij}}{g_j v_{ij}}.
\end{align} 
This is the ``spherical function for \(\pi_{ij}\)''.

\begin{proposition}
    \label[proposition]{prop:VertexIntegral}
    With the notations above, we have (up to sign)
    \begin{align}
        \label{eqn:vertex_formula}
        \begin{Bmatrix} 
            \chi_{\V{12}} & \chi_{\V{13}} & \chi_{\V{14}} \\ 
            \chi_{\V{34}} & \chi_{\V{24}} & \chi_{\V{23}} 
        \end{Bmatrix} &= 
        \frac{L(1,\mathrm{ad})}{L(2)^8\sqrt{L(\frac{1}{2},\Sigma)}}
        \cdot\VI(X,\varphi) \\
        & \left( =  \frac{L(1,\mathrm{ad})}{L(2)^8\sqrt{L(\frac{1}{2},\Sigma)}} \cdot
        \int_{g_i \in \RO^{\tetraV}}  \prod_{ij \in \tetraE}
        \IPair{g_i v_{ij}}{g_j v_{ij}}\right),
    \end{align}
    where, on the right-hand side, we take counting measure on the discrete set $X$, and
    put on $\RO$ the measure for which the volume of $K$ equals $1$, \(\Sigma\) is as
    in \Cref{LSdef}, and 
    \begin{align} \label{def:adjoint L}
        L(1,  \mathrm{ad}) = \prod_{ij \in \tetraE} 
        L(1) L(1, \chi_{ij}^2) L(1, \chi_{ij}^{-2}),
    \end{align}
    and recall the
    convention that $L(s) $ is the value of the $L$-function for the trivial
    character at $s$.
\end{proposition}

\begin{remark} \label[remark]{SignUnramifiedCase}
    In words, \Cref{prop:VertexIntegral} asserts that (up to the normalizing factors) we obtain
    the tetrahedral symbol by integrating over ``moduli of tetrahedra in $X$''
    the product of spherical functions labeled by the edges. Note that
    it implies that there is a coherent sign choice for the ``renormalized
    tetrahedral symbol'' \(L(\OneHalf,\Sigma)^{\OneHalf}\chisymbol\) when the $\pi$s vary
    through unramified representations;  our theorem \Cref{thm:main_duality_theorem} 
    computing the tetrahedral symbol in terms of a spinor cone is an explicit
    description of this sign choice.  
\end{remark}

\begin{proof}[Proof of \Cref{prop:VertexIntegral}]
    Again we follow the notation of \Cref{sub:definition_in_the_compact_case}.
    Let
    \begin{align}
        v = \bigotimes_{ij \in \tetraE} (v_{ij} \otimes v_{ij}) \in \Pi_G.
    \end{align}
    We shall compute separately \(\Lambda^H(v)\) and \((\Lambda^H)'(v)\). By
    definition,
    \begin{align}
        (\Lambda^H)'(v) = \int_{\RO \backslash H} \IPair{g_iv_{ij}}{g_jv_{ji}} =\int_{\RO\backslash X^{\tetraV}} \prod_{ij\in\tetraE} \varphi_{ij}(x_i, x_j)
        =\VI(\RO/K,\varphi).
    \end{align}

    On the other hand, \(\Lambda^H(v)\) must be evaluated by hand.
    It is known (a special case of \cite[Theorem 1.2]{II10}, or an easy if rather tedious
    computation in the case at hand) that
    \begin{align} \label{edge equation 2}
        \Lambda^H(v)
        =  L(2)^8 \frac{\sqrt{L(\OneHalf,\Sigma)}}{L(1,\mathrm{ad})}.
    \end{align}

    Comparing this with the definition \((\Lambda^H)' = \Pisymbol \Lambda^H\),
    we arrive at the claimed formula.
\end{proof}

The same reasoning also works in the case of $F$ archimedean, but we do not know
a reference for the computation \eqref{edge equation 2} there.

\subsection{Edge formula for principal series}\label{sub:edge_formula}

The following edge formula will be applicable to the case of the \emph{split} \(\RO\) and
the \(\pi_{ij}\) principal series induced by character \(\chi_{ij}\), for which we retain the notations from
\Cref{sub:the_principal_series_case}. 

We take \(Y=\bbP_F^1\). The \(\psi_i\) will
not be functions   on $Y \times Y
\times Y$ itself, but rather sections
of certain line bundles --- they will be functions on $F^2 \times F^2 \times
F^2$, with certain degrees of homogeneity in each factor.   Their {\em product} will define a \((-2)^\tetraE\)-homogeneous function on
\((F^2)^{\tetraE} \), i.e. a {\em density} --- this can be integrated over $Y^{\tetraE} \simeq
(\bbP_F^1)^{\tetraE}$ as explained in \Cref{ProjectiveIntegrationReview}.
 
The functions \(\psi_i\) are defined (as sections of suitable bundles)  on that open subset of
\(Y^{\tetraE}\) where adjacent edges are assigned distinct labels (in \(Y\)), 
in the following way:
 \begin{align}
    \psi_i = \psi_i^j \psi_i^k \psi_i^l,
\end{align}
where \( \Set{i,j,k,l}=\tetraV\) and
\begin{align}
    \psi_i^l \defeq \chi_{i-}^l\bigl(y_{ij}\wedge y_{ik}\bigr)
    = (\chi_i^l \abs{\blank}^{-\OneHalf}) \bigl(y_{ij}\wedge y_{ik}\bigr),
    \text{ and }
    \chi_i^l\defeq\frac{\chi_{ij}\chi_{ik}}{\chi_{il}}.
\end{align}
Here we write $x \wedge y$, where $x, y \in F^2$, for the determinant of the $2 \times 2$
matrix with $x$ as first row and $y$ as second row. This introduces a sign
ambiguity in the definition of \(\psi_i^l\) due to the order of \(j\) and
\(k\). There is a ``good'' way to make such a choice;
we will suppress this for now, but for details
we refer the reader to
\Cref{lem:edge_integrand_with_order_convention}.

\begin{proposition}
    \label[proposition]{prop:edge_integral}
    With the notations above, we have
    \begin{align}
        \Pisymbol \sqrt{\gamma\Bigl(\frac{1}{2}, \Sigma^-\Bigr)}
        &= \nu_\bbP^{-2} \EI(\bbP_F^1,\psi) \\
        & = \nu_\bbP^{-2}
        \int_{\RO \backslash (\bbP^1)^{\tetraE}} \prod_{il \in \tetraO}
        \left(\frac{\chi_{ij}\chi_{ik}}{\chi_{il}}\abs{\blank}^{-\OneHalf}\right)
        \bigl(y_{ij}\wedge y_{ik}\bigr),
        \label{eqn:edge_integral_formula}
    \end{align}
    where \(\nu_\bbP\) is as in \eqref{eqn:Haar_vs_P13_measures}.  The integrals
    are \emph{absolutely convergent} for $\chi$ unitary.
    (Note that the second expression is simply the explication of the notation;
    in it,  the $y$s are coordinates on $(\bbP^1)^{\tetraE}$, and
    \(\Set{i,j,k,l}=\tetraV\).)
\end{proposition}

An outline of the proof towards this formula will be sketched shortly in
\Cref{sub:outlinepage}, and the details will be contained in 
\Cref{sec:edge_formula_for_principal_series}.

Although \eqref{eqn:edge_integral_formula} is  compact, it is a little
opaque, so we offer two different reinterpretations
of it. The first will be given now, and the second in
\Cref{sec:hypergeometric_formulas_for_the_tetrahedral_symbol}. This first interpretation will also help us see
that the formula is absolutely convergent when $\psi$ is unitary. 
 
\subsubsection{Interpretation of \eqref{eqn:edge_integral_formula} in terms of moduli of six points on the projective line}
\label{sec:Moduli6Points}

Let us consider the space $\mathbf{M}^{\circ}$ of configurations of six points
on $\mathbb{P}^1$ indexed by $\tetraE$ where  the points indexed by  incident
edges are distinct (thus, the points indexed by opposite edges may collide).
$\PGL_2$ acts freely on this space, and upon taking the quotient we arrive at a
moduli space 
\begin{align}
    \mathcal{M}^{\circ} = \mathbf{M}^{\circ}/\PGL_2,
\end{align} 
which more concretely amounts to configurations of three points with certain distinctness properties, 
see \eqref{cMExplicit} below.

We first define a section $\Omega$ of the square of the canonical sheaf $\omega^2$
on $\mathcal{M}^{\circ}$. First of all,
consider the formula
 \begin{align} \label{OmegaMMdef}
    \Omega_{\mathbf{M}} =   \frac{\left( \bigwedge_e \dd z_e
    \right)^2}{\prod_{e\neq e'} (z_e \wedge z_{e'})},
\end{align}
where \(e,e'\) range over pairs of edges sharing a vertex and \(\dd z_e\) is
defined to be \(x\dd y-y\dd x\) if \(z_e\) has homogeneous coordinates
\([x,y]\); the overall sign depends
on the ordering in the numerator and denominator, but this will not matter for us. 
It defines a section of the square $\omega^2$ of the canonical bundle
on $\mathbf{M}^{\circ}$. 
The form $\Omega_{\mathbf{M}}$ is $\PGL_2$-invariant, and we can ``divide
it'' by the square of the fixed algebraic volume form on $\PGL_2$ (see \eqref{omegaA1def})
to arrive at a corresponding section $\Omega$ of the square of the canonical sheaf
on $\cM=\mathbf{M}/\PGL_2$ itself.

\begin{remark} \label[remark]{MExplicitRemark}
    For later use we explicate both $\cM^{\circ}$ and $\Omega$.
    We may identify
    \begin{equation} \label{cMExplicit}
        \cM^{\circ} \simeq \{ (x,y,z) \in \bbP^1: x \neq y, y \neq z, 
        x \notin \{1, \infty\}, y \notin \{\infty, 0\}, z \notin \{0,1\} \}
    \end{equation}
    
    To do so, we use the $\PGL_2$-action to identify 
    $\cM^{\circ}$ with the subset of $\mathbf{M}^{\circ}$ where $z_{\V{14}}=0,
    z_{\V{13}}=1, z_{\V{23}}=\infty$, and take $x=z_{\V{12}}, y=z_{\V{24}},
    z=z_{\V{34}}$ for the remaining coordinates. Upon comparing
    \eqref{OmegaMMdef} with \eqref{omegaA1def}, we deduce that the form $\Omega$ is given
    by
    \begin{align}\label{OmegaExplicit}
        \Omega = \frac{(\dd x \wedge \dd y \wedge \dd z)^2}
        {xyz (x-y) (y-z) (1-z)(1-x)}.
    \end{align}
\end{remark}

Next, we define a morphism
\begin{align} \label{mu inner def}
    \mu\colon\mathcal{M}^{\circ} \rightarrow \OddFunctions \otimes \Gm
    \simeq \Gm^6,
\end{align}
by requiring that, for $ij \in \tetraO$, the $ij$ coordinate of $\mu$ is given by the following rule:
move $z_{ij}$ to $\infty$ by means of a projective transformation, and then set
\begin{equation} \label{mudef}
    \mu_{ij} = \frac{z_{jk}\wedge z_{jl}}{z_{ik}\wedge z_{il}},
\end{equation}
where \(\Set{i,j,k,l}=\tetraV\).
Again, there are choices of ordering
here; we make them so that $\mu_{ij} \mu_{ji}=1$. This is not unique, but
different choices only affect the $\mu_{ij}$s by a sign.

Now $\abs{\Omega}^{\OneHalf}$ makes sense as a a volume form  on $\mathcal{M}^{\circ}(F)$
for $F$ a local field, so long as we fix a Haar measure on $F$.\footnote{Indeed, for
each point $z_0 \in \mathcal{M}^{\circ}(F)$ we choose an analytic isomorphism
$\Phi:U\simeq U_{z_0}$ from an open subset $U \subset F^n$ onto an analytic
neighbourhood $U_{z_0}$ of $z_0$ in $\mathcal{M}^{\circ}(F)$; analytic means that it is
given by convergent $F$-valued power series. The pullback of $\Omega$ by means
of this form has the shape $A(x_1, \dots, x_n) (\dd x_1 \wedge\dots \wedge
\dd x_n)^2$ for an analytic function $A$.   The integral of $f$ over the analytic
neighbourhood $U_{z_0}$ is then declared to be the integral of $(f \circ \Phi)
\times \abs{A}^{\OneHalf}$ against Haar measure on $U$.}
This permits us to state the following:

\begin{proposition} \label[proposition]{SixPointsModuli}
    We have an equality, up to sign, 
    \begin{align}
        \label{eqn:edge_integral_formula_M06}
        \EI(\bbP_F^1,\psi)=\nu_\bbP\int_{\cM^{\circ}(F)}\prod_{i<j}\chi_{ij}(\mu_{ij}(z)) \cdot
        \abs{\Omega}^{\OneHalf}
    \end{align}
    where \(\nu_\bbP\) is as in \eqref{eqn:Haar_vs_P13_measures}, $\mu$ is as in
    \eqref{mudef} and
    \(\mu_{ij}\) is the coordinate of \(\mu\) corresponding to
    \(ij\in\tetraE\). 
\end{proposition}

\begin{proof}
    This is simply a rewriting of the second line of
    \eqref{eqn:edge_integral_formula}.
    In carrying out this verification,
    it is convenient to use  a more intrinsic formulation of \eqref{mudef}:
    if \(z_{ij}\)   are homogeneous
    coordinates of \(\bbP^1\) corresponding to \(ij\), then
    \begin{align}
        \mu_{ij}=\frac{  (z_{ij} \wedge z_{ik}) (z_{ij} \wedge z_{il})  (z_{jk} \wedge z_{jl})}
        {(z_{ij} \wedge z_{jk}) (z_{ij} \wedge z_{jl}) (z_{ik} \wedge z_{il})}
    \end{align}
    where \(\Set{i,j,k,l}=\tetraV\);
    we recover \Cref{mudef} by moving $z_{ij}$ to $\infty$. Yet again, the sign of
    these wedges  above is of no concern to us. Finally, the factor \(\nu_\bbP\)
    comes from the fact that in \(\EI(\bbP_F^1,\psi)\) we used the Haar measure
    on \(\RO\) rather than the measure induced by \((\bbP^1)^3\) to
    define the measure on \(\RO\backslash(\bbP^1)^\tetraE\).
\end{proof}

With this intepretation at hand, we will prove the convergence claim of
\Cref{prop:edge_integral} in \Cref{sub:pf_of_edge_integral_convergence}.

\subsection{Outline of the argument for \eqref{eqn:edge_integral_formula}}
\label{sub:outlinepage}
The proof requires us to be careful about isomorphisms between different models
for the same representation, in a way that will seem  excessively pedantic. However this is the price
we pay for having a very short definition of the symbol:  complexity gets transferred
to the set-up phase of computations.

For each $ij \in \tetraO$, let $\pi_{ij}$ be the induced representation from
$\chi_{ij}$; it is realized in the space of sections of a line bundle on
$\mathbb{P}^1$, as described in detail in \Cref{sub:pseries}. Note that $\pi_{ij} \neq \pi_{ji}$, but they are isomorphic; in
fact, by definition of the principal series model, the fact that $\chi_{ij}
\chi_{ji}=1$ will give rise to a natural pairing
\begin{align}
    \pi_{ij} \times \pi_{ji} \longrightarrow \C.
\end{align}

Fix a splitting $\tetraE \to \tetraO$, that is to say, a choice of orientation
of each edge. For each such chosen orientation $e=ij \in \tetraO$, we fix  a
symmetric pairing on $\pi_{ij}$ and an isomorphism $I_e\colon \pi_{ij} \to
\pi_{ji}$, and we transport the symmetric pairing from $\pi_{ij}$ to $\pi_{ji}$ by means of the fixed isomorphism $I_e$.

Put 
\begin{gather}
    \Pi_G = \displaybigboxtimes_{ij\in\tetraO} \pi_{ij}=
    \displaybigboxtimes_{ij \in \tetraE} (\pi_{ij} \boxtimes \pi_{ji}),\\
    \Pi'_G = \displaybigboxtimes_{ij \in \tetraE}  (\pi_{ij} \boxtimes \pi_{ij}),
\end{gather}
where we use the splitting to choose an ordering in the second definition. 
Both $\Pi_G,\Pi'_G$ are endowed with symmetric self-pairings
arising from those on $\pi_{ij}$.
The representations $\Pi'_G$ and $\Pi_G$, together with these pairings,
are isomorphic
by means of the chosen isomorphisms $I_{e}$:
\begin{align} \label{IDEF}
    I=\otimes I_e\colon \bigl(\Pi_G',\IPair{-}{-}\bigr) \stackrel{\sim}{\longto} \bigl(\Pi_G,\IPair{-}{-}\bigr).
\end{align}

Our definition of the tetrahedral symbol is expressed in terms of $\Pi_G'$.
But unfortunately the contraction functional does not look nice in the
natural induced model for $\Pi_G'$.
This problem is remedied by switching to $\Pi_G$ instead:
it is realized as the space of sections of a line bundle on 
$(\bbP^1)^{\tetraO}$; and  a $D$-invariant
functional $\phi^D$ on $\Pi_G$ given simply by integration over
the ``diagonal'' $(\bbP^1)^{\tetraE}$, i.e. the {\em closed} $D$-orbit,
whereas an $H$-invariant functional $\phi^H$
on $\Pi_G$ is given by integration over the {\em open} $H$-orbit on $(\bbP^1)^{\tetraO}$.
After computing the constant of proportionality relating $I^* \phi^H$ and $\Lambda^H$,
as well as that relating $I^* \phi^D$ and $\Lambda^D$,
we are reducing to answering the question: when we average
$\phi^D$ over $H$, which multiple of $\phi^H$ do we get?

When we average  $\phi^D$ over $H$, we in effect are pushing forward by means of
the map
\begin{align}
    H \times (\bbP_F^1)^{\tetraE} \longrightarrow (\bbP_F^1)^{\tetraO}.
\end{align}
From this analysis, it follows that $\Pisymbol$ is an integral over the fibers
of this map. Now, the projection map from a general fiber onto
$(\bbP_F^1)^{\tetraE}$ is readily verified to be a birational morphism. This gives us an
expression for $\Pisymbol$ as an integral over \((\bbP_F^1)^{\tetraE}\). After
we work out the details in \Cref{sec:edge_formula_for_principal_series} we get
\eqref{eqn:edge_integral_formula}.

\section{Hypergeometric formulas for the tetrahedral symbol}
\label{sec:hypergeometric_formulas_for_the_tetrahedral_symbol}

In this section, we state some results relating the tetrahedral symbol to
special function theory. It generalizes the
well-knwon picture about the classical \(6j\) symbol; namely, there is an
explicit formula expressing the \(6j\) symbol as a generalized
hypergeometric series of type \(\pFq{4}{3}\), evaluated at the special point
\(z=1\).

Here we will focus on the case where \(\RO\) is split and \(\pi_{ij}\) are
irreducible principal series. We will first describe integral formulas for the
tetrahedral symbol of hypergeometric type (\Cref{thm:6j=Hypergeometric}), and in
the case
$F=\R$ we will simplify the resulting formula to a sum of $\pFq{4}{3}$
hypergeometric series evaluated at the special point $z=1$.

Compared to previous sections, this one focuses more on the explicit
formulas themselves and may be of independent interest.
The reader will also find some (loosely) related references in
\Cref{sub:symbol_for_SU2}.

\subsection{Hypergeometric functions on Grassmannians} \label{sec:HypergeometricIntro}

Now let \(k < n\) be integers. Let \(\ul\chi=(\chi_1, \ldots, \chi_n)\) be \(n\) characters of \(F^{\x}\)
with the property that
\begin{align}
    \prod_i \chi_i = \abs{\blank}^{-k}.
\end{align}
Let \(X \subset F^n\) be a \(k\)-dimensional subspace.
Then the integral
\begin{align} \label{IXchidef}
    \int_{\mathbb{P}X} \ul \chi := \int_{\mathbb{P}X} \prod_i \chi_i(x_i).
\end{align}
makes sense at least formally, as the function is \(-k\) homogeneous
and can be integrated over the projective space \(\mathbb{P}X\),
according to the general procedure of
\Cref{ProjectiveIntegrationReview}.
Such functions for \(F=\bbR\) were studied by Aomoto, Gelfand and others.
They can be considered as generalizations of Gauss's hypergeometric
function, which corresponds to the case when \(k=2, n=4\), and
\(X\) is spanned by the rows of the matrix
\begin{align}
    \begin{bmatrix}
        1 & 0 & 1 & 1\\
        0 & 1 & 1 & t
    \end{bmatrix}
\end{align}
where \(t\in\bbR\) is a fixed parameter.

\subsection{The hypergeometric formula for the tetrahedral symbol}
\label{sub:hypergeometric_for_tetra_symbol}
Now let \(X\) be the subspace of \(F^8\) (so that \(k=4,n=8\)) of the form
\begin{align}
    X = (x-y,y-z,z-w,w-x, x,y,z,w) \subset F^8,
\end{align}
and  define an $8$-tuple of quasicharacters
\begin{align} \label{ChiTupleDef}
    \ul\chi = 
    \left(
        \chi_{\V2-}^{\V3},
        \chi_{\V4-}^{\V1},
        \chi_{\V3-}^{\V2},
        \chi_{\V1-}^{\V4},
        \chi_{\V1-}^{\V3},
        \chi_{\V4-}^{\V3},
        \chi_{\V4-}^{\V2},
        \chi_{\V1-}^{\V2}
    \right),
\end{align}
where we use notation similar to that of \Cref{sub:edge_formula}:
\begin{align}
    \chi_{i}^l \defeq \frac{\chi_{ij}\chi_{ik}}{\chi_{il}}.
\end{align}
and recall the subscript ``$-$'' means that we multiply by $\abs{\blank}^{-\OneHalf}$.

\begin{theorem} \label[theorem]{thm:6j=Hypergeometric}
    Let $\chi\in \PrincipalSeries_0$, with the notations of \eqref{PS0def}, and
    let \(X\) and \(\ul{\chi}\) be as above.
    Then  the tetrahedral symbol  associated to $\chi$
    can be expressed as a hypergeometric integral,
    where we integrate $\ul{\chi}$ over the projectivization of $X$:
    \begin{align}
        \label{eqn:edge_integral_hypergeometric} 
        \chisymbolBig \cdot \sqrt{\gamma\Bigl(\OneHalf, \Sigma^-\Bigr)}  = 
        \nu_\bbP^{-1}\int_{\mathbb{P}X} \ul \chi,
    \end{align}
    and both sides are absolutely convergent for $\chi \in \PrincipalSeries_0$.
\end{theorem}

\begin{proof}[Proof of \Cref{thm:6j=Hypergeometric} using \Cref{SixPointsModuli}]   
    The right-hand side of \eqref{eqn:edge_integral_hypergeometric}, aside from
    the factor \(\nu_\bbP^{-1}\), is the following integral:
    \begin{align}
        \int_{[x,y,z,w]\in\bbP^3(F)}
        \chi_{\V2-}^{\V3}(x-y)
        \chi_{\V4-}^{\V1}(y-z)
        \chi_{\V3-}^{\V2}(z-w)
        \chi_{\V1-}^{\V4}(w-x)
        \chi_{\V1-}^{\V3}(x)
        \chi_{\V4-}^{\V3}(y)
        \chi_{\V4-}^{\V2}(z)
        \chi_{\V1-}^{\V2}(w).
    \end{align}
    which, more explicitly,  turns into the following, upon de-homogenization and expansion:
    \begin{align}
        & \mathbin{\phantom{=}}\int_{F^3}
        \chi_{\V2-}^{\V3}(x-y)\chi_{\V4-}^{\V1}(y-z) \chi^{\V2}_{\V3-}(z-1)
        \chi_{\V1-}^{\V4}(1-x) \chi_{\V1-}^{\V3}(x) \chi^{\V3}_{\V4-}(y)\chi^{\V2}_{\V4-}(z)
        \dd x\dd y\dd z \\
        &= \int_{F^3}
        \chi_{\V{12}}\left(\frac{x(1-x)}{x-y}\right)
        \chi_{\V{13}}\left(\frac{1-x}{x(z-1)}\right)
        \chi_{\V{14}}\left(\frac{x(y-z)}{(1-x)yz}\right)
        \chi_{\V{23}}\left(\frac{z-1}{x-y}\right)
        \chi_{\V{24}}\left(\frac{(x-y)z}{y(y-z)}\right)\\
        &\qquad\qquad\chi_{\V{34}}\left(\frac{y(z-1)}{z(y-z)}\right)
        \abs*{xyz(x-y)(y-z)(z-1)(1-x)}^{-\OneHalf}\dd x\dd y\dd z.
        \label{eqn:hypergeometric_edge_integral}
    \end{align}
    This amounts to the explicit content of \Cref{SixPointsModuli}, once we use
    the explicit coordinates provided by \Cref{MExplicitRemark}, drawn as below:
    
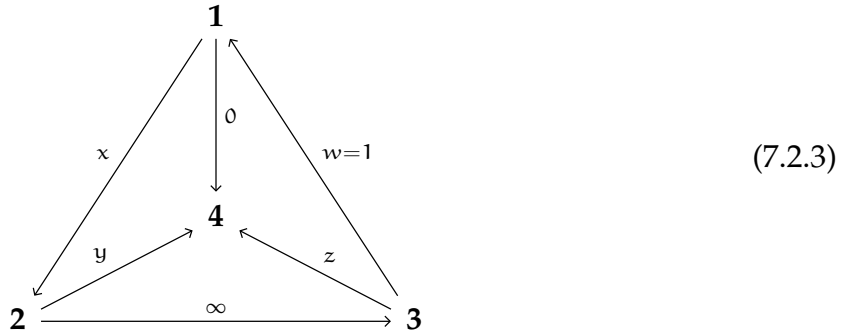
\begin{equation}
        \label{eqn:the_oriented_tetrahedron_with_dices_in_partI}
        \begin{tikzcd}
            && \V1
            \ar[llddd,"x"']\ar[dd,"0"]
            && \\
            && \phantom{x} &&\\
            && \V4  &&\\
            \V2\ar[rrrr,"\infty"]\ar[rru,"y"] && &&
            \V3\ar[llu,"z"']\ar[lluuu,"w=1"']
        \end{tikzcd}
    \end{equation}
    Examining the labelling from this remark,   we  readily verify that the
    arguments of the $\chi_{ij}$ are just the $\mu_{ij}$ from
    \Cref{SixPointsModuli}.
\end{proof}

\subsection{The tetrahedral symbol as a convolution of $\gamma$-functions}

The following striking formula can be derived, after some manipulation, from
\Cref{thm:6j=Hypergeometric}, and uses the same notation (we also remind the
reader of the notations in \Cref{MultisetNotation}):

\begin{proposition} \label[proposition]{PROPAB}
    Abridge $\rho_{ij} = \Set{\chi_{ij}, \chi_{ji}}$.
    Define four-element
    subsets $A, B$ of characters as, respectively, the union of $\chi_{\V{41}}
    \otimes\rho_{\V{12}}$ and  $\chi_{\V{23}}\otimes\rho_{\V{43}}$, and the  union of
    $\rho_{\V{13}}$ and $\chi_{\V{23}} \chi_{\V{41}}\otimes\rho_{\V{24}}$. Then
    \begin{align}
        \label{eqn:hypergeometric_formula_A_and_B} 
        \chisymbolBig
        =\frac{\nu_\bbP^{-1}\int_{\mu}
        \gamma(\frac{1}{2}+\epsilon, A\otimes\mu)
        \gamma(1-\epsilon, B^{-1}\otimes\mu^{-1}) \dd\mu}
        {\sqrt{\gamma(\OneHalf, A\otimes B^{-1})}},
    \end{align}
    where \(\mu\) ranges  over characters of \(F^\x\),
    whereas $\epsilon$ is any complex number with real part between $0$ and
    $\OneHalf$.
\end{proposition}

The proof will be given in \Cref{TetrahedralAsGammaConvolution}.
It would be interesting to give a proof of \Cref{thm:main_duality_theorem}
using this proposition as a starting point. In fact,
the right-hand side can be seen to reflect some of the geometry of the spinor cone,
a point we hope to return to in future work.  

\subsection{The case $F=\R$: expression in terms of special value of $\pFq{4}{3}$}
\label{sub:expression_in_4F3_statement}
Using Mellin transform, we can relate the tetrahedral symbol to
generalized hypergeometric series. For simplicity, we assume
\begin{align}
    \chi_{\V{12}}=\abs{\blank}^{J_{\EN{1}}},\quad
    \chi_{\V{31}}=\abs{\blank}^{J_{\EN{2}}},\quad
    \chi_{\V{14}}=\abs{\blank}^{J_{\EN{3}}},\\
    \chi_{\V{34}}=\abs{\blank}^{J_{\EN{4}}},\quad
    \chi_{\V{24}}=\abs{\blank}^{J_{\EN{5}}},\quad
    \chi_{\V{23}}=\abs{\blank}^{J_{\EN{6}}},
\end{align}
where \(J_{\EN{1}},\ldots,J_{\EN{6}}\in\bbC\). Here, we choose \(\chi_{\V{31}}\)
instead of \(\chi_{\V{13}}\) only because in
\eqref{eqn:chosen_oriented_tetrahedron} the blackbold \(\EN{2}\) is the label of
\(\V{31}\); the orientations themselves have no impact here.
The general case can be derived
with the same method, but we only give explicit statement about this special
case. We use shorthands such as
\(J_{\IEN{1}\EN{2}\EN{3}}=-J_{\EN{1}}+J_{\EN{2}}+J_{\EN{3}}\),
\(J_{\EN{2}\EN{3}\EN{5}\IEN{6}}=J_{\EN{2}}+J_{\EN{3}}+J_{\EN{5}}-J_{\EN{6}}\),
and so on.

Then in \Cref{sub:Mellin_of_tetrahedral_symbols} we will show that we can
express the tetrahedral symbol as a sum of \(\pFq{4}{3}s\) \emph{evaluated at
$1$}, in two different ways:
\begin{align}
    \Pisymbol\sqrt{L\Bigl(\OneHalf,\Sigma\Bigr)}
    &=-\frac{L(\OneHalf,\Sigma_1^*)}
        {\gamma(J_{\IEN{2}\IEN{3}\IEN{5}\EN{6}})\gamma(J_{\IEN{2}\IEN{3}\EN{5}\EN{6}})\gamma(J_{\IEN{2}\IEN{2}})}
        \pFq{4}{3}\!\left(\!
            \begin{array}{c}
                J_{\IEN{1}\EN{2}\EN{3}}+\OneHalf,
                J_{\EN{1}\EN{2}\EN{3}}+\OneHalf,
                J_{\EN{2}\EN{4}\IEN{6}}+\OneHalf,
                J_{\EN{2}\IEN{4}\IEN{6}}+\OneHalf
                \\
                J_{\EN{2}\EN{3}\EN{5}\IEN{6}}+1,
                J_{\EN{2}\EN{3}\IEN{5}\IEN{6}}+1,
                J_{\EN{2}\EN{2}}+1
            \end{array}\!\Big|\; 1\,
        \right)\\
    &\mathbin{\phantom{=}}-\frac{L(\OneHalf,\Sigma_2^*)}
        {\gamma(J_{\EN{2}\EN{3}\EN{5}\IEN{6})}\gamma(J_{\EN{5}\EN{5}})\gamma(J_{\IEN{2}\EN{3}\EN{5}\IEN{6}})}
        \pFq{4}{3}\!\left(\!
            \begin{array}{c}
                J_{\IEN{1}\IEN{5}\EN{6}}+\OneHalf,
                J_{\EN{1}\IEN{5}\EN{6}}+\OneHalf,
                J_{\IEN{3}\EN{4}\IEN{5}}+\OneHalf,
                J_{\IEN{3}\IEN{4}\IEN{5}}+\OneHalf
                \\
                J_{\IEN{2}\IEN{3}\IEN{5}\EN{6}}+1,
                J_{\IEN{5}\IEN{5}}+1,
                J_{\EN{2}\IEN{3}\IEN{5}\EN{6}}+1
            \end{array}\!\Big|\; 1\,
        \right)\\
    &\mathbin{\phantom{=}}-\frac{L(\OneHalf,\Sigma_3^*)}
        {\gamma(J_{\EN{2}\EN{3}\IEN{5}\IEN{6}})\gamma(J_{\IEN{5}\IEN{5}})\gamma(J_{\IEN{2}\EN{3}\IEN{5}\IEN{6}})}
    \pFq{4}{3}\!\left(\!
            \begin{array}{c}
                J_{\IEN{1}\EN{5}\EN{6}}+\OneHalf,
                J_{\EN{1}\EN{5}\EN{6}}+\OneHalf,
                J_{\IEN{3}\EN{4}\EN{5}}+\OneHalf,
                J_{\IEN{3}\IEN{4}\EN{5}}+\OneHalf
                \\
                J_{\IEN{2}\IEN{3}\EN{5}\EN{6}}+1,
                J_{\EN{5}\EN{5}}+1,
                J_{\EN{2}\IEN{3}\EN{5}\EN{6}}+1
            \end{array}\!\Big|\; 1\,
        \right)\\
    &\mathbin{\phantom{=}}-\frac{L(\OneHalf,\Sigma_4^*)}
        {\gamma(J_{\EN{2}\EN{2}})\gamma(J_{\EN{2}\IEN{3}\IEN{5}\EN{6}})\gamma(J_{\EN{2}\IEN{3}\EN{5}\EN{6}})}
    \pFq{4}{3}\!\left(\!
            \begin{array}{c}
                J_{\IEN{1}\IEN{2}\EN{3}}+\OneHalf,
                J_{\EN{1}\IEN{2}\EN{3}}+\OneHalf,
                J_{\IEN{2}\EN{4}\IEN{6}}+\OneHalf,
                J_{\IEN{2}\IEN{4}\IEN{6}}+\OneHalf
                \\
                J_{\IEN{2}\IEN{2}}+1,
                J_{\IEN{2}\EN{3}\EN{5}\IEN{6}}+1,
                J_{\IEN{2}\EN{3}\IEN{5}\IEN{6}}+1
            \end{array}\!\Big|\; 1\,
        \right),
\end{align}
and
\begin{align}
    \Pisymbol\sqrt{L\Bigl(\OneHalf,\Sigma\Bigr)}
    &=+\frac{L(\OneHalf,\Sigma_{1'}^*)}
        {\gamma(J_{\EN{1}\EN{1}})\gamma(J_{\EN{1}\IEN{3}\EN{4}\EN{6}})\gamma(J_{\EN{1}\IEN{3}\IEN{4}\EN{6}})}
        \pFq{4}{3}\!\left(\!
            \begin{array}{c}
                J_{\IEN{1}\EN{2}\EN{3}}+\OneHalf,
                J_{\IEN{1}\IEN{5}\EN{6}}+\OneHalf,
                J_{\IEN{1}\EN{5}\EN{6}}+\OneHalf,
                J_{\IEN{1}\IEN{2}\EN{3}}+\OneHalf
                \\
                J_{\IEN{1}\IEN{1}}+1,
                J_{\IEN{1}\EN{3}\IEN{4}\IEN{6}}+1,
                J_{\IEN{1}\EN{3}\EN{4}\IEN{6}}+1
            \end{array}\!\Big|\; 1\,
        \right)\\
    &\mathbin{\phantom{=}}+\frac{L(\OneHalf,\Sigma_{2'}^*)}
        {\gamma(J_{\IEN{1}\IEN{1}}\gamma(J_{\IEN{1}\IEN{3}\EN{4}\EN{6}})\gamma(J_{\IEN{1}\IEN{3}\IEN{4}\EN{6}})}
        \pFq{4}{3}\!\left(\!
            \begin{array}{c}
                J_{\EN{1}\EN{2}\EN{3}}+\OneHalf,
                J_{\EN{1}\IEN{5}\EN{6}}+\OneHalf,
                J_{\EN{1}\EN{5}\EN{6}}+\OneHalf,
                J_{\EN{1}\IEN{2}\EN{3}}+\OneHalf
                \\
                J_{\EN{1}\EN{1}}+1,
                J_{\EN{1}\EN{3}\IEN{4}\IEN{6}}+1,
                J_{\EN{1}\EN{3}\EN{4}\IEN{6}}+1
            \end{array}\!\Big|\; 1\,
        \right)\\
    &\mathbin{\phantom{=}}+\frac{L(\OneHalf,\Sigma_{3'}^*)}
        {\gamma(J_{\IEN{1}\EN{3}\IEN{4}\EN{6}})\gamma(J_{\EN{1}\EN{3}\IEN{4}\EN{6}})\gamma(J_{\IEN{4}\IEN{4}})}
    \pFq{4}{3}\!\left(\!
            \begin{array}{c}
                J_{\EN{2}\EN{4}\IEN{6}}+\OneHalf,
                J_{\IEN{3}\EN{4}\IEN{5}}+\OneHalf,
                J_{\IEN{3}\EN{4}\EN{5}}+\OneHalf,
                J_{\IEN{2}\EN{4}\IEN{6}}+\OneHalf
                \\
                J_{\EN{1}\IEN{3}\EN{4}\IEN{6}}+1,
                J_{\IEN{1}\IEN{3}\EN{4}\IEN{6}}+1,
                J_{\EN{4}\EN{4}}+1
            \end{array}\!\Big|\; 1\,
        \right)\\
    &\mathbin{\phantom{=}}+\frac{L(\OneHalf,\Sigma_{4'}^*)}
        {\gamma(J_{\IEN{1}\EN{3}\EN{4}\EN{6}})\gamma(J_{\EN{1}\EN{3}\EN{4}\EN{6}})\gamma(J_{\EN{4}\EN{4}})}
    \pFq{4}{3}\!\left(\!
            \begin{array}{c}
                J_{\EN{2}\IEN{4}\IEN{6}}+\OneHalf,
                J_{\IEN{3}\IEN{4}\IEN{5}}+\OneHalf,
                J_{\IEN{3}\IEN{4}\EN{5}}+\OneHalf,
                J_{\IEN{2}\IEN{4}\IEN{6}}+\OneHalf
                \\
                J_{\EN{1}\IEN{3}\IEN{4}\IEN{6}}+1,
                J_{\IEN{1}\IEN{3}\IEN{4}\IEN{6}}+1,
                J_{\IEN{4}\IEN{4}}+1
            \end{array}\!\Big|\; 1\,
        \right).
\end{align}
Here \(\Sigma_1^*,\ldots,\Sigma_4^*\) and \(\Sigma_{1'}^*,\ldots,\Sigma_{4'}^*\) are certain
slightly modified versions of \(\Sigma^-\) (details in
\Cref{sub:Mellin_of_tetrahedral_symbols}). For generic
\(J_{\EN{1}},\ldots,J_{\EN{6}}\), the eight \(\pFq{4}{3}\)s on
both right-hand sides turn out to be absolutely convergent at \(1\) and so the
expressions are well-defined.
See also \Cref{sub:symbol_for_SU2} for a discussion of
related results in the literature.

In particular, our results imply  that \emph{there exists a
\(4\)-term combination of \(\pFq{4}{3}\)s whose value at \(1\) satisfies
\(W(\TypeD_6)\)-symmetry.} We have not been able to find
this statement in the literature; it fits in a broader patterm
of extra symmetry enjoyed by evaluations of generalized hypergeometric series,
cf.~\cite{FGS11}.
In this  connection, we would like to point out that the Thomae symmetry of $\pFq{3}{2}$
hypergeometric series is, in the same fashion,  related to the Regge symmetries
of $3j$ symbols. 

We also verified the said \(W(\TypeD_6)\)-symmetry numerically. For the reader's
entertainment, we give approximate values of
\(\Pisymbol\sqrt{L(\OneHalf,\Sigma)}\) in two randomly chosen cases (note that
\(\Sigma\)s below depend on the \(J\)s):
\begin{gather}
    \sqrt{L\Bigl(\OneHalf,\Sigma\Bigr)}
    \begin{Bmatrix}
        0.713446\cdot i&
        0.172136\cdot i&
        0.036550\cdot i\\
        0.933153\cdot i&
        0.382368\cdot i&
        0.223157\cdot i
    \end{Bmatrix}\\
    \approx
    1730.348536844976745554895072481538288183,
    \\
    \sqrt{L\Bigl(\OneHalf,\Sigma\Bigr)}
    \begin{Bmatrix}
        0.953991 +
        0.458649\cdot i&
        0.284370 +
        0.590858\cdot i&
        0.777124 +
        0.667484\cdot i\\
        0.036922 +
        0.102931\cdot i&
        0.542879 +
        0.223050\cdot i&
        0.395720 +
        0.897602\cdot i
    \end{Bmatrix}\\
    \approx
    -70.0141698970658774227811653621105061501\\
    {}-
    344.8019718022410817955734451589062806451 \cdot i.
\end{gather}

\section{Langlands duality; the tetrahedral symbol as an interface}
\label{TetrahedralAsInterface}

We now sketch a picture that is
Langlands dual to the tetrahedral symbol. 
It does not shed much light on the proofs, and it may seem
as if we are simply translating into an arcane tongue. 
But it was instrumental to the discovery of our results; it brings
out the role of the {\em group} $\mathrm{Spin}_{12}$
rather than merely its maximal torus and Weyl group; and, 
  although we do not pursue it here, this language suggests natural
avenues of   generalization. 
 
\subsection{Tetrahedral symbol via a correspondence}
First of all, let us explain how the definition of the tetrahedral symbol
is a special case of an invariant attached to
a correspondence between spherical varieties. 
Recall the notation of \eqref{eq: group basic notation}:
\begin{align}
    G = \RO^{\tetraO},\quad
    D = \RO^{\tetraE}, \quad H =\RO^{\tetraV},
\end{align}
so that $G \simeq \RO^{12}, D \simeq \RO^{6}, H \simeq \RO^{4}$.
Let \(X, Y\) be the corresponding homogeneous spaces:
\begin{align}
    X = D \backslash G,\quad Y = H \backslash G.
\end{align}
Let \(C^\infty(X)\) (resp.~\(C^\infty(Y)\)) be the space of smooth
\(\bbC\)-valued functions on \(X\) (resp.~\(Y\)). We also regard elements in
\(C^\infty(X)\) as functions on \(G\) that are \(D\)-invariant on the left, and
similarly for \(C^\infty(Y)\).
With $\Pi_G$ as defined in \Cref{Tetrahedraldatum},
we consider the diagram of \(G\)-representations
\begin{equation}\label{JSymbolDiagram}
    \begin{tikzcd}
        & C^{\infty}(X) \ar[dd,"\Av", dashed]\\
        \Pi_G \ar[ru, "\Lambda^X"] \ar[rd, "\Lambda^Y",swap] & \\
            & C^{\infty}(Y)
    \end{tikzcd}
\end{equation}
where the respective morphisms \(\Lambda^X\) and \(\Lambda^Y\) of \(\Pi_G\) into
\(C^{\infty}(X)\) and \(C^{\infty}(Y)\)
are ``distinguished'' embeddings,
which is to say that their values at the identity coset
are given by the normalized functionals $\Lambda^D$ and $\Lambda^H$
that were specified earlier;
  the vertical
map is an averaging integral
with Haar measure is fixed as in the beginning of the paper:
\begin{align}
    \Av(f)(x)=\int_{(D\cap H)\backslash H}f(hx)\dd h.
\end{align}
We draw the vertical map with a dashed arrow because
the above integral is certainly not convergent for
an arbitrary function in $C^{\infty}(X)$.
We shall \emph{assume}
that \(\Av\) converges absolutely on the image of \(\Pi_G\) (which turns out to
be true for tempered $\Pi$).
We can reinterpret 
 the tetrahedral symbol \(\Pisymbol\) 
as \emph{describing the scalar by which this diagram fails to commute}. 
 
\subsection{Interface and its dual}
\label{sub:interfaces_and_dual}
The  key ingredients in the above definition are two spaces $X,Y$
and the correspondence
\begin{align}
    X \longleftarrow Z \longrightarrow Y
\end{align}
where $Z \defeq (D \cap H) \backslash G$; the averaging
operator amounts to pulling-back and pushing-forward along this diagram. 
Now an important theme of relative Langlands duality is that, in order
to achieve symmetry between automorphic and Galois sides, it is convenient
to switch to symplectic geometry, 
i.e. to ``code'' spaces by their cotangent bundles. From this point of view
we should replace $X$ and $Y$ by  $M=\CT X, N=\CT Y$
and replace $Z$ by  the \emph{Lagrangian correspondence}
\begin{equation} \label{LXY}
    L\defeq \mbox{conormal to $Z$} \subset M \times N.
\end{equation}
In the language of \cite{BZSV24}, this defines an \notion{interface} between
the hyperspherical varieties $M$ and $N$, and
it is reasonable to search for a \notion{dual interface}
\begin{equation} \label{LXY2}
    \check{L}  \subset \check{M} \times \check{N},
\end{equation}
where $\check{M}, \check{N}$ are the respective relative Langlands duals for $M,
N$. The paper \cite{BZV} does this for various  naturally arising classes of
$L$; however, in all examples of that paper, both $L$ and $\check{L}$ are
smooth. The tetrahedral symbol offers one of the first interesting
cases where this is not so.

 The dual group of \(G\) is 
\begin{align}
    \dual{G}=\SL_2(\bbC)^{\tetraO} \simeq \SL_2(\bbC)^{12}.
\end{align}
Consider also the dual $\dual{D}= \SL_2^{\tetraE} \simeq \SL_2(\bbC)^6$
which we regard as a subgroup of $\dual{G}$ in the obvious way,
i.e., corresponding to the map $\tetraO\rightarrow \tetraE$
where we forget orientation.
(Note that, in general, an embedding $D \subset G$
does not induce an embedding of dual groups; but
in the current case we just use the obvious embedding.)

According to relative Langlands duality, the Hamiltonian varieties \(M=\CT X\)
and \(N=\CT Y\) are respectively dual to the \(\dual{G}\)-spaces
\begin{align} \label{dualMdualNdef}
    \dual{M}=\CT(\dual{G}/\dual{D}),\quad
    \dual{N}\simeq\bigoplus_{i\in\tetraV}\bbC^2_{ij}\boxtimes\bbC^2_{ik}\boxtimes\bbC^2_{il}.
\end{align}

In fact,  $\dual{N}$
is the pullback  
 of a half-spin 
representation $S$ of $\Spin_{12}$ by means of
the homomorphism 
\begin{align}
    \tau \colon\check{D}=\SL_2^{\tetraE}\longto \Spin_{12}
\end{align}
that was described in \eqref{taudef2};
note that there are {\em two} half-spin representation of $\Spin_{12}$, only
one of which  restricts to $\check{N}$ (see \Cref{sub:halfspin} for
more details).
 Accordingly it makes sense to consider
within $\check{N}$ the cone
$    P \subset S\simeq \dual{N}$
of pure spinors. The reader can refer to \Cref{sub:the_dual_side} for a
description of \(P\) as a subvariety of \(S\).

We are now ready to describe what we think
is the picture in relative Langlands duality that underlies the theory of the tetrahedral symbol.
  
\begin{quote}
{\em Key proposal:} The dual interface to the tetrahedral symbol is the induction of
the cone of pure spinors $P \subset \check{N}$ 
from $\dual{D}$ to $\dual{G}$:
\begin{align} \label{L As Induction}
    \check{L}=\dual{G}\x^{\dual{D}}P.
\end{align}
where the projection from \(\dual{L}\) to \(\dual{M}\)  factors through the zero
section \(\dual{G}/\dual{D}\to\dual{M}\), and the projection to \(\dual{N}\) is given by
\((g,v)\mapsto gv\).
\end{quote}

This picture, for example, ``explains'' \Cref{thm:main_duality_theorem},
in a sense discussed further in  \Cref{functoriality} and \Cref{sheaves}.

\section{Further topics} \label{sec:Further}
What we have proved in this paper gives, we think, a deeper context
for  the Regge symmetry of the classical $6j$ symbols. However,
the classical $6j$ symbols has many other beautiful properties too, 
and it would be interesting to study these from the point of view taken in this paper.
We will give a brief discussion of a few such points here. 

The classical $6j$ symbol plays a distinguished role in the theory of orthogonal
polynomials: it gives the most general class of orthogonal polynomials in the {\em
Askey scheme} of orthogonal-hypergeometric polynomials. See \cite{Jac85,Ko88}.
Classical orthogonal polynomials satisfy both orthogonality relations
and difference equations; and in
\Cref{orthogonality} and \Cref{difference} we discuss, respectively,
the orthogonality and difference equations that hold in the context of this paper.

Secondly, the  tetrahedral symbol has played an important role in analytic
number theory (although apparently this connection has not been explicitly made previously). In the
first-named author's work on the subconvexity problem for $L$-functions, a key
role was played by a certain spectral reciprocity formula relating two sums of
$L$-functions. In this formula, as observed in \cite{MV10}, there is a somewhat
mysterious integral transform. This is, as we shall sketch in \Cref{Ass-ant},
precisely the integral transform associated by the two-variable function
obtained by fixing four of the six inputs to the tetrahedral symbol.

The relationship of the tetrahedral symbol to relative Langlands
duality could be more deeply understood in many ways. 
In \Cref{sub:review_of_local_Langlands_correspondence}
we use the general formulation of Langlands duality to 
discuss the situation beyond the principal series case.
In \Cref{functoriality}
we formulate a conjecture relating the tetrahedral symbol to the representation
theory of $\mathrm{Spin}_{12}$, and in \Cref{sheaves} we discuss associated
questions in geometric representation theory. 

We do not touch on another important role of the $6j$ symbol: namely, its role
in the Turaev--Viro \cite{TV92} invariant of \(3\)-manifolds; it would be interesting to look
at this, too,  from the dual viewpoint. Other interesting
directions include elaborating the analogous story for $3j$ symbols,
which are correspondingly related to the natural cone inside $\wedge^3 \bbC^6$;
or to study the asymptotic analysis of the tetrahedral symbol, in the spirit
of the Ponzano--Regge asymptotics \cite{Ro99}.

As this section is primarily meant to serve as a source of motivation for futher
work, {\em we will not always give full details, particularly around issues of
convergence.}

\subsection{Orthogonality} \label{orthogonality}
Like many classical special functions, the tetrahedral symbol
has remarkably rigid orthogonality properties.

Fix a pair of opposite edges of the tetrahedron: let us
take them to be $\Set{\V{13},\V{24}}$ and assign tempered representations to the
remaining $4$ edges. We can then regard the   tetrahedral symbol as a function of
$\sigma=\pi_{\V{13}}, \tau=\pi_{\V{24}}$,
which we denote by $\Pisymbol_{\sigma \tau}$,
and which we will regard as a function 
\begin{align}
    \cP_0\times \cP_0&\longrightarrow \bbC,\\*
    (\sigma, \tau) &\longmapsto \Pisymbol_{\sigma \tau},
\end{align}
where we recall from \Cref{Definition-general} that $\cP_0$ is the \emph{tempered dual}
of $\RO$, i.e. the set of irreducible tempered representations up to isomorphism.

The support of this function is a direct product, i.e. if $\Pisymbol_{\sigma_0
\tau_0}$ and $\Pisymbol_{\sigma_1 \tau_1}$ are nonzero, then $\Pisymbol$ is not
identically zero on the component of $\cP_0 \times \cP_0$ containing $(\sigma_0,
\tau_1)$. Indeed, the condition for $\Pisymbol$ to be not identically vanishing
on the component containing $(\sigma, \tau)$ is that all four of the
$\RO^3$-representations
\begin{align}
    \pi_{\V{12}} \boxtimes \sigma \boxtimes \pi_{\V{14}},\quad
    \pi_{\V{32}} \boxtimes \sigma \boxtimes \pi_{\V{34}},\quad
    \pi_{\V{21}} \boxtimes \tau\boxtimes \pi_{\V{23}},\quad
    \pi_{\V{41}} \boxtimes \tau \boxtimes \pi_{\V{43}}
\end{align}
admit $\RO$-invariant functionals.

Now this $\cP_0$ has a natural Borel structure; and a choice of Haar measure on $\RO$ determines upon $\cP_0$ a
canonical measure, the Plancherel measure, characterized by the fact that for a
continuous compactly supported function $f$ on $\RO$ we have
\begin{equation}
    \label{Plancherel}
    f(e) = \int_{\tau \in \cP_0} \Tr_{\tau}(f)\dd\tau.
\end{equation}

A collection of $N$ vectors in $\mathbb{R}^N$ are orthonormal
exactly when the matrix with those rows defines
an isometric transformation from $\mathbb{R}^N$ to itself.
In the same way, the following 
statement  can be considered as an orthogonality relation 
for $\Pisymbol_{\sigma\tau}$:

\begin{proposition}\label[proposition]{Orthogonality}
    The function $\Pisymbol_{\sigma \tau}$ is the kernel of a unitary transformation:
    \begin{align}
        L^2(\cP_0^{(a)}) &\longrightarrow L^2(\cP_0^{(b)}),\\*
        f &\longmapsto \left(\sigma\mapsto\int \Pisymbol_{\sigma \tau} f(\tau) \dd\tau\right).
    \end{align}
    and where $\cP_0^{(a)} \times \cP_0^{(b)}$ is the union of all
    components of $\cP_0 \times \cP_0$ on which $\Pisymbol_{\sigma\tau}$
    is not identically zero.
\end{proposition}

We sketch a proof in \Cref{OrthoProof}. Actually, we will establish a more precise
statement, which we now explain. Set
\begin{align}
    \Theta \defeq\pi_{\V{12}} \boxtimes \pi_{\V{23}} \boxtimes \pi_{\V{34}} \boxtimes \pi_{\V{41}},
\end{align}
an irreducible representation of $\RO^4$.
A choice of $\sigma=\pi_{\V{13}}$ and $\tau=\pi_{\V{24}}$ determines
$\RO$-invariant functionals
\begin{align} \label{Esdef}
    E_{\V{13}}=E_{\sigma}  \in \Theta^* \text{ and } E_{\V{24}}=E_{\tau}  \in
    \Theta^*
\end{align}
in the following way. 
First fix symmetric self-duality pairings on all $\pi_{ij}$
including $\sigma, \tau$.
Consider first
\begin{align}
    \Theta \boxtimes \sigma \boxtimes \sigma
    \simeq  \underbrace{ (\pi_{\V{32}} \boxtimes \pi_{\V{34}}  \boxtimes \pi_{\V{13}})}_{ 
        \pi_{\tetraO_\V3}} \boxtimes \underbrace{ 
    (\pi_{\V{13}} \boxtimes \pi_{\V{14}} \boxtimes \pi_{\V{12}})}_{\pi_{\tetraO_{\V1}}}
\end{align}
We normalize
functionals $\Lambda_{\V3}$ on the first factor $   \pi_{\tetraO_\V3}$ and $\Lambda_{\V1}$ on the second factor $    \pi_{\tetraO_\V1}$ by means of the same recipe as in
\eqref{eqn:def_of_Lambda_H}. To produce $E_{\sigma}$, contract in the $\sigma$
variables to produce a class in $\Theta^*$, that is to say, sum over $e_i
\otimes e_i\in\sigma\boxtimes\sigma$ where $e_i$ is an orthonormal basis with
respect to the fixed symmetric self-pairing on $\sigma$; it is presumably not difficult   to
verify the convergence of this summation. The construction of $E_{\tau}$ is
precisely parallel using instead
\begin{align}
    \Theta \boxtimes \tau \boxtimes \tau \simeq
      \pi_{\tetraO_\V2} \boxtimes   \pi_{\tetraO_\V4}  
\end{align}
and forming functionals $\Lambda_{\V2}, \Lambda_{\V4}$ on the two factors.

\begin{proposition}
    \label[proposition]{LemmaCyclicFour}
    There is a natural Hilbert structure on $(\Theta^*)^{\RO}$
    with the property that the rules  
    \begin{align}
        f \mapsto \int f(\sigma) E_{\sigma}\dd\sigma,\quad
        g \mapsto \int g(\tau) E_{\tau}\dd\tau
    \end{align}
    extend to isometries
    \begin{align}
        L^2(\cP_0^{(a)}) \mbox{ or } L^2(\cP_0^{(b)})
        \longrightarrow \text{Hilbert completion of $\Theta^*$}.
    \end{align}
    Moreover, the two isometries are intertwined
    by the transformation of \Cref{Orthogonality}.
\end{proposition}

\begin{remark}

    It may be helpful to discuss the situation  in a broader context, namely, 
    the {\em strong Gelfand formations} introduced by A.~Reznikov.
    By definition \cite[\S~1.2]{Reznikov}  such a formation is a pair of groups $I \subset G$
    along with two intermediate subgroups $H_1, H_2$, so that $I \subset H_i \subset G$,
    and such that {\em the restriction of irreducible representations
    from $G$ to $H_i$ and from $H_i$ to $I$ is multiplicity-free.}   

    In this situation, one can restrict representations from $G$ to $I$
    in two different ways: by going through $H_1$ or by going through $H_2$.
    One obtains, then, two different bases for the same multiplicity space --- one indexed
    by irreducible representations of $H_1$ and one indexed by irreducible representations of $H_2$.
    Our Proposition above relates these two different bases, in the case when 
  $I=\RO$ and $G=\RO^4$,  taking the $H_i$ to be two different copies of $\RO^2$. 
   It would be interesting to study examples arising from other strong Gelfand formations.

\end{remark}

\subsection{Difference and differential equations} \label{difference}

Specialize to the case $F=\mathbb{R}$; a variant of the following discussion
applies for $F=\mathbb{C}$ too. In these cases, 
like many classical special functions,
the tetrahedral  symbol then satisfies a large system of difference equations:

Quasi-characters  of $F^{\times}$ all have the form $x \mapsto \abs{x}^s$ or $x
\mapsto \sgn(x) \abs{x}^s$, which we parameterize by $(-1, s)$ or $(1,s)$ in
$\Set{\pm 1} \times \bbC$ respectively. Restricted to principal series, then,
the tetrahedral symbol defines a meromorphic function on 
\begin{align}
    \bigl( \Set{\pm 1} \times \bbC \bigr)^6.
\end{align}
 
Fixing a choice of signs, we get simply a function on $\bbC^6$. This function
satisfies a \emph{holonomic system of difference equations}, where holonomic
means, roughly,  that the difference equation has only a finite-dimensional
space of solutions if one imposes suitable growth constraints.  This follows
from \Cref{prop:edge_integral}, or even more conveniently
from its  reformulation in \eqref{eqn:hypergeometric_edge_integral}.

Indeed, as we see from \eqref{eqn:hypergeometric_edge_integral}, what we are
doing is pushing forward  Lebesgue measure from (an open subset of) $\bbR^3$  to
$(\bbR^{\times})^6$, by means of the map $$(x,y,z) \mapsto (x,1-x,x-y,y-z, z,
1-z),$$ and taking a Mellin transform of the resulting measure $\mu$. Difference
equations for that Mellin transform correspond to \emph{differential}
equations for the pushforward measure $\mu$; and  that such differential
equations exist in plenty follow from the fact that holonomic $D$-modules are
stable under direct image. See \cite{Oa18} for a very  explicit discussion.

It would not be  difficult to explicitly write
down these difference equations, for example, starting with some of our
hypergeometric representations. What would be particularly interesting would be to
``index'' the resulting system by the geometry of the spinor cone.

\subsection{The associativity kernel and analytic number theory}
\label{Ass-ant}
The unitary integral transform indicated in \Cref{Orthogonality} has played a
significant role in number theory; we sketch this in a typical example.
We will, for this subsection, freely assume familiarity with the language of automorphic forms. 

Let us fix an anisotropic quadratic form $Q$
over a totally real\footnote{This enables us to avoid some irrelevant analytical
issues because the Ramanujan conjecture is known.} \emph{global} field \(F\); let
$\SO_3$ denote the associated orthogonal group over $F$. For each place \(v\) of \(F\),
we let \(F_v\) be the corresponding local field. Let $S$ be a finite set of places. 

Now let  $\pi_{\V{12}}, \pi_{\V{23}}, \pi_{\V{34}}, \pi_{\V{41}}$ be  \emph{automorphic}
representations of $\SO_3$, all of which are unramified outside the set $S$. 
For each $v \in S$, let $f_v$ be a continuous function of compact support on the tempered dual
of $\SO_3(F_v)$. Then we have the ``associativity'' formula (the
terminology used in \cite[\S~4.5.2]{MV10}):

\begin{equation}
    \label{quadruple}
    \sum_{\pi_{\V{24}}}   \frac{\mathcal{L}_{\V2} \mathcal{L}_{\V4}}{L^{(S)}(1, \pi_{\V{24}}, \ad)}
    \prod_{v\in S} f_v(\pi_{\V{24},v})
    = \sum_{\pi_{\V{13}}} \frac{\mathcal{L}_{\V3}\mathcal{L}_{\V1}}{L^{(S)}(1, \pi_{\V{13}}, \ad)}
    \prod_{v \in S} \check{f}_v(\pi_{\V{13},v}),
\end{equation}
where $f_v \leftrightarrow \check{f}_v$ is the transform of Proposition
\ref{Orthogonality}; the sum is taken over all automorphic representations
$\pi_{\V{24}}$ or $\pi_{\V{13}}$, with local constituents
$\pi_{\V{24},v}$ and $\pi_{\V{13},v}$ that are representations of
$\SO_3(F_v)$, and
\begin{align}
    \mathcal{L}_i =\sqrt{L^{(S)}\Bigl(\OneHalf, \pi_{ij} \boxtimes \pi_{ik} \boxtimes \pi_{il} \Bigr)},
\end{align}
where $\Set{i,j,k,l}=\tetraV$, and $L^{(S)}$ means that we take $L$-factors only outside $S$. The meaning of \eqref{quadruple}
is that there exists a choice of signs for the various square roots such that it is valid (understanding the choice
of sign, from the point of view of $L$-functions, is a very interesting and rather unexplored question).

We sketch a proof of \eqref{quadruple} here, at least for a class of test functions $f_v$. 
 Take factorizable $\psi_i \in \pi_i$ and expand
\begin{align} \label{CC0}
    \int \psi_{\V1} \psi_{\V2} \psi_{\V3} \psi_{\V4}
    = \sum_{\pi_{\V{24}}}\sum_{\psi \in \pi_{\V{24}}} \int\psi_{\V1}\psi_{\V2}\psi\times
    \int\psi\psi_{\V3}\psi_{\V4},
\end{align}
where we sum first over automorphic representations $\pi_{\V{24}}$ and then over
orthogonal bases for (a real structure inside) $\pi_{\V{24}}$; the integrals
are over the  adelic quotient associated to $\SO_3$.
 Using the Ichino--Ikeda formula (see \cite{II10})
we rewrite this as
\begin{align}
    c  \cdot\mathcal{L}_2 \mathcal{L}_4 \cdot \prod_{v \in S}\underbrace{ E_{\V{24}}^{(v)}(\psi_{\V1,v}
    \otimes \psi_{\V2,v} \otimes \psi_{\V3,v} \otimes \psi_{\V4,v})}_{f_v(\pi_{\V{24},v})}
\end{align}
where $c$ is a constant depending on measure normalization, and the other notation is as in \eqref{Esdef}, with a superscript
$(v)$ to remind that we are working with $\SO_3(F_v)$; note that
$E_{\V{24}}^{(v)}$ depends  on $\pi_{\V{24},v}$  and so indeed defines a function
 $f_v$  on the tempered dual of $\SO_3(F_v)$.
 Now, compare this expansion with a similar expansion of \eqref{CC0}  
according to the $\Set{\V{14},\V{23}}$ grouping;
one gets a similar structure, now involving a factor
$\check{f}_v(\pi_{\V{13},v})$ given by
$E_{\V{13}}^{(v)}(\psi_{\V1,v} \otimes \psi_{\V2,v} \otimes \psi_{\V3,v} \otimes
\psi_{\V4,v})$. We conclude
by \Cref{LemmaCyclicFour}.

\subsection{Tetrahedral and \(6j\) symbols for \(\SU_2\)}
\label{sub:symbol_for_SU2}

There is a significant body of prior work related to this paper
concerning the question of defining $6j$ symbols for $\SL_2(\bbR)$ or $\SL_2(\bbC)$.
To our knowledge, all this work
is based on  definitions parallel to that of \Cref{orthogonality}
rather than \Cref{sub:definition_in_the_non_compact_case},
that is to say, realizing
\(6j\) symbols as ``associativity kernels,'' rather than
in a fashion that manifestly has the symmetry of a tetrahedron.
We will briefly review some of this work. 

Let us first note that, although the map $\SL_2 \rightarrow \PGL_2$ is almost an isomorphism,
the representation theory has significant differences, because of the failure
of multiplicity-one: given irreducible representations \(V_1\) and
\(V_2\) of \(\SL_2\) over a local field, the irreducible representations
appearing in the continuous part of \(V_1\otimes V_2\) may have multiplicity
\(2\).     However, there are a number of cases
related to $\SL_2$ where the multiplicity one holds, in full or in part:

\begin{itemize}
    \item  The case of \(\SL_2(\bbC)\). Here we have multiplicity-one in general
        (\cite{Na59}). The $6j$ symbols were first defined by Ismagilov
        \cite{Is06,Is07},  who gives  explicit hypergeometric formulas (as a sum
        of the products of two $\pFq{4}{3}$ hypergeometric series) for
        representations that descend to \(\PGL_2(\bbC)\). Mellin--Barnes
        integrals for bona fide \(\SL_2(\bbC)\)-representations were obtained
        and further studied by Derkachov and Spiridonov
        \cite{DeSp19}.\footnote{ We expect that it would be straightforward to
            check that their formula agrees with the ones in
            \Cref{sec:hypergeometric_formulas_for_the_tetrahedral_symbol}, but
        we have not done so.}
        Finally, relations to elliptic hypergeometric
        functions were studied in \cite{DSS22}.

    \item  The case of $\SL_2(\bbR)$. Here one does not have multiplicity one in
        general, but it remains valid in various situations where some of the
        representations are discrete series.  Such are the situations studied by
        Groenevelt \cite{Gr03,Gr06}, who expressed the relevant $6j$ symbols in
        terms of 
        Wilson functions, which can be written as the value at \(1\) of certain
        \(\pFq{7}{6}\)-function, or an equivalent form of a sum of
        \(\pFq{4}{3}\)-functions. Note that the latter form may be compared with
        our results in \Cref{sub:expression_in_4F3_statement}, and one may ask
        whether a similar consolidation into \(\pFq{7}{6}\)-functions is
        available.
\end{itemize}

For some further discussion of how to address such examples within our
framework, see \Cref{remarkSU}.

Genuine multiplicity two arises 
when tensoring two principal series of \(\SL_2(\bbR)\).
Such cases have been studied by
Derkachov and Ivanov \cite{DeIv23}; 
the definition implicitly depends on the choice
of bases for various multiplicity spaces (cf.  \cite{Iv18}). 
It would be interesting to 
analyze this from the point of view of relative Langlands duality,
which indeed suggests in many instances the existence of preferred
bases for multiplicity spaces.

\subsection{The local Langlands correspondence}
\label{sub:review_of_local_Langlands_correspondence}
The local Langlands correspondence
permits a generalization of our prior notions of $\gamma$, \(L\),
and $\epsilon$-factors, which seem well adapted to the general study of the tetrahedral
symbol.

Let $F$ be a local field.  Attached
to $F$ is a certain modification of the Galois group,
the \notion{Weil--Deligne} group \(W_F\) of $F$ (\cite[\S~8.3.6]{De73}).
All that is important for us now is that 
there is a canonical isomorphism
\begin{align}
    W_F^{\mathrm{ab}} \simeq F^{\times}
\end{align} 
and so characters of $F^{\times}$ can also be considered as characters of $W_F$;
we will do this without comment. Now, given a representation
\begin{align}
    \rho\colon W_F \longrightarrow \GL_n(\C)
\end{align}
we can define a $\gamma$-factor $\gamma(s, \rho)$, the $L$-function $L(s, \rho)$,
and the \(\epsilon\)-factor $\epsilon(s, \rho)$, each meromorphic function of a
complex variable $s$. In the case when $\rho$ is one-dimensional, and therefore
a quasi-character of $W_F^{\mathrm{ab}} \simeq F^{\times}$, they coincide with
the prior definitions from
\Cref{sub:intro_review_three_factors}. They also satisfy a relation analogous to
\eqref{epsilon_and_gamma_and_L}, replacing now $\chi^{-1}$ by the
contragredient.
In the general case, $L(s, \rho)$ always has the form $\prod_{i=1}^{m} L(s,
\chi_i)$ for various characters $\chi_i$, and $\epsilon(s, \rho)$ always has the
form $a\cdot b^s$, but $m$ can be smaller than $n$, and the values of $a$ and $b$ are
in general difficult to determine. 

The significance of the representation theory of $W_F$ comes from the
\notion{local Langlands correspondence}; it asserts that there is a map $\pi
\mapsto \rho_{\pi}$ from  irreducible representations of $\RO$, to
representations
\begin{align} \label{WeilRep}
    \rho\colon W_F \longto \SL_2(\C),
\end{align}
where \(\SL_2(\bbC)\) is the \emph{Langlands dual group} to $\RO$. The
representation \(\rho_\pi\) is also called the \notion{Langlands parameter} of
\(\pi\). For example, if $\pi$ is the principal series attached to a character
$\chi$ of $F^{\times}$, the corresponding representation is given by $\rho =
\chi \oplus \chi^{-1}$ where we identify $\chi$ with a character of $W_F$
through $W_F^{\mathrm{ab}} \simeq F^{\times}$. 
The map that associates to $\pi$
its Langlands parameter is injective.

The tetrahedral datum $\Pi$  of \Cref{Tetrahedraldatum} gives rise to a parameter
$\rho_{\Pi}\colon W_F \to  \SL_2^{\tetraE}$,
and this can be composed with the map \(\tau\colon\SL_2^\tetraE\to\Spin_{12}\)
in \eqref{taudef2} to obtain a map
\begin{align}
    \rho_{\Pi}^{\Spin}\colon W_F \longto \mathrm{Spin}_{12}.
\end{align} 
One aspiration that underlies much of this paper is to
\begin{quote}
    \emph{Express the theory of the tetrahedral symbol $\Pisymbol$
        entirely in terms of $\rho_{\Pi}^{\Spin}$ and the resulting
    action of $W_F$ on the spinor cone.}
\end{quote}
  Our main theorem \Cref{thm:main_duality_theorem} has accomplished this in the
unramified case. We will give now some further examples in this direction.

\subsubsection{Components on which $\Pisymbol$ vanishes identically} \label{EpsilonFactors}
The inclusion of the center $\Set{\pm 1} \hookrightarrow \SL_2$ induces
\begin{align}
    Z\defeq \Set{\pm 1}^\tetraE \longrightarrow \Spin_{12},
\end{align}
whose image commutes with $\rho_{\Pi}^{\Spin}$. By a general construction of
Gross and Prasad, explained in \cite[\S~10]{gross_prasad_1992}, we obtain from
this a character
\begin{align}
    \psi\colon Z \rightarrow \bbC^{\times},
\end{align}
namely, we associate
to each $z \in Z$ the $\epsilon$-factor $\epsilon(\frac{1}{2}, S^{z=-1})$ for
the action of $W_F$, acting via $\rho_{\Pi}^{\Spin}$, on the $(-1)$-eigenspace of $z \in Z$
acting on the half-spin representation.
 
\begin{lemma} \label[lemma]{GenericLemma}
    If a  nontrivial functional $\Lambda^H$ as in
    \Cref{sub:definition_in_the_non_compact_case} exists, $\psi$ is trivial.
    Conversely, if $\psi$ is trivial for a given $\rho: W_F \rightarrow
    \SL_2^{\tetraE}$, then there exists $\Pi_G$ with Langlands parameter $\rho$
    (possibly after replacing $\RO$  by the isometry group of a different quadratic form) for which
    $\Lambda^H$ is nontrivial.
\end{lemma}

\begin{proof}
    To check $\psi$ is trivial, it is enough
    to check that its value on the $ij$
    copy of $-1$ is trivial. The $(-1)$-eigenspace for
    this element is given, as a representation of $\SL_2^\tetraE$, by
    \begin{align}
        \left( \bbC^2_{ij}\boxtimes\bbC^2_{ik}\boxtimes\bbC^2_{il} \right) \oplus 
        \left( \bbC^2_{ji}\boxtimes\bbC^2_{jl}\boxtimes\bbC^2_{jk} \right),
    \end{align} 
    where $\Set{i,j,k,l}=\tetraV$. Therefore, the condition is that
    \begin{align}
        \epsilon\Bigl(\OneHalf, \rho_{ij}\boxtimes\rho_{ik}\boxtimes\rho_{il}\Bigr) =
        \epsilon\Bigl(\OneHalf,\rho_{ji}\boxtimes\rho_{jl}\boxtimes\rho_{jk}\Bigr)
    \end{align}
    and since this holds for each $i,j$ we find that the value of
    $\epsilon(\frac{1}{2}, \rho_{ij} \boxtimes \rho_{ik} \boxtimes \rho_{il})$ is
    independent of $i \in \tetraV$.
    We now apply the beautiful result of Prasad \cite{PrasadCompositio}
    characterizing trilinear invariant functionals.
\end{proof}

\subsubsection{Generalization to $\SL_2$ from the point of view of Langlands parameters}\label{remarkSU}

 Let us  sketch an approach, within our framework, of how to extend
  the tetrahedral symbol to the $\SL_2$ case, and how it fits with the duality formalism. 
The key role is played by the algebraic group
 $$G^* = \SL_2^{\tetraE}/ Z',$$
 where   the subgroup
 $Z' \subset Z=(\pm 1)^{\tetraE}$  of order $8$ within the center
 of $\SL_2^{\tetraE}$ consists of 
 those elements whose product around each face is trivial; equivalently, $Z'$
 is \emph{generated} by elements with precisely three $-1$s around a given
 vertex.

The dual group
$\dual{G}^*$  has a remarkably similar description, but now one quotients
$\SL_2(\C)^{\tetraE}$ now by 
 those elements whose product around each vertex is trivial.
 The morphism $\tau$ does not descend to $\dual{G}^*$,
 {\em but} the action of
$\SL_2(\C)^{\tetraE}$ on the half-spin $S$ does.
What this means is that, given a parameter
\begin{equation}
    \label{Gstarparm}
    W_F \rightarrow \dual{G}^*
\end{equation}
one still define a representation of $W_F$ on $S$,
even though one cannot define in general a $\Spin_{12}$-valued representation of it. 
A parameter \eqref{Gstarparm} gives in particular an  $L$-packet of representations of $\SL_2(F)$
for each edge, {\em such that
  the product of central characters around each vertex is trivial};
  that is precisely the situation in which the classical $6j$ symbol
  for $\SU_2$ is meaningful.

It seems likely that most of the results of this paper would carry over
to this situation, i.e., attach a tetrahedral symbol $\Pisymbol$
to a datum  as in \eqref{Gstarparm}.  The embedding $H \rightarrow G$
that played a crucial role earlier is to be replaced by
$H \rightarrow  G^* \times G^*/Z^*$
where $Z^*$ is the antidagonal copy of the center. 
 We have not verified the details but expect that multiplicity one holds in this context,
 permitting us to carry over our definitions verbatim; and it is likely
 that the same theorems also hold with cosmetic modifications.

\subsubsection{Completing Regge's  original symmetry to a $W(\TypeD_6)$-symmetry}
\label{ClassicalRegge}

The prior discussion has an interesting manifestation related
to the  Regge symmetries as originally envisioned by Regge. 
Restrict now to the case when $F=\bbR$, and let us 
consider the classical $6j$ symbol attached to 
the matrix $J$ of non-negative integers given by
\begin{align}
    J=\begin{bmatrix}
        j_1 & j_2 &  j_3\\ 
        j_4 &  j_5 &  j_6 
    \end{bmatrix}
    \longmapsto
    \begin{Bmatrix}
        j_1 & j_2 &  j_3\\ 
        j_4 &  j_5 &  j_6 
    \end{Bmatrix}\in\bbC,
 \end{align}
i.e., in our language, the tetrahedral symbol attached to the
$2j+1$-dimensional
representations of $\SO_3$. The condition for this to be defined is, with our
convention, that the triangle inequalities associated to all four vertices are
satisfied, e.g. $j_1, j_2, j_3$ are the lengths of a Euclidean triangle, and so
forth.

The Weil group of the real numbers is an extension of $\mathbb{C}^*$
by an element $c$ (for ``conjugation'') satisfying $c^2=-1$, whose action on $\mathbb{C}^*$
is given by conjugation, i.e. $c z c^{-1} = \bar{z}$.
The representation $\rho_j$, as in \eqref{WeilRep}, associated
to $V_{2j+1}$, is given by
\begin{equation} \label{SL2parameter}
    \rho_{j}\colon z=re^{i \theta} \mapsto \begin{bmatrix} e^{i (2j+1) \theta} &
    0 \\ 0 & e^{-i (2j+1) \theta} \end{bmatrix},\quad
    c \mapsto c_{0,j}\defeq\begin{bmatrix} 0 & 1  \\ -1 & 0  \end{bmatrix}.
\end{equation}
Note that $\rho_{j} \simeq \rho_{-j-1}$.

Let us consider $J$ as defining
a class in $\OddFunctions$ (defined in \eqref{OddFunctionsDef}) by sending
each edge with positive orientation to the corresponding value of $2j+1$ (and
the opposite orientation to \(-2j-1\));
using the isogeny $\OddFunctions \rightarrow \TypeD_6$
we shall think of this as a cocharacter $\bJ$ for $\Spin_{12}$.
The representation $\rho_\Pi^{\Spin}$
is then given by
\begin{align} \label{rhoSpinReggeDS}
    r e^{i \theta} \mapsto \bJ(e^{i \theta}),\quad
    c \mapsto \dot{w}_0,
\end{align}
where $\dot{w}_0$ is a representative (induced by \eqref{SL2parameter}) for the
longest element of $W(\TypeD_6)$, which acts by inversion on the torus.
Let us observe that any two such representatives are in fact conjugate under
the maximal torus, and so
the conjugacy class of  \eqref{rhoSpinReggeDS} is independent of this choice.

Although we are no longer in the principal series case,
it remains of interest to examine what symmetries
of the tetrahedral symbol might be induced by $W(\TypeD_6)$.
If we replace $\bJ$  in \eqref{rhoSpinReggeDS} by $w(\bJ)$ for $w \in W(\TypeD_6)$,
the resulting homomorphism  remains conjugate to \eqref{rhoSpinReggeDS}.
The various $w(\bJ)$ correspond to 
various collections  
\begin{align}
    J' \defeq \begin{bmatrix} 
        j_1' & j_2'&  j_3'\\ 
        j_4' &  j_5' &  j_6 '
    \end{bmatrix}.
\end{align}
In general, the $J'$ are only half-integral.  However, they satisfy a parity
constraint:  the sums of $J'$s along triples of edges emanating from a single
vertex are integers. In fact, this parity condition is closed under
\(W(\TypeD_6)\)-symmetry; in other words, we may relax the assumption on \(J\)
to that \(J\) has
half-integral entries satisfying the same parity conditions.

One checks by direct computation that, for generic $J$, there are $15$ possibilities 
for $J'$ modulo \(\FRI\rtimes \FRT\);
here $15$ arises from the index of $\FRI\rtimes\FRT$ inside $W(\TypeD_6)$
(cf.~\Cref{lem:some_subgroups_of_WD6}).
The various $J'$ do not correspond to homomorphisms $W_F \rightarrow \SL_2^{\tetraE}$,
but the parity properties noted above imply that
they {\em do} correspond to homomorphisms $W_F \rightarrow \dual{G}^*$,
with notation as in   \Cref{remarkSU}, and thus
define $L$-packets of
discrete series representations of $G^*(\bbR)$.
The tetrahedral symbol  can then be defined according to the discussion of  \Cref{remarkSU}.\footnote{
    However,   one needs to switch between $G^*(\bbR)$ and its compact form to cover all cases.
    Indeed, the reasoning of   \Cref{GenericLemma} implies the following
    all-or-none property for
    $J'$ as above: either the triangle inequalities
are satisfied for all the vertices, or for none of them. }
In the six cases
(represented by the subgroup \(\FRS_3\) in
\Cref{lem:Regge_group_in_a_narrow_sense}) where all
triangle inequalities are satisfied, we interpret
entries of $J'$ as indexing representations of $\mathrm{SU}_2$,
and use the classical $6j$ symbol;
in the remaining nine where no triangle inequality is satisfied, we interpret them
as indexing holomorphic-antiholomorphic pairs of discrete series  for $\SL_2(\bbR)$. 
With this setup it becomes natural to ask:

\begin{question}
    Are all of  these extended classical \(6j\) symbols for the \(15\) $J'$s ---
     six attached to \(\SU_2\) and nine attached to
    \(\SL_2(\bbR)\) --- actually equal, up to sign
    and a normalizing $\gamma$-factor?
\end{question}

 If true, this means in effect that the original Regge symmetries can indeed be ``completed''
to a $W(\TypeD_6)$ symmetry, even though the Regge group is very much smaller.
Quite possibly this question is accessible through some of our hypergeometric formulas,
or through the prior work of  Groenevelt \cite{Gr03,Gr06},
but we have not examined it.

\subsubsection{Formulas outside the principal series case}
\Cref{PROPAB} can be generalized beyond the principal series case
using the Langlands formalism.  Let us suppose that $\pi_{\V{14}}$
and $\pi_{\V{23}}$ are principal series,  but make no
supposition about the remaining $\pi$s.
In the statement of \Cref{PROPAB}  we can  then simply replace the roles of $\rho_{ij}$
by the Langlands parameter of the representation $\pi_{ij}$.
Then we anticipate that formula \eqref{eqn:hypergeometric_formula_A_and_B}
remains valid; we have an argument for this, but have not verified the details, constants and so forth, and will return to it elsewhere.

\subsection{Functoriality and $\mathrm{Spin}_{12}$} \label{functoriality}

Langlands duality provides a ``lifting''
\begin{equation} \label{Langlands Lifting}
    \text{irreducible representations of $\RO^6$}
    \longrightarrow
    \text{$L$-packets of representation of $\mathrm{PSO}_{12}(F)$}
\end{equation}
associated with the morphism $\SL_2^6 \rightarrow \mathrm{Spin}_{12}$ of dual
groups. Our discussion of duality suggests that \emph{the tetrahedral symbol
factors through the lifting \eqref{Langlands Lifting}}. 

There is a natural function on the right
which pulls back to the tetrahedral symbol, as we now sketch.
Namely, we can use the
  same general setup as
\Cref{TetrahedralAsInterface}:  we consider two
$\mathrm{PSO}_{12}$-spaces $X,Y$ and an averaging intertwiner from $X$ to $Y$,
and compute the scalar by which \eqref{JSymbolDiagram2} below fails to commute:
\begin{equation}\label{JSymbolDiagram2}
    \begin{tikzcd}
        & C^{\infty}(X) \ar[dd,"\Av", dashed]\\
        \Pi \ar[ru, "\Lambda^X"] \ar[rd, "\Lambda^Y",swap] & \\
                                                           & C^{\infty}(Y)
    \end{tikzcd}
\end{equation}
where  $\Lambda^X, \Lambda^Y$ are ``normalized'' embeddings of the
\(\mathrm{PSO}_{12}\)-representation $\Pi$ into $X$ and $Y$.
Our proposal is, for well-chosen $X$ and $Y$, the resulting function
on  (certain) $\mathrm{PSO}_{12}$-representations
agrees with the
tetrahedral symbol after pullback via \eqref{Langlands Lifting}.

Relative Langlands duality suggests natural choices for $X$ and $Y$:
For $X$ we
take the Langlands dual to the $32$-dimensional (hyperspherical!)
half-spin representation \(S\)
of $\mathrm{Spin}_{12}$; it is a generalized Bessel model for
$\mathrm{PSO}_{12}$. For $Y$ we take the Whittaker model for
$\mathrm{PSO}_{12}$, so its Langlands dual is a single point.
Finally, to define the  intertwiner $\mathrm{Av}$,
we average over  over the smallest $G$-orbit $Z$ in $X \times Y$
that supports an invariant distribution.

In the language of \Cref{sub:interfaces_and_dual},
$Z$ defines a Lagrangian
\begin{align}
    L\defeq\text{conormal of $Z$}  \subset M\x N\defeq \CT{X}\x\CT{Y},
\end{align}
whose dual  $
    \check{L} \subset \check{M} \times \check{N}=S$
 we expect to be precisely the cone of pure spinors. There is
no other reasonable candidate for $\check{L}$: the preimage of $0$ under the moment map for the
half-spin representation is already an irreducible Lagrangian.

\subsection{Geometric representation theory}\label{sheaves}
We will now indicate a geometrization of 
\Cref{thm:main_duality_theorem}.
 We  work now over the field
 $\FRf=\overline{\mathbb{F}}_p\lauser{t}$, and write
 $\FRo=\overline{\mathbb{F}}_p\powser{t}$.
We use the same notations $G, D, H, X, Y, Z,
M=\CT{X}, N=\CT{Y}, 
\check{M}, \check{N} $ as from \Cref{TetrahedralAsInterface};
we have a Lagrangian $\check{L} \subset \check{M} \times \check{N}$
``induced from'' the cone of pure spinors $P \subset \check{N}$. 

Now, the geometric conjecture of relative Langlands duality asserts that there are equivalences of 
categories:
\begin{align}
    \text{constructible sheaves on $X_{\FRf}/G_{\FRo}$} &\sim (\mbox{coherent sheaves on $\check{M}/\check{G}$})^{\mathrm{shear}},\\
    \text{constructible sheaves on $Y_{\FRf}/G_{\FRo}$} &\sim (\mbox{coherent sheaves on
    $\check{N}/\check{G}$})^{\mathrm{shear}}.
\end{align}
(We omit technical details about the exact categories, which can be found in
\cite{BZSV24}; also, the superscript ``shear'' means that the category is to be
regraded as in \textit{loc.~cit.}; the details are not presently important.)
Moreover, the equivalence above carries the ``basic sheaf'', the constant sheaf
on $X_{\FRo}$ or $Y_{\FRo}$, to the structure sheaf on the coherent side, and is
compatible, in a natural way, with the Satake equivalence.

The diagram $X \leftarrow Z \rightarrow Y$ gives rise, by pullback followed by
pushforward, to a functor $I_{\mathrm{aut}}$ (for ``automorphic intertwiner'')
from constructible sheaves on $Y_{\FRf}/G_{\FRo}$ to ind-constructible sheaves on
$X_{\FRf}/G_{\FRo}$. (Note that to actually define this requires examination of technical
details that we have not carried out;  the spaces $X_{\FRf}, Y_{\FRf},
Z_{\FRf}$ are not pleasant.) A natural
conjecture is that \emph{this functor is equivalent, with respect to the
equivalences above, to the push-pull $I_{\mathrm{spec}}$ along the diagram}
\begin{align} \label{push-pull-spectral}
    \check{M}/\check{G} \longleftarrow \check{L}/\check{G} \longrightarrow \check{N}/\check{G}.
\end{align}

As we sketch in \Cref{section:geometric_to_numerical}, our 
\Cref{thm:main_duality_theorem},   in the case of nonarchimedean $F$ of
finite residue characteristic, would result from the statement by taking
Frobenius trace.

\part{Proofs}

\section{Convergence and the analytic continuation}
\label{sec:convergence_and_analytic_continuation}

We prove all results related to convergence and analytic continuation in this
section.

\subsection{Proof of \Cref{prop:AC}}
\label{sub:pf_of_AC}
Let $X$ be a smooth algebraic variety over the local field \(F\) and $f_s$ a family of smooth
\(\bbC\)-valued functions
on $X(F)$ depending analytically on the complex parameter $s \in \C^n$,
or, more generally, a parameter $s$ belonging to a complex analytic manifold.

\begin{definition}
    \label[definition]{def:controlled_functions}
    We say $f_s$ is ``a controlled family of functions on $X(F)$ depending holomorphically on $s$,''  if there is a smooth
    compactification $X \rightarrow \bar{X}$, with normal crossing boundary,
    with the following property: 
    
    For each point $x \in \bar{X}(F)$, parameter $s=s_0$ and $N \geq
    0$, there exists:
    \begin{itemize}
        \item an analytic neighbourhood $U_x \subset \bar{X}(F)$, \item   an open neighbourhood $S$ of $s_0$ in the parameter space $s$,
        and analytic functions $a_i(s)$ on this neighbourhood,
        \item local equations  $z_1,
    \ldots, z_r$ for the various divisorial components of the boundary passing
    through $x$
    \end{itemize}
     such that we have an equality
    \begin{align} \label{asymptotic}
        f_s =  \sum_i  h_{i,s} \times \abs{z_1}^{a_1(s)} \cdots
        \abs{z_n}^{a_n(s)}  +  \mathcal{E}_s
    \end{align}
    valid on $U_x \times S$, where $h_{i,s}$ and $\mathcal{E}_s$ are smooth function on
    $U_x \times S$ varying holomorphically in $s$\footnote{ Note that by a ``smooth function varying holomorphically,''
what we mean is that it is a holomorphic function of $s$
with values in the space of smooth functions with its natural topology;
in the real or complex case, this is the topology induced
by requiring uniform convergence of all derivatives on compact sets.
} and the   ``error term'' $\mathcal{E}_s$ is bounded
on $U_x \times S$
by a constant multiple of  $ \abs{z_1 z_2 \cdots z_r}^N$.

Similarly, we say that $h_s$ is a controlled family depending {\em meromorphically}  on $s$ 
if there is an analytic function $P(s)$ such that $P(s) h(s)$ 
satisifes the above definition.   In this case, we simply 
say that $h_s$  is for short \notion{controlled}.
\end{definition}

\begin{lemma}
    If $Y$ is a closed smooth subvariety of $X$ and \(f_s\) is controlled on \(X\)
    then the restriction of $f_s$ to $Y$ is also controlled.
\end{lemma}
\begin{proof}
    Let $Y^*$ be the closure of $Y$ within $\bar{X}$. The boundary $Z=Y^*-Y$ is
    a closed set, but need not be a normal crossing divisor. By
    \cite[Theorem~3.36]{Ko07}, we may find a resolution of singularity \(Y'\to
    Y^*\) such that the preimage of the singular locus of \(Y^*\) (which is
    contained in \(Z\)) is a union of normal crossing divisors. Then by
    \cite[Theorem~3.35]{Ko07}, we may find a further resolution \(\bar{Y}\to
    Y'\) such that it is an isomorphism over \(Y\) and the preimage of \(Y'-Y\)
    becomes a union of normal crossing divisors. This way we find an embedded
    resolution of singularities, i.e., a map $\pi\colon  \bar{Y} \rightarrow
    Y^*$ with the property that $\pi^{-1}(Z)$ is a union of normal crossing
    divisors.

    Let us fix $\bar{y} \in \bar{Y}(F)$ with image $x \in \bar{X}(F)$;
    we will verify the desired asymptotic conditions hold near $\bar{y}$.
    In that case, the pullback to $Y$ of any local equation $z_i$,
    defining a component of $\bar{X}-X$ through $x$, 
    must have the form $\prod_j w_j^{b_{ij}}$ where the $w_i$ are local equations for
    boundary divisors on $\bar{Y}$ that pass through $\bar{y}$; the exponents
    $b_{ij}$ are non-negative
    and, for any $j$, there  is at least one $i$ with $b_{ij}>0$.
    From this we see that $f_s$ on $Y$ satisfies
    the same condition.
\end{proof}

\begin{lemma}
    \label[lemma]{lem:abstract_AC}
    Suppose \(f_s\) is a controlled family of functions on \(X\), and let
    \(\omega\) be a volume form on \(X\).
    Suppose the integral
    \begin{align}
        I(f_s)\defeq \int_{X(F)} f_s\abs{\omega}
    \end{align}
    converges absolutely for some \(s\). Then \(I(f_s)\) can be meromorphically
    continued to all \(s\).
\end{lemma}
\begin{proof}
    We analyze this by local computation, following Igusa,
    using the asymptotic expansion \eqref{asymptotic}.
    Using a partition of unity reduces our statement to the  meromorphicity of
    \begin{align}
        \int_{F^n} \Phi_s(z) \times \abs{z_1}^{a_1(s)} \cdots \abs{z_n}^{a_n(s)}
    \end{align}
    where $\Phi_s(z)$ is a smooth compactly supported function of $z_1, \dots, z_n \in F^n$
    that varies holomorphically  in $s$, and, by the assumed absolute
    convergence, there exists some value
    $s_0$ of $s$ for which the real part of $a_i(s_0)$ is larger than $-1$.

    If $F$ is nonarchimedean, we reduce to the case when $\Phi_s$
    is the characteristic function of a product of intervals
    $\abs{z_i-z_{i_0}}\leq C$, and we leave the verification in that case to the
    reader. 

    In the archimedean case, we repeatedly integrate by parts
    to replace an integral by one in which all the $a_i(s)$
    have real part larger than zero, and the integral becomes absolutely convergent. For example, in the case $F=\mathbb{R}$ and
    $n=1$ the relevant identity is
    \begin{align}
        \int_{\mathbb{R}}  \frac{d^{2N} \phi}{dx^{2N}} \abs{x}^{s+2N}
        = (s+1) \dots (s+2N) \int  \phi(x)  \abs{x}^s
    \end{align}
    This shows that the integral becomes holomorphic
    upon multiplication by a polynomial of the form $\prod_{i=1}^n
    (a_i(s)+1)(a_i(s)+2) \dots (a_i(s)+2N)$. Our assumption on the $s_0$
    shows that this polynomial does not vanish identically.
\end{proof}

For us, an important example arises as follows. Suppose, to simplify the
dicussion, that $F$ is nonarchimedean, and fix an open compact subgroup $\Omega$
of $\RO$. We need some language to be able to speak of ``holomorphically varying
families of vectors in holomorphic families of representations.'' Let $\cP$
be the complex analytic space described in \eqref{Pdef}. By an abuse of
notation, we will understand a point $\pi \in \cP$ to index
the corresponding representation if $\pi$ is discrete series,
and the associated principal series representation if $\pi$ 
corresponds to a quasi-character of $F^{\times}$.
With this understanding, for each $\pi
\in \cP$, the space of fixed vectors $\pi^{\Omega}$ is finite-dimensional, and the
various $\pi^{\Omega}$ fit together to give a finite-dimensional holomorphic vector
bundle over $\cP$, the ``bundle of $\Omega$-invariants.''  The holomorphic structure is described as follows: 
positive dimensional components of $\cP$ correspond to families of principal
series representations, and the condition for a section of the vector bundle to
be holomorphic is that all its point evaluations at points of $F^2-\{0\}$ be
holomorphic. Similarly, the family of \emph{Jacquet modules} $\bar{\pi}$, i.e.,
the largest $U$-invariant \emph{quotient} of the various $\pi$s, again fit
together to a holomorphic vector bundle over $\cP$, where the holomorphic
structure is determined by requiring that the map $\pi^K \mapsto \bar{\pi}$
induces a holomorphic map of bundles for every $K$.

Having fixed this,  we fix a self-pairing $\IPair{-}{-}_{\pi}$
for each $\pi \in \cP$ that varies meromorphically (i.e.,
it induces a meromorphic self-pairing on the vector bundle
$\pi^K$ for each $K$); for example, such a pairing
will be fixed in \Cref{sec:edge_formula_for_principal_series}.
It induces, by work of Casselman, a self-pairing
$\IPair{-}{-}_{\bar{\pi}}$ on the Jacquet modules $\bar{\pi}$, 
characterized by the following property: 
for any  $v_1, v_2 \in \pi$  we have
\begin{equation}
    \label{Jacquet}
    \IPair{a_x v_1}{v_2}_{\pi} = \IPair{a_x \bar{v}_1}{\bar{v}_2}_{\bar{\pi}}
\end{equation}
for \(a_x=\Diag(x,1)\) whenever $\abs{x}$ is sufficiently small.
In fact, for each open compact subgroup $\Omega$ there exists $\varepsilon(\Omega)$
so that \eqref{Jacquet}
holds whenever  $v_1, v_2 \in \pi^\Omega$ and $\abs{x} < \varepsilon(\Omega)$.
In particular, $\IPair{-}{-}_{\bar{\pi}}$ also varies meromorphically. 
With this setup we can assert:

\begin{lemma}
    \label[lemma]{ControlledMatrixCoefficients}
    Suppose that $\pi \mapsto v_{\pi}$ and $\pi \mapsto w_{\pi}$ are holomorphic
    sections of the bundles of $K$-invariant vectors over $\cP$. Then the
    function $\cP \times \PGL_2(F)\rightarrow \mathbb{C}$, associating to $\pi
    \in \cP$ and $g \in \PGL_2(F)$ the number $(g v_{\pi}, w_{\pi})$, is a
    controlled family of functions on $\PGL_2(F)$ parameterized by the complex
    manifold $\cP$.
\end{lemma}
\begin{proof}
    We use the embedding $\PGL_2 \to \mathbb{P}(\Mat_2)$. A local
    analytic  coordinate system for a point on the boundary has the form $k_1
    a_x k_2$ where $k_1,k_2$ range through open neighbourhoods of $K$, the
    maximal compact subgroup of $\RO$;  \(a_x=\Diag(x,1)\) as above, and
    $\abs{x} \leq c$. The desired
    properties follow from \eqref{Jacquet} and the following fact: for the
    principal series induced from $\chi$, the Jacquet module $\bar{\pi}$ is
    two-dimensional and the eigenvalues of $a_x$ upon it are given by
    $\chi_{+}$ and $(\chi^{-1})_{+}$.
\end{proof}

\begin{proof}[Proof of \Cref{prop:AC}]
    For each $e \in \tetraE$ choose a holomorphic family of vectors $v_{\pi}^{(e)}$ 
    parameterized by $\pi \in \cP$. Tensoring together, this induces a holomorphic family
    of vectors in the family of representations of $\RO^{\tetraE}$
    parameterized by $\cP^{\tetraE}$.
    
    It is enough to show that
    both the integrals
     $ \int_{H}  \IPair{h v_1}{v_2}\dd h $ of
    \eqref{eqn:def_of_Lambda_H} 
    and
    the integral $   \int_{H \cap D \backslash H}
    \Lambda^D(hv)\dd h$ of
    \eqref{eqn:LambdaHprimedef}, where $v_1, v_2, v$ are chosen
    to be holomorphic families of the type just described, 
    vary meromorphically over $\cP^{\tetraE}$.
 
    The first integral amounts to taking  a family of matrix coefficients on
    $\PGL_2(F)^{\tetraO}$ parameterized by $\cP^\tetraE$, pulling it back to
    $H$, and integrating. The second amounts to taking a family of matrix
    coefficients on $\PGL_2(F)^{\tetraO}$, again parameterized by
    $\cP^{\tetraE}$,  pulling it back to $(D \cap H) \backslash H \simeq
    (\PGL_2(F))^3$, and integrating. In both cases, family of matrix
    coefficients is controlled after pull-back by applying
    \Cref{ControlledMatrixCoefficients} and the desired result thereby follows
    from \Cref{lem:abstract_AC}.
\end{proof}

\subsection{Proof of the absolute convergence in \Cref{prop:edge_integral}}
\label{sub:pf_of_edge_integral_convergence}

We start from the right-hand side of \eqref{eqn:edge_integral_formula_M06}.
For absolute convergence, since all \(\chi_{ij}\) are unitary, it suffices to
assume that they are trivial. Thus it reduces to show that \(\cM\) has finite
volume with respect to the volume form
\begin{align}
    \abs{\Omega}^{\OneHalf}=\abs*{\frac{\bigwedge_e\dd x_e}{\sqrt{\prod_{e\neq e'}x_e\wedge
    x_{e'}}}},
\end{align}
in which \(e,e'\) are unoriented edges that share exactly one vertex.
(Recall that $\Omega$ is a section of the square of the canonical bundle.)
Let us recall from \eqref{OmegaExplicit} that
the statement to be proven  can be put in a concrete form, namely, 
\begin{align}
    \int_{F^3}\frac{1}{\abs{xyz (x-y) (y-z) (1-z)(1-x)}^{\OneHalf}} < \infty.
\end{align}

\emph{First argument via high school calculus.} We only sketch this. Suppose
that $F=\mathbb{R}$ for ease of visualization. We need to choose coordinates
around each potential singularity of the integral and check local convergence.
As an illustration, let us check conergence in $x^2+y^2+z^2 \leq 1$.  We work in
spherical coordinates. We can neglect the terms $(1-x)(1-z)$. Discarding these,
the function is homogeneous of homogeneity degree $-5/2$; and since $\int_0^1
r^{-5/2} r^2 \dd r <\infty$  it suffices then to verify that  $\abs{xyz (x-y)
(y-z)}^{-\OneHalf}$ is integrable on the unit sphere. The singularities of this
function lie along   five great circles, and again it is enough to check local
convergence.  The most problematic singularity is $x=y=0$ where three great
circles meet; but in local coordinates $(u,v)$ near that point, the function
looks like $\abs{u v (u+v)}^{\OneHalf}$, which is verified to be integrable by passing to
polar coordinates.

\emph{ Second argument via algebraic geometry.}
Algebraic geometry offers a more systematic way to 
write the prior type of reasoning, at the cost of requiring rather more background. 
Using the classical theory of moduli stack of marked curves (see \cite{Kn83}),
we can compactify $\cM^{\circ}$ to a projective \emph{variety} $\cM$ (because
we are considering genus \(0\) curves) whose 
boundary
\(\cM-\cM^\circ\) is a union of smooth divisors with normal crossings. 
When this is done, we have
\begin{claim}
    $\Omega$ has at most simple poles along each boundary divisor.
\end{claim}
The absolute convergence
then follows from the simple fact that the function $\abs{\epsilon}^{-\OneHalf}$
is locally integrable near $\epsilon=0$.

The irreducible components of the divisor \(\cM-\cM^\circ\) can be put in three types:
\begin{enumerate}
    \item \(2\) points (i.e., \(x_e\)s) collide;
    \item \(3\) points collide in a generic fashion;
    \item \(4\) points collide in a generic fashion.
\end{enumerate}

We leave the (easy) first case to the reader, and explain why $\Omega$
has at most a simple pole along all divisors of the other two types. 
It suffices in both cases to choose a coordinate chart $(a,b,\epsilon)$
near any generic point of the divisor, where $\epsilon=0$ defines the divisor,
and compute $\Omega$ in local coordinates.

In the three-point case, let $e,e', e''$ be the three edges that label the three
colliding points;
we take our coordinate chart to be given by
\begin{align}
    x_e=b+\epsilon, x_{e'}=b+a\epsilon, x_{e''}=b,
    \text{ remaining $x$s} = 0,1,\infty,
\end{align}
where $a \neq 1$ and $b \notin \Set{0,1}$.
In the four-point case, we similarly  let $e,e', e'', e'''$ be the four edges
that label the four colliding points, and  we take our coordinate chart to be 
\begin{align}
    x_e=0, x_{e'}=  \epsilon, x_{e''}= a \epsilon,  x_{e'''}= b \epsilon,
    \text{ remaining $x$s} = 1, \infty,
\end{align}
where  $a,b$ must avoid  
the loci \(a=0\), \(a=1\), \(b=0\), \(b=1\) and \(a=b\).
In the three-point (resp. four-point) cases, the expression for $\Omega$ is then
\begin{align}
    \frac{(\epsilon \dd \epsilon \wedge \dd a \wedge \dd b)^2}{\epsilon^k},
    \text{ resp. }
    \frac{(\epsilon^2\dd \epsilon \wedge \dd a \wedge \dd b)^2}{\epsilon^k},
\end{align}
where $k$ is the number of adjacencies amongst the edges $e,e',e''$.
Evidently $k \leq 3$ in the three-point case. In the four point case
we note that not all four of $e,e', e'', e'''$ can be simultaneously adjacent,
and so $k \leq 5$. In both cases the order of pole of $\Omega$ is then at most $1$
as desired.

\subsection{Proof of the absolute convergence of \eqref{eqn:LambdaHprimedef}}
\label{Hprimedefconvergence}

We must verify that the integral of $\Lambda^D(hv)$ over $H \cap D \backslash H$
is absolutely convergent. Now $\Lambda^D(hv)$ is a certain product of matrix
coefficients of tempered representation; by the results of Cowling, Haagerup and
Howe \cite[Theorem 2]{CowlingHaagerupHowe}, these are all majorized by a matrix
coefficient of a suitable tempered principal series representation. Therefore it
suffices to check the absolute convergence when all $\pi_{ij}$ are tempered
principal series, i.e. principal series associated to unitary characters
$\chi_{ij}$; but in that case the integral has been computed explicitly and
proved to be absolutely convergent in \Cref{prop:edge_integral} (actually, the
proof of absolute convergence was was just given now, in
\Cref{sub:pf_of_edge_integral_convergence}).

\section{Computations with principal series}
\label{sec:edge_formula_for_principal_series}

Our goal here is to prove the edge formula for principal series, stated in
\Cref{sub:edge_formula}. We already outlined the general plan of the computation
at that point; the main issue is to be very careful about the normalizations of
various functionals, because it is otherwise rather easy to compute the answer
only up to an unspecified constant.

We follow the notation set up in \Cref{sub:edge_formula}, so that we assign
characters $\chi_{ij}$ of $F^{\x}$ to oriented edges such that $\chi_{ij}
\chi_{ji}=1$, and let $\pi_{ij}$ be the corresponding unitarily induced principal
series representations. Let \(B_{\RO}\subset \RO=\PGL_2\) be the standard
upper-triangular Borel in \(\RO\) and \(B=B_{\RO}^{\tetraO}\) the Borel in
\(G\).

\subsection{The plan of the computation}
Recall that $\Pisymbol$ is defined using a  $D$-functional on $\Pi$, denoted by
$\Lambda^D$ and an $H$-invariant functional denoted by $\Lambda^H$. What we will
do, here, is  to compute with a different normalizations of such  invariant
functionals, which we call $\phi^D, \phi^H$, and then deduce what we want about
$\Lambda^D, \Lambda^H$. Then we compare $\Lambda^H$
with another $H$-invariant functional obtained by averaging $\Lambda^D$:
\begin{equation} \label{LLA}
    (\Lambda^H)' \colon v \longmapsto  \int_{H \cap D \backslash H}
    \Lambda^D(hv)\dd h.
\end{equation}
Due to the technicality of this section, we give an outline of the argument for
the reader's convenience:
\begin{itemize} 
    \item In \Cref{sub:pseries} we set up notations related
        to principal series. 
    \item In \Cref{sub:pseriesB,sub:pseriesT} we carry out computations
        related to normalizing bilinear and trilinear pairings
        on principal series.
    \item In \Cref{sub:phi_D_and_H_def} we  define $\phi^D$ (resp.~$\phi^H$) and 
        write $\Lambda^D$ (resp.~$\Lambda^H$) in terms of it.
    \item in \Cref{sub:averaging}  we compare the $H$-average of $\phi^D$ with
        $\phi^H$.
    \item In \Cref{sub:summary_for_pseries} we  conclude the proof, conditional on a computational lemma,
    which is proved in \Cref{uglylemmaproof}.
\end{itemize}

\subsection{Setup on principal series} \label{sub:pseries}
We will summarize some essential properties and notation related
to principal series for $\RO$. In what follows, we abridge $\chi(\det g)$ to $\chi(g)$. 
We have fixed a local field \(F\), and characters
and measures are fixed as in
\Cref{Characters And Measures}.

There is a very useful way to parameterize vectors by Schwartz functions, namely
there is a natural projection
 $\CcInf(F^2) \rightarrow \pi_{\chi}$
where we force a function \(f\) on $F^2$ to have the desired degree of
homogeneity along central directions, namely, we send
\begin{equation}
    \label{eqn:Phichi}
    f(z) \longmapsto  f_{\chi}(z)
    \defeq \int_{\lambda\in F^\x} \chi^{-2}(\lambda)  \abs{\lambda} f(\lambda z) \dd^{\x}\lambda
    = \int_{\lambda\in F} f(\lambda z) \chi^{-2}(\lambda) \dd\lambda.
\end{equation}

The normalized principal series
$\pi_{\chi}$ is realized in the space of functions
$\Phi\colon \RO \rightarrow \C$ that transform
on the left by means of the character
\begin{align}
    \begin{bmatrix}
        a & b\\
        0 & c
    \end{bmatrix}
    \longmapsto (\chi^{-1})_{+}(a/c)
    =\chi\left(\frac{c}{a}\right)\abs*{\frac{a}{c}}^{\OneHalf}.
\end{align} 
This coincides with the definition given in \Cref{DefinitionNontempered} in
terms of functions on the punctured plane. Indeed, given  $\Phi$ as above,  pull
it back to $F^2-\Set{0}$ by means of $(x,y) \mapsto g_{xy}$, where $g_{xy}$ is
any matrix with bottom row $(x,y)$ and determinant equal to $1$, and then extend
it to \(F^2\) by \(0\). The result $f$ is independent of choice of $g_{xy}$ and
satisfies $f (\lambda \cdot z) = \chi^2(\lambda) \abs{\lambda}^{-1} f(z)$.

Now let us discuss how to rigidify this $\pi_{\chi}$, that is, how to endow it with a self-duality pairing.  First, let us observe that
on the space of $(-2)$-homogeneous functions on $F^2$ there is an action of
$\RO$ whose pullback to \(\GL_2\) is \(g\colon f(z)\mapsto f(zg) \abs{\det g}\),
and an invariant functional given by integrating $f$ on $\mathbb{P}^1_F$ (see
\Cref{ProjectiveIntegrationReview}). Composing this with the product of two
functions gives an invariant pairing $ \pi_{\chi} \otimes \pi_{\chi^{-1}}
\rightarrow \C$, which we denote simply by $(f,g) \mapsto \int_{\mathbb{P}^1}
fg$.

For $z_1,z_2\in F^2$, we denote by $z_1 \wedge z_2 \in F$
the determinant of the $2 \times 2$ matrix whose
rows are respectively $z_1$ and $z_2$. Now put
\begin{align}
    K(z_1, z_2) = \frac{ \chi^{-2}(z_1 \wedge z_2)}{\abs{z_1\wedge z_2}}
\end{align}
and note that $K(z_1, z_2)=K(z_2, z_1)$.
Then for $f \in \pi_{\chi}$ the function
\begin{align}
    I(f)(z) \defeq \int_{\mathbb{P}^1} K(z,z') f(z')
\end{align}
defines an element of $\pi_{\chi^{-1}}$ and the rule $f \mapsto I(f)$
intertwines $\pi_{\chi}$ with $\pi_{\chi^{-1}}$. We also denote \(I\) by
$I_{\chi}$ when we need to be explicit (note that $I_{\chi}$ and $I_{\chi^{-1}}$
are not inverse to each other).  The rule
\begin{equation} \label{rigidification for principal series}
    \Pair{f_1}{f_2}
    \defeq \int_{\mathbb{P}^1} f_1 \cdot I(f_2)
    = \int_{\mathbb{P}^1 \times \mathbb{P}^1} K(z_1,z_2) f_1(z_1) f_2(z_2)
\end{equation}
defines (at least for generic $\chi$) a rigidification of $\pi_{\chi}$.

 \subsection{The bilinear pairing on principal series}
\label{sub:pseriesB}

Continue with notation as above.  We have defined
a rigidification of $\pi_{\chi}$ and $\pi_{\chi^{-1}}$.
There is, up to sign, a unique isomorphism
between these that preserves rigification, that we shall
call the  rigidified isomorphism; it is the essential uniqueness of the
rigidified isomorphisms that make rigidifications so useful!

We can then transport the rigidification,
either on $\pi_{\chi}$ or on $\pi_{\chi^{-1}}$,
by means of a rigidified isomorphism in one factor, 
and get (up to sign) the same pairing
\begin{align} \label{normalizedpairing}
    \pi_{\chi} \times \pi_{\chi^{-1}} \longrightarrow \mathbb{C},
\end{align}
which we will call the ``normalized'' pairing, and 
which we shall  compute to be  
\begin{align}
    (f, g) \longmapsto \sqrt{ \gamma(1,\chi^2) \gamma(1,\chi^{-2})} \int_{\mathbb{P}^1} fg.
\end{align}

Denote by $\Schwartz(F^2)$ the Schwartz space on \(F^2\), namely the space of
smooth rapidly decreasing functions if \(F\) is archimedean, or locally
constant functions with compact support if \(F\) is nonarchimedean.
We normalize the Fourier transform $\FT_2$ on $F^2$ by means of
\begin{align}
    \label{eqn:2D_Fourier_skew_version}
    \FT_2 \Phi(z_1) = \int_{z_2\in F^2} \Psi(z_1 \wedge z_2) \Phi(z_2).
\end{align}
Note that since we use a \emph{skew-symmetric} pairing inside \(\Psi\) instead
of a symmetric one, one has \(\FT_2^2=\Id\) without also negating \(z\)
(cf.~\Cref{sec:Fadjoint}).

\begin{lemma}
    [Intertwiner versus Fourier transform]
    \label[lemma]{lem:IvsFourier}
    The following diagram commutes:
    \begin{equation}
        \begin{tikzcd}
            \Schwartz(F^2) \ar[d, "f\mapsto f_\chi", swap] \ar[rr, "{\FT_2}"] && \Schwartz(F^2) \ar[d, "f\mapsto
            f_{\chi^{-1}}"] \\
            \pi_{\chi} \ar[rr,"{\gamma(1,\chi^2)^{-1} I_\chi}"]  &&\pi_{\chi^{-1}}
        \end{tikzcd}
    \end{equation}
    Moreover, $I_{\chi^{-1}}I_\chi$ is the same as multiplication by
    $\gamma(1, \chi^2)\gamma(1, \chi^{-2})$. 
\end{lemma}
\begin{proof}
    We compute
    \begin{align}
        I_\chi(f_{\chi})(z)
        &= \int_{\lambda\in F, z'\in\mathbb{P}^1} \chi^{-2}(\lambda)f(\lambda z')
        \chi_+^{-2}(z \wedge z') \dd z' \dd\lambda\\
        &= \int_{(x,y) \in F^2} f(x,y) \chi_+^{-2}\bigl(z \wedge (x,y)\bigr) \dd x\dd y.
    \end{align}
    On the other hand, we also have
    \begin{equation} \label{eqn:IFourier} 
        \FT_2(f)_{\chi^{-1}}(z)
        = \int_{\lambda\in F, z'\in F^2}  \chi^2(\lambda) \Psi(\lambda z
        \wedge z') f(z') \dd z'\dd\lambda
        = \gamma(1, \chi^{2})^{-1} I_\chi(f_{\chi})(z),
    \end{equation}
    where we first carried out $\lambda$-integral to give
    \(\gamma(1, \chi^{2})^{-1}\chi_+^{-2}(z \wedge z')\).
    Here we used from \eqref{CheckChiEquation} the fact that the Fourier transform
    of a character $\chi$, considered as a function on the additive group, is
    $ (\chi(-1)\gamma(1,\chi)\chi_1)^{-1}$ (and that \(\chi^2(-1)=1\)).
    This proves the  commutativity of the diagram. 

    Moreover,  applying the above equation with $\chi$ replaced by $\chi^{-1}$
    and $f$ by $\FT_2(f)$ we get
    \begin{equation}\label{eqn:I2} 
        \FT_2^2(f)_{\chi}
        = \gamma(1, \chi^{-2})^{-1}  I_{\chi^{-1}}\bigl(\FT_2(f)_{\chi^{-1}}\bigr)
        = \gamma(1, \chi^{-2})^{-1} \gamma(1, \chi^2)^{-1} I_{\chi^{-1}}I_\chi(f_{\chi}),
    \end{equation}
    from which we obtain the second claim.
\end{proof}

\begin{lemma}
    \label[lemma]{lem:twobilinearpairings}
    A \emph{rigidified} isomorphism
    \begin{align}
        \pi_{\chi} \otimes \pi_{\chi^{-1}} \longto \pi_{\chi} \otimes \pi_{\chi}
    \end{align}
    is  given by
    $\gamma(1, \chi^2)^{-\OneHalf} \gamma(1, \chi^{-2})^{-\OneHalf}(\id
    \otimes I_{\chi^{-1}})$.
    Correspondingly, a normalized pairing on $\pi_{\chi} \otimes \pi_{\chi^{-1}}$
    (as in \eqref{normalizedpairing})
    is given by
    \begin{align}
        (f,g) \longmapsto \sqrt{\gamma(1, \chi^2)\gamma(1, \chi^{-2})}
        \int_{\mathbb{P}^1} fg.
    \end{align}
\end{lemma}
\begin{proof}
    Use $v, v'$  to denote for vectors in $\pi_\chi$ and $w, w'$ for vectors in
    $\pi_{\chi^{-1}}$.
    Now, integration against the kernel $K(x,y)$ above, or its analogue for
    $\chi^{-1}$, give respectively
    \begin{align}
        I_\chi&\colon \pi_\chi \longto \pi_{\chi^{-1}},\\
        I_{\chi^{-1}}&\colon \pi_{\chi^{-1}} \longto \pi_\chi.
    \end{align}

    The self-duality structures are given by $\int_{\POne}{v} \cdot I_\chi(v') $ on
    $\pi_\chi$ and similarly $\int_{\POne} {w} \cdot I_{\chi^{-1}}(w') $ on
    $\pi_{\chi^{-1}}$. Consider now
    \begin{align}
        \id \otimes I_{\chi^{-1}}: \pi_{\chi} \otimes \pi_{\chi^{-1}}
        &\longto \pi_{\chi} \otimes \pi_{\chi}\\
        v \otimes w &\longmapsto v \otimes I_{\chi^{-1}}(w).
    \end{align}
    Transporting back the self-duality form
    $\int_{\POne}{v_1}\cdot I_\chi(v_1')\times \int_{\POne}{v_2}\cdot
    I_\chi(v_2')$
    on the right hand side, we get on the left the self-duality form
    \begin{multline}
        \int_{\POne}v\cdot I_\chi(v')\times
        \int_{\POne}I_{\chi^{-1}}(w)\cdot I_\chi I_{\chi^{-1}}(w')\\
        \stackrel{\eqref{eqn:I2}}{=}
        \int_{\POne}v\cdot I_\chi(v')\times \int_{\POne}I_{\chi^{-1}}(w)\cdot w'
        \cdot  \gamma(1, \chi^2) \gamma(1, \chi^{-2}).
    \end{multline}
    Since $K$ is symmetric in its two arguments,  $I_{\chi^{-1}}$ is adjoint to $I_{\chi}$
    with respect to the pairing $\int_{\POne}$, and so 
    we can rewrite the above  as $\gamma(1, \chi^2) \gamma(1,
    \chi^{-2})$
    multiplied by the 
    standard self-duaity pairing on $\pi_{\chi} \otimes \pi_{\chi^{-1}}$. 
    The first claim follows, and the second claim is a consequence of the first. 
\end{proof}

\subsection{Normalizing the trilinear functional} \label{sub:pseriesT}
 Now fix three characters $\chi_1, \chi_2, \chi_3$
 with corresponding principal series representations $\pi_{\chi_i}$, 
and write $H'$ for the diagonal copy of $\RO$ inside $G'=\RO^3$. Put
\begin{align}
    \Pi' = \pi_{\chi_1} \otimes \pi_{\chi_2} \otimes \pi_{\chi_3}.
\end{align}
We also let \(\chi_{123}=\chi_1\chi_2\chi_3\).

We rigidify each $\pi_{\chi_i}$ and so also $\Pi'$ according to the discussion
above. Our interest here will be to compute the normalized (in the sense of
\eqref{eqn:def_of_Lambda_H}) $H'$-invariant functional on the space of $\Pi'$.
The open orbit of $H'$ on \((\bbP^1)^3\) is the locus $O$ where the three
points are distinct, which we denote by \([x_j,y_j]\) for \(j=1,2,3\). We fix a
basepoint $p_0 \in O$ and a lift $\tilde{p}_0 \in (F^2)^3$;
for concreteness we take $p_0=(0,1,\infty)$ and
\begin{align}
    \tilde{p}_0 =\bigl((0,1),(1,1),(1,0)\bigr)\in (F^2)^3.
\end{align}

The $H'$-invariant measure is induced by the
\(3\)-form on $(\mathbb{P}^1)^3$ whose pullback to $(F^2-\Set{0})^3$ is
\begin{align}
    \label{eqn:HaarP13}
   \dd p =  \frac{\prod_{j=1}^3(x_j\dd y_j-y_j\dd x_j)}{(x_1y_2-x_2y_1)(x_2y_3-x_3y_2)(x_3y_1-x_1y_3)}.
\end{align}
The choice of order in the product does not matter, because the measure
induced by a differential form does not depend on sign.
It is easily seen that this formula is independent of the choices of \(x_j\) and
\(y_j\), and is invariant under \(H'\). 
We transport this measure to $H'$
by means of $h \in H' \mapsto p_0 h$. The resulting measure \(\dd^\bbP h\)
is independent of the choice of $p_0$. We remind the reader
that if \(F\) is nonarchimedean, this measure assigns \(H'(\cO)=\RO(\cO)\)
volume \(1-q^{-2}\), hence differs from the normalized Haar measure
(cf.~\Cref{subsec:Group}).
Evidently, the following form on $\Pi'$ is $H'$-invariant:
\begin{equation}
    \label{eqn:taudef}
    \tau(f_1 \otimes f_2 \otimes f_3)
    = \int_{H'} (hf_1 \otimes hf_2 \otimes hf_3)(\tilde{p}_0) \dd^\bbP h
    =
    \int_{H'} \psi(\tilde{h}) f(\tilde{p}_0 \tilde{h}) \dd^\bbP h,
\end{equation}
where we choose some \(\tilde{h}\in\GL_2\) lifting \(h\),
and  $\psi=(\chi_{123}^{-1}\abs{\blank}^{3/2})\circ \det$.

\begin{lemma}
    \label[lemma]{lem:trilinear_lemma}
    With the choice of \(\tilde{p}_0\) above, we have
    \begin{align}
        \taunorm = \alpha^{\OneHalf} \tau,
    \end{align}
    where $\taunorm $ is the  ``normalized''
    trilinear functional on $\Pi'$  defined via $\taunorm(v_1) \taunorm(v_2) =
    \int_{h \in H'} \IPair{hv_1}{v_2}\dd h$, and
    \begin{equation} \label{alphadef}
        \alpha = \nu_\bbP^{-1}
        \prod_{i=1}^3 \gamma(1, \chi_i^2) \cdot
        \gamma\Bigl(\OneHalf, \chi_{\bar{1}\bar{2}\bar{3}}\Bigr)
        \gamma\Bigl(\OneHalf,\chi_{1\bar{2}\bar{3}}\Bigr)
        \gamma\Bigl(\OneHalf,\chi_{\bar{1}2\bar{3}}\Bigr)
        \gamma\Bigl(\OneHalf,\chi_{\bar{1}\bar{2}3}\Bigr).
    \end{equation}
    where we use the shorthand
    $\chi_{1\bar{2}\bar{3}}$ for $\chi_1\chi_2^{-1}\chi_3^{-1}$, and so on,
    and \(\nu_\bbP\) is defined by \eqref{eqn:Haar_vs_P13_measures}.
\end{lemma}
\begin{proof}
    First, note that
    \begin{align} \label{tauf}
        \tau(f) =\int_{F^3} f_1(x_1,1) f_2(x_2,1) f_3(x_3,1)
        \frac{\chi_{1 \bar{2}\bar{3}}(x_2-x_3) \chi_{\bar{1} 2 \bar{3}}(x_3-x_1)
        \chi_{\bar{1} \bar{2}
    3}(x_1-x_2)}{\abs{(x_2-x_3)(x_3-x_1)(x_1-x_2)}^{\OneHalf}}\dd x_1\dd x_2\dd x_3.
    \end{align}
    To see this we  set
    \begin{align}
        \label{eqn:h_as_x1x2x3}
        \tilde{h} =
        \begin{bmatrix}
             x_3 b & b \\
             x_1 d & d
        \end{bmatrix}
    \end{align}
    with $b=(x_1-x_2)$ and $d=(x_2-x_3)$.
    Note that $h$ maps $p_0$ to $(x_1, x_2, x_3)$; in fact,
    \begin{equation} 
        \tilde{p}_0 \tilde{h} =
        ((x_2-x_3)[x_1,1], -(x_3-x_1)[x_2,1], (x_1-x_2)[x_3,1])\in (\bbP^1)^3.
    \end{equation}
    Note also that 
    $\det{\tilde{h}}=(x_2-x_3)(x_3-x_1)(x_1-x_2)$. Using these coordinates
    for $h$  in \eqref{eqn:taudef}, the claim follows.

     Now, take $f \in \Pi'$ and $g \in \tilde{\Pi}' \defeq\pi_{\chi_1^{-1}} \boxtimes
    \pi_{\chi_2^{-1}} \boxtimes \pi_{\chi_3^{-1}}$. Then we have
    \begin{align} \label{tauftaug}
        \tau(f) \tau(g) &= \int_{h_1, h_2\in H'} \psi(\tilde{h}_1) \psi'(\tilde{h}_2)
        f(\tilde{p}_0 \tilde{h}_1) g(\tilde{p}_0 \tilde{h}_2)\dd^\bbP h_1\dd^\bbP h_2\\
        &= \int_{x, h_2\in H'}   (xf)(\tilde{p}_0 \tilde{h}_2) g(\tilde{p}_0 \tilde{h}_2)
        \abs{\det(\tilde{h}_2)}^3\dd^\bbP x\dd^\bbP h_2,
     \end{align}
    where we substituted $x=\tilde{h}_2^{-1} \tilde{h}_1$,
    $\psi = (\chi_{123}^{-1} \abs{\blank}^{3/2})\circ\det$
    and $\psi' = (\chi_{123} \abs{\blank}^{3/2})\circ\det$.
    The product $\phi \defeq (xf) \cdot g$ is a function of homogeneous degree $-2$
    and for such a function we have
    \begin{align}\label{H'P1}
        \int_{H'} \phi(\tilde{p}_0 \tilde{h}) \abs{\det(\tilde{h})}^3\dd^\bbP h
        = \int_{(\mathbb{P}^1)^3} \phi.
     \end{align}
    To check \eqref{H'P1} we express both as integrals over $F^3$. 
    The right hand side equals $\int_{F^3} \phi(x_1, x_2, x_3)\dd x_1 \dd x_2 \dd x_3$
    by \eqref{Pnexplicit}.
    On the left hand
    side,   we parametrize elements $\tilde{h}$ by means of $(x_1, x_2, x_3)$ as
    in \eqref{eqn:h_as_x1x2x3}. The function $p_0 h\mapsto \phi(\tilde{p}_0
    \tilde{h})\abs{\det(\tilde{h})}^3$, considered as a function on the orbit
    $O$, therefore assigns to $(x_1, x_2, x_3) \in F^3 \subset
    \mathbb{P}^1(F)^3$ the value $\Delta\phi(x_1, x_2, x_3)$, where $\Delta=
    \abs{(x_2-x_3) (x_3-x_1)(x_1-x_2)}$, and therefore its integral against the
    measure $\Delta^{-1}\dd x_1 \dd x_2 \dd x_3$ (cf. \eqref{eqn:HaarP13})
    coincides with the right hand side of \eqref{H'P1}. That concludes the proof
    of \eqref{H'P1}.

    Combining \eqref{tauftaug} with \eqref{H'P1}
    we find that $\tau(f) \tau(g) $ coincides with the integral of the pairing
    \begin{align}
        \int_{x\in H'}\dd^\bbP x\int_{(\POne)^3}(xf)\cdot g
        =\nu_\bbP\int_{x\in H'}\dd x\int_{(\POne)^3}(xf)\cdot g.
    \end{align}
    Now, recall that the normalized pairing on $\pi_{\chi}$ is defined by
    $\Pair{f}{g} = \int_{\POne} f \cdot I(g)$;
    we therefore obtain for $f,g \in \Pi'$ the equality
    \begin{align}
        \taunorm(f)\taunorm(g)= \int_{H'}\dd x\int_{(\POne)^3}{xf} \cdot {I(g)}
        = \tau(f) \tau(I(g)),
    \end{align}
    where $I$ now connotes the product of intertwining operators in each of the three tensor variables.
    It remains now to verify that
    $$ \tau(I(g)) = \alpha \tau(g).$$
    To prove this, we may suppose that \(g=\otimes_i\Phi_{i,\chi_i}\) for some
    $\Phi_i\in\CcInf(F^2)$ (cf.~\eqref{eqn:Phichi}), and then from \eqref{tauf}
    \begin{align}
        \tau( g)
        = \int_{(z_1,z_2,z_3)\in (F^2)^3} \Phi_1(z_1) \Phi_2(z_2) \Phi_3(z_3)
        \frac{\chi_{1\bar{2}\bar{3}}(z_2 \wedge z_3)\chi_{\bar{1}2\bar{3}}(z_3 \wedge
        z_1)\chi_{\bar{1}\bar{2}3}(z_1 \wedge z_2)}{\abs{(z_2 \wedge z_3)(z_3
    \wedge z_1)(z_1 \wedge z_2)}^{\OneHalf}}.
    \end{align}
    Using \Cref{lem:IvsFourier}, we have
    \begin{align}
        \tau(I(g))
        &=
        \prod_{i} \gamma(1, \chi_i^2)\\
        &\mathbin{\phantom{=}}\x
        \int_{(z_1,z_2,z_3)} \FT_2(\Phi_1\otimes\Phi_2\otimes\Phi_3)
        \frac{\chi_{\bar{1}23}(z_2 \wedge z_3)\chi_{1\bar{2}3}(z_3 \wedge
        z_1)\chi_{12\bar{3}}(z_1\wedge z_2)}{\abs{(z_2\wedge z_3)(z_3\wedge z_1)(z_1\wedge z_2)}^{\OneHalf}},
        \label{eqn:tau_I_g_as_Fourier}
    \end{align}
    where we still use \(\FT_2\) to denote the Fourier transform on \((F^2)^3\)
    induced by \(\FT_2\) in \eqref{eqn:2D_Fourier_skew_version}.
    For \(\Phi\in\CcInf((F^2)^3)\) and a tempered
    distribution \(\chi\) on \((F^2)^3\),  adjointness of the Fourier
    transform
    gives
    \begin{align}
        \int_{(F^2)^3}\Phi\cdot \bar{\chi}=\int_{(F^2)^3}\FT_2(\Phi)\cdot
        \overline{\FT_2(\chi)},
    \end{align}
    where \(\bar{\chi}\) means the complex conjugation.

    For additive character
    \(\Psi\), \(\overline{\Psi(x)}=\Psi(x)^{-1}=\Psi(-x)\), we have then
    \(\FT_2(\bar{\chi})(y)=\overline{\FT_2(\chi)}(-y)\). As a result,
    \eqref{eqn:tau_I_g_as_Fourier} is equal to
    \begin{align}
        \prod_{i} \gamma(1, \chi_i^2)
        \int_{(z_1,z_2,z_3)\in(F^2)^3}
        \Phi_1(z_1)\Phi_2(z_2)\Phi_3(z_3)\cdot
        \FT_2(\chi)(-z_1,-z_2,-z_3),
    \end{align}
    where \(\chi\) is the distribution given by
    \begin{align}
        (z_1,z_2,z_3)\longmapsto \frac{\chi_{\bar{1}23}(z_2 \wedge
            z_3)\chi_{1\bar{2}3}(z_3 \wedge
        z_1)\chi_{12\bar{3}}(z_1\wedge z_2)}{\abs{(z_2\wedge z_3)(z_3\wedge
    z_1)(z_1\wedge z_2)}^{\OneHalf}}.
    \end{align}
    The result then follows from
    \Cref{lem:Fourier_of_certain_form_for_trilinear} below.
\end{proof}

\begin{lemma}
    \label[lemma]{lem:Fourier_of_certain_form_for_trilinear}
    Let $\alpha_1, \alpha_2, \alpha_3$ be characters of $F^{\times}$ and
    \(\alpha_{123}=\alpha_1\alpha_2\alpha_3\). Then the distribution on
    $(F^2)^3$ given by $\alpha_1(z_2 \wedge z_3) \alpha_2(z_3 \wedge z_1)
    \alpha_3(z_1 \wedge z_2)$ has Fourier transform
    \begin{align}
        \bigl(\gamma(1, \alpha_{123}\abs{\blank}) \prod_i \gamma(1, \alpha_i)\bigr)^{-1}
        \times
        \frac{\alpha_1^{-1}(z_2 \wedge z_3)\alpha_2^{-1}(z_3 \wedge
        z_1)\alpha_3^{-1}(z_1 \wedge z_2)}{\abs{(z_2 \wedge z_3)(z_3 \wedge z_1)(z_1 \wedge z_2)}}.
    \end{align}
\end{lemma}
\begin{proof}
    We will proceed formally, leaving the routine analysis details of dealing
    with distributions (in contrast to functions) to the reader. 
    The basic strategy is simply to carry out the Fourier transform first in the $z_1$ variable,
    then $z_2$, then $z_3$; each of these will be straightforward after a change of coordinates. 
    
    Write \(z_i=(x_i,y_i)\) and \(\dd z_i=\dd x_i\wedge\dd y_i\). Using the
    equalities
    \begin{align}
        \bigl((z_2 \wedge z_3) \dd z_1\bigr)\wedge\dd z_2\wedge\dd z_3
        &= \bigl(\dd(z_2 \wedge z_1)\wedge \dd(z_3 \wedge z_1)\bigr)\wedge\dd
        z_2\wedge\dd z_3,\\
        z_1
        &= \frac{(z_1 \wedge z_3) z_2 - (z_1 \wedge z_2) z_3}{z_2 \wedge z_3},
    \end{align}
    we carry out the first change of coordinates, fixing \(z_2\) and \(z_3\) but replacing
    $z_1$ by
    $$z_1'=(x_1',y_1'), \ \ x_1'=z_2\wedge z_1, y_1'=z_3\wedge z_1.$$

    Perform Fourier transformation with respect to \(z_1\) with \(k_1\) being
    the dual coordinate of \(z_1\), in other words, we compute the integral
    \begin{align}
        &\mathbin{\phantom{=}}\int_{F^2}\Psi(k_1\wedge z_1)\alpha_1(z_2 \wedge z_3) \alpha_2(z_3
        \wedge z_1)\alpha_3(z_1 \wedge z_2)\dd z_1\\
        &=
        \int_{F^2}\Psi\Bigl(\frac{k_1\wedge z_2}{z_2\wedge
        z_3}(-y_1')+\frac{k_1\wedge z_3}{z_2\wedge z_3}x_1'\Bigr)
        \alpha_1(z_2 \wedge z_3) \alpha_2(y_1')\alpha_3(-x_1')\frac{\dd
        z_1'}{\abs{z_2\wedge z_3}}\\
        &=
        \frac{\alpha_1(z_2 \wedge z_3)}{\abs{z_2\wedge z_3}}
        \int_{F}\Psi\Bigl(\frac{k_1\wedge z_2}{z_2\wedge z_3}(-y_1')\Bigr)
        \alpha_2(y_1')\dd y_1'
        \int_F\Psi\Bigl(\frac{k_1\wedge z_3}{z_2\wedge z_3}x_1'\Bigr)
        \alpha_3(-x_1')\dd x_1'\\
        &=\alpha_{23}(-1)\frac{\alpha_1(z_2 \wedge z_3)}{\abs{z_2 \wedge z_3}}
        \FT(\alpha_{3})\Bigl(\frac{k_1 \wedge z_3}{z_2 \wedge z_3}\Bigr)
        \FT(\alpha_{2})\Bigl(\frac{k_1 \wedge z_2}{z_2 \wedge z_3}\Bigr) \\
        &\stackrel{\eqref{CheckChiEquation}}{=}
        \bigl(\gamma(1,\alpha_2)\gamma(1,\alpha_3)\bigr)^{-1}
        \alpha_{\bar{3};-1}(k_1\wedge z_3)
        \alpha_{\bar{2};-1}(k_1\wedge z_2)\alpha_{123;1}(z_2\wedge z_3),
    \end{align}
    where we use a semicolon to separate the indexing subscripts (\(1,2,3, 123\), 
    etc.) from the twisting subscripts (\(\pm 1\)).

    Similarly, we write
    \begin{align}
        z_2 = \frac{(z_2 \wedge z_3)k_1-(z_2 \wedge k_1)z_3}{k_1 \wedge z_3},
    \end{align}
    and transform with respect to \(z_2\).
    We get
    \begin{align}
        &\mathbin{\phantom{=}}\int_{F^2}\Psi(k_2\wedge z_2)\alpha_{\bar{3};-1}(k_1\wedge z_3)
        \alpha_{\bar{2};-1}(k_1\wedge z_2)\alpha_{123;1}(z_2\wedge z_3)\dd z_2\\
        &=
        \frac{\alpha_{\bar{3};-1}(k_1 \wedge z_3)}{\abs{k_1\wedge z_3}}
        \FT(\alpha_{123;1})\Bigl(\frac{k_2 \wedge k_1}{k_1 \wedge z_3}\Bigr)
        \FT(\alpha_{\bar{2};-1})\Bigl(\frac{k_2 \wedge z_3}{k_1 \wedge z_3}\Bigr) \\
        &=
        \alpha_{13}(-1)\bigl(\gamma(1,\alpha_{123;1})\gamma(1,\alpha_{\bar{2};-1})\bigr)^{-1}
        \alpha_{\bar{1}\bar{2}\bar{3};-2}(k_2 \wedge k_1)
        \alpha_2(k_2 \wedge z_3) \alpha_1(k_1 \wedge z_3).
    \end{align}

    Finally, we now write
    \begin{align}
        z_3 = \frac{(z_3\wedge k_2)k_1-(z_3\wedge k_1)k_2}{k_1\wedge k_2},
    \end{align}
    and transform with respect to \(z_3\). We get
    \begin{align}
        &\mathbin{\phantom{=}}\int_{F^2}\Psi(k_3\wedge z_3)
        \alpha_{\bar{1}\bar{2}\bar{3};-2}(k_2 \wedge k_1)
        \alpha_2(k_2 \wedge z_3) \alpha_1(k_1 \wedge z_3)\\
        &=
        \alpha_2(-1)\frac{\alpha_{\bar{1}\bar{2}\bar{3};-2}(k_2\wedge k_1)}{\abs{k_1\wedge
        k_2}}
        \FT(\alpha_2)\Bigl(\frac{k_3\wedge k_1}{k_1\wedge k_2}\Bigr)
        \FT(\alpha_1)\Bigl(\frac{k_3\wedge k_2}{k_1\wedge k_2}\Bigr)\\
        &=\alpha_{23}(-1)\bigl(\gamma(1,\alpha_2)\gamma(1,\alpha_1)\bigr)^{-1}
        \alpha_{\bar{2};-1}(k_3\wedge k_1)\alpha_{\bar{1};-1}(k_3\wedge
        k_2)\alpha_{\bar{3};-1}(k_1\wedge k_2)\\
        &=\alpha_{123}(-1)\bigl(\gamma(1,\alpha_2)\gamma(1,\alpha_1)\bigr)^{-1}
        \alpha_{\bar{1};-1}(k_2\wedge k_3)
        \alpha_{\bar{2};-1}(k_3\wedge k_1)
        \alpha_{\bar{3};-1}(k_1\wedge k_2).
    \end{align}

    Combining three steps together, the Fourier transform on \((F^2)^3\) of the
    distribution $\alpha_1(z_2 \wedge z_3) \alpha_2(z_3 \wedge z_1) \alpha_3(z_1
    \wedge z_2)$ is then
    \begin{align}
        \alpha_{\bar{1};-1}(k_2\wedge k_3)
        \alpha_{\bar{2};-1}(k_3\wedge k_1)
        \alpha_{\bar{3};-1}(k_1\wedge k_2)
    \end{align}
    times
    \begin{align}
        \alpha_{2}(-1)
        \bigl(\gamma(1,\alpha_2)\gamma(1,\alpha_3)\gamma(1,\alpha_{123;1})
            \gamma(1,\alpha_{\bar{2};-1})
        \gamma(1,\alpha_1)\gamma(1,\alpha_2)\bigr)^{-1}.
    \end{align}
    Notice that
    \(\gamma(1,\alpha_2)\gamma(1,\alpha_{\bar{2};-1})=\alpha_2(-1)\), and so we
    see that the Fourier transform of \(\alpha_1(z_2 \wedge z_3) \alpha_2(z_3 \wedge z_1) \alpha_3(z_1 \wedge z_2)\)
    is precisely
    \begin{align}
        \bigl(\gamma(1,\alpha_{123;1})\prod_i\gamma(1,\alpha_i)\bigr)^{-1}
        \alpha_{\bar{1};-1} (k_2 \wedge k_3)
        \alpha_{\bar{2};-1}(k_3 \wedge k_1)
        \alpha_{\bar{3};-1}(k_1 \wedge k_2),
    \end{align}
    as desired.
\end{proof}

\subsection{Definition of $\phi^D, \Lambda^D$ and $\phi^H, \Lambda^H$}
\label{sub:phi_D_and_H_def}

We now follow the notation of \Cref{sub:outlinepage}, and
the reader may find the outline there
helpful before reading this and the next subsections.

We define the naive functional $\phi^D$ on $\Pi$
 so that its value on the
vector \(f=\otimes_{i\neq j}f_{ij}\) 
is given by the formula:
\begin{align}
    \phi^D(f)= 
     \prod_{ij\in\tetraE}\int_{\bbP^1(F)}
         f_{ij}(z_{ij})f_{ji}(z_{ij}) \dd z_{ij},
\end{align}
where \(z_{ij}=(x_{ij},y_{ij})\) is any representative of the homogeneous
coordinate \([x_{ij},y_{ij}]\) on \(\bbP^1\),
and denote $\dd z_{ij}=x_{ij}\dd y_{ij}-y_{ij}\dd x_{ij}$ (note that this
notation is different from the previous section where \(\dd z=\dd x\wedge \dd y\)).
By \Cref{lem:twobilinearpairings} we obtain
the formula for the normalized pairing $\Lambda^D$:
\begin{align}
    \label{eqn:phiD_and_lambdaD}
    \Lambda^D = \phi^D \cdot \prod_{ij \in \tetraO}
    \gamma(1,\chi_{ij}^2)^{\OneHalf}.
\end{align}

Next, we 
construct an $H$-invariant functional $\phi^H$,
annd relate it to the normalized functional $\Lambda^H$.
The
\(H\)-orbits on \(B\backslash G = (\mathbb{P}^1)^{\tetraO}\) can be described by looking at the triads
pointing outwards (or equivalently, inwards) from a given vertex. 
There is a
unique open dense \(\RO\)-orbit 
on the product of three copies of $\bbP^1$ indexed by each triad; we denote
this open orbit corresponding to vertex \(i\) by
\(O_i\). The Haar measure on \(\RO\) induces an invariant measure on \(O_i\), which has been described in \eqref{eqn:HaarP13}.

Fix a base point  $p_0=(p_{ij})_{ij\in\tetraO}$
in \(\prod_i O_i\subset (\bbP^1)^{\tetraO}\).
and a lifting $\tilde{p}_0=(\tilde{p}_{ij}) \in (\mathbb{A}^2)^{\tetraO}$. To
interface with the
computation in \Cref{lem:trilinear_lemma}, we make our choices
 as follows:
\begin{align}
    \tilde{p}_{\V{12}}=\tilde{p}_{\V{21}}=\tilde{p}_{\V{34}}=\tilde{p}_{\V{43}}=(0,1),\\
    \tilde{p}_{\V{13}}=\tilde{p}_{\V{31}}=\tilde{p}_{\V{24}}=\tilde{p}_{\V{42}}=(1,1),\\
    \tilde{p}_{\V{14}}=\tilde{p}_{\V{41}}=\tilde{p}_{\V{23}}=\tilde{p}_{\V{32}}=(1,0).
\end{align}
In words, nonadjacent pair of edges (i.e., edges not sharing any vertex)
regardless of orientation
(e.g. $\V{12},\V{21},\V{34},\V{43}$) are all assigned the
same point; these points are $0,1,\infty$
and their ``standard'' lifts to \(F^2\).

With this choice of \(\tilde{p}_0\), we define  $\phi^H$ as the product
of the previously defined $\tau$-functionals from \eqref{eqn:taudef} on the
various triads of representations; symbolically, 
\begin{align} \label{phiHformula}
    \phi^H(f)=\prod_i\int_{H}f_{ij}(\tilde{p}_{ij} \tilde{g}_i)
    \psi_i(\tilde{g}_i)\dd^\bbP g_i
\end{align}
where $\psi_i = (\chi_{ij} \chi_{ik} \chi_{il})^{-1} \abs{\blank}^{3/2}$.
By \Cref{lem:trilinear_lemma}, we have also
\begin{align}
    \Lambda^H
    &= \nu_\bbP^{-2} \phi^H \x \Biggl( \prod_{ij\in\tetraO} \gamma(1, \chi_{ij}^2) \cdot
         \prod_{i \in \tetraV}
         \gamma\Bigl(\frac{1}{2}, \chi_{ij}^{-1}\chi_{ik}^{-1}\chi_{il}^{-1}\Bigr)
         \gamma\Bigl(\frac{1}{2}, \chi_{ij}\chi_{ik}^{-1}\chi_{il}^{-1}\Bigr)\\
    &\qquad\qquad\qquad\qquad\qquad\gamma\Bigl(\frac{1}{2}, \chi_{ij}^{-1}\chi_{ik}\chi_{il}^{-1}\Bigr)
         \gamma\Bigl(\frac{1}{2}, \chi_{ij}^{-1}\chi_{ik}^{-1}\chi_{il}\Bigr)
     \Biggr)^{\OneHalf}\\
    &= \nu_\bbP^{-2}\phi^H \x \Biggl( \prod_{ij\in\tetraO} \gamma(1, \chi_{ij}^2) \cdot
        \gamma\Bigl(\OneHalf, \Sigma^-\Bigr)
    \Biggr)^{\OneHalf}.
    \label{eqn:phiH_and_lambdaH}
\end{align}

\subsection{Averaging}
\label{sub:averaging}
Again, the reader may refer to the outline in \Cref{sub:outlinepage}.
First, the $H$-average of $\phi^D$, denoted by $(\phi^H)'$, is given by
\begin{align}
    (\phi^H)'(f)
    &=\int_{[g_i]\in \RO\backslash H}
    \Biggl[\prod_{ij\in\tetraE}\int_{\bbP^1}
        \underbrace{
        \chi_{ij-}^{-1}(\tilde{g}_{i})
        \chi_{ij-}^{-1}(\tilde{g}_{j})
    f_{ij}(z_{ij}\tilde{g}_{i})f_{ji}(z_{ij}\tilde{g}_{j})}_{P_{ij}}
    \dd z_{ij}\Biggr] \dd [g_i],
\end{align}
where we have written $\chi_{ij-}^{-1}(\tilde{g})$ as an abbreviation of
\(\chi_{ij-}^{-1}(\det\tilde{g})\);
the function $P_{ij}$ is $(-2)$-homogeneous in coordinate \(z_{ij}\).
In other words $(\phi^H)'$ is  the integral of the $15$-form
\begin{align} \label{omega1formula}
    \omega_1= \frac{\prod_{ij\in\tetraE} P_{ij}\cdot\dd z_{ij}\prod_{i\in\tetraV}\dd g_i}{\dd g},
\end{align}
over the space 
\begin{align}
    \label{eqn:RO_quotient_15_form}
    \RO\backslash \bigl((\bbP^1)^{\tetraE}\x H\bigr),
\end{align}
where \(\RO\) acts on each \(\bbP^1\)-factor by \(h\colon [x,y]\mapsto [x,y]h^{-1}\),
and where  \(\dd g\) corresponds to the Haar measure on the diagonal \(\RO\).

On the other hand the integral \eqref{phiHformula} $\phi^H$ can be rewritten as an integral
of a $12$-form
\begin{align} \label{12form}
    \omega_2 = \prod_{ij\in\tetraO}Q_{ij}\cdot\prod_{i\in\tetraV}\dd^\bbP g_i
    = \nu_\bbP^{4}\prod_{ij\in\tetraO}Q_{ij}\cdot\prod_{i\in\tetraV}\dd g_i
\end{align}
over \(H\), where
$Q_{ij}=f_{ij}(\tilde{p}_{ij}\tilde{g}_i)\chi_{ij-}^{-1}(\tilde{g}_{i})$ and
\(\nu_\bbP\) is as in \eqref{eqn:Haar_vs_P13_measures}.

To relate these, let us first relate more carefully the spaces over which we are integrating.
There is a natural map
\begin{align} \label{omega2formula}
    A\colon (\bbP^1)^{\tetraE}\x H\longto (\bbP^1)^{\tetraO}, \\
     (z_{ij}=z_{ji}, g_i) \longmapsto (z_{ij} g_i).
\end{align}
It is $\RO$-invariant where $\RO$ is acting on $(\bbP^1)^\tetraE\x H$ as in
\eqref{eqn:RO_quotient_15_form}. We examine its restriction
over the open subset $O = \prod_i O_i$, and descend it to the quotient of the domain by $\RO$:
\begin{align}
    \bar{A}\colon\RO\backslash A^{-1}(O)\longto O \simeq H,
\end{align}
where the final identification uses the orbit map for $p_0$, i.e., the inverse of $h \mapsto p_0 h$.
Then $\bar{A}$ is submersive; $(\phi^H)'$ is given by an integral over the domain of $\bar{A}$, whereas
$\phi^H$ is given by an integral over its range:
\begin{equation} \label{phiHforms} (\phi^H)' = \int_{\RO \backslash A^{-1}(O)} \omega_1 \mbox{ and } \phi^H = \int_{H} \omega_2.\end{equation}
Note that the forms $\omega_1, \omega_2$ are not quite algebraic because they involve the $f_{ij}$
that are simply smooth functions. It is more convenient to relate them in a more algebraic version;

\begin{lemma}     \label[lemma]{lem:edge_integrand_with_order_convention}
    Define
    \begin{align}
        \tilde{\Omega} =  \nu_\bbP^{-4} \prod_{ij \in \tetraE} dz_{ij} \cdot \prod_{ij\in\tetraO}
        \chi_{ij-}\bigl((z_{ij}\wedge z_{ik})(z_{il}\wedge z_{ij})(z_{ik}\wedge z_{il})^{-1}\bigr),
    \end{align}
    which we consider  as a differential $6$-form on $(\bbP^1)^\tetraE$, where
    \begin{enumerate}
                \item In the product over $ij \in \tetraV$, we always require \(\Set{i,j,k,l}=\tetraV\), and the ordering of \(k\) and \(l\) is
            \emph{opposite} for any pair of \(\chi_{ij}\) and \(\chi_{ji}\);
        \item we choose for $ij\in\tetraE$ coordinates  \((x_{ij},y_{ij})\) in
            \((\bbA^2)^{\tetraE}\) with $z_{ij}$ the corresponding points in
            $(\bbP^1)^{\tetraE}$, and write
            \begin{align}
                z_{ij}\wedge z_{ik}\defeq \begin{vmatrix}
                    x_{ij} & y_{ij}\\
                    x_{ik} & y_{ik}
                \end{vmatrix},\quad
                \dd z_{ij}=x_{ij}\dd y_{ij}-y_{ij}\dd x_{ij}.
            \end{align}
    \end{enumerate}

    Then we have an equality of differential (regular) $18$-forms on \(A^{-1}(O)\):
    \begin{align}
        \tilde{\omega}_1^{\mathrm{alg}}
        = A^*(\omega_2^{\mathrm{alg}}) \wedge \tilde{\Omega},
    \end{align}
    where  $\tilde{\omega}_1^{\mathrm{alg}}$ is the numerator of \eqref{omega1formula}, but replacing the $f_{ij}$
            by an {\em algebraic} function $f_{ij}^{\mathrm{alg}}$ defined on a
            Zariski-open subset, with the same degree of homogeneity,  and
            making the same substitution in the definition of $\omega_2^{\mathrm{alg}}$.
\end{lemma}

\subsection{Conclusion of the proof, assuming \Cref{lem:edge_integrand_with_order_convention}}
\label{sub:summary_for_pseries}
From the lemma we readily deduce that
 \begin{align}
    \int_{\RO\backslash A^{-1}(O)} \omega_1
    = \int_{O \simeq H} \omega_2 \cdot \biggl(\int_{\bar{A}^{-1}(p)}
    \tilde{\Omega}{\dd g} \biggr),
  \end{align}
for any choice whatsoever of $f_{ij}$. (Just integrate both sides of the lemma,
but after multiplying by the ratio
$f_{ij}/\abs{f_{ij}^{\mathrm{alg}}}$.)
Taking $f_{ij}$
to be supported in a very small neighbourhood of $p$, and using
\eqref{phiHforms},
we deduce that 
\begin{align} \label{FiberOmega}
    (\phi^H)' =  (\phi^H)
    \times \biggl(\int_{\bar{A}^{-1}(p_0)} \Omega(p_0) \biggr).
\end{align}

If we combine  this with
\Cref{lem:edge_integrand_with_order_convention}, we find that $(\phi^H)'$
equals $\phi^H$ multiplied by the normalization factor:

\begin{align}
    \nu_\bbP^{-4}\int_{\RO\backslash (\bbP^1)^{\tetraE}}\prod_{ij\in\tetraO}
    \chi_{ij-}\left((z_{ij}\wedge z_{ik})(z_{il}\wedge
    z_{ij})(z_{ik}\wedge z_{il})^{-1}\right)
    \frac{\prod_{ij\in\tetraE}\dd
    z_{ij}}{\dd g},
\end{align}
with the ordering convention for \(k\) and \(l\) as stated in
\Cref{lem:edge_integrand_with_order_convention}.
Combining this with the known relationships \eqref{eqn:phiD_and_lambdaD} and
\eqref{eqn:phiH_and_lambdaH} between $\Lambda^H,\Lambda^D$ and $\phi^H, \phi^D$
completes the proof of \Cref{prop:edge_integral}, hence also
\Cref{thm:6j=Hypergeometric}.

\subsection{The proof of \Cref{lem:edge_integrand_with_order_convention}} \label{uglylemmaproof}
Working Zariski locally on $\bbP^1$ we fix
algebraic choices of homogeneous coordinate representatives $z \mapsto \tilde{z}$
and similarly working locally on $\PGL_2$ we choose representatives $g \mapsto \tilde{g}$
in $\GL_2$.
Define functions $\lambda_{ij}$ Zariski-locally on $(\bbP^1)^{\tetraE} \times H$
by requiring
\begin{align} \label{Zlambda}
    \lambda_{ij} z_{ij} \tilde{g}_i =\tilde{p}_{ij},\quad
    \lambda_{ji} z_{ij} \tilde{g}_j = \tilde{p}_{ji}.
\end{align}
 In view of   \eqref{omega1formula} and
\eqref{12form} what we must prove is

    \begin{align}
        \prod_{ij\in\tetraE} \chi_{ij-}^{-1}(\tilde{g}_i\lambda_{ij})
        = \prod_{ij\in\tetraO}
        \chi_{ij-}\bigl((z_{ij}\wedge z_{ik})(z_{il}\wedge z_{ij})(z_{ik}\wedge z_{il})^{-1}\bigr)
    \end{align} 
    Writing \begin{align}
        \tilde{g}_i^{-1}= \begin{bmatrix}
            a_i & b_i\\
            c_i & d_i
        \end{bmatrix}.
    \end{align}
    the condition \eqref{Zlambda} then amounts to  
    \begin{multline}
        \begin{bmatrix}
        & \lambda_{\V1\V2}(x_{\V1\V2},y_{\V1\V2}) &
            \lambda_{\V1\V3}(x_{\V1\V3},y_{\V1\V3}) &
            \lambda_{\V1\V4}(x_{\V1\V4},y_{\V1\V4}) \\
            \lambda_{\V2\V1}(x_{\V1\V2},y_{\V1\V2}) & &
            \lambda_{\V2\V3}(x_{\V2\V3},y_{\V2\V3}) &
            \lambda_{\V2\V4}(x_{\V2\V4},y_{\V2\V4}) \\
            \lambda_{\V3\V1}(x_{\V1\V3},y_{\V1\V3}) &
            \lambda_{\V3\V2}(x_{\V2\V3},y_{\V2\V3}) & &
            \lambda_{\V3\V4}(x_{\V3\V4},y_{\V3\V4}) \\
            \lambda_{\V4\V1}(x_{\V1\V4},y_{\V1\V4}) &
            \lambda_{\V4\V2}(x_{\V2\V4},y_{\V2\V4}) &
            \lambda_{\V4\V3}(x_{\V3\V4},y_{\V3\V4}) & 
        \end{bmatrix}\\
        = \begin{bmatrix}
        & (c_{\V1},d_{\V1}) &
            (a_{\V1}+c_{\V1},b_{\V1}+d_{\V1}) &
            (a_{\V1},b_{\V1}) \\
            (c_{\V2},d_{\V2}) & &
            (a_{\V2},b_{\V2}) &
            (a_{\V2}+c_{\V2},b_{\V2}+d_{\V2}) \\
            (a_{\V3}+c_{\V3},b_{\V3}+d_{\V3}) &
            (a_{\V3},b_{\V3}) & &
            (c_{\V3},d_{\V3}) \\
            (a_{\V4},b_{\V4}) &
            (a_{\V4}+c_{\V4},b_{\V4}+d_{\V4}) &
            (c_{\V4},d_{\V4}) & 
        \end{bmatrix}.
    \end{multline}
    Observe that
    \begin{align}
        \det{\tilde{g}_{\V1}^{-1}}
    &=-\lambda_{\V{12}}\lambda_{\V{13}}\begin{vmatrix}
        x_{\V{12}} & y_{\V{12}}\\
        x_{\V{13}} & y_{\V{13}}
    \end{vmatrix}
    =-\lambda_{\V{13}}\lambda_{\V{14}}\begin{vmatrix}
        x_{\V{13}} & y_{\V{13}}\\
        x_{\V{14}} & y_{\V{14}}
    \end{vmatrix}
    =\lambda_{\V{14}}\lambda_{\V{12}}\begin{vmatrix}
        x_{\V{14}} & y_{\V{14}}\\
        x_{\V{12}} & y_{\V{12}}
    \end{vmatrix}\\
    &=\lambda_{\V{12}}^2\frac{\begin{vmatrix}
            x_{\V{12}} & y_{\V{12}}\\
            x_{\V{13}} & y_{\V{13}}
            \end{vmatrix}\begin{vmatrix}
            x_{\V{14}} & y_{\V{14}}\\
            x_{\V{12}} & y_{\V{12}}
            \end{vmatrix}}{\begin{vmatrix}
            x_{\V{13}} & y_{\V{13}}\\
            x_{\V{14}} & y_{\V{14}}
    \end{vmatrix}},
    \end{align}
    and so we have
    \begin{align}
        \psi_{\V{12}}(\tilde{g}_{\V1})\chi_{\V{12}}^{-2}(\lambda_{\V{12}})\abs{\lambda_{\V{12}}}
        =\psi_{\V{12}}^{-1}\left(\begin{bmatrix}
                x_{\V{12}} & y_{\V{12}}\\
                x_{\V{13}} & y_{\V{13}}
                \end{bmatrix}\begin{bmatrix}
                x_{\V{14}} & y_{\V{14}}\\
                x_{\V{12}} & y_{\V{12}}
                \end{bmatrix}\begin{bmatrix}
                x_{\V{13}} & y_{\V{13}}\\
                x_{\V{14}} & y_{\V{14}}
        \end{bmatrix}^{-1}\right).
    \end{align}
    Similarly, for any \(i\neq j\), if we let \(k,l\) be the remaining two vertices,
    and let \(\epsilon_{ijkl}\) be the sign of the permutation that sends
    \((\V1,\V2,\V3,\V4)\) to \((i,j,k,l)\), then we have
    \begin{align}
        \psi_{ij}(\tilde{g}_i)\chi_{ij}^{-2}(\lambda_{ij})\abs{\lambda_{ij}}
        =\chi_{ij}(\epsilon_{ijkl})\psi_{ij}^{-1}\left(\begin{bmatrix}
                x_{ij} & y_{ij}\\
                x_{ik} & y_{ik}
                \end{bmatrix}\begin{bmatrix}
                x_{il} & y_{il}\\
                x_{ij} & y_{ij}
                \end{bmatrix}\begin{bmatrix}
                x_{ik} & y_{ik}\\
                x_{il} & y_{il}
        \end{bmatrix}^{-1}\right),
    \end{align}
    where for any \(i>j\) we let \(x_{ij}=x_{ji}\), and so on.
    This expression is independent of the ordering of \(k\) and \(l\), as expected.
    Note that we always have \(\chi_{ij}(\epsilon_{ijkl})^2=1\), and so
    $\chi_{ij}(\epsilon_{ijkl})\chi_{ji}(\epsilon_{jilk})=1.$
 This concludes the proof.

\section{Weyl symmetry: proof of \Cref{thm:W_D6_symmetry}}
\label{sec:proof_of_Weyl_symmetry}
We will prove \Cref{thm:W_D6_symmetry}, namely,
that the tetrahedral symbol for principal series has a $W(\TypeD_6)$-symmetry,
up to an explicit cocycle.

\subsection{A Fourier duality}
We continue with the notation of
\Cref{sec:HypergeometricIntro}
but now specialize to the case $n=2k$, i.e. $X$ is one-half the
dimension of its ambient space. Let $X^{\perp}$
be the orthogonal complement to $X$ inside $F^n$ with respect
to the usual pairing. 
We choose Haar measures \(\dd x\) on \(X\) and \(\dd y\) on \(X^\perp\) so that
\(\dd x\dd y=(\dd\mu)^n\), and \(\dd x\) and \(\dd y\) are Fourier
transforms of each other. Recall from \eqref{CheckChiEquation} that
\(\check{\chi}=\bigl(\chi(-1)\gamma(1,\chi)\chi_1\bigr)^{-1}\) is
the Fourier transform of any quasi-character \(\chi\) with respect to \(\Psi\) and \(\dd\mu\).
\begin{proposition}
    \label[proposition]{prop:Fourier_duality}
    Suppose \(n=2k\), all the characters
    \(\chi_i\) have the form \(\alpha_{i}^{-1}\abs{\blank}^{-\OneHalf}\)
    with the \(\alpha_i\) unitary, and the integrals of
    \(\abs{x}^{-\OneHalf}\defeq\prod_{i=1}^n\abs{x_i}^{-\OneHalf}\) (where
    \(x_i\) are the standard coordinates on \(F^n\)) over both \(\bbP X\) and
    \(\bbP X^\perp\) are convergent. Then we have the
    following \emph{Fourier duality}
    \begin{align}
        \int_{\mathbb{P}X} \ul \chi
        =\prod_{i=1}^n\chi_i(-1)\cdot \int_{\mathbb{P}X^\perp} \ul{\check\chi}.
    \end{align}
    where $\int_{\mathbb{P}X^\perp} \ul{\check\chi}$ is, as in \eqref{IXchidef}, the integral of
    $\prod_i\check{\chi}_i$ over the projectivation of $X^\perp$.
\end{proposition}

Let us remark that {\em formally} this is a simple application of Fourier
duality: the equality follows readily from the definition if $\ul \chi$ is
replaced by a Schwartz function on $F^n$. The proof is longer  only because of
issues of convergence.

\begin{proof}
    Let \(\Phi\) be a Schwartz function on \(F^n\). Let \(T\) be the torus
    $(F^{\times})^n$ which we understand to act on $F^n$ in the obvious way. We
    equip $T$ with the Haar measure \(\prod_{i=1}^n\dd\mu(t_i)/\abs{t_i} \) and
    for \(t \in T\) we write
    \begin{align}
        \abs{t} = \prod_{i=1}^n \abs{t_i},\quad \alpha(t) = \prod_{i=1}^n \alpha_i(t_i).
    \end{align}
    We extend them by zero to functions on \(F^n\).

    Let $T'\subset T$ be any complement to the 
    central scaling copy of $F^{\times}$, e.g.
    we can take $T'$ to be the copy of $(F^{\times})^{n-1}$
    which scales the first $n-1$ coordintaes. The Haar measure on \(T'\) is chosen such
    that its product with \(\dd\mu/\abs{\blank}\) on the central copy of
    \(F^\x\) is the Haar measure on \(T\) above.

    We first prove that we have an equality of \emph{absolutely convergent} integrals:
    \begin{align} \label{eqn:Phiclaim}
        J_X(\alpha, \Phi) \defeq \int_{T' \times X} \alpha(t) \abs{t}^{\OneHalf}
        \Phi(tx)
        = \int_{\bbP X} \alpha^{-1}(x) \abs{x}^{-\OneHalf} \cdot \int_{F^n}
        \alpha(v) \abs{v}^{-\OneHalf}\Phi(v),
    \end{align} 
    with \(v=tx\) and the measures specified above.

    We first note that the two integrals on the right are absolutely convergent,
    the first by assumption, and the second because it is bounded
    by a product of one-variable integrals of the general form $\int_{\Omega}
    \abs{a}^{-\OneHalf} \dd a$
    for a compact set $\Omega $. To verify that the integral on the left is  absolutely convergent,
    it is sufficient to verify that, for positive $\Phi$, the iterated integral
    \begin{align}
        \int_{X} \int_{T'}  \abs{t}^{\OneHalf} \Phi(tx) < \infty.
    \end{align}
    We can rewrite this as (iterated integrals)
    \begin{align}
        \int_{\mathbb{P}X } \int_{\lambda \in F^{\times}} \int_{T'}
        \abs{t}^{\OneHalf} \abs{\lambda}^{\frac{n}{2}} \Phi(\lambda tx)
        = \int_{\mathbb{P} X} \int_{T} \abs{t}^{\OneHalf} \Phi(tx) .
    \end{align}
    The inner integral is, by what we already noted, finite 
    as long as $x$ lies on no coordinate hyperplane, and its value
    is equal to a constant multiple of $\abs{x}^{-\OneHalf}$; and by
    assumption $\int_{\bbP X} \abs{x}^{-\OneHalf} <\infty$. So the left hand integral, too, is absolutely convergent. 

    Since the integral $J(\alpha,\Phi)$ is absolutely convergent,
    it can be evaluted in whatever order we please, and
    we may now compute:
    \begin{align}
        J_X(\alpha, \Phi) &= \int_{T'} \alpha(t) \abs{t}^{\OneHalf} \int_X
        \Phi(tx)\dd x\\
        &= \int_{T'} \alpha(t) \abs{t}^{-\OneHalf} \int_{X^{\perp}}
        \check{\Phi}(t^{-1}y)\dd y\\
        &= \int_{T'} \alpha^{-1}(t) \abs{t}^{\OneHalf} \int_{X^{\perp}}
        \check{\Phi}(ty)\dd y\\
        &= J_{X^\perp}(\alpha^{-1}, \check{\Phi}),
    \end{align}
    where, at the third step, we inverted $t$, and at the second step, we used
    Fourier duality on \(F^n\) for the Schwartz function $\Phi(tx)$ and the
    distribution \(\dd x\), and the fact that the Fourier transform of
    \(\Phi(tx)\) is \(\abs{t}^{-1}\check{\Phi}(t^{-1}y)\).

    That is to say, we have proved that
    \begin{align}
        \int_{\bbP X} \alpha^{-1}(x) \abs{x}^{-\OneHalf} \cdot  \int_{F^n}
        \alpha(v) \abs{v}^{-\OneHalf} \Phi(v)
    \end{align}
    is symmetric under replacement of \(X\) by \(X^\perp\), $\Phi$ by
    $\check{\Phi}$ and $\alpha$ by $\alpha^{-1}$. Combined with another Fourier
    duality from \Cref{sec:Fadjoint}:
    \begin{align}
        \int_{F^n}\alpha(v)\abs{v}^{-\OneHalf}\Phi(v)
        =\int_{F^n}\FT(\alpha\abs{\blank}^{-\OneHalf})(-v)\check{\Phi}(v)
        \stackrel{\eqref{CheckChiEquation}}{=}\gamma(1,\alpha\abs{\blank}^{-\OneHalf})^{-1}
        \int_{F^n}\alpha^{-1}(v)\abs{v}^{-\OneHalf}\check{\Phi}(v),
    \end{align}
    we have
    \begin{align}
        \gamma(1,\alpha\abs{\blank}^{-\OneHalf})^{-1}\int_{\bbP X} \alpha^{-1}(x) \abs{x}^{-\OneHalf}
        =\int_{\bbP X^\perp} \alpha(x) \abs{x}^{-\OneHalf},
    \end{align}
    which translates to the desired equality.
\end{proof}

\subsection{The Weyl symmetry}\label{sub:the_weyl_symmetry}

In order to explain the indexing of characters here,
we combine \eqref{eqn:chosen_oriented_tetrahedron} and
\eqref{eqn:the_oriented_tetrahedron_with_dices_in_partI} in the following diagram:
\begin{equation}
    \label{eqn:the_oriented_tetrahedron_with_dices}
    \begin{tikzcd}
        && \V1
        \ar[llddd,"x","\EN{1}"']\ar[dd,"0"',"\EN{3}"] && \\
        && \phantom{x} &&\\
        && \V4  &&\\
        \V2\ar[rrrr,"\infty"',"\EN{6}"]\ar[rru,"y"',"\EN{5}"] && &&
        \V3\ar[llu,"z","\EN{4}"']\ar[lluuu,"w","\EN{2}"']
    \end{tikzcd}
\end{equation}
where we label each edge with both a blackboard bold number and a
coordinate that is consistent
with \Cref{thm:6j=Hypergeometric}, which we proved in
\Cref{sec:edge_formula_for_principal_series}.

We express \(\Pisymbol\) in terms of the hypergeometric integral
\eqref{eqn:edge_integral_hypergeometric}.
Let us first recall some abbreviations that we will use.
We shall use the following type of abbreviation (given by an example;
cf.~\Cref{sub:halfspin}):
\begin{align}
    \chi_{\V{2 \bar{3}\bar{4}}} \defeq \chi_{\V1 \V2} \chi_{\V1 \V3}^{-1}
    \chi_{\V1 \V4}^{-1},\quad
    \gamma_{\V{2\bar{3}\bar{4}}}\defeq
    \gamma\Bigl(\OneHalf,\chi_{\V1\V2}\chi_{\V1\V3}^{-1}\chi_{\V1\V4}^{-1}\Bigr).
\end{align}
We further abridge  the special cases when there are one or two inverted
characters:
\begin{align}
    \chi_i^l \defeq \chi_{j k \bar{l}}= \frac{\chi_{ij}\chi_{ik}}{\chi_{il}}
    \quad\text{ and }\quad
    \tilde{\chi}_{i}^l \defeq \chi_{\bar{j} \bar{k} l} = 
    \frac{\chi_{il} }{\chi_{ij} \chi_{ik}},
\end{align}
and by extension
\begin{align}
    \gamma_i^l\defeq \gamma_{j k \bar{l}} = \gamma\Bigl(\OneHalf,\chi_i^l\Bigr).
\end{align}

The convergence claim of \Cref{prop:edge_integral} (proved in
\Cref{sub:pf_of_edge_integral_convergence}) then allows us to
apply \Cref{prop:Fourier_duality} to
\eqref{eqn:edge_integral_hypergeometric}, and we arrive at
\begin{multline}
    \Pisymbol\sqrt{\gamma\Bigl(\OneHalf,\Sigma^-\Bigr)}
    = \nu_\bbP^{-1}\bigl[\gamma_{\V2}^{\V3}
    \gamma_{\V4}^{\V1}
    \gamma_{\V3}^{\V2}
    \gamma_{\V1}^{\V4}
    \gamma_{\V1}^{\V3}
    \gamma_{\V4}^{\V3}
    \gamma_{\V4}^{\V2}
    \gamma_{\V1}^{\V2}\bigr]^{-1}\\
    \x \int_{[x,y,z,w]\in\bbP^3(F)}
    \tilde{\chi}_{\V2-}^{\V3}(x)
    \tilde{\chi}_{\V4-}^{\V1}(y)
    \tilde{\chi}_{\V3-}^{\V2}(z)
    \tilde{\chi}_{\V1-}^{\V4}(w)
    \tilde{\chi}_{\V1-}^{\V3}(w-x)
    \tilde{\chi}_{\V4-}^{\V3}(x-y)
    \tilde{\chi}_{\V4-}^{\V2}(y-z)
    \tilde{\chi}_{\V1-}^{\V2}(z-w)
    ,
    \label{eqn:Jsymbol_after_Fourier_duality}
\end{multline}
where we used \eqref{CheckChiEquation}
\begin{align}
    \check{\chi}=\bigl(\chi(-1)\gamma(1,\chi)\chi_1\bigr)^{-1},
\end{align}
and that
\begin{align}
    \bigl(\chi_{\V2}^{\V3}
    \chi_{\V4}^{\V1}
    \chi_{\V3}^{\V2}
    \chi_{\V1}^{\V4}\bigr)(-1)
    =\bigl(\chi_{\V1}^{\V3}
    \chi_{\V4}^{\V3}
    \chi_{\V4}^{\V2}
    \chi_{\V1}^{\V2}\bigr)(-1)
    =1.
\end{align}

The new integral in \eqref{eqn:Jsymbol_after_Fourier_duality} may be interpreted
as the integral \eqref{eqn:edge_integral_hypergeometric} with a different set of
\(\chi_{ij}\), the latter obtained by performing 
what we will call an
\emph{inv}-Regge symmetry through the  pair of opposite edges
$\EN{2}$ and \(\EN{5}\).
(for discussion, see \Cref{GroupTheory2};
it relates to the usual Regge symmetry by composition with inversion of each character). 
 The said transformation may formally be written as follows:
\begin{alignat}{2}
    \EN{1}&\mapsto \EN{1}'\defeq\frac{\EN{1}+\IEN{3}+\IEN{4}+\IEN{6}}{2},
          &\qquad
    \EN{2}&\mapsto \EN{2}'\defeq\IEN{2},\\
    \EN{3}&\mapsto \EN{3}'\defeq\frac{\IEN{1}+\EN{3}+\IEN{4}+\IEN{6}}{2},
          &\qquad
    \EN{4}&\mapsto \EN{4}'\defeq\frac{\IEN{1}+\IEN{3}+\EN{4}+\IEN{6}}{2},\\
    \EN{5}&\mapsto \EN{5}'\defeq\IEN{5},
          &\qquad
    \EN{6}&\mapsto \EN{6}'\defeq\frac{\IEN{1}+\IEN{3}+\IEN{4}+\EN{6}}{2},
\end{alignat}
where \(\IEN{1}\) means formally negating \(\EN{1}\) and so on.
Thus, for example, the new character \(\chi_{\EN{1}'}\) associated to \emph{oriented} edge
\(\V{12}\) after the transformation satisfies the relation
\begin{align}
    \chi_{\EN{1}'}^2
    =\chi_{\scriptstyle\EN{1}}\chi_{\EN{3}}^{-1}\chi_{\EN{4}}^{-1}\chi_{\EN{6}}^{-1}
    =\chi_{\V{12}}\chi_{\V{41}}\chi_{\V{43}}\chi_{\V{32}},
\end{align}
and similarly for other edges.
This does not uniquely determine \(\chi_{\EN{1}'}\) and so on:
indeed, the character appearing on the right need not even have a square root.
However, this is not an issue: it is simply a reflection of the fact that the
map \eqref{eqn:Principal_to_type_D} is not surjective; rather, this
transformation does determine the (new) integrand in
\eqref{eqn:edge_integral_hypergeometric}. For example, the first character
\(\chi_{\V2-}^{\V3}\) in the integrand is the same
as \begin{align}
    \chi_{\IEN{1}\EN{5}\IEN{6}-}
    \defeq \chi_{\EN{1}}^{-1}\chi_{\EN{5}}\chi_{\EN{6}}^{-1}\abs{\blank}^{-\OneHalf},
\end{align}
and after the inv-Regge symmetry, we have
\begin{align}
    \chi_{\IEN{1}'\EN{5}'\IEN{6}'-}
    =\chi_{\EN{3}}\chi_{\EN{4}}\chi_{\EN{5}}^{-1}\abs{\blank}^{-\OneHalf}
    =\chi_{\EN{3}\EN{4}\IEN{5}-}
    =\tilde{\chi}_{\V4-}^{\V2}.
\end{align}
Thus, the new edge integral after inv-Regge symmetry is the integral in the expression:
\begin{align}
    \EI(\bbP_F^1,\psi')=\nu_\bbP\int
    \tilde{\chi}_{\V4-}^{\V2}(x-y)
    \tilde{\chi}_{\V4-}^{\V3}(y-z)
    \tilde{\chi}_{\V1-}^{\V3}(z-w)
    \tilde{\chi}_{\V1-}^{\V2}(w-x)
    \tilde{\chi}_{\V3-}^{\V2}(x)
    \tilde{\chi}_{\V4-}^{\V1}(y)
    \tilde{\chi}_{\V2-}^{\V3}(z)
    \tilde{\chi}_{\V1-}^{\V4}(w)
\end{align}

which is easily seen equal to the integral in
\eqref{eqn:Jsymbol_after_Fourier_duality} by exchanging \(x\) with \(z\) (and
using the fact that
\(\tilde{\chi}_{\V4}^{\V2}\tilde{\chi}_{\V4}^{\V3}\tilde{\chi}_{\V1}^{\V3}\tilde{\chi}_{\V1}^{\V2}(-1)=1\)).

Let \(\PisymbolPrime\) be the tetrahedral symbol associated with the
new set of characters after performing the said inv-Regge symmetry, and
\(\gamma(\OneHalf,(\Sigma^-)')\) be the corresponding \(\gamma\)-factor (see
\Cref{sub:the_principal_series_case}). Then we have by \eqref{eqn:edge_integral_formula}
\begin{align}
    \PisymbolPrime\sqrt{\gamma\Bigl(\OneHalf,(\Sigma^-)'\Bigr)}=\nu_\bbP^{-2}\EI(\bbP_F^1,\psi').
\end{align}
Therefore, we showed that
\begin{align}
    \label{eqn:inv_Regge_pre}
    \Pisymbol
    =\PisymbolPrime\x
    \frac{\bigl[\gamma_{\V2}^{\V3} \gamma_{\V4}^{\V1} \gamma_{\V3}^{\V2} \gamma_{\V1}^{\V4}
        \gamma_{\V1}^{\V3} \gamma_{\V4}^{\V3} \gamma_{\V4}^{\V2}
    \gamma_{\V1}^{\V2}
\bigr]^{-1}\sqrt{\gamma(\OneHalf,(\Sigma^-)')}}{\sqrt{\gamma(\OneHalf,\Sigma^-)}}.
\end{align}

\begin{proposition}
    \label[proposition]{prop:inv_Regge_symmetry}
    We have an equality up to sign:
    \begin{align}
        \Pisymbol=\PisymbolPrime.
    \end{align}
\end{proposition}
\begin{proof}
    It amounts to showing that the fraction on the right-hand side of
    \eqref{eqn:inv_Regge_pre} equals \(\pm 1\).
    Indeed, by definition,
    \begin{align}
        \gamma\Bigl(\OneHalf,\Sigma^-\Bigr)
        &=\gamma_{\V{2\bar{3}\bar{4}}}\gamma_{\V{\bar{2}3\bar{4}}}\gamma_{\V{\bar{2}\bar{3}4}}\gamma_{\V{\bar{2}\bar{3}\bar{4}}}
        \x
        \gamma_{\V{1\bar{3}\bar{4}}}\gamma_{\V{\bar{1}3\bar{4}}}\gamma_{\V{\bar{1}\bar{3}4}}\gamma_{\V{\bar{1}\bar{3}\bar{4}}}\\
        &\x \gamma_{\V{1\bar{2}\bar{4}}}\gamma_{\V{\bar{1}2\bar{4}}}\gamma_{\V{\bar{1}\bar{2}4}}\gamma_{\V{\bar{1}\bar{2}\bar{4}}}
        \x \gamma_{\V{1\bar{2}\bar{3}}}\gamma_{\V{\bar{1}2\bar{3}}}\gamma_{\V{\bar{1}\bar{2}3}}\gamma_{\V{\bar{1}\bar{2}\bar{3}}},
    \end{align}
    whereas after the inv-Regge symmetry we find
    \begin{align}
        \gamma\Bigl(\OneHalf,(\Sigma^-)'\Bigr)
        &=\gamma_{\V{23\bar{4}}}\gamma_{\V{1\bar{2}4}}\gamma_{\V{\bar{2}34}}\gamma_{\V{\bar{1}\bar{2}4}}
        \x
        \gamma_{\V{\bar{1}\bar{2}\bar{3}}}\gamma_{\V{1\bar{2}3}}\gamma_{\V{\bar{1}\bar{3}\bar{4}}}\gamma_{\V{\bar{1}\bar{3}4}}\\
        &\x
        \gamma_{\V{\bar{1}\bar{2}\bar{4}}}\gamma_{\V{2\bar{3}4}}\gamma_{\V{\bar{2}\bar{3}\bar{4}}}\gamma_{\V{1\bar{2}\bar{4}}}
        \x
        \gamma_{\V{12\bar{3}}}\gamma_{\V{1\bar{3}4}}\gamma_{\V{\bar{1}23}}\gamma_{\V{1\bar{3}\bar{4}}}.
    \end{align}
    As a result, we have (up to sign)
    \begin{align}
        \label{eqn:difference_of_16_factors}
        \sqrt{\frac{\gamma(\OneHalf,(\Sigma^-)')}{\gamma(\OneHalf,\Sigma^-)}}
        &=
        \sqrt{
        \frac{\gamma_{\V{23\bar{4}}}}{\gamma_{\V{\bar{2}\bar{3}4}}}
        \frac{\gamma_{\V{2\bar{3}4}}}{\gamma_{\V{\bar{2}3\bar{4}}}}
        \frac{\gamma_{\V{\bar{2}34}}}{\gamma_{\V{2\bar{3}\bar{4}}}}
        \frac{\gamma_{\V{1\bar{3}4}}}{\gamma_{\V{\bar{1}3\bar{4}}}}
        \frac{\gamma_{\V{1\bar{2}4}}}{\gamma_{\V{\bar{1}2\bar{4}}}}
        \frac{\gamma_{\V{\bar{1}23}}}{\gamma_{\V{1\bar{2}\bar{3}}}}
        \frac{\gamma_{\V{1\bar{2}3}}}{\gamma_{\V{\bar{1}2\bar{3}}}}
        \frac{\gamma_{\V{12\bar{3}}}}{\gamma_{\V{\bar{1}\bar{2}3}}}}\\
        &=
        \gamma_{\V{23\bar{4}}}
        \gamma_{\V{2\bar{3}4}}
        \gamma_{\V{\bar{2}34}}
        \gamma_{\V{1\bar{3}4}}
        \gamma_{\V{1\bar{2}4}}
        \gamma_{\V{\bar{1}23}}
        \gamma_{\V{1\bar{2}3}}
        \gamma_{\V{12\bar{3}}}\\
        &=
        \gamma_{\V1}^{\V4}
        \gamma_{\V1}^{\V3}
        \gamma_{\V1}^{\V2} 
        \gamma_{\V2}^{\V3}
        \gamma_{\V3}^{\V2}
        \gamma_{\V4}^{\V1}
        \gamma_{\V4}^{\V2}
        \gamma_{\V4}^{\V3},
    \end{align}
    where we again used the symmetry properties \eqref{GammaSymmetry} of the $\gamma$-factor; and the fact that
    \begin{align}
        \bigl(\chi_{\V2}^{\V3}
            \chi_{\V4}^{\V1}
            \chi_{\V3}^{\V2}
            \chi_{\V1}^{\V4}\chi_{\V1}^{\V3}
            \chi_{\V4}^{\V3}
            \chi_{\V4}^{\V2}
        \chi_{\V1}^{\V2}\bigr)(-1)
        =1.
    \end{align}
    This finishes the proof.
\end{proof}

\subsection{The proof of $W(\TypeD_6)$-symmetry, completed} \label{thmD6proofsection}

\begin{proof}[Proof of \Cref{thm:W_D6_symmetry}]
    Let us begin with the first claim (1), that $\chisymbol^2$ descends to
    $\TypeD_6 \otimes \widehat{F^{\times}}$ and is $W(\TypeD_6)$-invariant.
    Formula  \eqref{eqn:hypergeometric_edge_integral} shows that
    \(\chisymbol^2\) depends only on characters of the form
    \(\chi_{ij}^\pm\chi_{ik}^\pm\chi_{il}^\pm\), which gives the descent.   In
    the language of \Cref{GroupTheory2}, $\chisymbol^2$ is evidently invariant
    both by the group of orientation reversals (since these do not change the
    isomorphism class of the underlying representation) and also by the group of
    tetrahedral symmetries (by the way we defined it).    Finally,
    \Cref{prop:inv_Regge_symmetry} shows that $\chisymbol^2$ is invariant by at
    least one Regge symmetry; together with the prior groups, this generates all
    of $W(\TypeD_6)$. This concludes the proof of the first claim.

    We now pass to claim (2).  By \eqref{eqn:edge_integral_formula} we may take
    \(I(\chi)\) to be the integral \(\EI(\bbP_F^1,\psi)\), multiplied by
    $\nu_{\bbP}^{-2}$. Then \(a(w,\psi)=I(w^{-1}\chi)/I(\chi)\) is automatically
    a cocycle of $W(\TypeD_6)$ valued in functions of $\psi$.  By what we have already proved, namely, that $\chisymbol^2$ is
    $W(\TypeD_6)$-invariant, the validity of
    \eqref{eqn:the_cocycle_with_sign_formula} for {\em some} choice of signs
    $\iota(w, \chi)$ follows; it remains to describe the signs.
  
    Since \(a\) is a cocycle, it suffices to compute \(\iota\) on the generators
    of \(W\). By \Cref{lem:generators_of_WD6}, it suffices to consider \(\FRI\),
    \(\FRTV\) and a single element in \(\FRR\). Some of
    them are easy:
    the \(\tetraV\)-tetrahedral group \(\FRTV\) preserves $\Sigma^+$ and $\Sigma^-$,
    and evidently also \(\EI(\bbP_F^1,\psi)\), so they have
    trivial signs; by \eqref{eqn:Jsymbol_after_Fourier_duality}, the one
    particular inv-Regge symmetry in \Cref{prop:inv_Regge_symmetry} also has
    trivial sign.

    We then compute the sign for \(w\) being the edge
    flipping operations, and it suffices to assume that the edge is \(\V{12}\),
    and that all \(\chi_{ij}\) are unitary. We will need to use the results from
    \Cref{sub:Mellin_transform_of_hypergeometric_over_local_field,sub:Mellin_of_tetrahedral_symbols}
    (whose argument is purely analytic based on the integral formula
    \eqref{eqn:edge_integral_formula} and independent of
    \Cref{thm:W_D6_symmetry}). In short, \eqref{eqn:edge_integral_formula} may
    be written as \(f(1)\) for some smooth function \(f\) on \(F\) whose Mellin
    transform (as a function on \(\widehat{F^\x}\)) is
    \begin{align}
        \lambda\longmapsto
        \bigl[\gamma(\chi_{\V{1}+}^{\V{4}})\gamma(\chi_{\V{2}+}^{\V{3}})
        \gamma(\chi_{\V{4}+}^{\V{1}})\gamma(\chi_{\V{3}+}^{\V{2}})\bigr]^{-1}
        \frac{\gamma(\lambda\otimes A_+)}{\gamma(\lambda\otimes B)},
    \end{align}
    where
    \begin{align}
        A &= \Set*{\chi_{\V{12}}\chi_{\V{31}}\chi_{\V{41}},
            \chi_{\V{21}}\chi_{\V{41}}\chi_{\V{31}},
            \chi_{\V{43}}\chi_{\V{31}}\chi_{\V{23}},
        \chi_{\V{34}}\chi_{\V{31}}\chi_{\V{23}}}\\
        B &= \Set*{\chi_{\V{31}}^2,
            \chi_{\V{42}}\chi_{\V{41}}\chi_{\V{31}}\chi_{\V{23}},
        \chi_{\V{41}}\chi_{\V{31}}\chi_{\V{24}}\chi_{\V{23}}, 1}.
    \end{align}
    It is clear that exchanging \(\V{12}\) with \(\V{21}\) only changes the
    constant factor
    \(\bigl[\gamma(\chi_{\V{1}+}^{\V{4}})\gamma(\chi_{\V{2}+}^{\V{3}})\bigr]^{-1}\)
    in the Mellin transform into
    \(\bigl[\gamma\bigl((\chi_{\V{1}-}^{\V{3}})^{-1}\bigr)\gamma\bigl((\chi_{\V{2}-}^{\V{4}})^{-1}\bigr)\bigr]^{-1}\).
    Therefore, we have
    \begin{align}
        \label{eqn:flipping_cocycle}
        \frac{I(w^{-1}\chi)}{I(\chi)}
        &=\frac{\gamma(\chi_{\V{1}+}^{\V{4}})\gamma(\chi_{\V{2}+}^{\V{3}})}
        {\gamma\bigl((\chi_{\V{1}-}^{\V{3}})^{-1}\bigr)\gamma\bigl((\chi_{\V{2}-}^{\V{4}})^{-1}\bigr)}\\
        &=\chi_{\V{1}}^{\V{3}}(-1)\chi_{\V{2}}^{\V{4}}(-1)\gamma(\chi_{\V{1}+}^{\V{4}})\gamma(\chi_{\V{2}+}^{\V{3}})\gamma(\chi_{\V{1}+}^{\V{3}})\gamma(\chi_{\V{2}+}^{\V{4}})\\
        &=\bigl(\chi_{\V{13}}\chi_{\V{14}}\chi_{\V{23}}\chi_{\V{24}}\bigr)(-1)\gamma(\chi_{\V{1}+}^{\V{4}})\gamma(\chi_{\V{2}+}^{\V{3}})\gamma(\chi_{\V{1}+}^{\V{3}})\gamma(\chi_{\V{2}+}^{\V{4}}).
    \end{align}
    This shows that the sign for flipping the edge \(\V{12}\) is
    \(\bigl(\chi_{\V{13}}\chi_{\V{14}}\chi_{\V{23}}\chi_{\V{24}}\bigr)(-1)\).

    Let \(s_{ij}\) be the element flipping edge \(ij\). We show here that
    \begin{align}
        \label{eqn:sign_of_triple_flips}
        \iota(s_{ij}s_{ik}s_{il},\chi)=\iota(s_{ij}s_{jk}s_{ki},\chi)
        =\chi_{ij}\chi_{ik}\chi_{il}(-1)=\chi_{jkl}(-1),
    \end{align}
    where \(\Set{i,j,k,l}=\tetraV\).
    Then \eqref{eqn:flipping_cocycle} implies that
    \begin{align}
        \frac{I(s_{ik}^{-1}s_{ij}^{-1}\chi)}{I(s_{ij}^{-1}\chi)}
        &=\bigl(\chi_{ij}^{-1}\chi_{il}\chi_{kj}\chi_{kl}\bigr)(-1)
        \gamma\Bigl(\OneHalf,\frac{\chi_{ik}\chi_{il}}{\chi_{ij}^{-1}}\Bigr)
        \gamma\Bigl(\OneHalf,\frac{\chi_{ik}\chi_{ij}^{-1}}{\chi_{il}}\Bigr)
        \gamma\Bigl(\OneHalf,\frac{\chi_{ki}\chi_{kl}}{\chi_{kj}}\Bigr)
        \gamma\Bigl(\OneHalf,\frac{\chi_{ki}\chi_{kj}}{\chi_{kl}}\Bigr)\\
        &=\bigl(\chi_{ij}\chi_{il}\chi_{kj}\chi_{kl}\bigr)(-1)
        \gamma(\chi_{jkl+})
        \gamma(\chi_{\bar{j}k\bar{l}+})
        \gamma(\chi_{i\bar{j}l+})
        \gamma(\chi_{ij\bar{l}+})
    \end{align}
    And one can compute other quotients
    \(I(s_{il}^{-1}s_{ik}^{-1}s_{ij}^{-1}\chi)/I(s_{ik}^{-1}s_{ij}^{-1}\chi)\)
    and so on. On the other hand, the weights in \(\Sigma^+\cap
    s_{ij}s_{ik}s_{il}(\Sigma^-)\) are precisely
    \begin{align}
        \chi_{i\bar{j}k},\chi_{ij\bar{k}},\quad
        \chi_{i\bar{k}l},\chi_{ik\bar{l}},\quad
        \chi_{i\bar{j}l},\chi_{ij\bar{l}},\quad
        \chi_{jkl}, \chi_{\bar{j}kl}, \chi_{j\bar{k}l}, \chi_{jk\bar{l}}.
    \end{align}
    The claim \eqref{eqn:sign_of_triple_flips} is then an easy (albeit a bit
    tedious) exercise.

    The longest element \(w_0\), which flips
    all edges simultaneously, can be written as the product of
    \(s_{ij}s_{ik}s_{il}\) and \(s_{jk}s_{kl}s_{lj}\). Repeating the same
    argument for these two elements, we see that \(\iota(w_0,\chi)=1\).

    Lastly, the inv-Regge symmetry in \Cref{prop:inv_Regge_symmetry} is the
    element \(r_{\EN{2}\EN{5}}w_0=w_0r_{\EN{2}\EN{5}}\). Using the fact that
    \(\iota(w_0,\chi)=1\) and \(\Sigma^+\cap w_0(\Sigma^-)=\Sigma^+\), we then can show that
    \(\iota(r_{\EN{2}\EN{5}},\chi)=1\). Conjugating using \(\FRTV\), or
    more precisely by cyclically permuting \(\Set{\V1,\V2,\V3}\), the
    same also holds for \(r_{\EN{1}\EN{4}}\) and \(r_{\EN{3}\EN{6}}\). This
    finishes the proof.
\end{proof}

\section{Computations for unramified principal series}\label{sec:computations_for_unramified_principal_series}

In this section, we prove \Cref{thm:main_duality_theorem} by direct
computations, partially assisted by a computer. 
Using \Cref{prop:VertexIntegral}, the statement to be proved is as follows:
\begin{align}
    \label{eqn:vertex_final_equality_in_proof}
    (1-q^{-2})^5\VI(\RO/K,\varphi)L(1,\mathrm{ad})=\Tr(q^{-\OneHalf}\Frob,\bbC[P]).
\end{align}

Our proof breaks into several steps:

\begin{enumerate}
    \item We first compute $\VI(\RO/K,\varphi)$ using the Bruhat--Tits tree of
        \(\RO\cong\PGL_2\) as a sum of terms, each of which is a
        product of several geometric series; see \Cref{VIXcomp}.
    \item On the dual side, the representation \(\bbC[P]\) of \(\Spin_{12}\)
        decomposes in a very nice way by using the Cartan map. Then the trace of
        \(q^{-\OneHalf}\Frob\) can be computed using Weyl character formula, see
        \Cref{WeylFormula}.
\end{enumerate}

At this point, we should be able to prove     \eqref{eqn:vertex_final_equality_in_proof} by
direct comparison, since both sides are algebraic expressions. However, the
length of the expressions seems too large for our computer to handle
efficiently, and so we opt for a more indirect approach using analysis. Regard
both sides of  \eqref{eqn:vertex_final_equality_in_proof} as functions of on the
six-dimensional torus $\TypeD_6 \otimes \bbT$ of possible $\sigma$. We shall
then verify that:

\begin{enumerate}
    \setcounter{enumi}{2}
    \item  Both  sides of  \eqref{eqn:vertex_final_equality_in_proof} have at most simple poles,
        and these simple poles can occur only along
        the locus where some eigenvalue of $\sigma$ acting on $\Sigma$
        equals $q^{\OneHalf}$. See   \Cref{prop:poles_of_VI} and  \Cref{WeylFormula}.

    \item These poles have  the same residues.  See \Cref{sub:comparison_using_residues} for 
        the (computer-assisted) computation. 
\end{enumerate}  

Therefore, 
the difference between the two sides defines a regular
function on $\TypeD_6 \otimes \bbT \simeq \Gm^6$. 
To conclude it is constant, we use the following observations:

\begin{enumerate}
    \setcounter{enumi}{4}
    \item  If we restrict either the left-hand side or the right-hand side
        to a generic one-parameter subgroup of the torus  $\TypeD_6 \otimes \bbT$
        they remain bounded at infinity.
        (See \Cref{BoundedLHS} and   \Cref{BoundedRHS}). 
\end{enumerate}
Here ``generic'' means that the coordinates of the one-parameter subgroup is permitted
to avoid a finite set of hyperplanes. One readily verifies that a regular function on $\Gm^n$
that remains bounded along a generic one-parameter subgroup is constant. 
So the difference between the left-hand and right-hand sides
must be constant. 
We can conclude by showing that the desired equality holds at a single value of $\sigma$.

Let us set up notation. In the statement of \Cref{thm:main_duality_theorem},
we have
$\chi\in\PrincipalSeries_0$
with image $\Frob\in \TypeD_6 \otimes \bbT$; we denote by 
\(x_{ij}=\chi_{ij}(\varpi)\)  the value of \(\chi_{ij}\) at the uniformizer,
so that the various $x_{ij}$ for $i<j$ provide coordinates on $ \OddFunctions \otimes \bbT$.
Let $\Frob \in \TypeD_6 \otimes \bbT$ be as in the statement of the theorem,
whereas the image of $\sigma$ under the six standard coordinate functionals are  
\begin{align}
    x_{\V{12}}x_{\V{34}},x_{\V{12}}x_{\V{34}}^{-1},\quad
    x_{\V{13}}x_{\V{24}},x_{\V{13}}x_{\V{24}}^{-1},\quad
    x_{\V{14}}x_{\V{23}},x_{\V{14}}x_{\V{23}}^{-1},
\end{align}
as well as their inverses. 

\subsection{The computation based on Bruhat--Tits tree}
\label{sub:the_computation_based_on_bruhat_tits_tree}

Recall that the quotient set \(\RO/K\) can be naturally identified with the
set of vertices of a tree, namely the Bruhat--Tits tree of \(\RO\). The
\(K\)-orbits of those vertices, in other words, the double quotient
\(K\backslash \RO/K\), can be identified with \(\bbN\):
the identity coset $K$ is sent to \(0\), and
in general we send a double coset to its distance from the orbit $K$.

With the notations in \Cref{sub:spherical_vertex_formula}, we note that the
spherical vector \(v_{ij}\) can be represented by the 
eigenfunction of the adjacency matrix of the tree, with
eigenvalue \(q^{-\OneHalf}(x_{ij}+x_{ij}^{-1})\), where
\(x_{ij}=\chi_{ij}(\varpi)\) is the value of \(\chi_{ij}\) at the uniformizer.
Recall that for any \(x\in\bbC\), the
corresponding Hecke eigenspace is generated by a function on \(\bbN\cong
K\backslash \RO/K\):
\begin{align}
    f_x(n)=\frac{q x-x^{-1}}{(1+q)(x-x^{-1})}(q^{-\OneHalf}x)^n+\frac{x-q
    x^{-1}}{(1+q)(x-x^{-1})}(q^{-\OneHalf}x^{-1})^n,
\end{align}
where we, for now, naively normalize \(f_x\) so that \(f_x(0)=1\). Let
\(f_{ij}=f_{x_{ij}}\).

The space \(X^{\tetraV}\) is simply the moduli space of \(4\) labeled vertices on the
Bruhat--Tits tree, and so \(\RO\backslash X^{\tetraV}\) may be identified
with the subspace where the vertex \(\V1\) sits at the root (that is, the
trivial coset \(K\)). The function \(\varphi_{ij}\) is simply the value of
\(f_{ij}\) evaluated at the distance between vertices \(i\) and \(j\). Note
that the opposite orientation of the edge \(ij\) (in other words, using
\(f_{ji}\) instead of \(f_{ij}\)) does not change \(\varphi_{ij}\),
because the formula above is invariant under $x \leftrightarrow x^{-1}$. For
definiteness we shall
choose the orientations so that \(i<j\), which is consistent
with other parts of this paper.

\subsubsection{}
The generic
configuration patterns of \(4\) vertices on a tree (with \(\V1\) at the root)
are, as follows, put into three (not disjoint) classes \(\bA\), \(\bB\), \(\bC\):
\newsavebox{\tetraII}
\newsavebox{\tetraIII}
\newsavebox{\tetraIV}
\sbox{\tetraII}
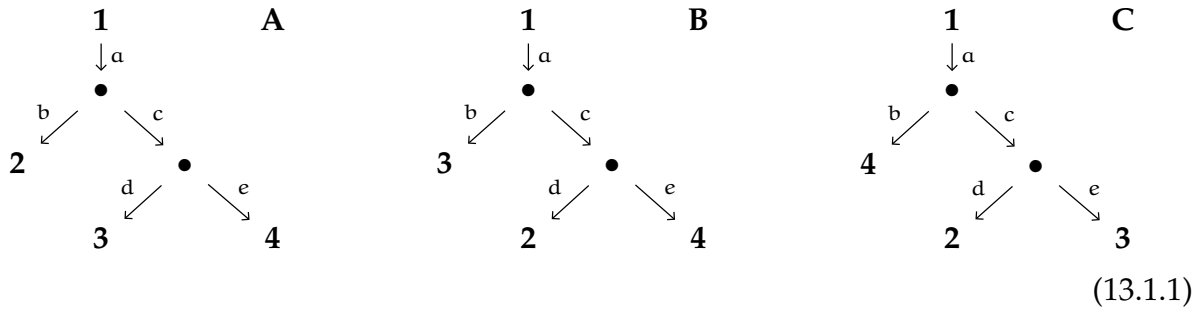
{
    \begin{tikzcd}[ampersand replacement = \&, column sep=small, row sep=small]
        \& \V1 \ar[d, "a"] \& \& \bA\\
        \& \bullet  \ar[ld, "b"']\ar[rd, "c"] \& \&\\
        \V2 \&   \& \bullet \ar[ld, "d"']\ar[rd, "e"] \&\\
         \& \V3 \&  \& \V4
    \end{tikzcd}
}
\sbox{\tetraIII}{
    \begin{tikzcd}[ampersand replacement = \&, column sep=small, row sep=small]
        \& \V1 \ar[d, "a"] \& \& \bB\\
        \& \bullet  \ar[ld, "b"']\ar[rd, "c"] \& \&\\
        \V3 \&   \& \bullet \ar[ld, "d"']\ar[rd, "e"] \&\\
         \& \V2 \&  \& \V4
    \end{tikzcd}
}
\sbox{\tetraIV}{
    \begin{tikzcd}[ampersand replacement = \&, column sep=small, row sep=small]
        \& \V1 \ar[d, "a"] \& \& \bC\\
        \& \bullet  \ar[ld, "b"']\ar[rd, "c"] \& \&\\
        \V4 \&   \& \bullet \ar[ld, "d"']\ar[rd, "e"] \&\\
         \& \V2 \&  \& \V3
    \end{tikzcd}
}
\begin{equation}
    \label{eqn:three_tree_patterns}
    \begin{tikzcd}
        \usebox{\tetraII}& \usebox{\tetraIII}& \usebox{\tetraIV}
    \end{tikzcd}
\end{equation}
where any ``line'' above signifies a connecting path rather
than just an edge of the tree, and the letters \(a\) through \(e\) are the
lengths of the paths. Any of these numbers can be \(0\); in particular when
\(c=0\), the three classes coincide, and in this case we will use extra care.

Given any pattern as in \eqref{eqn:three_tree_patterns}, say in class \(\bA\),
with fixed \(a,\ldots,e\), then its contribution towards the
vertex integral \(\VI(X,\varphi)\) is given by
\begin{align}
    f_{\V{12}}(a+b)f_{\V{13}}(a+c+d)f_{\V{14}}(a+c+e)f_{\V{23}}(b+c+d)
    f_{\V{24}}(b+c+e)f_{\V{34}}(d+e)
\end{align}
multiplied by the volume \(\boldsymbol{\mu}\) of the subset of \(\RO\backslash
X^\tetraV\) that gives this pattern. Note that we put \(\V1\) at the root, which
is the same as identifying \(\RO\backslash X^\tetraV\) with the double quotient
\(K\backslash X^3\), where \(X^3\) is indexed by \(\V2, \V3, \V4\), and \(K\)
acts diagonally.

Since the volume of \(K\) is normalized to be \(1\), and since \(K\) acts
transitively on the set of all configurations in class \(\bA\) with the same
pattern (that is, same numbers \(a,\ldots,e\)),
the volume \(\boldsymbol{\mu}\) is just the
number of elements in this set of configurations.
It is then easy to see that \(\boldsymbol{\mu}\) is of the order
\(q^{a+b+c+d+e}\), but with an additional factor \(C_{abcde}\) depending on
whether any of \(a,\ldots,e\) is \(0\) or not.
For example, if none of the numbers are \(0\), then there are
\((1+q^{-1})q^{a+b}\) different choices to place \(\V2\), and then
\((1-q^{-1})q^{c+d}\) choices to place \(\V3\), and finally \((1-q^{-1})q^e\)
choices to place \(\V4\). In this case, \(C_{abcde}=(1+q^{-1})(1-q^{-1})^2\).
The values of \(C_{abcde}\) are listed in
\Cref{tab:volume_constant_6j_tree}.

The classes \(\bB\) and \(\bC\) are treated similarly, but with the complication
that when \(c=0\), they duplicate cases already considered for \(\bA\). To
eliminate such duplications, we let \(Z_c=1/3\) if \(c=0\) and \(1\)
otherwise, and multiply everything in all three classes by \(Z_c\).

\subsubsection{} \label{VIXcomp}
Thus, \(\VI(X,\varphi)\) is equal to the series
\begin{multline}
    \label{eqn:tree_summation}
    \sum_{a,b,c,d,e}\Bigl[f_{\V{12}}(a+b)f_{\V{13}}(a+c+d)f_{\V{14}}(a+c+e)f_{\V{23}}(b+c+d)f_{\V{24}}(b+c+e)f_{\V{34}}(d+e)\\
        +f_{\V{13}}(a+b)f_{\V{12}}(a+c+d)f_{\V{14}}(a+c+e)f_{\V{23}}(b+c+d)f_{\V{34}}(b+c+e)f_{\V{24}}(d+e)\\
        +f_{\V{14}}(a+b)f_{\V{13}}(a+c+d)f_{\V{12}}(a+c+e)f_{\V{34}}(b+c+d)f_{\V{24}}(b+c+e)f_{\V{23}}(d+e)
    \Bigr]
    q^{a+b+c+d+e}C_{abcde}Z_c,
\end{multline}
where \(a,\ldots,e\) range in \(\bbN\).
\begin{table}
    \begin{center}
        \begin{tabular}[c]{|c|c|c|}
            \hline
            \# & Zero conditions & \(C_{abcde}\)\\
            \hline
            1 & \(a=b=c=d=e=0\) & 1\\
            \hline
            2 & \(abcde\neq 0\) & \((1+q^{-1})(1-q^{-1})^2\)\\
            \hline
            3 & \(c=0\), \(abde\neq 0\) & \((1+q^{-1})(1-q^{-1})(1-2q^{-1})\)\\
            \hline
            4 & \(ab=0\), \(de\neq 0\), \(a+b+c\neq 0\) & \((1+q^{-1})(1-q^{-1})\)\\
            \hline
            5 & \(ab\neq 0\), \(de=0\), \(c+d+e\neq 0\) & \((1+q^{-1})(1-q^{-1})\)\\
            \hline
            6 & otherwise & \((1+q^{-1})\)\\
            \hline
        \end{tabular}
    \end{center}
    \caption{Values of \(C_{abcde}\)}\label{tab:volume_constant_6j_tree}
\end{table}

\subsubsection{Boundedness of the vertex integral} \label{BoundedLHS}

We now show that
 \begin{align}
     (1-q^{-2})^5\VI(\RO/K,\varphi)L(1,\mathrm{ad}) 
\end{align}
remains bounded when $\sigma$ varies through a {\em generic}
one-parameter torus.

Explicitly, let \(x_{ij}=t^{n_{ij}}\) for \(n_{ij}\in\bbZ\);
we will prove the boundedness as $t \rightarrow 0$ so long
as all the $n_{ij}$ are nonzero.
The factor $L(1, \mathrm{ad})$ is, up to constants,
a product of terms $(1-q^{-1} t^{\pm 2n_{ij}})^{-1}$.
Such factors approach $1$ as $t \rightarrow 0$ if the exponent
is positive, $1-q^{-1}$ if it is zero, and zero if it is negative. 
Therefore, it suffices to prove the boundedness for  $ \VI(\RO/K,\varphi)$.

The constants \(C_{abcde}\) and \(Z_{c}\) are independent of
\(x_{ij}\) and for \(abcde\neq 0\) they stay the same constants respectively. It
suffices to look at the summation
\begin{align}
    \sum_{a,b,c,d,e}&f_{\V{12}}(a+b)f_{\V{13}}(a+c+d)f_{\V{14}}(a+c+e)\\
                    &\x f_{\V{23}}(b+c+d)f_{\V{23}}(b+c+e)f_{\V{34}}(d+e)q^{a+b+c+d+e},
\end{align}
because the other two summands are similar. Each \(f_{ij}(n)\) is the sum of two
terms:
\begin{align}
    \frac{qx_{ij}-x_{ij}^{-1}}{(1+q)(x_{ij}-x_{ij}^{-1})}(q^{-\OneHalf}x_{ij})^n,\quad
    \frac{x_{ij}-q x_{ij}^{-1}}{(1+q)(x_{ij}-x_{ij}^{-1})}(q^{-\OneHalf}x_{ij}^{-1})^n,
\end{align}
where the coefficients in front of \((q^{-\OneHalf}x_{ij})^n\) and
\((q^{-\OneHalf}x_{ij}^{-1})^n\)
are uniformly bounded when \(x_{ij}\) goes to \(0\) or \(\infty\).
Expanding the products of \(f_{ij}\)s in the summation, we see that we only need
to bound the series
\begin{align}
    \sum_{a,b,c,d,e}
    &(q^{-\OneHalf}t^{n_{\V{12}}})^{a+b}
    (q^{-\OneHalf}t^{n_{\V{13}}})^{a+c+d}
    (q^{-\OneHalf}t^{n_{\V{14}}})^{a+c+e}\\
    &\x (q^{-\OneHalf}t^{n_{\V{23}}})^{b+c+d}
    (q^{-\OneHalf}t^{n_{\V{24}}})^{b+c+e}
    (q^{-\OneHalf}t^{n_{\V{34}}})^{d+e}
    q^{a+b+c+d+e}\\
    =
    \sum_{a,b,c,d,e}
    &\bigl(q^{-\frac{a}{2}}t^{a(n_{\V{12}}+n_{\V{13}}+n_{\V{14}})}\bigr)
    \bigl(q^{-\frac{b}{2}}t^{b(n_{\V{12}}+n_{\V{23}}+n_{\V{24}})}\bigr)
    \bigl(q^{-c}t^{c(n_{\V{13}}+n_{\V{14}}+n_{\V{23}}+n_{\V{24}})}\bigr)\\
    &\x \bigl(q^{-\frac{d}{2}}t^{d(n_{\V{13}}+n_{\V{23}}+n_{\V{34}})}\bigr)
    \bigl(q^{-\frac{e}{2}}t^{e(n_{\V{14}}+n_{\V{24}}+n_{\V{34}})}\bigr)
    \\
    ={}&
    \frac{1}{1-q^{-\OneHalf}t^{n_{\V{12}}+n_{\V{13}}+n_{\V{14}}}}\x
    \frac{1}{1-q^{-\OneHalf}t^{n_{\V{12}}+n_{\V{23}}+n_{\V{24}}}}\x
    \frac{1}{1-q^{-1}t^{n_{\V{13}}+n_{\V{14}}+n_{\V{23}}+n_{\V{24}}}}\\
                    &\x\frac{1}{1-q^{-\OneHalf}t^{n_{\V{13}}+n_{\V{23}}+n_{\V{34}}}}
    \x\frac{1}{1-q^{-\OneHalf}t^{n_{\V{14}}+n_{\V{24}}+n_{\V{34}}}}.
    \label{eqn:tree_based_6j_bound_estimate}
\end{align}
Clearly when \(t\to 0\) or \(t\to\infty\), the above series is a
product of \(1\), \(0\), \((1-q^{-\OneHalf})^{-1}\) or \((1-q^{-1})^{-1}\),
hence bounded.

\subsubsection{Poles of the vertex integral} \label{Poles}
\begin{proposition}
    \label[proposition]{prop:poles_of_VI}
    The expression \(\VI(\RO/K,\varphi)L(1,\mathrm{ad})\)
     viewed as a rational function in
    variables \(x_{ij}\) and \(q^{-\OneHalf}\), has poles at hypersurfaces
    \begin{align}
        1-q^{-\OneHalf}x_{ij}^\pm x_{ik}^\pm x_{il}^\pm,
    \end{align}
    where \(\tetraV=\Set{i,j,k,l}\), and nowhere else,
    that is to say, only poles at points where   $\sigma$ has an eigenvalue
    $q^{\OneHalf}$ in the half-spin representation\(\Veven=S\).
\end{proposition}
\begin{proof}
    Since the summation is a sum of products of
    geometric series, we know it can only have simple poles at the hypersurfaces
    defined by the denominators in \eqref{eqn:tree_based_6j_bound_estimate}, such as
    \(1-q^{-\OneHalf}x_{\V{12}}x_{\V{13}}x_{\V{14}}\), etc.
    We need to show that the residue vanishes along:
    \begin{itemize}
        \item zeroes of the factors involving \(4\) different \(x_{ij}\)s, such
            as \(1-q^{-1}x_{\V{13}}x_{\V{14}}x_{\V{23}}x_{\V{24}} =0 \),  as
            well as
        \item  zeroes of factors coming from $L(1, \mathrm{ad})$, i.e. \(1\pm
            q^{-\OneHalf}x_{ij}^{\pm} =0\).
    \end{itemize}
    These are checked by routine computer computation (see remark below for some discussion of why this is easier than just checking the original result directly). 
\end{proof}

\begin{remark}
    \label[remark]{rmk:residues_are_easier_for_computer}
    Although the computation cost is highly dependent on the algorithm and
    implementation, it is reasonable to expect that the residues are
    significantly easier to handle computationally, because the ``length'' of
    the symbolic expression, vaguely speaking, is at least expected to be
    between \(1/16\) to \(1/24\) of that of the full tetrahedral symbol: for
    example, there are \(2^4=16\) cases in \(1-q^{-1}x_{\V{13}}^\pm
    x_{\V{14}}^\pm x_{\V{23}}^\pm x_{\V{24}}^\pm \), and \(2^2\x 6=24\) cases in
    \(1\pm q^{-\OneHalf}x_{ij}^{\pm}\). On the other hand, the expression is
    mainly products of two-term polynomials, therefore the time (and memory)
    cost of the computation can potentially grow exponentially with respect to
    the length.
\end{remark}

\subsection{The dual side}\label{sub:the_dual_side}

On the dual side, we consider \(\Spin_{12}\), and one of its two
half-spin representations (the one relevant to us is the last fundamental
representation of highest weight \(\Wt_6\)).   We will now describe 
in more detail the 
  the Lagrangian \(P\) in the half-spin representation \(\Veven\)
defined by the pure spinors.

\subsubsection{}
Indeed, equip \(\bbC^{12}\) with the anti-diagonal bilinear form
\begin{align} \label{split2}
    Q(x,y)=\sum_{i=1}^{12}x_iy_{12-i+1},
\end{align}
so that it decomposes
into the direct sum of two maximal isotropic spaces
\begin{align}
    \bbC^{12}=V_6\oplus V_6^*,
\end{align}
where \(V_6\) is the first six coordinates.
The exterior algebra \(\wedge^\bullet V_6\) is a module of the Clifford
algebra \(\Clif(\bbC^{12},Q)\) (the quotient of the tensor algebra of
\(\bbC^{12}\) by the relation \(x\otimes y+y\otimes x=Q(x,y)\)). The action of
\(\Clif(\bbC^{12},Q)\) is as follows: vectors in \(V_6\) act by wedging, and
vectors in \(V_6^*\) act by contracting. 
It is not hard to see (by fixing bases in \(V_6\) and \(V_6^*\)) that this
induces an isomorphism of associative algebras
\begin{align}
    \Clif(\bbC^{12},Q)\simeq \End(\wedge^\bullet V_6).
\end{align}
Since up to isomorphism \(\wedge^\bullet V_6\) is the unique simple module of
\(\End(\wedge^\bullet V_6)\), we see that as an abstract
\(\Clif(\bbC^{12},Q)\)-module this construction is independent of 
choice of the splitting \eqref{split2}
up to isomorphism. Consequently,  the automorphisms of $(\bbC^{12},Q)$
act projectively upon it, and this actually lifts to a genuine
action of $\Spin_{12}$ --- this is 
the $64$-dimensional spin representation.
The spin representation decomposes into two  $32$-dimensional half-spin
representations $\Veven$ and $\Vodd$, 
that is, the subspace of even- and odd- degree elements;
in the labeling of Bourbaki  \cite[Plate IV]{Bo02} these are respectively
the representations of highest weight $\Wt_6$ and $\Wt_5$.

\subsubsection{}
\label{ssub:pure_spinor_cone}

The distinguished element
\begin{align}
    1\in\bbC=\wedge^0 V_6
\end{align}
which we will denote by \(\bFv_0\) for better visibility, 
is annihilated by the  subspace \(V_6^*\) under the
Clifford action, and one readily verifies that 
this characterizes it up to scalars.  From the uniqueness claim above, it follows that 
{\em any} Lagrangian (i.e., maximal $Q$-isotropic) subspace of $\bbC^{12}$
annihilates a one-dimensional subspace in  \(\wedge^\bullet V_6\)
under the Clifford action. 
A \notion{pure spinor} is any vector 
in \(\wedge^\bullet V_6\)  belonging to such a line;
thus $\bFv_0$ is a pure spinor. 

Clearly, pure spinors form a cone $\Spinor$.
Let $\Spinor^{\times}$ be the
open subset of nonzero pure spinors. Then the quotient of $\Spinor^{\times}$
by scaling, i.e. the associated projective subvariety of the projectivization
of \(\wedge^\bullet V_6\), is evidently identified
with the space of {\em all} Lagrangian subspaces of $\bbC^{12}$,
that is to say,  the Lagrangian Grassmannian $\LGr(\bbC^{12})$ of $\bbC^{12}$.
This Lagrangian Grassmanian splits into  two orbits under $\SO_{12}$, i.e.
\begin{align}
    \LGr(\bbC^{12}) = \LGr^+ \coprod \LGr^{-}.
\end{align}
In fact, two isotropic subspaces $A, B$ belong to the same orbit
if and only if the dimension of $A \cap B$ is even.

One gets, therefore, a corresponding decomposition
\begin{align}
    \Spinor = P_+ \cup P_-
\end{align}
of the cone of spinors into sub-cones that intersect precisely at the origin.
In fact, one readily verifies from the description above that
$P_+$ is precisely the subcone of pure spinors in $\Veven \oplus 0\simeq \Veven$ 
and \(P_-\) is the cone of pure spinors 
in \(0\oplus \Vodd \simeq \Vodd \). Indeed, using the $\Spin_{12}$-action,
it is enough to verify this for a single point of $P_+$ and a single point of $P_-$,
which one does by explicit computation.

Since we mostly care about the cone \(P_+\), we henceforth denote it simply by
\(P\).

\begin{lemma} \label[lemma]{SpinorsHighestWeight}
    The vector space  $\bbC[P]_n$ of degree $n$ homogeneous functions on $P$
    is identified with the irreducible representation of $\Spin_{12}$
    of highest weight $n \Wt_6$.
\end{lemma}
\begin{proof}
    Let $P^{\times}$ be the nonzero elements of $P$. 
    As we have seen above, $P^{\times}$
    is the total space of a line bundle \(L\) over the flag variety
    $\LGr^+$ minus the zero section.\footnote{The square has a nice description: \(L^{\otimes(-2)}\) is
    the pullback of the determinant bundle over \(\LGr^+\).}
    By the Borel--Weil theorem,  sections of line bundles on flag varieties
    are highest weight representations, with highest weight determined
    by the isotropy representation. This shows that $\bbC[P]_n$
    is a subrepresentation of the highest weight representation of weight $n \nu$
    for {\em some} $\nu$, hence must be the said representation itself by
    irreducibility. To compute $\nu$ it is easiest to note
    that    $\C[P]_1$ is by definition a quotient of the irreducible
    representation $\Veven$, and so is in fact $\Veven$; thus $\nu$ is the
    highest weight of $\Veven$.
\end{proof}

\subsubsection{Poles of the trace on the spinor cone} \label{WeylFormula}

\Cref{SpinorsHighestWeight} permits us to compute the character of $\bbC[P]$ by means of the Weyl character formula:
 \begin{align}
    \Tr(q^{-\OneHalf}\Frob,\bbC[P])
    &=\frac{\Frob^\rho}{\prod_{\Rt>0}(\Frob^\Rt-1)}\sum_{w\in
        W}\frac{(-1)^{\ell(w)}\Frob^{w(\rho)}}{1-q^{-\OneHalf}\Frob^{w(\Wt_6)}}.
\end{align}

Note that this function is \(W\)-invariant, and so none of the term \(\Frob^\Rt-1\)
contributes to a pole (because you can conjugate any given \(\Rt\)
away). Therefore, the poles are only given by \(1-q^{-\OneHalf}\Frob^{w(\Wt_6)}\).
We see that the hypersurfaces supporting those poles coincide with the
ones on the automorphic side. Moreover, just as in 
      \Cref{rmk:residues_are_easier_for_computer}, it is easy to see
    that the residues are significantly simpler expressions: there are \(32\)
    weights in \(S\), so the symbolic length of a residue is only \(1/32\) of
    the full trace.

\subsubsection{} \label{BoundedRHS}
We compute the behavior of the trace when \(\Frob=t^\mu\) for some coweight
\(\mu\). It suffices to assume that \(\mu\) is dominant by
\(W\)-invariance,  and since  it is enough to consider
\emph{generic} directions, we assume that $\mu$ is \emph{strictly} dominant.
Looking at each summand in the trace, and we want to show that for any \(w\in
W\), the limit of
\begin{align}
    \label{eqn:term_to_estimate_in_Galois_6j}
    \left[\frac{1}{\prod_{\Rt>0}(t^{\Pair{\Rt}{\mu}}-1)}\right]
    \left[\frac{t^{\Pair{\rho+w(\rho)}{\mu}}}{1-q^{-\OneHalf}t^{\Pair{w(\Wt_6)}{\mu}}}\right]
    =
    \left[\frac{1}{\prod_{\Rt>0}(1-t^{-\Pair{\Rt}{\mu}})}\right]
    \left[\frac{t^{\Pair{-\rho+w(\rho)}{\mu}}}{1-q^{-\OneHalf}t^{\Pair{w(\Wt_6)}{\mu}}}\right]
\end{align}
when \(t\to 0\) or \(\infty\) is bounded.

We first consider the case \(t\to 0\).
In this case, we use the left-hand side of \eqref{eqn:term_to_estimate_in_Galois_6j}.
Since \(\mu\) is strictly dominant, then we have
\((t^{\Pair{\Rt}{\mu}}-1)\to -1\) for any positive root \(\Rt\), and
\begin{align}
    t^{\Pair{\rho+w(\rho)}{\mu}}\longto 0\text{ or }1,
\end{align}
because \(\rho+w(\rho)\) is a sum of positive roots. Lastly, the denominator
\(1-q^{-\OneHalf}t^{\Pair{w(\Wt_6)}{\mu}}\) goes to either \(1\) or
$1-q^{-\OneHalf}$ or  \(\infty\), and
so \(\Tr(q^{-\OneHalf}\Frob,\bbC[P])\) stays uniformly bounded when \(t\to 0\). The
case \(t\to\infty\) is similarly proved by using the right-hand side of
\eqref{eqn:term_to_estimate_in_Galois_6j}.

\subsection{Comparison using residues}\label{sub:comparison_using_residues}
We have now located the poles of both sides of
\eqref{eqn:vertex_final_equality_in_proof}. They are, as we have seen,
located on the locus where an eigenvalue of $\sigma$ on $S$ coincides with
$q^{\OneHalf}$.
The corresponding residues, on either side, can be computed using 
a computer as well --- this computation is substantially smaller
than computing the full expressions --- and it then
turns out these residues can be explicitly factorized. 

For example, we record the residue of
both sides of \eqref{eqn:vertex_final_equality_in_proof}
when \(1-q^{-\OneHalf}\Frob^{-\Wt_6}=1-q^{-\OneHalf}(x_{\V{12}}x_{\V{13}}x_{\V{14}})^{-1}=0\):
\begin{multline}
    \label{eqn:residues_example}
    -L(1,\mathrm{ad})x_{\V{23}}^4 x_{\V{24}}^4 x_{\V{34}}^4\Bigl[(x_{\V{12}}^2-1)  (x_{\V{13}}^2-1)
         (x_{\V{14}}^2-1)  (x_{\V{12}} x_{\V{13}} x_{\V{24}}-x_{\V{34}})\\
        (x_{\V{12}} x_{\V{13}} x_{\V{34}}-x_{\V{24}}) (x_{\V{12}} x_{\V{13}}-x_{\V{24}} x_{\V{34}})
        (x_{\V{12}} x_{\V{13}} x_{\V{24}} x_{\V{34}}-1) (x_{\V{12}} x_{\V{14}} x_{\V{23}}-x_{\V{34}})\\
        (x_{\V{12}} x_{\V{14}} x_{\V{34}}-x_{\V{23}}) (x_{\V{12}} x_{\V{14}}-x_{\V{23}} x_{\V{34}})
        (x_{\V{12}} x_{\V{14}} x_{\V{23}} x_{\V{34}}-1) (x_{\V{13}} x_{\V{14}} x_{\V{23}}-x_{\V{24}})\\
        (x_{\V{13}} x_{\V{14}} x_{\V{24}}-x_{\V{23}}) (x_{\V{13}} x_{\V{14}}-x_{\V{23}} x_{\V{24}})
    (x_{\V{13}} x_{\V{14}} x_{\V{23}} x_{\V{24}}-1)\Bigr]^{-1},
\end{multline}
in which we really meant to replace, for example, \(x_{\V{12}}\) by
\(q^{-\OneHalf}(x_{\V{13}}x_{\V{14}})^{-1}\) (so that \(q^{-\OneHalf}\) is not
treated as a variable but a constant); however, replacing \(q^{-\OneHalf}\)
by \(x_{\V{12}}x_{\V{13}}x_{\V{14}}\) makes the expression look more symmetric.

This completes the proof, according to the general plan outlined at the start of the section.

\section{Hypergeometric evaluations of the tetrahedral symbol} \label{HypergeometricProofs}

Our goal here is to express the tetrahedral symbol in terms of generalized
hypergeometric functions, proving
both \Cref{PROPAB} and the formulas given in  
\Cref{sub:expression_in_4F3_statement}.

\subsection{Classical hypergeometric functions and Mellin--Barnes integrals}
For positive integers \(k<l\) and parameters
\(\ul{a}=(a_1,\ldots,a_l)\in\bbC^l\) and \(\ul{b}=(b_1,\ldots,b_k)\in\bbC^k\),
the generalized hypergeometric function of parameters \(\ul{a}\) and \(\ul{b}\)
is the analytic continuation of the series
\begin{align}
    \pFq{l}{k}(\ul{a},\ul{b};x)
    =\pFq{l}{k}\!\left(\!\begin{array}{c}
        a_1,\ldots,a_l\\
        b_1,\ldots,b_k
    \end{array}\!\Big|\; x\,\right)
    \defeq \sum_{n=0}^\infty\frac{\prod_{j=1}^l(a_j)_n}
    {\prod_{j=1}^k(b_j)_n}\frac{x^n}{n!},
\end{align}
where \((a)_n\) is the \emph{rising} factorial:
\begin{align}
    (a)_0\defeq 1,\quad (a)_n\defeq a(a+1)\cdots (a+n-1).
\end{align}

Classical \(6j\) symbols (i.e., compact \(\RO\) and \(F=\bbR\)) enjoy various
interpretations at the value at the singular point \(x=1\)
of \(\pFq{4}{3}(\ul{a},\ul{b};x)\) for certain \(\ul{a}\) and
\(\ul{b}\); our goal is to  describe a similar result for the tetrahedral
symbols for $F=\bbR$.

Using Mellin transform and its inverse, one can rewrite \(\pFq{l}{k}\) into an
\emph{Mellin--Barnes} type integral (with some assumptions on the parameters):
\begin{align}
    \frac{\Gamma(a_1)\cdots\Gamma(a_l)}
    {\Gamma(b_1)\cdots\Gamma(b_k)}\pFq{l}{k}(\ul{a},\ul{b};x)=
    \frac{1}{2\pi
    i}\int_{c-i\infty}^{c+i\infty}\frac{\Gamma(a_1+s)\cdots\Gamma(a_l+s)}
    {\Gamma(b_1+s)\cdots\Gamma(b_k+s)}\Gamma(-s)(-x)^s\dd s,
\end{align}
where the vertical line \((c-i\infty,c+i\infty)\) separates the poles of all
\(\Gamma(a_j+s)\) from those of \(\Gamma(-s)\).\footnote{We will not use this general
fact here; in the case relevant to us, namely \(\pFq{4}{3}\), details are
contained in \Cref{sub:relation_with_classical_4F3}.}  
One can readily generalize the notion of Mellin--Barnes integrals to any local field because
\(\Gamma\)-functions are essentially the local \(L\)-factors over \(\bbR\)
(cf. \Cref{GammaReview}). Our procedure will, in fact, 
be to first derive a Mellin--Barnes  representation (in this generalized sense) 
of the tetrahedral symbol in the principal series case,
and then derive the various desired consequences from it.

\subsection{Review on Mellin transforms}
We review the properties of Mellin transforms over an arbitrary local field.

\begin{lemma}
    \label[lemma]{lem:basic_Mellin}
    Suppose \(a,b>0\) and \(a+b<1\), then we have equality
    \begin{align}
        \int_{F} \abs{x}^{a-1} \abs{1-x}^{b-1} \dd x= \frac{\gamma(a+b)}{\gamma(a)\gamma(b)},
    \end{align}
    where the left-hand side is absolutely convergent.
\end{lemma}
\begin{proof}
    Since \(a+b-2< -1\) (resp.~\(a-1> -1\), resp.~\(b-1> -1\)), the integral is
    absolutely convergent near \(\infty\) (resp.~\(0\), \(1\)). Thus the whole
    integral is absolutely convergent. For the equality, we note that
    \begin{align}
        \int_{F} \abs{x}^{a-1} \abs{1-x}^{b-1} \dd x\cdot\gamma(a+b)^{-1}
        &\stackrel{\eqref{Tate4gamma}}{=} \int_{F} \abs{x}^{a-1} \abs{1-x}^{b-1} \dd x \int_F\abs{y}^{a+b-1}\Psi(y)\dd y\\
        &= \int_{F} \abs*{\frac{x}{y}}^{a-1} \abs*{1-\frac{x}{y}}^{b-1}
        \dd\Bigl(\frac{x}{y}\Bigr) \int_F\abs{y}^{a+b-1}\Psi(y)\dd y\\
        &= \int_{F} \abs{x}^{a-1} \abs{y-x}^{b-1} \dd x \int_F\Psi(y)\dd y,
    \end{align}
    which is the Fourier transform of the convolution of \(\abs{x}^{a-1}\) and
    \(\abs{x}^{b-1}\). Here we used the fact that \(0<a+b<1\) implies that
    the integral form of \(\gamma(a+b)^{-1}\) is also absolutely convergent,
    hence we are free to use Fubini theorem to manipulate the integrals above.
    Applying \eqref{Tate4gamma} again we arrive at
    \begin{align}
        \int_{F} \abs{x}^{a-1} \abs{1-x}^{b-1} \dd x\cdot
        \gamma(a+b)^{-1}=\gamma(a)^{-1}\gamma(b)^{-1},
    \end{align}
    and this finishes the proof.
\end{proof}

\begin{definition}
    \label[definition]{def:Mellin_transform}
    For a function \(f\) on \(F^\x\), we define its \notion{Mellin transform} to
    be the  integral
    \begin{align}
        \cM_f(\mu)=\int_{F^\x}f(x)\mu(x)\dd^\x x,
    \end{align}
    assuming the integral is absolutely convergent for all (unitary) characters
    \(\mu\) of \(F^\x\). 
    We will allow ourselves to speak of the integral for quasi-characters
    by means of analytic continuation, when applicable.
    When \(F=\bbR\), we denote \(\cM_f^+(s)\defeq
    \cM_f(\abs{\blank}^s)\) and \(\cM_f^-(s)\defeq \cM_f(\sgn\abs{\blank}^s)\).
    In the reverse direction, given a function $\hat{f}(\chi)$ on the character
    group $\widehat{F^{\times}}$, we define its inverse Mellin transform as
    \begin{align}
        \cM^{-1}_{\hat{f}}(x) = \int_{\widehat{F^{\times}}} \hat{f}(\chi) \chi(x) \dd\chi,
    \end{align}
    where the measure $\dd\chi$ is dual to $\dd^{\x}x$, which amounts to asking
    that this is indeed inverse to the Mellin transform.
\end{definition}
 
\begin{lemma} \label[lemma]{MellinL1L1}
    Suppose that $f$ is an $L^1$-function on $F^{\times}$ with the property that
    $\cM_f(\mu)$ defines an $L^1$-function on the character group of
    $F^{\times}$. Then the inverse Mellin transform of $\cM_f$ defines a
    continuous function that agrees with $f$ almost everywhere.
\end{lemma}
\begin{proof}
    This is a standard fact of harmonic analysis. In the \(F=\bbR\) case, using
    the fact that \(\bbR^\x=\bbR_{>0}\x \Set{\pm 1}\) and change of variables,
    the lemma is a standard precise form of Fourier inversion; see for example
    \cite[Corollary~1.21]{StWe71}. The \(F=\bbC\)
    case can be similarly derived from the fact that \(\bbC^\x=\bbR_{>0}\x
    \bbT\) (\(\bbT\) denoting the unit circle).
    The nonarchimedean case is relatively easy to deduce because
    continuous functions are locally constant. It can be proved by first
    verifying that the inverse Mellin transform of $\cM_f$ defines a continuous
    function, call it $f_1$, and then verifying that the pairings of $f$ or
    \(f_1\) with the characteristic function of any compact open subset of
    \(F^\x\) coincide. We leave the details to the reader.
\end{proof}

For a nonempty open interval $J \subset \mathbb{R}$, and let $L^1_{J}$ be the
functions on $F^{\x}$ with the property that $\int_{F^{\x}} \abs{f(x)}\cdot
\abs{x}^{\sigma} \dd^{\x}{x} < \infty$ whenever $\sigma \in J$. 
In practice, we
will be interested in the interval $J=(0,\OneHalf)$.
Many statements can be reduced to the case when $J$ contains zero, simply by
multiplying $f$ by a suitable power of $\abs{x}$.

\begin{lemma}
    For $f \in L^1_J$, the Mellin transform $\cM_f(\mu_s)$ is absolutely
    convergent for a character $\mu$ and any $s$ whose real part belongs to
    $J$.
\end{lemma}

 When we write $\cM_f$ for
$f\in L^1_J$, we will always regard it as a function on the set of
quasi-characters specified by this Lemma. Next, $L^1_J$ behaves well with
respect to convolution:

\begin{lemma}
    \label[lemma]{lem:basic_properties_of_Mellin}
    If $f_1, f_2 \in L^1_{J}$, the \emph{multiplicative} convolution
    \begin{align}
        f(y)\defeq\int_{F^\x} f_1(x)f_2(yx^{-1})\dd^\x x
    \end{align}
    is absolutely convergent for  almost all $y$, and also defines an element of $L^1_{J}$.
     Moreover, the Mellin transforms are multiplicative: we have
     \begin{align}
         \cM_f = \cM_{f_1} \cM_{f_2},
     \end{align}
    where, as noted above, we regard both sides as functions on quasi-characters
    $\mu_s$
    where the real part of $s$ lies in $J$.
\end{lemma}

We omit the straightforward proof.

\begin{lemma}\label[lemma]{alphabetaMellin}
    For two characters \(\alpha\) and \(\beta\) of \(F^\x\), 
    the product $\alpha(x) \beta_{-}(1-x)$ belongs to the space
    $L^1_{(0,\OneHalf)}$ defined in \Cref{lem:basic_properties_of_Mellin}.
    Its Mellin transform is given by 
    \begin{align}
        \cM_{\alpha(x)\beta_-(1-x)}(\mu_s)
        =\frac{\gamma(\alpha\beta_+\mu_s)}{\gamma(\alpha\mu_s)\gamma(\beta_+)},
    \end{align}
    valid for $\Re(s) \in (0,\OneHalf)$.
\end{lemma}
\begin{proof}
    The proof of \Cref{lem:basic_Mellin} still works: the real part of
    the exponent of quasi-character \(\alpha\mu_s\) is between \(0\) and
    \(\OneHalf\), while that of \(\beta_+\) is \(\OneHalf\), and so the sum of
    two exponents has real part between \(0\) and \(1\); this way the relevant
    integrals all converge absolutely.
\end{proof}

\subsection{Hypergeometric functions over local fields}
\label{sub:Mellin_transform_of_hypergeometric_over_local_field}
The edge integral we have described in \eqref{eqn:edge_integral_hypergeometric},
up to scaling by powers of \(\nu_\bbP\),
has the following alternative form by letting \(w=1\) (see
\Cref{sec:edge_formula_for_principal_series} for how the measure is properly
dealt with):
\begin{align}
    I=\int_{F^3} \alpha_{1-}(x)  \alpha_{2-}(1-x) \alpha_{3-}(x-y) \alpha_{4-} (y-z)
    \alpha_{5-}(y) \alpha_{6-}(z)  \alpha_{7-}(z-1),
\end{align}
for certain characters \(\alpha_i: F^{\x} \rightarrow \bbC^\x\), and the measure is
the usual additive Haar measure \(\dd x\dd y\dd z\), which we omit for
simplicity. We will explain how to evaluate this integral in terms of a
generalized hypergeometric function.
Note that we have already proven that $I$ is absolutely convergent, and thereby,
by Fubini's theorem, it can be evaluated as an iterated integral, in any way we please.

Use the shorthand  \(\alpha_{12} = \alpha_1 \alpha_2\), etc., and rewrite the
integral as
\begin{align}
    I=\int_{F^3} \alpha_{13--}(x)  \alpha_{2-}(1-x) \alpha_{3-}(1-y/x)
    \alpha_{4-} (1-z/y)
    \alpha_{45--}(y) \alpha_{67--}(z)  \alpha_{7-}(1-1/z).
\end{align}

We are going to repeatedly apply \Cref{lem:basic_properties_of_Mellin}
with the interval $J$ taken to be $(0,\OneHalf)$.
First of all, we take $f_1(x) = \alpha_{13}(x) \alpha_{2-}(1-x)$
and $f_2(x)=\alpha_{3-}(1-x)$; they both belong to $L^1_J$
by \Cref{alphabetaMellin}, and therefore their multiplicative convolution
\begin{align}
    f_{123}(y)\defeq \int_F \alpha_{13--}(x)\alpha_{2-}(1-x)\alpha_{3-}(1-y/x)\dd x
    \left( = \int_{F} \alpha_{1-}(x) \alpha_{2-}(1-x) \alpha_{3-}(x-y)\dd x \right).
\end{align}
defines also a class in $L^1_J$. 
By  \Cref{lem:basic_properties_of_Mellin}  and  \Cref{alphabetaMellin} we have
\begin{align}
    \cM_{f_{123}}(\mu_s)
    =\frac{\gamma(\alpha_{13}\alpha_{2+}\mu_s)}{\gamma(\alpha_{13}\mu_s)\gamma(\alpha_{2+})}
    \cdot
    \frac{\gamma(\alpha_{3+}\mu_s)}{\gamma(\mu_s)\gamma(\alpha_{3+})}.
\end{align} 
where we assume that the real part of $s$ belongs to $(0,\OneHalf)$; we will continue
to impose this assumption below.

Multiplication by a unitary character of course preserves the property of
belonging to $L^1_J$. Consequently, we can apply
\Cref{lem:basic_properties_of_Mellin} to  analyze the multiplicative convolution
\(f_{12345}(z)\) of \(f_{123}(y) \cdot \alpha_{45}(y)\) and
\(\alpha_{4-}(1-y)\). Then $f_{12345}$ belongs to $L^1_J$ and its Mellin
transform is given by
\begin{align}
    \cM_{f_{12345}}(\mu_s)
    &=\cM_{f_{123}(z)\alpha_{45}(z)}(\mu_s)\cM_{\alpha_{4-}(1-z)}(\mu_s)\\
    &=
    \frac{\gamma(\alpha_{12345+}\mu_s)}{\gamma(\alpha_{1345}\mu_s)\gamma(\alpha_{2+})}
    \cdot
    \frac{\gamma(\alpha_{345+}\mu_s)}{\gamma(\alpha_{45}\mu_s)\gamma(\alpha_{3+})}
    \cdot
    \frac{\gamma(\alpha_{4+}\mu_s)}{\gamma(\mu_s)\gamma(\alpha_{4+})}.
\end{align}

Finally, multiplying \(f_{12345}(z)\) by \(\alpha_{67}(z)\) and then convolving
with \(\alpha_{7-}(1-z)\)  gives, for exactly the same reason as before, a
function \(f_{1234567}(w) \in L^1_J\) such that
\begin{align}
    \cM_{f_{1234567}}(\mu_s)
    &=\cM_{f_{123}(w)\alpha_{4567}(w)}(\mu_s)\cM_{\alpha_{67}(w)\alpha_{4-}(1-w)}(\mu_s)\cM_{\alpha_{7-}(1-w)}\\
    &=
    \frac{\gamma(\alpha_{1234567+}\mu_s)}{\gamma(\alpha_{134567}\mu_s)\gamma(\alpha_{2+})}
    \cdot
    \frac{\gamma(\alpha_{34567+}\mu_s)}{\gamma(\alpha_{4567}\mu_s)\gamma(\alpha_{3+})}
    \cdot
    \frac{\gamma(\alpha_{467+}\mu_s)}{\gamma(\alpha_{67}\mu_s)\gamma(\alpha_{4+})}
    \cdot
    \frac{\gamma(\alpha_{7+}\mu_s)}{\gamma(\mu_s)\gamma(\alpha_{7+})},
\end{align} 
which we rewrite as
\begin{equation}
    \label{eqn:hypergeometric_Mellin_final}
    \bigl[\gamma(\alpha_{2+})\gamma(\alpha_{3+})\gamma(\alpha_{4+})\gamma(\alpha_{7+})\bigr]^{-1}
    \frac{\gamma(\alpha_{1234567+}\mu_s)\gamma(\alpha_{34567+}\mu_s)\gamma(\alpha_{467+}\mu_s)\gamma(\alpha_{7+}\mu_s)}
    {\gamma(\alpha_{134567}\mu_s)\gamma(\alpha_{4567}\mu_s)\gamma(\alpha_{67}\mu_s)\gamma(\mu_s)}.
\end{equation} 

Note that \(f_{1234567}(w)\) is defined by replacing \(\alpha_{7-}(z-1)\) by
\(\alpha_{7-}(z-w)\) in the definition of \(I\). Therefore, the edge integral
\(I\) is the value at \(1\) of a function \(f_{1234567}\) whose Mellin transform
is given by \eqref{eqn:hypergeometric_Mellin_final}. 

Moreover, the Mellin inversion formula is applicable, in the following form:
\begin{align} \label{MellinFormulaI}
    I=f_{1234567}(1)
    =\int\cM_{f_{1234567}}(\mu_s)\dd \mu,
\end{align}
where \(0<s<\OneHalf\) is any fixed number, and the integral is taken over all
characters \(\mu\) (see~\Cref{Characters And Measures} for the measure on $\mu$).
To verify the applicability of the Mellin inversion formula, we verify that $\cM_{f_{1234567}}(\mu_s)$
is absolutely integrable as a function of $\mu$ and invoke
\Cref{MellinL1L1}. 
We will check this absolute integrability in the nonarchimedean case, leaving the archimedean cases to the reader;
a similar check for $F=\mathbb{R}$ is carried out after
\eqref{eqn:Mellin_plus_step_3}.

Suppose, then, that the cardinality of the residue field of $F$ equals $q$. Let
$\cO \subset F$ be the ring of integers. The decomposition $F^{\times} \simeq
\cO^{\times} \times \bbZ$ (after fixing a uniformizer) gives a corresponding
decomposition
\begin{align}
    \widehat{F^{\times}} \simeq \widehat{\cO^{\times}} \times \bbT,
\end{align}
where \(\bbT\) is the unit circle.
As usual, we say that a character $\mu$ of $\cO^{\times}$ has conductor $n$ if
it is trivial on $1+\varpi^n \cO$ but not on any larger subgroup of this form.
The number of such characters equals $q (1-2/q)$ for $n=1$ and $q^n(1-1/q)^2$
for $n > 1$; all that matters is that it is $O(q^n)$. On the other hand, the
formulas of \Cref{Gamma Evaluation} imply that the the absolute value of the
term
\begin{align}
    \frac{\gamma(\alpha_{1234567+}\mu_s)\gamma(\alpha_{34567+}\mu_s)\gamma(\alpha_{467+}\mu_s)\gamma(\alpha_{7+}\mu_s)}
    {\gamma(\alpha_{134567}\mu_s)\gamma(\alpha_{4567}\mu_s)\gamma(\alpha_{67}\mu_s)\gamma(\mu_s)}
\end{align}
from \Cref{eqn:hypergeometric_Mellin_final} equals $q^{-2n}$ when $\mu$ has
conductor $n$, where $n$ is chosen strictly larger than the conductor of any
$\alpha_i$ (this ensures that the conductor of any $\alpha_{\bullet}\mu$
involved in the fraction above equals the conductor of $\mu$). Therefore, the
integral of \Cref{eqn:hypergeometric_edge_integral} is absolutely bounded by a
constant multiple of $\sum_{n \geq 0} q^n \cdot q^{-2n}$ and is absolutely
convergent. This concludes our justification of \eqref{MellinFormulaI}.

\subsection{Relationship with classical \(\pFq{4}{3}\)}
\label{sub:relation_with_classical_4F3}
When \(F=\bbR\), we may use \eqref{eqn:hypergeometric_Mellin_final} to relate
the tetrahedral symbol for the principal series with generalized hypergeometric
functions.

For simplicity, we assume characters \(\alpha_{1234567+}\), etc. in the
numerator of \eqref{eqn:hypergeometric_Mellin_final} and \(\alpha_{134567}\),
etc. in the denominator are of the forms \(\abs{\blank}^{a_j}\) and
\(\abs{\blank}^{b_j}\) (\(j=1,\ldots, 4\)) respectively, where
\begin{align}
    a_j \in \OneHalf + i \mathbb{R},\quad b_j \in i \mathbb{R},
\end{align}
which is the case that is relevant
to evaluating the tetrahedral symbol for unramified characters with $F=\mathbb{R}$.
The general case can be analyzed similarly, where one adds various signs;
the answer itself will look different because the contours of integration  used
later in the argument need to be chosen differently.

The Mellin transform over \(\bbR^\x\) consists of two disjoint components
\(\cM^+\) and \(\cM^-\) (see \Cref{def:Mellin_transform}). For a function \(f\) on
\(\bbR^\x\), write \(f=f^++f^-\), where \(f^+\) is even and \(f^-\) is odd,
then we have
\begin{align}
    \cM^-_{f^+}=\cM^+_{f^-}=0.
\end{align}
Moreover, \(\cM^+_{f^+}\) (resp.~\(\cM^-_{f^-}\)) is twice the classical Mellin
transform of \(f^+|_{(0,\infty)}\) (resp.~\(f^-|_{(0,\infty)}\)).
As a result, the sum
\(\cM^+_f+\cM^-_f\) is twice the classical Mellin transform of the function
\(f|_{(0,\infty)}\).

Therefore, to recover the value
of \(f=f_{1234567}\) at \(1\), we can use the inverse Mellin transform on
the sum of
\begin{align}
    \cM_{f}^+(s)=
    \prod_{j=1}^4\frac{\gamma(a_j+s)}{\gamma(a_j-b_j)\gamma(b_j+s)}
\end{align}
and
\begin{align}
    \cM_{f}^-(s)=
    \prod_{j=1}^4\frac{\gamma^-(a_j+s)}{\gamma(a_j-b_j)\gamma^-(b_j+s)},
\end{align}
where \(\gamma^-(\mu)\) means \(\gamma(\mu\cdot\sgn)\) for any quasi-character
\(\mu\).

We will now use the relations from \eqref{gammaReval}:
\begin{gather} \label{Gamma+Again}
    \gamma^{\pm}(s)^{-1}=\Gamma(s)(\II^{-s}\pm\bar{\II}^{-s})
    = (2\pi)^{-s} \Gamma(s) (i^{-s}\pm i^{s}),\\
    \gamma^{\pm}(s) = \Gamma(1-s)(\II^{s-1} \pm \bar{\II}^{s-1})
    = (2\pi)^{s-1} \Gamma(1-s)(i^{s-1} \pm i^{1-s})
\end{gather}
with $\II \defeq 2 \pi i$ and \(i^s\defeq e^{i\pi s/2}\),
to rewrite \(\cM_f^+\) and \(\cM_f^-\) in terms of \(\Gamma\)-functions and
exponential functions:
\begin{align}
    \label{eqn:Mellin_plus_step_2}
    \cM_f^{\pm} = \underbrace{\prod_{j=1}^4 \frac{ (2
    \pi)^{a_j-b_j-1}}{\gamma(a_j-b_j)}}_{\eqdef A}
    \x
    \prod_{j=1}^4 \Gamma(b_j+s)\Gamma(1-a_j-s)
\underbrace{\prod_{j=1}^{4} (i^{-b_j-s}\pm i^{b_j+s})(i^{a_j+s-1}\pm
i^{1-a_j-s})}_{\eqdef B_{\pm}}
\end{align}
Therefore, to compute \(\cM_f^+(s)+\cM_f^-(s)\), we expand
\begin{align} \label{CkDefinition}
    A (B_+ + B_-) =  \sum_{k=-4}^4 C_k i^{2ks} =
    \sum_{k=-2}^{2} C_{2k}i^{4ks}
\end{align}
in powers of $i^s$, where 
the $C_k$s are various constants that are sums of products of
\((2\pi)^{a_j-b_j-1}\), $i^{\pm a_j},
i^{\pm b_j}$ and $\gamma(a_j-b_j)$; for the last equality,
it is easy to see that \(C_k=0\) for all odd \(k\) because terms from \(B_+\)
and from \(B_-\) cancel.

Now fix $0 < \sigma < \OneHalf$.
Taking  classical inverse Mellin transform, we have
\begin{align}
    \label{eqn:Mellin_plus_step_3}
    f(x)&= 
    \sum_{k=-2}^2\frac{C_{2k}}{4\pi i}\int_{\sigma-i\infty}^{\sigma+i\infty}\prod_{j=1}^4
        \Gamma(b_j+s)\Gamma(1-a_j-s)i^{4ks}x^{-s}\dd s
\end{align}
where we take the straight line contour from  \(-\sigma-i\infty\) to \(-\sigma+i\infty\).
For simplicity, we will restrict to \(x\in\bbR_{>0}\), since we will be
most interested in the value at \(x=1\); in particular, power functions are
well-defined.

In order to justify the application of inverse Mellin transform, observe that the integral is indeed absolutely convergent.
To handle the asymptotics, it is convenient to rewrite, using the relation
$\Gamma(s) \Gamma(1-s)=\frac{\pi}{\sin(\pi s)}$,
the product of $\Gamma$-functions as
\begin{equation}  \label{GammaProduct}
    \prod_{j=1}^4 \frac{\Gamma(1-a_j-s)}{\Gamma(1-b_j-s)} \frac{\pi}{\sin(\pi
    (b_j+s))}
\end{equation}
and use the fact (\cite{TrEr51}) that the ratio $\Gamma(s+a)/\Gamma(s+b)$ is asymptotic to
$s^{a-b}$ so long as we restrict the argument of $s$ to be within
\((-\pi+\delta,\pi-\delta)\) for any fixed \(\delta>0\).\footnote{
That is to say, \(s^{b-a}\Gamma(s+a)/\Gamma(s+b)\) approaches $1$ as $s$
approaches infinity in this region.}  In particular, when \(s\)
is restricted to any vertical line, this same ratio is bounded by
$(1+\abs{t})^{\Re(a)-\Re(b)}$ where \(t\) is the imaginary part of \(s\).
Consequently, the integrand has the asymptotic behavior $\abs{t}^{-2}$, and so is absolutely convergent. 
Changing variables \(s\mapsto -s, k \mapsto -k\), we obtain
\begin{align}
    \label{eqn:Mellin_plus_step_4}
    f(x)&=
    \sum_{k=-2}^2\frac{C_{-2k}}{4\pi i}\int_{-\sigma-i\infty}^{-\sigma+i\infty}\prod_{j=1}^4
    \Gamma(b_j-s)\Gamma(1-a_j+s)i^{4ks}x^{s}\dd s.
\end{align}

The integral
\begin{align}
    \label{eqn:integral_G_function}
    G_{4,4}^{4,4}\!\left(\!\begin{array}{c}
        a_1,\ldots,a_4\\
        b_1,\ldots,b_4
    \end{array}\!\Big|\; x\,\right)
    \defeq\frac{1}{2\pi i}\int_L
    \prod_{j=1}^4 \Gamma(b_j-s)\Gamma(1-a_j+s)x^{s}\dd s
\end{align}
taken along certain admissible paths \(L\), is known as a Meijer \(G\)-function.
To understand its connection with generalized hypergeometric functions, we need
to review some of its properties. The general theory of \(G\)-functions is rich,
but what we need can be easily derived from basic complex analysis.

For our choice of $L$, namely,  the vertical contour with real part $-\sigma$,
  all the poles of \(\Gamma(1-a_j+s)\) appears to the left of \(L\)
and those of \(\Gamma(b_j-s)\) to the right.
We then use residues to evaluate the integral; but we will shift
the contour in different ways according to whether $\abs{x}$ is less than unity
or greater than unity.

We first assume that \(\abs{x}\le 1\). In this case, we shift the integral to
the line where the real part of $x$ equals some large positive real $c$, where
$c$ is chosen (``pole avoidance'') so that its distance from the series of
points $b_j, b_j+1, \dots$ is at least $1/8$ (which is possible because \(j\)
ranges from \(1\) to \(4\)). In order to shift contours in this way,  we
consider the segment of our integral from $-\sigma-iT$ to $-\sigma+iT$ and
connect it to the line segment from $c-iT$ to $c+iT$ by means of horizontal
segments. The same reasoning that is given after \eqref{GammaProduct} shows
that, as we take $T \rightarrow \infty$, the contribution of both horizontal
segments vanish. Moreover, if we then take $c \rightarrow \infty$, the
contribution of the right-hand segment vanishes too; here we use both the fact
that $\abs{x} \leq 1$ and the uniformity of the asymptotic for
$\frac{\Gamma(s+a)}{\Gamma(s+b)}$ in the relevant region. The importance of the
choice of $c$ is to ensure that the term $\sin(\pi (b_j-s))$ is bounded away
from zero.

This shows that when \(\abs{x}\le 1\), the integral
\eqref{eqn:integral_G_function} is equal to the \emph{negative} of the series whose
terms are the residues at all \(b_j+n\), in the sense that one converges
absolutely if and only if the other does. More explicitly, this series is
\begin{align}
    &-\sum_{n=0}^\infty \sum_{h=1}^4\frac{(-1)^n
        \prod_{j=1}^4\Gamma(1-a_j+b_h+n)\prod_{j\neq h}\Gamma(b_j-b_h-n)}{n!}x^{b_h+n}\\
    &\qquad =-\sum_{h=1}^4\prod_{j=1}^4\Gamma(1-a_j+b_h)\prod_{j\neq
    h}\Gamma(b_j-b_h)x^{b_h}
    \x \sum_{n=0}^\infty\frac{\prod_{j=1}^4(1-a_j+b_h)_n}{\prod_{j\neq
    h}(1-b_j+b_h)_n}\frac{x^n}{n!}.
\end{align}
Using the definition of \(\pFq{4}{3}\), we then obtain 
\begin{align}
    G_{4,4}^{4,4}\!\left(\!\begin{array}{c}
        a_1,\ldots,a_4\\
        b_1,\ldots,b_4
    \end{array}\!\Big|\; x\,\right)
    &=-
    \sum_{h=1}^4\prod_{j=1}^4\Gamma(1-a_j+b_h)\prod_{j\neq
    h}\Gamma(b_j-b_h)x^{b_h}\\
    &\x\pFq{4}{3}\!\left(\!\begin{array}{c}
        1+b_h-a_j, j=1,\ldots,4\\
        1+b_h-b_j, j\neq h
    \end{array}\!\Big|\; x\,\right).
\end{align}
The series defining these \(\pFq{4}{3}\) converge absolutely when \(\abs{x}\le
1\), and so the whole equality is valid in the same domain.

Similarly, when \(\abs{x}\ge 1\), we shift the contour to the {\em left},
i.e. choose $c$ to be very negative, now incurring poles when $s=a_h-1-n$ for $n \geq 0$.
The resulting formula is
\begin{align}
    G_{4,4}^{4,4}\!\left(\!\begin{array}{c}
        a_1,\ldots,a_4\\
        b_1,\ldots,b_4
    \end{array}\!\Big|\; x\,\right)
    &=
    \sum_{h=1}^4\prod_{j=1}^4\Gamma(1-a_h+b_j)\prod_{j\neq
    h}\Gamma(a_h-a_j)x^{a_h-1}\\
    &\x\pFq{4}{3}\!\left(\!\begin{array}{c}
        1-a_h+b_j, j=1,\ldots,4\\
        1-a_h+a_j, j\neq h
    \end{array}\!\Big|\; x^{-1}\,\right).
\end{align}
This in particular also allows us to analytically continue the \(G\)-function
for all \(x\) within our chosen domain containing \(1\).

Finally, we go back to \eqref{eqn:Mellin_plus_step_4}.  For the summand where \(k=0\), the integral involved is
exactly the Meijer \(G\)-function discussed above. For any \(k\neq 0\), we may
use the same contour integral argument and the fact that \(i^{4kn}=1\) for any
\(n\in\bbZ\), and see that (we use the \(\abs{x}\le 1\)
formula here for example)
\begin{align}
    &\frac{1}{2\pi i}\int_{-\sigma-i\infty}^{-\sigma+i\infty} \prod_{j=1}^4
    \Gamma(b_j-s)\Gamma(1-a_j+s)i^{4ks}x^{s}\dd s\\
    &\qquad=-\sum_{h=1}^4\prod_{j=1}^4\Gamma(1-a_j+b_h)\prod_{j\neq
    h}\Gamma(b_j-b_h)i^{4kb_h}x^{b_h}
    \x \sum_{n=0}^\infty\frac{\prod_{j=1}^4(1-a_j+b_h)_n}{\prod_{j\neq
    h}(1-b_j+b_h)_n}\frac{x^n}{n!}.
\end{align}
Combining everything together, we have for \(\abs{x}\le 1\):
\begin{align}
    \label{eqn:4F3_formula_1}
    f(x)&=\frac{1}{2}
     \sum_{h=1}^4
    B_h(x)\,\pFq{4}{3}\!\left(\!\begin{array}{c}
        1+b_h-a_j, j=1,\ldots,4\\
        1+b_h-b_j, j\neq h
    \end{array}\!\Big|\; x\,\right),
\end{align}
where
\begin{align}
    B_h(x)
    &=-x^{b_h}\prod_{j=1}^4\Gamma(1-a_j+b_h)\prod_{j\neq h}\Gamma(b_j-b_h)
    \left(\sum_{k=-2}^2C_{-2k}i^{4kb_h}\right).
\end{align}
Recall by \eqref{CkDefinition}, we have
\begin{align}
    \sum_{k=-2}^2C_{-2k}i^{4kb_h}=A(B_+ + B_-)|_{s=-b_h}.
\end{align}
But for \(B_-\), we have
\begin{align}
    B_-|_{s=-b_h} = \prod_{j=1}^{4} (i^{-b_j+b_h} - i^{b_j-b_h})(i^{a_j-b_h-1}-
    i^{1-a_j+b_h})=0
\end{align}
because when \(j=h\) the factor
\begin{align}
    i^{-b_h+b_h} - i^{b_h-b_h}=0.
\end{align}
Therefore, only the term \(AB_+\) survives, and so we can simplify:
\begin{align}
    B_h(x)
    &= -2x^{b_h}  \prod_{j=1}^4 \frac{ (2 \pi)^{a_j-b_j-1}}{\gamma(a_j-b_j)}
    \prod_{j=1}^4\Gamma(1-a_j+b_h)(i^{a_j-b_h-1}+i^{1-a_j+b_h})\\
    &\qquad\qquad\prod_{j\neq h}\Gamma(b_j-b_h)(i^{-b_j+b_h}+i^{b_j-b_h})\\
    &\stackrel{\eqref{gammaReval}}{=} -2x^{b_h} \sideset{}{'}\prod_{j=1}^{4}
    \frac{ \gamma(a_j-b_h)}{\gamma(a_j-b_j) \gamma(b_j-b_h)},
\end{align}
where $\prod'$ denotes that we omit any evaluations of $\gamma$ at polar points; in
the case above, this means $\gamma(b_j-b_h)$ for $j=h$.
As a result, we finally have for \(\abs{x}\le 1\):
\begin{align}
    \label{eqn:4F3_formula_1_final}
    f(x)&=-\sum_{h=1}^4
    x^{b_h} \left( \sideset{}{'}\prod_{j=1}^4 \frac{  \gamma(a_j-b_h)}{\gamma(b_j-b_h)\gamma(a_j-b_j)}\right)
    \,\pFq{4}{3}\!\left(\!\begin{array}{c}
        1+b_h-a_j, j=1,\ldots,4\\
        1+b_h-b_j, j\neq h
    \end{array}\!\Big|\; x\,\right).
\end{align}

Similarly, for \(\abs{x}\ge 1\), we have
\begin{align}
    \label{eqn:4F3_formula_2}
    f(x)&=\frac{1}{2}\sum_{h=1}^4
    A_h(x)\,\pFq{4}{3}\!\left(\!\begin{array}{c}
        1-a_h+b_j, j=1,\ldots,4\\
        1-a_h+a_j, j\neq h
    \end{array}\!\Big|\; x^{-1}\,\right),
\end{align}
where
\begin{align}
    A_h(x)=x^{a_h-1}\prod_{j=1}^4\Gamma(1-a_h+b_j)\prod_{j\neq
    h}\Gamma(a_h-a_j)\left(\sum_{k=-2}^2C_{-2k}i^{4k(a_h-1)}\right),
\end{align}
and we may simplify it as
\begin{align}
    \label{eqn:4F3_formula_2_final}
    f(x)&=\sum_{h=1}^4 x^{a_h-1}\left(\sideset{}{'}\prod_{j=1}^4
    \frac{\gamma(a_h-b_j)}{\gamma(a_h-a_j)\gamma(a_j-b_j)}\right)
    \,\pFq{4}{3}\!\left(\!\begin{array}{c}
        1-a_h+b_j, j=1,\ldots,4\\
        1-a_h+a_j, j\neq h
    \end{array}\!\Big|\; x^{-1}\,\right).
\end{align}

\subsection{The tetrahedral symbol as a convolution of $\gamma$-factors, for general $F$}
\label{TetrahedralAsGammaConvolution}

Return for a moment to the case of general $F$. We shall prove the formula
\Cref{PROPAB}, which we recall here: 
\begin{align}
    \label{eqn:hypergeometric_formula_A_and_B2} 
    \Pisymbol=\nu_\bbP^{-1}
    \int_{\mu}
    \frac{\gamma(\frac{1}{2}+\epsilon, A\otimes\mu)
    \gamma(1-\epsilon, B^{-1}\otimes\mu^{-1})}
    {\sqrt{\gamma(\OneHalf, A\otimes B^{-1})}}\dd\mu
\end{align}
for four-element sets of characters $A, B$.

We shall apply the results of the former subsection with the following \(\alpha_i\)
\begin{align}
    \alpha_1 = \chi_{\V{12}}\chi_{\V{14}}\chi_{\V{13}}^{-1},\quad
    \alpha_2 = \chi_{\V{12}}\chi_{\V{13}}\chi_{\V{14}}^{-1},\quad
    \alpha_3 = \chi_{\V{21}}\chi_{\V{24}}\chi_{\V{23}}^{-1},\quad
    \alpha_4 = \chi_{\V{42}}\chi_{\V{43}}\chi_{\V{41}}^{-1},\\
    \alpha_5 = \chi_{\V{42}}\chi_{\V{41}}\chi_{\V{43}}^{-1},\quad
    \alpha_6 = \chi_{\V{41}}\chi_{\V{43}}\chi_{\V{42}}^{-1},\quad
    \alpha_7 = \chi_{\V{34}}\chi_{\V{31}}\chi_{\V{32}}^{-1}.
\end{align}

We compute the following combinations:
\begin{alignat}{4}
    \alpha_{1234567} &= \chi_{\V{12}}\chi_{\V{31}}\chi_{\V{41}},\quad &
    \alpha_{34567} &= \chi_{\V{21}}\chi_{\V{41}}\chi_{\V{31}},\quad &
    \alpha_{467} &= \chi_{\V{43}}\chi_{\V{31}}\chi_{\V{23}},\quad &
    \alpha_7 &= \chi_{\V{34}}\chi_{\V{31}}\chi_{\V{23}},\\
    \alpha_{134567} &= \chi_{\V{31}}^2,\quad &
    \alpha_{4567} &= \chi_{\V{42}}\chi_{\V{41}}\chi_{\V{31}}\chi_{\V{23}},\quad &
    \alpha_{67} &= \chi_{\V{41}}\chi_{\V{31}}\chi_{\V{24}}\chi_{\V{23}}.\quad &
\end{alignat}
We now rewrite \eqref{eqn:hypergeometric_Mellin_final}
in the form
\begin{align} \label{Mellin Redone}
    \bigl[\gamma(\alpha_{2+})\gamma(\alpha_{3+})\gamma(\alpha_{4+})\gamma(\alpha_{7+})\bigr]^{-1}
    \frac{\gamma(\frac{1}{2} + s, A'\otimes\mu)}{\gamma(s,B'\otimes\mu)},
\end{align}
where  we are to integrate over characters $\mu$ and a fixed real $s$ between
$0$ and $\OneHalf$; and
\begin{align}
    A'= \{\alpha_{1234567}, \alpha_{34567}, \alpha_{467}, \alpha_7\}
    &=\chi_{\V{21}}\chi_{\V{41}}\chi_{\V{31}}\otimes \Set{ \chi_{\V{12}}^2, 1,
    \chi_{\V{43}}\chi_{\V{23}}\chi_{\V{12}}\chi_{\V{14}}, \chi_{\V{34}}\chi_{\V{23}}\chi_{\V{12}}\chi_{\V{14}} },\\
    B' = \Set*{\alpha_{134567}, \alpha_{4567}, \alpha_{67}, 1}
    &= \chi_{\V{31}}^2 \otimes \Set{1, \chi_{\V{42}}\chi_{\V{41}}\chi_{\V{13}}\chi_{\V{23}},
    \chi_{\V{41}}\chi_{\V{13}}\chi_{\V{24}}\chi_{\V{23}}, \chi_{\V{13}}^2}.
\end{align}

Now, \eqref{Mellin Redone}, when integrated over $\mu$,
is invariant under a common translation of $A', B'$. Doing such 
translating $A'$ and $B'$ both by $\chi_{\V{13}}$, we arrive at:
\begin{align}
    A \defeq \chi_{\V{13}}\otimes A'
    &= \bigl(\chi_{\V{41}}\otimes\Set{ \chi_{\V{12}} , \chi_{\V{21}}}\bigr)
    \cup \bigl(\chi_{\V{23}}\otimes\Set{\chi_{\V{43}}, \chi_{\V{34}}}\bigr),\\
    B \defeq \chi_{\V{13}}\otimes B'
    &= \Set{\chi_{\V{31}}, \chi_{\V{13}}}\cup
    \bigl(\chi_{\V{41}} \chi_{\V{23}}\otimes\Set{\chi_{\V{42}},\chi_{\V{24}}}\bigr).
\end{align}
Using the fact that \(\gamma(\mu_+)\gamma(\mu_-^{-1})=\mu(-1)\), we can replace
the inverse of $\gamma(s, B\otimes\mu)$ by $\gamma(1-s, B^{-1}\otimes\mu^{-1})$.
Consequently we rewrite the edge integral in
\eqref{eqn:edge_integral_hypergeometric} as the inverse Mellin transform of
\begin{align}
    \mu \longmapsto \frac{\gamma(\frac{1}{2}+s, A\otimes\mu)
    \gamma(1-s,  B^{-1}\otimes\mu^{-1})}
    {\sqrt{\gamma(\OneHalf, \Sigma^*)}},
\end{align}
evaluated at \(1\),
where \(\Sigma^*\) is modified from \(\Sigma^-\) by including the eigenspaces of
eigenvalues \(x_{\V{12}}x_{\V{13}}x_{\V{14}}^{-1}\),
\(x_{\V{21}}x_{\V{23}}^{-1}x_{\V{24}}\), \(x_{\V{41}}^{-1}x_{\V{42}}x_{\V{43}}\),
\(x_{\V{31}}x_{\V{32}}^{-1}x_{\V{34}}\) (cf.~\Cref{sub:halfspin})
instead of their inverses; these arise from the \(\gamma(\alpha_{2+}), \ldots,
\gamma(\alpha_{7+})\) factors. 
 To conclude we note that \(A\otimes B^{-1}\) is exactly equal to \(\Sigma^*\) (the underlined
ones below are those in \(\Sigma^*\) but not \(\Sigma^-\)):
\begin{align}
    \label{eqn:S_ast_entries}
    \begin{array}{c|cccc}
         & \alpha_{1234567} & \alpha_{34567} & \alpha_{467} & \alpha_7\\
         \hline& \\[\dimexpr-\normalbaselineskip+5pt]
        \alpha_{134567}^{-1} 
         &\ul{\chi_{\V{12}}\chi_{\V{13}}\chi_{\V{14}}^{-1}}
         &\chi_{\V{12}}^{-1}\chi_{\V{13}}\chi_{\V{14}}^{-1}
         &\chi_{\V{31}}^{-1}\chi_{\V{32}}^{-1}\chi_{\V{34}}^{-1}
         &\chi_{\V{31}}^{-1}\chi_{\V{32}}^{-1}\chi_{\V{34}}\\[5pt]
        \alpha_{4567}^{-1}
         &\chi_{\V{21}}^{-1}\chi_{\V{23}}^{-1}\chi_{\V{24}}
         &\ul{\chi_{\V{21}}\chi_{\V{23}}^{-1}\chi_{\V{24}}}
         &\chi_{\V{41}}^{-1}\chi_{\V{42}}^{-1}\chi_{\V{43}}
         &\chi_{\V{41}}^{-1}\chi_{\V{42}}^{-1}\chi_{\V{43}}^{-1}\\[5pt]
        \alpha_{67}^{-1}
         &\chi_{\V{21}}^{-1}\chi_{\V{23}}^{-1}\chi_{\V{24}}^{-1}
         &\chi_{\V{21}}\chi_{\V{23}}^{-1}\chi_{\V{24}}^{-1}
         &\ul{\chi_{\V{41}}^{-1}\chi_{\V{42}}\chi_{\V{43}}}
         &\chi_{\V{41}}^{-1}\chi_{\V{42}}\chi_{\V{43}}^{-1}\\[5pt]
        1
         &\chi_{\V{12}}\chi_{\V{13}}^{-1}\chi_{\V{14}}^{-1}
         &\chi_{\V{12}}^{-1}\chi_{\V{13}}^{-1}\chi_{\V{14}}^{-1}
         &\chi_{\V{31}}\chi_{\V{32}}^{-1}\chi_{\V{34}}^{-1}
         &\ul{\chi_{\V{31}}\chi_{\V{32}}^{-1}\chi_{\V{34}}}\\
    \end{array}
\end{align}

\subsection{The tetrahedral symbol for the case $F=\mathbb{R}$}
\label{sub:Mellin_of_tetrahedral_symbols}

The \(F=\bbR\) case is the most interesting when combined with our discussion in
\Cref{sub:relation_with_classical_4F3}. We assume for simplicity that
\(\chi_{ij}\) (\(ij\in\tetraO\)) is of the form \(\abs{\blank}^{J_{ij}}\) for
\(J_{ij}\in i\bbR\) such that \(J_{ij}+J_{ji}=0\). Using
\eqref{eqn:chosen_oriented_tetrahedron}, we also write
\(J_{\EN{1}}=J_{\V{12}}, J_{\EN{2}}=J_{\V{31}}\), and so on.
Then with \(a_i,b_i\) (\(i=1,\ldots,4\)) as in
\Cref{sub:relation_with_classical_4F3}, we have
\begin{alignat}{4}
    a_1 &= J_{\EN{1}\EN{2}\IEN{3}}+\OneHalf,\quad &
    a_2 &= J_{\IEN{1}\EN{2}\IEN{3}}+\OneHalf,\quad &
    a_3 &= J_{\EN{2}\IEN{4}\EN{6}}+\OneHalf,\quad &
    a_4 &= J_{\EN{2}\EN{4}\EN{6}}+\OneHalf,\\
    b_1 &= J_{\EN{2}\EN{2}}, \quad &
    b_2 &= J_{\EN{2}\IEN{3}\IEN{5}\EN{6}},\quad &
    b_3 &= J_{\EN{2}\IEN{3}\EN{5}\EN{6}},\quad &
    b_4 &= 0,
\end{alignat}
where \(J_{\EN{1}\EN{2}\IEN{3}}\) means \(J_{\EN{1}}-J_{\IEN{2}}-J_{\EN{3}}\),
and so on.
Plugging in the formula \eqref{eqn:4F3_formula_1_final} and simplifying using
the properties of \(\gamma\) and \(L\)-factors from \Cref{GammaReview}, we have
\begin{align}
    \Pisymbol\sqrt{L\Bigl(\OneHalf,\Sigma\Bigr)}
    &=-\frac{L(\OneHalf,\Sigma_1^*)}
        {\gamma(J_{\IEN{2}\IEN{3}\IEN{5}\EN{6}})\gamma(J_{\IEN{2}\IEN{3}\EN{5}\EN{6}})\gamma(J_{\IEN{2}\IEN{2}})}
        \pFq{4}{3}\!\left(\!
            \begin{array}{c}
                J_{\IEN{1}\EN{2}\EN{3}}+\OneHalf,
                J_{\EN{1}\EN{2}\EN{3}}+\OneHalf,
                J_{\EN{2}\EN{4}\IEN{6}}+\OneHalf,
                J_{\EN{2}\IEN{4}\IEN{6}}+\OneHalf
                \\
                J_{\EN{2}\EN{3}\EN{5}\IEN{6}}+1,
                J_{\EN{2}\EN{3}\IEN{5}\IEN{6}}+1,
                J_{\EN{2}\EN{2}}+1
            \end{array}\!\Big|\; 1\,
        \right)\\
    &\mathbin{\phantom{=}}-\frac{L(\OneHalf,\Sigma_2^*)}
        {\gamma(J_{\EN{2}\EN{3}\EN{5}\IEN{6})}\gamma(J_{\EN{5}\EN{5}})\gamma(J_{\IEN{2}\EN{3}\EN{5}\IEN{6}})}
        \pFq{4}{3}\!\left(\!
            \begin{array}{c}
                J_{\IEN{1}\IEN{5}\EN{6}}+\OneHalf,
                J_{\EN{1}\IEN{5}\EN{6}}+\OneHalf,
                J_{\IEN{3}\EN{4}\IEN{5}}+\OneHalf,
                J_{\IEN{3}\IEN{4}\IEN{5}}+\OneHalf
                \\
                J_{\IEN{2}\IEN{3}\IEN{5}\EN{6}}+1,
                J_{\IEN{5}\IEN{5}}+1,
                J_{\EN{2}\IEN{3}\IEN{5}\EN{6}}+1
            \end{array}\!\Big|\; 1\,
        \right)\\
    &\mathbin{\phantom{=}}-\frac{L(\OneHalf,\Sigma_3^*)}
        {\gamma(J_{\EN{2}\EN{3}\IEN{5}\IEN{6}})\gamma(J_{\IEN{5}\IEN{5}})\gamma(J_{\IEN{2}\EN{3}\IEN{5}\IEN{6}})}
    \pFq{4}{3}\!\left(\!
            \begin{array}{c}
                J_{\IEN{1}\EN{5}\EN{6}}+\OneHalf,
                J_{\EN{1}\EN{5}\EN{6}}+\OneHalf,
                J_{\IEN{3}\EN{4}\EN{5}}+\OneHalf,
                J_{\IEN{3}\IEN{4}\EN{5}}+\OneHalf
                \\
                J_{\IEN{2}\IEN{3}\EN{5}\EN{6}}+1,
                J_{\EN{5}\EN{5}}+1,
                J_{\EN{2}\IEN{3}\EN{5}\EN{6}}+1
            \end{array}\!\Big|\; 1\,
        \right)\\
    &\mathbin{\phantom{=}}-\frac{L(\OneHalf,\Sigma_4^*)}
        {\gamma(J_{\EN{2}\EN{2}})\gamma(J_{\EN{2}\IEN{3}\IEN{5}\EN{6}})\gamma(J_{\EN{2}\IEN{3}\EN{5}\EN{6}})}
    \pFq{4}{3}\!\left(\!
            \begin{array}{c}
                J_{\IEN{1}\IEN{2}\EN{3}}+\OneHalf,
                J_{\EN{1}\IEN{2}\EN{3}}+\OneHalf,
                J_{\IEN{2}\EN{4}\IEN{6}}+\OneHalf,
                J_{\IEN{2}\IEN{4}\IEN{6}}+\OneHalf
                \\
                J_{\IEN{2}\IEN{2}}+1,
                J_{\IEN{2}\EN{3}\EN{5}\IEN{6}}+1,
                J_{\IEN{2}\EN{3}\IEN{5}\IEN{6}}+1
            \end{array}\!\Big|\; 1\,
        \right),
\end{align}
where \(\Sigma_i^*\) is obtained from \eqref{eqn:S_ast_entries} by inverting the
\(i\)-th \emph{row}. Note also that the first row of parameters in each
\(\pFq{4}{3}\) corresponds precisely to the respective row in
\eqref{eqn:S_ast_entries} up to a universal inversion and half twist.

Similarly, plugging in \eqref{eqn:4F3_formula_2_final}, we also have
\begin{align}
    \Pisymbol\sqrt{L\Bigl(\OneHalf,\Sigma\Bigr)}
    &=+\frac{L(\OneHalf,\Sigma_{1'}^*)}
        {\gamma(J_{\EN{1}\EN{1}})\gamma(J_{\EN{1}\IEN{3}\EN{4}\EN{6}})\gamma(J_{\EN{1}\IEN{3}\IEN{4}\EN{6}})}
        \pFq{4}{3}\!\left(\!
            \begin{array}{c}
                J_{\IEN{1}\EN{2}\EN{3}}+\OneHalf,
                J_{\IEN{1}\IEN{5}\EN{6}}+\OneHalf,
                J_{\IEN{1}\EN{5}\EN{6}}+\OneHalf,
                J_{\IEN{1}\IEN{2}\EN{3}}+\OneHalf
                \\
                J_{\IEN{1}\IEN{1}}+1,
                J_{\IEN{1}\EN{3}\IEN{4}\IEN{6}}+1,
                J_{\IEN{1}\EN{3}\EN{4}\IEN{6}}+1
            \end{array}\!\Big|\; 1\,
        \right)\\
    &\mathbin{\phantom{=}}+\frac{L(\OneHalf,\Sigma_{2'}^*)}
        {\gamma(J_{\IEN{1}\IEN{1}}\gamma(J_{\IEN{1}\IEN{3}\EN{4}\EN{6}})\gamma(J_{\IEN{1}\IEN{3}\IEN{4}\EN{6}})}
        \pFq{4}{3}\!\left(\!
            \begin{array}{c}
                J_{\EN{1}\EN{2}\EN{3}}+\OneHalf,
                J_{\EN{1}\IEN{5}\EN{6}}+\OneHalf,
                J_{\EN{1}\EN{5}\EN{6}}+\OneHalf,
                J_{\EN{1}\IEN{2}\EN{3}}+\OneHalf
                \\
                J_{\EN{1}\EN{1}}+1,
                J_{\EN{1}\EN{3}\IEN{4}\IEN{6}}+1,
                J_{\EN{1}\EN{3}\EN{4}\IEN{6}}+1
            \end{array}\!\Big|\; 1\,
        \right)\\
    &\mathbin{\phantom{=}}+\frac{L(\OneHalf,\Sigma_{3'}^*)}
        {\gamma(J_{\IEN{1}\EN{3}\IEN{4}\EN{6}})\gamma(J_{\EN{1}\EN{3}\IEN{4}\EN{6}})\gamma(J_{\IEN{4}\IEN{4}})}
    \pFq{4}{3}\!\left(\!
            \begin{array}{c}
                J_{\EN{2}\EN{4}\IEN{6}}+\OneHalf,
                J_{\IEN{3}\EN{4}\IEN{5}}+\OneHalf,
                J_{\IEN{3}\EN{4}\EN{5}}+\OneHalf,
                J_{\IEN{2}\EN{4}\IEN{6}}+\OneHalf
                \\
                J_{\EN{1}\IEN{3}\EN{4}\IEN{6}}+1,
                J_{\IEN{1}\IEN{3}\EN{4}\IEN{6}}+1,
                J_{\EN{4}\EN{4}}+1
            \end{array}\!\Big|\; 1\,
        \right)\\
    &\mathbin{\phantom{=}}+\frac{L(\OneHalf,\Sigma_{4'}^*)}
        {\gamma(J_{\IEN{1}\EN{3}\EN{4}\EN{6}})\gamma(J_{\EN{1}\EN{3}\EN{4}\EN{6}})\gamma(J_{\EN{4}\EN{4}})}
    \pFq{4}{3}\!\left(\!
            \begin{array}{c}
                J_{\EN{2}\IEN{4}\IEN{6}}+\OneHalf,
                J_{\IEN{3}\IEN{4}\IEN{5}}+\OneHalf,
                J_{\IEN{3}\IEN{4}\EN{5}}+\OneHalf,
                J_{\IEN{2}\IEN{4}\IEN{6}}+\OneHalf
                \\
                J_{\EN{1}\IEN{3}\IEN{4}\IEN{6}}+1,
                J_{\IEN{1}\IEN{3}\IEN{4}\IEN{6}}+1,
                J_{\IEN{4}\IEN{4}}+1
            \end{array}\!\Big|\; 1\,
        \right),
\end{align}
where \(\Sigma_{i'}^*\) is obtained from \eqref{eqn:S_ast_entries} by inverting the
\(i\)-th \emph{column}, and the first row of parameters in each
\(\pFq{4}{3}\) corresponds to the respective column in
\eqref{eqn:S_ast_entries} up to a universal inversion and half twist.

Finally, it is not hard to deduce from definition that all the eight
\(\pFq{4}{3}\) functions on the right-hand sides above converge absolutely at
\(1\) for any generic \(J_{\EN{1}},\ldots,J_{\EN{6}}\in\bbC\) (see
\cite[\S~2.2]{Sl66}), and so by
analytic continuation, the two equalities above continue to hold without
assuming these \(J\)s being purely imaginary.

\section{Some proof sketches for \Cref{sec:Further} }
In this section, we provide sketches of proofs of some statements in
\Cref{sec:Further}. We are confident that they turn into real proofs once proper
technical care is given, but nevertheless we have not done it.  Also, we 
do in fact wonder if any human being will read this.

\subsection{A sketch of the proof of \Cref{Orthogonality}} \label{OrthoProof}

The following proof is rigorous --- but can also be written more easily,
as in the traditional theory of the tetrahedral symbol ---
when $\RO$ compact. 
However, we have chosen to write the proof in a way should be  valid in the general case.
We say ``should be'' because there are details we have not tried to fill in: we expect them
to be routine but rather tedious and notationally cumbersome to handle.  
Let us pre-emptively confess to these  sins:

\begin{itemize}
    \item We shall not distinguish between the space of smooth vectors in an
        $\RO$-representation and its Hilbert completion; similarly we will not
        clearly distinguish between maps that are defined on the smooth part and
        on the Hilbert completion.
    \item We will be freely using the theory of unitary decomposition. In
        particular, this theory handles various set- and measure-theoretic
        issues that we will not even allude to. 
        
        As a typical example, let
        $f(\pi)$ be a rule that assigns to  each class $[\pi]$ in the unitary
        dual of $\RO$ an element of ``the'' corresponding unitary representation
        $\pi$. In this situation is a reasonable way to talk about such $f$s, and there is a
        reasonable way to talk about them being measurable, and taking the inner
        products of two such; but all this we will omit. 
    \item Many of the functions below are defined in the sense of measure
        theory, i.e., off zero measure sets; again, we will ignore this entirely
        in our language.
    \item We shall not justify absolute convergence of expressions below. This
        is the only sin that we think is not venial, because it requires
        estimates that we did not carry out.
\end{itemize}

Suppose $\pi, \sigma$ are tempered representations, upon which we fix
self-duality pairings. For any tempered representation $\tau\in\cP_0$, on which
we we also fix a self-duality pairing, we may normalize an invariant functional
$\Lambda$ on $\pi \otimes \sigma \otimes \tau$ according to
\Cref{eqn:def_of_Lambda_H}. We dualize it to obtain a map $\pi \otimes \sigma
\rightarrow \tau$. (The first sin!) 

Recall  (\Cref{PrasadLemma}) that there are real structures $\pi_{\bbR}, \sigma_{\bbR}, \tau_{\bbR}$
on which the duality structures are inner products. We then define an inner
product on $\pi, \sigma, \tau$ by complex-linear extension.

\begin{lemma} \label[lemma]{PlancherelTriple}
    The induced map
    \begin{align}
        \pi  \otimes \sigma  \longto \int \tau  \dd\tau
    \end{align}
    extends to an isometry of  Hilbert spaces upon completing the
    left-hand side. On the right, the measure is Plancherel measure; and we integrate over the subset of \(\cP_0\) with
    the property that there is a nontrivial
    invariant functional on $\pi \otimes \sigma \otimes \tau$.
\end{lemma}
\begin{proof}
    We prove the corresponding statement for real Hilbert spaces, from which the
    claim follows by complex-linear extension. This amounts to the following
    assertion for $v, v'\in\pi_{\R}$ and $w, w'\in\sigma_{\R}$:
    \begin{align} \label{PConv}
        \IPair{v}{v'}\IPair{w}{w'}
        = \int \dd\tau
        \int_{h\in\RO}
        \sum_{e \in \mathcal{B}_{\tau}}
        \Pair{hv}{v'}\Pair{hw}{w'}\Pair{h e}{e}\dd h.
    \end{align}
    where $\mathcal{B}_{\tau}$ is an orthonormal basis for $\tau_{\R}$, which is a
    consequence of the Plancherel formula \eqref{Plancherel} applied to the
    function $h \mapsto \IPair{hv}{v'}\IPair{hw}{w'}$.
\end{proof}

We now put ourselves in the situation of \Cref{orthogonality}; having fixed
self-duality pairings on all the $\pi_{ij}$ we fix inner products as
above. \emph{In what follows, $\pi_{\V{12}}, \pi_{\V{23}}, \pi_{\V{34}}$ should
be regarded as fixed,} but the remaining $\pi$s will be ``varying'' --- that is,
we will not explicitly include in the notation dependence on $\pi_{\V{12}},
\pi_{\V{23}}, \pi_{\V{34}}$.

Applying the lemma (to various choices of $\pi,\sigma$),  
we find isometries:
\begin{align} \label{IsometryA}
    \pi_{\V{12}} \otimes \pi_{\V{23}}
    \simeq \int_{\pi_{\V{24}}}\pi_{\V{24}}
    \implies 
    \pi_{\V{12}} \otimes \pi_{\V{23}}  \otimes \pi_{\V{34}}
    \simeq \int_{\pi_{\V{24}}}\pi_{\V{24}}\otimes \pi_{\V{34}}
    \simeq\int_{\pi_{\V{14}},\pi_{\V{24}}}\pi_{\V{14}}.
\end{align}

Here, $\simeq$ means an isometry of Hilbert spaces. The range of integration on
the right hand side consists of those $\pi_{\V{14}},\pi_{\V{24}}$ for which the
representation both $\pi_{\V{21}}\otimes\pi_{\V{23}}$ and $\pi_{\V{14}}
\otimes \pi_{\V{34}}$ admit nonzero invariant maps to $\pi_{\V{24}}$. For a
given $\pi_{\V{14}}$ let us call this set of $\pi_{\V{24}}$ by the name
$A(\pi_{\V{14}})$ (it depends on the other $\pi$s too, but we are regarding them
as fixed).

We get a similar decomposition with $\V{24}$ replaced by $\V{13}$:
\begin{align} \label{IsometryB}
    \pi_{\V{23}}  \otimes \pi_{\V{34}}
    \simeq \int_{\pi_{\V{13}}}   \pi_{\V{13}}
    \implies
    \pi_{\V{12}} \otimes \pi_{\V{23}}\otimes \pi_{\V{34}}
    \simeq\int_{\pi_{\V{12}}}   \pi_{\V{12}} \otimes \pi_{\V{13}} \simeq \int_{\pi_{\V{14}},\pi_{\V{13}}} \pi_{\V{14}}.
\end{align}
where the range of integration consists of pairs $\pi_{\V{14}},\pi_{\V{13}}$
with the property that both $\pi_{\V{23}} \otimes \pi_{\V{34}}$ and
$\pi_{\V{14}} \otimes \pi_{\V{12}}$ admit maps to $\pi_{\V{13}}$. For a given
$\pi_{\V{14}}$ let us call this set of $\pi_{\V{13}}$ by the name $
B(\pi_{\V{14}})$ (it depends on the other $\pi$s too, but we are regarding them
as fixed). 

Therefore, there is an isometry
\begin{equation}\label{GaGL}
    \int_{ \pi_{\V{14}}, \pi_{\V{24}}}  \pi_{\V{14}}
    \simeq\int_{\pi_{\V{14}},\pi_{\V{13}}}\pi_{\V{14}},
\end{equation}
where the ranges of integration are as specified above. Being equivariant for
the group action this map necessarily has a rather special form: for each
$\pi_{\V{14}}$ we must have an isometry
\begin{align}
    L^2(A(\pi_{\V{14}})) \simeq L^2(B(\pi_{\V{14}})),
\end{align}
which is necessarily given by a scalar-valued kernel function
$K^{\pi_{\V{14}}}(\pi_{\V{24}},\pi_{\V{13}})$
on the product $A(\pi_{\V{14}}) \times B(\pi_{\V{14}})$; explicitly, 
the isometry of \Cref{GaGL} necessarily has the form
\begin{equation}
    \label{KernelCharacterization}
    f(\pi_{\V{24}}, \pi_{\V{14}})
    \longmapsto
    \int_{\pi_{\V{24}}}
    K^{\pi_{\V{14}}}(\pi_{\V{24}},\pi_{\V{13}}) f(\pi_{\V{24}}, \pi_{\V{14}}),
\end{equation}
where the function $f$  is vector-valued: it takes inputs $\pi_{\V{14}}$ and $
\pi_{\V{24}}$ and returns a vector in the space of $\pi_{\V{14}}$; on the
right, the input parameters are \(\pi_{\V{14}}\) and \(\pi_{\V{13}}\) with the
output a vector in \(\pi_{\V{14}}\).

In what follows, we write for typical vectors
\begin{align}
    x \in \pi_{\V{12}}, y \in \pi_{\V{23}}, z \in \pi_{\V{34}}, w \in
    \pi_{\V{41}}.
\end{align}

We will denote the image of $x \otimes y \otimes z$ on the right hand side of \eqref{IsometryA}
by $E_{\pi_{\V{24}}}^{\pi_{\V{14}}}(x \otimes y \otimes z)$;
it is an element of $\pi_{\V{14}}$ that depends also on the choice of $\pi_{\V{24}}$.
Similarly we denote its image on the right hand side of \eqref{IsometryB}
by $E_{\pi_{\V{13}}}^{\pi_{\V{14}}}(x \otimes y \otimes z)$. Thus
the $E$s are linear maps
\begin{align}
    E_{\pi_{\V{24}}}^{\pi_{\V{14}}}, E_{\pi_{\V{13}}}^{\pi_{\V{14}}}\colon
    \pi_{\V{12}} \otimes \pi_{\V{23}} \otimes \pi_{\V{34}} \longrightarrow
    \pi_{\V{14}}
\end{align}
Let $x,y,z$ vary through orthonormal bases of $K$-finite vectors for the
respective representations.  Then $x \otimes y \otimes z$ varies through an
orthonormal basis for $\pi_{\V{12}} \otimes \pi_{\V{23}} \otimes \pi_{\V{34}}$.
Therefore,  $E^{\pi_{\V{14}}}_{\pi_{\V{13}}}(x \otimes y \otimes z)$ and
$E^{\pi_{\V{14}}}_{\pi_{\V{24}}}(x \otimes y \otimes z)$, considered as
$\pi_{\V{14}}$-valued functions of $(\pi_{\V{14}}, \pi_{\V{13}})$ or
$(\pi_{\V{14}}, \pi_{\V{24}})$, form orthonormal bases for the two Hilbert
spaces appearing on corresponding sides of \eqref{GaGL}. The unitary
transformation sending one basis to the other sends a function $f(\pi_{\V{24}},
\pi_{\V{14}})$, taking values in $\pi_{\V{14}}$, to
\begin{align}
    f'\colon (\pi_{\V{13}}, \pi_{\V{14}})
    \longmapsto \sum_{x,y,z} E_{\pi_{\V{13}}}^{\pi_{\V{14}}}
    (x \otimes y \otimes z)
    \int_{\pi_{\V{24}}, \pi_{\V{14}}'}
    \Pair
    {f(\pi_{\V{24}}, \pi_{\V{14}}')}{E_{\pi_{\V{24}}}^{\pi_{\V{14}}'} (x \otimes y \otimes z)}.
\end{align}
Let  $f(\pi_{\V{24}},\pi_{\V{14}}), f'(\pi_{\V{13}},\pi_{\V{14}})$
be arbitrary, but ``real-valued,'' i.e. valued in a fixed real form $\pi_{\V{14}, \bbR}$;
this permits us to ignore  complex conjugates in inner products.
Compare the above equation with \eqref{KernelCharacterization}
 to conclude 
\begin{multline}
    \int_{\pi_{\V{14}},\pi_{\V{24}}} K^{\pi_{\V{14}}}(\pi_{\V{24}}, \pi_{\V{13}})
    \Pair{f(\pi_{\V{24}},\pi_{\V{14}})}{f'(\pi_{\V{13}},\pi_{\V{14}})}
   \\  =\int_{\pi_{\V{24}}, \pi_{\V{14}},\pi_{\V{14}}'}
    \sum_{x,y,z}\Pair{f(\pi_{\V{24}},\pi_{\V{14}}')
    \otimes f'(\pi_{\V{13}},\pi_{\V{14}})}
    {E_{\pi_{\V{24}}}^{\pi_{\V{14}}'}(x \otimes y \otimes z)\otimes
    E_{\pi_{\V{13}}}^{\pi_{\V{14}}}(x \otimes y \otimes z)}.
\end{multline}
Since \(f\) is arbitrary, the same equality still holds without integrating over
\(\pi_{\V{24}}\); so 

\begin{multline}
    \label{eqn:kernalK_and_f_fprime}
    \int_{\pi_{\V{14}}} K^{\pi_{\V{14}}}(\pi_{\V{24}}, \pi_{\V{13}})
    \Pair{f(\pi_{\V{24}},\pi_{\V{14}})}{f'(\pi_{\V{13}},\pi_{\V{14}})}
   \\  =\int_{\pi_{\V{14}}, \pi_{\V{14}}'}
    \sum_{x,y,z}\Pair{f(\pi_{\V{24}},\pi_{\V{14}}')
    \otimes f'(\pi_{\V{13}},\pi_{\V{14}})}
    {E_{\pi_{\V{24}}}^{\pi_{\V{14}}'}(x \otimes y \otimes z)\otimes
    E_{\pi_{\V{13}}}^{\pi_{\V{14}}}(x \otimes y \otimes z)}.
\end{multline}

We now take several steps to simplify notation:
\begin{itemize}
    \item[-]Previously we fixed $\pi_{\V{12}}, \pi_{\V{23}}, \pi_{\V{34}}$. We now
        additionally fix $\pi_{\V{13}}$ and $\pi_{\V{24}}$, so that we are now regarding
        \emph{all the $\pi_{ij}$ for $\Set{ij} \neq \Set{\V{14}}$ as fixed}, and will
        not include in the notation dependence on these representations.
    \item[-]Having fixed these, we put
        \begin{align}
            K^{\pi_{\V{14}}} \defeq K^{\pi_{\V{14}}}(\pi_{\V{24}},\pi_{\V{13}}).
        \end{align}
        Again, here, we have a measure-theoretic issue, since $K$ is defined only off a set of measure zero;
        in actuality, all the statements hold almost everywhere, and are to be extended by continuity. 

    \item[-] We 
        abridge $\pi_{\V{14}}$ to $\pi$ and $\pi_{\V{14}}'$ to $\pi'$.
    \item[-] We switch the letters $f,f'$ for nicer typography aligning with
        \(\pi,\pi'\) respectively.
\end{itemize}

Having done all of this we can rewrite
\eqref{eqn:kernalK_and_f_fprime} as:
\begin{align}
    \label{eqn:kernalK_and_f_fprime2}
    K^{\pi} \cdot \int_{\pi}  
    \Pair{f(\pi)}{f'(\pi)}
    =\int_{\pi, \pi'}
    \sum_{x,y,z}\Pair{f'(\pi') \otimes f(\pi)}
    {E_{\pi_{\V{24}}}^{\pi'}(x \otimes y \otimes z)\otimes
    E_{\pi_{\V{13}}}^{\pi}(x \otimes y \otimes z)}.
\end{align}

Let us reinterpret the right hand side as follows.
Consider  the \(G\)-representation
\begin{align} \label{PiGstardef}
    \Pi_G^* = (\pi \boxtimes \pi')\boxtimes
    \displaybigboxtimes_{ij\in\tetraE-\Set{\V{14}}}\pi_{ij} \boxtimes \pi_{ij},
\end{align}
i.e., similar to \(\Pi_G\) as defined in
\Cref{Tetrahedraldatum}
but now taking
$\pi_{\V{14}}=\pi$ and $\pi_{\V{41}}=\pi'$, that is, 
  two previously isomorphic copies of
  $\pi_{\V{14}}$ are taken to be distinct representations.
We can define as before $\Lambda^{H*}$ in the dual
of $\Pi_G^*$  by tensoring
the invariant trilinear functionals on triples of representations indexed
by edges sharing a common vertex.
Note that both $\Pi_G^*$ and $\Lambda^{H*}$ depend on both $\pi$ and $\pi'$,
but to simplify the notation we will not explicitly denote this.

Let $\Pi_{G,\mathrm{big}}^*$ be the same expression \eqref{PiGstardef} integrated over $\pi, \pi'$;
this defines a Hilbert representation of $G$.
Define $\tilde{\Lambda}^H$ in the dual of $\Pi_{G,\mathrm{big}}^*$ by
similarly integrating $\Lambda^{H*}$, where, in both cases, all integrals are taken
with respect to Plancherel measure in both $\pi$ and $\pi'$:
\begin{align}
    \Pi_{G,\mathrm{big}}^* = \int_{\pi,\pi'} \Pi_G^* ,\quad
    \tilde{\Lambda}^H = \int_{\pi,\pi'} \Lambda^{H*} .
\end{align}

Consider the expression $\sum_x x \otimes x \in \pi_{\V{12}} \otimes \pi_{\V{12}}$
and its analogues for $\pi_{\V{23}}$ and so on; let $\Delta$ be obtained by
tensoring together all these expressions for all $ij\in\tetraE-\Set{\V{14}}$.
Finally, let $Y \defeq \int f' \otimes \int f \in
\int_{\pi,\pi'}\pi' \otimes \pi$. Then we may rewrite
\eqref{eqn:kernalK_and_f_fprime2} as a pairing inside
$\Pi^*_{G,\mathrm{big}}$:
\begin{equation} \label{PiBigFormula}
    \Pair{Y \otimes \Delta} {\tilde{\Lambda}^H}_{\Pi^*_{G,\mathrm{big}}}.
\end{equation}

To proceed we must observe an alternative way of writing the tetrahedral symbol.

\begin{lemma} \label[lemma]{alternateformula3}
    Notation as above, we can rewrite \eqref{PiBigFormula} 
    \begin{equation}\label{alternateformula2}
        ( Y \otimes \Delta,\tilde{\Lambda}^H)_{\Pi^*_{G,\mathrm{big}}}
        = \int_{\pi}  \Pair{f(\pi)}{f'(\pi)}  \Pisymbol 
    \end{equation}
    where $\Pisymbol$ takes as arguments the fixed $\pi_{ij}$ for $ij \neq \V{14}$,
    and $\pi_{\V{14}}=\pi$.
\end{lemma}

Given this, it is easy to finish the proof: starting at
\eqref{eqn:kernalK_and_f_fprime2} we find
\begin{align}
    \int_{\pi} K^{\pi} 
    \Pair{f(\pi)}{f'(\pi)} = \eqref{PiBigFormula} = \int_{\pi}
    \Pisymbol  \Pair{f(\pi)}{f'(\pi)},
\end{align}
and, this being valid for all choices of $f,f'$,  we
see that $\Pisymbol$ coincides with $K^{\pi}$, which is the kernel of an isometric
isomorphism of $L^2(A(\pi))$ and $L^2(B(\pi))$, as desired. 

\begin{proof}[Proof of \Cref{alternateformula3}]
    First let us prove an easier-to-grasp version. Suppose that we can find a
    functional $\lambda^D$ on $\Pi_G$ such that
    \begin{equation}
        \label{unaverage}
        \int_{\delta \in D \cap H} \delta \cdot \lambda^D=\Lambda^D,
    \end{equation}
    where we can think of $\lambda^D$ as an un-averaging of $\Lambda^D$. Then we
    have, simply, 
    \begin{equation}\label{alternateformula}
        \Pisymbol = (\Lambda^H,\lambda^D),
    \end{equation}
    where the right hand side  means $\lim_{i} \Lambda^H(v_i)$ when $\lambda^D =
    \lim_i (v_i, -)$.
    To prove \eqref{alternateformula} we choose $v_H \in \Pi_G$ with the
    property that $\Lambda^H(v_H) =1$ and successively rewrite $\Pisymbol$ via
    \begin{align}
        \Pisymbol=  (\Lambda^H)'(v_H)
            &= \int_{h \in   D \cap H\backslash H}  \Lambda^D(h v_H)
            = \int_{h \in H} \lambda^D(h v_H) \\
            &= \int_{H} \lim_{i} (v_i, h v_H)
            = \lim_{i}  \int_{H} (v_i, h v_H)
            = \lim_{i} \Lambda^H(v_i).
    \end{align}

    Write $u(\pi) = \langle f(\pi), f'(\pi)\rangle$,
    so that the desired \eqref{alternateformula2}
    can be written as
    $\int_{\pi} u(\pi) \Pisymbol = ( \tilde{\Lambda}^H, Y \otimes \Delta)  $.
    Note that this is a version of  \eqref{alternateformula}
    with $\Pi_G$ replaced by $\Pi_{G, \mathrm{big}}^*$ and
    with $Y \otimes \Delta$ playing the role of $\lambda^D$, and 
    we  are going to use similar reasoning to prove it. 
    Note, first of all, the following analogue of \eqref{unaverage}:
    \begin{equation} \label{unaverage2}
        \int_{\delta \in D \cap H} \delta \cdot (Y \otimes \Delta)
        = \int_{\pi} u(\pi) \Lambda^D.
    \end{equation}
    where  we regard both sides as functionals on $\Pi^*_{G,\mathrm{big}}$;
    to make sense of the right hand side one first restricts to the locus
    where $\pi=\pi'$  so that the functional $\Lambda^D$ of \Cref{Tetrahedraldatum} makes sense;
    this construction, which involves restriction to a measure zero set
    inside the unitary dual,
    only really makes sense on a certain dense
    subset of vectors of $\Pi^*_{G, \mathrm{big}}$.
    The equation \eqref{unaverage2} is a consequence of the Plancherel
    formula, which gives the analogue of the ``Schur orthogonality relations'' in
    the current context. \footnote{  The point is that the  averaged vector $\int_{h \in \RO} h
    \cdot Y$ represents, on $\int_{\pi, \pi'} \pi \otimes \pi'$, the functional that
    corresponds to restricting to the diagonal, contracting, and integrating against
    $\Pair{f}{f'}$ times the Plancherel measure. }

    To prove \eqref{alternateformula2} from here, we  repeat
    the reasoning of \eqref{unaverage} but replacing $v_H$ now by a
    $(\pi,\pi')$-dependent family of vectors  $v_H(\pi,\pi') \in \Pi_G^*$, chosen
    to have the property that $\Lambda^{H*}(v_H)=1$ for all $\pi,\pi'$; write
    $\mathbf{v}_H$ for the integrated vector $\int v_H(\pi, \pi') \in
    \Pi^*_{G,\mathrm{big}}$. The result is
    \begin{align}
        \int_{h \in H} (Y \otimes \Delta, h \mathbf{v}_H)_{\Pi^*_{G,\mathrm{big}}}
        \stackrel{\eqref{unaverage2}}{=}
        \int_{D \cap H \backslash H} \int_{\pi} u(\pi) (\Lambda^D, h v_H(\pi, \pi))
        = \int_{\pi} u(\pi) \Pisymbol.
    \end{align}
    This finishes the proof of \eqref{alternateformula2} and so also the Proposition.
\end{proof}

\subsection{Relationship between geometric representation theory
and \Cref{thm:main_duality_theorem}, a sketch}
\label{section:geometric_to_numerical}

We follow notation as in \Cref{sheaves}, and will write as usual
$F=\mathbb{F}_q((t)) \supset \cO=\mathbb{F}_q[[t]]$ for the
analogues of $\FRf,\FRo$ over $\mathbb{F}_q$. We write for this argument
\begin{align}
    \alpha = (1-q^{-2}) = L(2)^{-1}.
\end{align}

Let us consider, as in \eqref{JSymbolDiagram}, integration over $(H \cap D)
\backslash H$ as defining an averaging intertwiner $\mathrm{Av}$ from $X=D
\backslash G $ to $Y=H \backslash G$. We will use the normalized version
\begin{align}
    I_{\mathrm{aut}}= \alpha^3 \mathrm{Av},
\end{align}
which comes from the fact that it is geometrically natural to use the
``point-counting'' measure where $\PGL_2(\cO)$ has mass $1-q^{-2}$,
coming from comparing the point count over $\mathbb{F}_q$ 
to an affine space of the same dimension. We regard $I_{\mathrm{aut}}$
an interwiner from functions on $X_{F}/G_{\cO}$ to
$Y_{F}/G_{\cO}$. We will give  these spaces
$X_{F}/G_{\cO}$ and $Y_{F}/G_{\cO}$
also the point-counting measures; for example the measure of $X_{\cO}/G_{\cO}$ is
then equal to $\alpha^{-6}$, whereas the measure of $Y_{\cO}/G_{\cO}$ is equal to
$\alpha^{-4}$.

Let $\check{D}_0$ be a maximal compact subgroup of the dual group of $\check{D}$,
which we regard as embedded in $\check{G}$
in the diagonal fashion. 
For any $\sigma \in \check{D}_0$
let $\varphi^X_{\sigma}$ be the corresponding  spherical function
on  $X_{F}/G_{\cO}$,  normalized so that its value at the identity equals $1$
--- that is to say, the unique  function on $X_F/G_{\cO}$
that transforms according to the unramified representation of $G_F$
parameterized by $\sigma$, and has value $1$ at the identity.  
 Similarly, let $\varphi^Y_\sigma$ be the  spherical function on
$Y_{F}/G_{\cO}$ with the same parameter $\sigma$,
normalized so that its value at the identity coset equals, instead,
$\alpha^{-8}\sqrt{\frac{L(\frac{1}{2}, \Sigma)}{L(1, \check{\mathfrak{g}})}}$, as in
\eqref{edge equation 2}.
By definition,
\begin{align}
    I_{\mathrm{aut}} \varphi^X_{\sigma} =  \alpha^3 \Pisymbol \varphi^Y_{\sigma},
\end{align} 
where $\Pi=\Pi(\sigma)$ is the unramified representation with parameter $\sigma$.
By the Plancherel formula,  we get\footnote{
{To compute the normalization factor of $\alpha^{12}$ that appears here,
perhaps the simplest way is to evaluate both sides at $1$, and use the 
equality
$\int_{\sigma} \sqrt{L(1,\check{\mathfrak{g}})} = \bigl( \int_{\tau \in \mathrm{SU_2}} 
\det(1-q^{-1} \mathrm{Ad}(\tau))^{-1} \bigr)^6$; the final integral
is readily evaluated to $(1-q^{-2})^{-1}$.}}
\begin{equation} 
    \label{second}
    1_{X_\cO} = 
    \int_{\sigma \in \check{D}_0}
    \Pair{1_{X_{\cO}}}{ \varphi_{\sigma}^X}_{X_{F}/G_{\cO}} \varphi_{\sigma}^X  \cdot
    \alpha^{12} \sqrt{L(1, \check{\mathfrak{g}})}
\end{equation}

Applying $I_{\mathrm{aut}}$ and pairing we find
\begin{align}
    \Pair{I_{\mathrm{aut}}  1_{X_\cO}}{1_{Y_\cO}}
    &=
    \int_{\check{D}_0} \alpha^3 \Pisymbol  \underbrace{
    \langle
1_{X_\cO}, \varphi_{\sigma}^X \rangle_{X_F/G_\cO}  \langle \varphi_{\sigma}^Y,1_{Y_\cO} \rangle_{Y_F/G_\cO}}_{\alpha^{-10} \varphi^X_{\sigma}(1)\varphi^Y_{\sigma}(1)} \alpha^{12}
    \sqrt{L(1, \check{\mathfrak{g}})}  \\
    &=\alpha^{-3} \int_{\check{D}_0} \Pisymbol \sqrt{L\Bigl(\frac{1}{2}, \Sigma\Bigr)}.
\end{align}
where the exponent $-3$ arises from $3-10-8+12$. There is a similar formula with
a Hecke operator inserted.

Let us compare this with what we get from the conjecture enunciated around
\eqref{push-pull-spectral}. Computing $\Hom$-spaces\footnote{Now, if we apply
    $I_{\mathrm{spec}}$ to the structure sheaf on $\check{M}/\check{G}$,
    the result is simply the structure sheaf of $L/\check{G}$, but considered
    as a sheaf on $\check{N}/\check{G}$. Therefore 
    \begin{align}
        \Hom(\mathcal{O}_{\check{N}}, I_{\mathrm{spec}} \mathcal{O}_{\check{M}})
        \simeq k[\check{L}]^{\check{G}} 
        \simeq k[P]^{\check{D}}
    \end{align}
    where we used \eqref{L As Induction}. Pass this equivalence to the
    automorphic side. The structure sheaves above correspond to constant sheaves
    on $X_\cO$ and $Y_\cO$, and so the above formula computes
    $\Hom(1_{X_\cO}, I_{\mathrm{aut}} 1_{Y_\cO})$.
}
and then passing to Frobenius trace, we find
\begin{equation}  \label{first}
    \langle  I_{\mathrm{aut}}  1_{X_\cO},1_{Y_\cO} \rangle
    = \Tr(q^{-\OneHalf} \sigma, k[P]^{\check{D}})
    = \int_{\sigma \in \check{D}_0}
    \mbox{character of $q^{-\OneHalf} \sigma$ on $k[P]$}.
\end{equation}
Here  $\check{D}_0$ is a maximal compact subgroup of $\check{D}$,  and the
$q^{-\OneHalf}$ arises from interpreting the effect of shearing. Again,
there is a corresponding formula with a Hecke operator inserted.
Comparing \eqref{second} and \eqref{first},
and moreover the versions with Hecke modifications, one  
can identify the integrands, and not merely the integrals;
so we find 
\begin{align}
    \Pisymbol \cdot \sqrt{L\Bigl(\frac{1}{2}, \Sigma\Bigr)}
    = \alpha^3 \cdot \mbox{character of $k[P]$ at $q^{-\OneHalf} \sigma$}
\end{align}
which agrees with \Cref{thm:main_duality_theorem}.
\appendix

\bibliographystyle{alpha}
\bibliography{main}

\end{document}